%% file: main.tex
\documentclass{article}

\input{preamble}

\title{Deformations of harmonic maps with conical singularities}
\author{Dominik Gutwein\textsuperscript{$*$} and Thibault Langlais\textsuperscript{$\dagger$}}
\date{}

\begin{document}

\maketitle

\footnotetext[1]{Universität Hamburg, Germany. Email: \texttt{dominik.gutwein[]uni-hamburg.de}.}
\footnotetext[2]{Humboldt-Universität, Berlin, Germany. Email: \texttt{thibault.langlais[]hu-berlin.de}.}

\begin{abstract}
    We study the deformation theory of harmonic maps with isolated singularities between compact Riemannian manifolds, in the case where all the tangent maps are smooth (away from the origin) and the decay to the tangents is polynomial. We will refer to these as conically singular harmonic maps. Under certain conditions on the Morse index and the nullity of the tangent maps and a non-degeneracy assumption on a weighted kernel of the Jacobi operator, we prove that such harmonic maps persist under a small $C^2$-perturbation of the background metric.

    In the second part of the paper, we construct examples of conically singular harmonic maps satisfying the assumptions of our deformation theorem. We first prove that the desired conditions on the tangent maps are satisfied by the radial projections $\R^m \setminus \{0\} \to \Sp^{m-1}$ ($m \geq 4$) and the complex Hopf fibrations $\C^{n+1} \setminus \{0\} \to \mathbb{CP}^{n}$ ($n \geq 1$). We then construct explicit examples of conically singular maps $\Sp^{m} \setminus \{N,S\} \to \Sp^{m-1}$, $\Sp^{4} \setminus \{N,S\} \to \Sp^2$ and $\mathbb{CP}^2 \setminus \{[0:0:1]\} \to \mathbb{CP}^1$ which are harmonic and satisfy the assumptions of the deformation theorem with respect to an appropriate metric on the domain and the round metric on the target.
\end{abstract}

\vspace{1cm}
\tableofcontents

\setcounter{footnote}{0}
\renewcommand{\thefootnote}{\arabic{footnote}}

\newpage

\input{section1_introduction}
\input{section2_conical}

\input{section3_banach}
\input{section4_proof}

\input{section5_modelmaps}
\input{section6_examples}

\appendix

\input{appendix_stationary}

\small

\addcontentsline{toc}{section}{References}

\bibliographystyle{amsplain}
\bibliography{bibliography_harmonic}

\end{document}

%% file: preamble.tex
\usepackage[english]{babel}	
\usepackage[T1]{fontenc}		
\usepackage[utf8]{inputenc}	
\usepackage{amsmath}
\usepackage{amsthm}	
\usepackage{amssymb}
\usepackage{mathrsfs}
\usepackage{appendix}
\usepackage{mathtools}
\usepackage{paralist}	
\usepackage{geometry}			
\usepackage[colorlinks = true,
            linktocpage = true,
            linkcolor = blue,
            urlcolor  = blue,
            citecolor = blue,
            anchorcolor = blue]{hyperref}
\usepackage[capitalize,nameinlink,noabbrev]{cleveref}
\usepackage[hang,flushmargin,symbol,bottom]{footmisc}

\numberwithin{equation}{section}        
\newtheoremstyle{bfnote}                
	{}{}
	{\itshape}{}
	{\bfseries}{.}
	{ }
	{\thmname{#1}\thmnumber{ #2}\thmnote{ (#3)}}

\newtheoremstyle{bfremark}              
	{}{}
	{}{}
	{\bfseries}{.}
	{ }
	{\thmname{#1}\thmnumber{ #2}\thmnote{ (#3)}}

\theoremstyle{bfnote}

\newtheorem{thm}{Theorem}[section]          

\newtheorem{prop}[thm]{Proposition}         

\newtheorem{lem}[thm]{Lemma}                

\newtheorem{cor}[thm]{Corollary}            

\newtheorem{hyp}[thm]{Assumption}           

\theoremstyle{definition}

\newtheorem{Def}[thm]{Definition}           

\theoremstyle{bfremark}

\newtheorem{rem}[thm]{Remark}               

\newtheorem{ex}[thm]{Example}               

\newcommand{\col}{\colon \thinspace}        

\newcommand{\N}{\mathbb{N}}                 
\newcommand{\Z}{\mathbb{Z}}                 
\newcommand{\R}{\mathbb{R}}                 
\newcommand{\C}{\mathbb{C}}                 
\renewcommand{\P}{\mathbb{P}}               

\newcommand{\Ball}{\mathbb{B}}              
\newcommand{\Sp}{\mathbb{S}}                

\newcommand{\II}{\mathrm{I\!I}}             

\newcommand{\Map}{\mathrm{Map}}             
\newcommand{\Met}{\mathrm{Met}}             

\newcommand{\pr}{\mathrm{pr}}               

\newcommand{\pairl}{\langle \! \langle}     
\newcommand{\pairr}{\rangle \! \rangle}     

\newcommand{\diff}{\mathrm{d}}                      
\newcommand{\bmap}[1]{\boldsymbol{\mathrm{#1}}}     

\DeclareMathOperator{\vol}{vol}                 
\DeclareMathOperator{\tr}{tr}                   
\DeclareMathOperator{\Ric}{Ric}                 
\DeclareMathOperator{\id}{id}                   
\DeclareMathOperator{\im}{im}                   
\DeclareMathOperator{\inj}{inj}                 
\DeclareMathOperator{\Hom}{Hom}                 
\DeclareMathOperator{\End}{End}                 
\DeclareMathOperator{\Ind}{Ind}                 
\DeclareMathOperator{\Spec}{Spec}               
\DeclareMathOperator{\Hess}{Hess}               
\DeclareMathOperator{\divergence}{div}          


%% file: section1_introduction.tex
    \section{Introduction and background}

    \subsection{Context and main results}

A harmonic map between Riemannian manifolds $(M,g)$ and $(N,h)$ is by definition a critical point of the Dirichlet energy functional
\begin{equation*}
    E(\phi) \coloneqq \frac{1}{2} \int_M \vert \diff \phi \vert^2 \vol_g
\end{equation*}
naturally defined for maps of regularity $W^{1,2}(M,N)$. While \emph{weakly} harmonic maps have poor regularity properties in general (see \cite{Riviere1992DiscontHarmonicMaps}, which constructs weakly harmonic maps that are nowhere continuous), it is known that \emph{stationary} maps are smooth on the complement of a set $S\subset M$ of vanishing $(m-2)$-dimensional Hausdorff measure \cite{Bethuel1993Regularity,Evans1991Regularity} (where $m= \dim(M)$). Moreover, the singular set of an \emph{energy-minimising} harmonic map is of Hausdorff codimension at least 3 \cite{SchoenUhlenbeck1982RegularityHarmonicMaps}. See also \cite{chang1999regularity,chanillo1991sobolev,lin1999gradient,Naber2017RegularityHarmonicMaps,riviere2007partial} for a non-exhaustive list of related results, as well as the surveys \cite{hardt1997singularities,helein2007harmonic} and references therein.

In order to analyse the behaviour of a stationary harmonic map around its singular set $S$, one typically studies sequences of rescalings of the map around a point $s \in S$. By the monotonicity formula, such sequences of blow-ups sub-converge to a radially-invariant harmonic map $\phi_{s} \col \mathbb{R}^m\setminus S_s \to N$, a so-called \emph{tangent map}. A celebrated result by Simon \cite{simon1983asymptotics,simon1993theorems} (see also \cite{CaniatoParise2025logepiperimetric} for an alternative proof) states that if $\phi\in W^{1,2}(M,N)$ is energy-minimising with $N$ analytical and $s \in S$ is an isolated singularity such that at least one of the tangent maps $\phi_s$ is smooth outside the origin (i.e. $S_s=\{0\}$), then the tangent map $\phi_s$ at $s \in S$ is unique, i.e. independent of the choice of blow-up sequence. In fact, under the additional hypothesis that all infinitesimal harmonic deformations of $\varphi_s \coloneqq \left. \phi_s \right|_{\Sp^{m-1}}$ are integrable, \cite{adams1988rates} proved that the convergence rate to the tangent map $\phi_s$ is polynomial. We will call such harmonic maps with isolated singularities, where all the tangent maps are smooth (away from the origin) and the decay to the tangents is polynomial, \emph{conically singular} (cf. \cref{Def: Conically singular maps}).

The goal of the present paper is to study the deformation theory of conically singular harmonic maps. In the smooth setting, a classical result of Eells--Lemaire \cite{eells1981deformations} states that if $\phi \col (M,g) \to (N,h)$ is a smooth harmonic map between compact Riemannian manifolds which has no non-trivial Jacobi fields, then for any sufficiently small perturbations $\tilde{g}$ of $g$ and $\tilde{h}$ of $h$ there exists a smooth harmonic map $\tilde{\phi} \col (M,\tilde{g}) \to (N,\tilde{h})$ close to $\phi$. Our first main result exhibits sufficient conditions for a conically singular harmonic map to satisfy a similar persistence result with respect to a variation of the domain metric. More precisely, we prove:

\begin{thm}     \label{thm: main theorem introduction}
    Let $(M^m,g)$, $(N^n,h)$ be compact Riemannian manifolds such that $m = \dim(M) \geq 4$ and $n = \dim(N) \geq 2$ and let $S \subset M$ be a finite subset. Let $\phi \col (M \setminus S, g) \to (N,h)$ be a conically singular harmonic map with tangent maps $\{ \phi_s \col \R^m \setminus \{0\} \to N\}_{s \in S}$. We further assume that for any $s \in S$, $\varphi_s \coloneqq \left. \phi_s \right|_{\Sp^{m-1}} \col (\Sp^{m-1},g_{\Sp^{m-1}}) \to (N,h)$ is a smooth harmonic map satisfying:
    \begin{enumerate}
        \item $-(m-3)$ and $0$ are the only non-positive eigenvalues of the Jacobi operator $J^{\varphi_s}$.
        \item $\dim \ker(J^{\varphi_s} + m-3) = m$.
        \item All the Jacobi fields of $J^{\varphi_s}$ are integrable.
    \end{enumerate}
    Suppose moreover that $\phi$ satisfies the following non-degeneracy condition:
    \begin{enumerate}
        \item[4.] Any Jacobi field $u \in \ker(J^\phi_g)$ of $\phi$ satisfying
        \begin{equation*}
            |u| = \mathcal{O}(r^{-1}),
        \end{equation*}
        where $r \col M \to [0,+\infty)$ is the distance to the singular set $S \subset M$, is of the form $u = \phi^* \xi$ where $\xi \in \mathfrak{X}(N)$ is a Killing field for the metric $h$. Furthermore, we have $\dim \phi^*\mathfrak{iso}(N) = \dim \mathfrak{iso}(N)$.
    \end{enumerate}
    Then there exists $\varepsilon > 0$ such that if $\tilde{g}$ is a Riemannian metric on $M$ with $\|\tilde{g} - g\|_{C^2} < \varepsilon$, there exists a finite subset $\tilde{S} \subset M$ with $|\tilde{S}| = |S|$ and a harmonic map $\tilde{\phi} \col (M \setminus \tilde{S}, \tilde{g}) \to (N,h)$ conically singular along $\tilde{S}$. 
\end{thm}

\begin{rem}
    The set $\tilde{S}$ and the map $\tilde{\phi}$ in the previous theorem are not arbitrary: they can be chosen close to the original $S$ and $\phi$ in a sense that will be made precise in \cref{subsec: the deformation theorem}.
\end{rem}

Note that the first three conditions only concern the tangent maps of $\phi$. In particular, conditions 1 and 2 impose strong restrictions on their Morse index. Indeed, for any non-trivial \emph{smooth} harmonic map $\varphi \col \Sp^{m-1} \to N$, the sections of the form $\diff \varphi(v^\perp)$ (where $v \in \R^m$ and $v^\perp \in \mathfrak{X}(\Sp^{m-1})$ is its orthogonal projection onto $T\Sp^{m-1}$) span an $m$-dimensional subspace of $\ker(J^\varphi + m-3)$ \cite{elsoufi1995indice}. Hence the first two conditions ensure that this is the only source of negative eigenvalues for the tangent maps of $\phi$, and are equivalent to requiring that their Morse index is exactly $m$. The third condition has the following precise meaning: for any $s \in S$, there exists a finite-dimensional embedded submanifold $\mathcal{M}_s \subset C^\infty(\Sp^{m-1},N)$ consisting of harmonic maps (with respect to the round metric on $\Sp^{m-1}$ and the metric $h$ on $N$) such that $\varphi_s \in \mathcal{M}_s$ and $T_{\varphi_s} \mathcal{M}_s = \ker(J^{\varphi_s})$. Such integrability conditions on the tangent maps are very common in this type of problems; for instance, this is one of the conditions under which Adams--Simons \cite{adams1988rates} prove polynomial convergence to the tangent map at an isolated singularity. 

By contrast, the last condition genuinely sees both the map $\phi$ and the underlying metric $g$ on $M$. Its role is to ensure that the linearised deformation operator -- which in this case corresponds to the Jacobi operator of $(\phi,g)$ -- induces an invertible map between appropriate Banach spaces. Remark that in the smooth case, the deformation theorem of Eells--Lemaire assumes that the Jacobi operator does not admit any non-trivial Jacobi fields, while in the singular case there are asymptotic conditions to take into account near the singularities. 

The proof of \cref{thm: main theorem introduction} relies on the Implicit Function Theorem. In order to apply it, we will define in \cref{sec: Banach manifolds} suitable Banach manifolds $\Map^{k,p}_{\mu}(M,N;\mathfrak{S})$ consisting of conically singular maps $\phi \col M\setminus S\to N$ of Sobolev regularity $W^{k,p}$, where $\mu > 0$ is a small exponent controlling the rate of decay towards the tangent maps, and $\mathfrak{S}$ is a data set consisting of the singular set $S \subset M$ (which is fixed) and a finite collection of submanifolds $\mathcal{M}_s \subset C^{\infty}(\Sp^{m-1},N)$ parametrising the possible tangent maps. The Euler--Lagrange operator of the harmonic map equation (i.e. the tension field) can then be interpreted as a differentiable section $\mathcal{T}$ of a suitable Banach vector bundle lying over the entire parameter space $\Map^{k,p}_{\mu}(M,N;\mathfrak{S}) \times \Ball_1^{S} \times \Met^k(M)$, where $\Ball_1^S \subset (\R^m)^S$ parametrises the local deformations of the singular set and $\Met^k(M)$ is the space of Riemannian metrics of regularity $C^k$ on $M$. By definition, the zero locus of $\mathcal{T}$ consists precisely of the conically singular harmonic maps in this parameter space. 

In order to prove that the linearisation of $\mathcal{T}$ is surjective at a given harmonic map, we first analyse the derivative of $\mathcal{T}$ with respect to a fixed metric and singularity data. This amounts to studying the Jacobi operator of a fixed singular harmonic map as an operator between weighted function spaces. For any sufficiently small rate $\mu > 0$, this is always a Fredholm operator of negative Fredholm index. In particular, this operator can never be surjective. However, we show that under conditions 1--3, the deformations of the singular set and of the tangent maps can be designed in order fix this defect, so that the Fredholm section $\mathcal{T}$ has index zero. The last assumption in \cref{thm: main theorem introduction} (condition 4) then ensures that its linearisation at a given harmonic map is invertible: in that sense, condition 4 is a \emph{non-degeneracy} condition. Since the Jacobi operator is formally self-adjoint, one might expect that this condition is satisfied for sufficiently generic pairs $(g,\phi)$ consisting of a metric on $M$ and a conically singular harmonic map $\phi \col (M,g)\to (N,h)$; this observation underlies our construction of explicit examples, which we describe next.

\begin{rem}
    Let us briefly explain the origin of the conditions $m = \dim(M) \geq 4$ and the $C^2$-closeness required in the previous theorem.

    The condition $\dim(M) \geq 4$ is natural, since in that case any conically singular harmonic map $\phi \col M \setminus S \to N$ can be seen as a $W^{1,2}$-map from $M$ to $N$, which is weakly harmonic and even stationary (cf. \cref{prop: equivalent conditions cs harmonic} and \cref{app: cs as stationary maps}). By contrast, if $m \leq 2$ a conically singular map always has infinite Dirichlet energy, so we cannot talk about a conically singular harmonic map in a meaningful sense. In the borderline case where $m = 3$, conically singular maps do have finite energy and a conically singular map which is harmonic on its smooth locus can still be regarded as a weakly harmonic map (in the $W^{1,2}$-sense). While it is possible to set up the deformation problem when $m = 3$, it turns out that the Jacobi operator always has many Jacobi fields of order $\mathcal{O}(r^{-1})$ besides the ones coming from the lifts of Killing fields, and therefore we cannot write a geometrically meaningful condition ensuring that the linearised deformation operator is invertible between the appropriate weighted Sobolev spaces.

    On the other hand, the assumption of $C^2$-closeness of the background metrics on $M$ is stronger than the assumption of the deformation theorem of Eells--Lemaire in the smooth case (which only requires $C^{1,\alpha}$-closeness). Ultimately, it comes from the necessity to move around the singularities in order to make the deformation problem invertible for conically singular harmonic maps. Indeed, we shall see in \cref{sec: Banach manifolds} that it is necessary to control at least two derivatives of the metric in order to prove that the Banach section induced by the tension field is continuously differentiable across the whole parameter space (see \cref{rem: derivatives of the metric}).
\end{rem}

The second part of this article is dedicated to the construction of conically singular harmonic maps satisfying the assumptions of \cref{thm: main theorem introduction}. We first study possible models for the tangent maps satisfying assumptions 1--3. The simplest example is to consider the identity map $\id \col \Sp^{m-1} \to \Sp^{m-1}$ for any $m \geq 4$, where we equip $\Sp^{m-1}$ with the standard round metric on the domain and the target. In this case, the study of the spectrum of the Jacobi operator essentially boils down to studying the Hodge-de Rham Laplacian on $1$-forms, and properties 1--3 are satisfied. A more interesting example is the standard Hopf fibration $\varphi \col \Sp^3 \to \Sp^2$, where we again equip the spheres with the round metrics. In this case, the description of the spectrum of $\varphi$ is due to Urakawa \cite{urakawa1987stability}, who notably proved that it has Morse index $4$ and nullity $8$. Moreover, all the Jacobi fields of $\varphi$ can be explicitly integrated by pre-composing $\varphi$ with an isometry of $\Sp^3$ and post-composing $\varphi$ by a biholomorphism of $\Sp^2 \simeq \mathbb{CP}^1$, so that properties 1--3 are satisfied. Our second main result extends this to the higher-dimensional complex Hopf fibrations $\Sp^{2n+1} \to \mathbb{CP}^n$, which in particular all satisfy condition 1--3:

\begin{thm}     \label{thm: hopf fibrations intro}
    For any $n \geq 1$, the standard Hopf fibration $\varphi \col \Sp^{2n+1} \to \mathbb{CP}^n$ has Morse index 
    \begin{equation*}
        \textup{Ind}(\varphi) = 2(n+1)
    \end{equation*} 
    and nullity 
    \begin{equation*}
        \textup{Null}(\varphi) = n(3n+5)
    \end{equation*}
    with respect to the round metric on $\Sp^{2n+1}$ and the Fubini--Study metric on $\mathbb{CP}^n$. Moreover, all the Jacobi fields of $\varphi$ are integrable, and the harmonic deformations of $\varphi$ can be obtained by pre-composing $\varphi$ with an isometry of $\Sp^{2n+1}$ and post-composing $\varphi$ with a biholomorphism of $\mathbb{CP}^n$.
\end{thm}

\begin{rem}
    Recent work of Rivière \cite{Riviere2023HarmonicMapsS3toS2} proves that, up to pre-composing with an isometry of $\Sp^3$ and post-composing with a holomorphic map of $\mathbb{CP}^1$ to itself, the Hopf fibration $\varphi\col \Sp^3 \to \Sp^2$ is the only harmonic map $\varphi \col \Sp^3 \to \Sp^2$ with Morse index less or equal to $4$. Thus, for a conically singular harmonic map $\phi \col M^4\setminus S \to \Sp^2 \simeq \mathbb{CP}^1$ (where we equip $\Sp^2$ with the round metric) to satisfy the first two assumptions of \cref{thm: main theorem introduction}, we need all the tangent maps to be given by the Hopf fibration, possibly up to post-composition with a holomorphic map $\mathbb{CP}^1 \to \mathbb{CP}^1$. 
    
    The authors were also informed of upcoming work of Lara and Loubeau \cite{laraloubeau1} proving the analogous result for all complex Hopf fibrations $\Sp^{2n+1} \to \mathbb{CP}^n$, and in particular recovering the value $2(n+1)$ for the Morse index. However, they also prove that the index of the quaternionic Hopf fibration $\Sp^7 \to \Sp^4$ is strictly greater than $8$ \cite{laraloubeau2}, so that conditions 1--3 in \cref{thm: main theorem introduction} can fail even in very simple cases.
\end{rem}

We then turn to construct explicit examples of singular harmonic maps to which \cref{thm: main theorem introduction} applies. The main difficulty hereby comes from the fourth condition, which can be exemplified in the following situation. Consider the rational map
\begin{equation*}
    \phi_1 \col \mathbb{CP}^2\setminus \{[0:0:1]\} \to \mathbb{CP}^1, ~ [z_0:z_1:z_2] \mapsto [z_0:z_1]. 
\end{equation*}
If both projective spaces are equipped with their respective Fubini--Study metrics, then $\phi$ is harmonic (as a holomorphic map between Kähler manifolds \cite{eells1964harmonic}) and has an isolated conical singularity at $[0:0:1]$ modelled on the Hopf fibration $\Sp^3 \to \mathbb{CP}^1 \simeq \Sp^2$. Thus, the first three conditions are satisfied by the previous discussion. However, since both manifolds are Kähler, every infinitesimal biholomorphism on $\mathbb{CP}^1$ contributes a Jacobi field that is of order $\mathcal{O}(r^0)$ close to $[0:0:1]$. In particular, there are Jacobi fields that are of order $\mathcal{O}(r^{-1})$ close to the singular point that are not given by the pull-back of a Killing vector field. Thus, for harmonic maps arising from rational (holomorphic) maps $(M,g) \to \mathbb{CP}^1$ from a Kähler manifold, condition 4 has no chance of being satisfied. 

A similar issue occurs for other simple examples of conically singular harmonic maps. For instance, the equatorial projection $\phi_2 \col \Sp^{m} \setminus \{N,S\} \to \Sp^{m-1}$ ($m \geq 4$) is a conically singular harmonic map with respect to the round metrics, with conical singularities at the poles modelled on the identity map $\id \col \Sp^{m-1} \to \Sp^{m-1}$. When $m = 4$, we can also compose this map with the Hopf fibration $\Sp^3 \to \Sp^2$ to obtain a conically singular harmonic map $\Sp^4 \setminus \{N,S\} \to \Sp^2$. In these cases, pre-composing $\phi_i$ with arbitrary isometries of $\Sp^{m}$ yield Jacobi fields of order $\mathcal{O}(r^{-1})$ which are not the pull-back of Jacobi fields on the target. In this sense, the situation is not generic enough for condition 4 to be satisfied.

Nevertheless, we shall prove the following in \cref{sec: Examples from deformed suspensions}:

\begin{thm}     \label{thm: examples intro}
    For each of the conically singular maps
    \begin{equation*}
        \phi_1 \col \mathbb{CP}^2 \setminus \{[0:0:1]\} \to \mathbb{CP}^1, ~~ \phi_2 \col \Sp^m \setminus \{N,S\} \to \Sp^{m-1} ~ (m \geq 4), ~~ \phi_3 \col \Sp^4 \setminus \{N,S\} \to \Sp^2
    \end{equation*}
    described above, there exists a metric $g$ on the domain such that $\phi_i$ is harmonic (with respect to $g$ on the domain and the round metric on the target) and all the assumptions of \cref{thm: main theorem introduction} are satisfied.
\end{thm}

The idea to prove the above theorem is to find a metric on the domain with respect to which $\phi_i$ is still harmonic, and which is sufficiently symmetric that we can understand the Jacobi fields but sufficiently non-generic for condition 4 in \cref{thm: main theorem introduction} to be satisfied. Indeed, when the metric on the domain is homogeneous, this condition cannot hold as we previously pointed out. We will therefore consider metric which are invariant under a \emph{cohomogeneity-one} action of a Lie group on the domain.

More specifically, our strategy of proof will be to equip the domain manifold in each case (i.e. $\mathbb{CP}^2$, $\Sp^m$, and $\Sp^4$, respectively) with a 1-parameter family of metrics $g_T$ that satisfy the following properties:
\begin{enumerate}[(i)]
    \item The maps $\phi_i$ are still harmonic with respect to $g_T$ (and the round metric on on the target).
    \item The metrics $g_T$ are all invariant under a natural cohomogeneity-one group action on the domain.
    \item There exists an open region $U$ inside each domain such that $(U,g_T)$ is isometric to the cylinder $(-T,T) \times \Sp^3$ (for $\phi_1$ and $\phi_3$) and $(-T,T) \times \Sp^{m-1}$ (for $\phi_2$). 
\end{enumerate}
Property (ii) ensures that the Jacobi fields of $\phi_i$ with respect to $g_T$ are given as solutions of countably many finite-dimensional ODE systems. We then consider the space of solutions of these ODEs over the cylindrical region $(-T,T) \times \Sp^3$ (or $(-T,T) \times \Sp^{m-1}$, respectively). This solution space has two subspaces that are relevant for our purpose: the space of solutions that have the appropriate decay when extended past the left end of the cylinder and the space of solutions that have the appropriate decay when extended past the right end of the cylinder. Checking condition 4 in \cref{thm: main theorem introduction} then becomes a `matching problem': we want these two subspaces to intersect precisely along the space spanned by the pull-backs of Killing vector fields. For the map $\phi_2$, it turns out that this matching condition is satisfied for an open dense subset of values $T \in [T_0,\infty)$ as long as $T_0$ is sufficiently large. For the maps $\phi_1$ and $\phi_3$ however, we need to introduce a further small $\Sp^3$-invariant deformation $g_{\varepsilon,T}$ of the metric in the cylindrical neck region in order to get rid of the Jacobi fields induced by the (non-Killing) holomorphic vector fields on $\Sp^2$. Through a careful ODE analysis we verify that in these two examples there again exists an open dense set of values $T \in [T_0,\infty)$ (for $T_0$ sufficiently large) such the matching condition is satisfied with respect to the metric $g_{\varepsilon,T}$ for any small enough $\varepsilon > 0$.

\paragraph{Organisation of the paper.} Let us close this introduction with an overview of the contents of the article and its organisation.
\begin{itemize}
    \item The second part of this introductory section (\cref{subsec: background}) recalls some background on harmonic maps and sets the conventions and notations that will be used throughout the paper. \cref{subsubsec: definitions and notations} treats smooth harmonic maps, the tension field and the Jacobi operator, while \cref{subsubsec: cs harmonic maps} introduces maps with conical singularities.
    
    \item \cref{sec: Analysis in weighted spaces} lays the analytical foundations underlying the proof of \cref{thm: main theorem introduction}. We begin with the study of dilation-invariant harmonic maps $\phi \col \R^m \setminus \{0\} \to N$ in \cref{subsec: Dilation-invariant harmonic maps}, where we derive an expression for the Jacobi operator of such maps and introduce the fundamental notions of indicial roots and polyhomogeneous sections. This serves as motivation for the introduction of a class of weighted Sobolev spaces which are well-adapted to the analysis of conically singular problems in \cref{subsec: Weighted spaces}. There, we also define a special class of non-linear bundle maps, called \emph{bundle maps with dilation-invariant ends}, which will play a key role in setting up the non-linear deformation problem. Finally, we use Lockhart--McOwen theory in order to describe the mapping properties of the Jacobi operator of a conically singular harmonic map in \cref{subsec: The Jacobi operator of a CS harmonic map}. In particular, we analyse the critical rates of the Jacobi operator, and introduce a duality pairing which will help us understand how to compensate the obstructions coming from the cokernel of the linearised deformation problem.

    \item \cref{sec: Banach manifolds} is the technical heart of the proof of \cref{thm: main theorem introduction}. We first define Banach manifolds of conically singular maps with Sobolev regularity in \cref{subsec: manifold structure}, which we combine with a finite-dimensional manifold parametrising the deformations of the singular set and the space of Riemannian metrics with finite regularity on the domain in order to form the entire parameter space of our deformation problem. Then in \cref{subsec: variations of the tension} we prove, under certain assumptions on the allowed tangents maps, that the tension field defines a continuously differentiable section of a Banach vector bundle lying over this parameter space.
    
    \item \cref{sec: The deformation theorem} proves \cref{thm: main theorem introduction}, building upon the work of the previous sections. We first explain in \cref{subsec: the tension revisited} how to work transversely to the subspace of the cokernel of the linearised deformation operator generated by the pull-back of the Killing fields of the target space. Then, in \cref{subsec: the deformation theorem}, we state a more precise deformation theorem (\cref{thm: deformation theorem with sobolev regularity}) for $C^k$-variations of the domain metric (where $k \geq 2$), whose proof relies on the Implicit Function Theorem. 
    
    \item \cref{sec: Models for the tangent maps} studies several model harmonic maps satisfying the first three assumptions of \cref{thm: main theorem introduction}. We begin with a general discussion of harmonic fibrations in \cref{subsec: harmonic fibrations}, before describing in \cref{subsec: The identity map} the example of the identity map $\id \col \Sp^{m-1} \to \Sp^{m-1}$ ($m \geq 4$), where the desired properties can be deduced from the classical description of the spectrum of the Lapalcian operator on round spheres. We subsequently consider the Hopf fibration $\Sp^3 \to \Sp^2$ in \cref{subsec: Hopf fibration as model map}, where we give a detailed description of the Jacobi operator and its spectrum with respect to the round metric on $\Sp^3$. Moreover, we analyse the effect of certain symmetric deformations of the background metric which will play a key role in \cref{sec: Examples from deformed suspensions}. Finally, we study the higher-dimensional Hopf fibrations $\Sp^{2n+1} \to \mathbb{CP}^n$ and prove \cref{thm: hopf fibrations intro} in \cref{subsec: higher Hopf fibrations}.
    
    \item \cref{sec: Examples from deformed suspensions} constructs background metrics on the domains of each of the conically singular maps $\phi_i$ described above, with respect to which the maps are harmonic and all the assumptions of \cref{thm: main theorem introduction} are satisfied, thereby proving \cref{thm: examples intro}. We first discuss our general strategy in \cref{subsec: General stategy of examples}, before treating the map $\phi_2$ in \cref{subsec: map S^m to S^{m-1}}, the map $\phi_3$ in \cref{sec: map S^4 to S^2} and $\phi_1$ in \cref{sec: map CP2 to CP1}.
    
    \item Finally, in \cref{app: cs as stationary maps} we prove that a conically singular map $\phi \col M \setminus S \to N$ whose tension vanishes on the regular locus is weakly harmonic if  $m = \dim(M) \geq 3$ and stationary once $m \geq 4$.
\end{itemize}

\paragraph{Acknowledgements.} D.G. would like to thank Simon Donaldson, who suggested that singular harmonic maps whose tangents are modelled on the Hopf fibration should persist under small perturbations of the metric \cite{Donaldson2021talk}, for making him aware of the deformation problem for singular harmonic maps. Moreover, the authors are grateful to \'Eric Loubeau for valuable discussions concerning his upcoming work.  D.G. is supported by the Deutsche Forschungsgemeinschaft (DFG, German Research Foundation) under SFB-Geschäftszeichen 1624 – Projektnummer 506632645. T.L. is supported by a fellowship from the Alexander-von-Humboldt Foundation.

    \subsection{Background}     \label{subsec: background}

In the next section, we gather some background material on harmonic maps. In \cref{subsubsec: definitions and notations}, we give some reminders on the Dirchlet energy, the tension and the Jacobi opertor of smooth maps, and set the notations that will be used throughout this article. Conically singular maps are introduced in \cref{subsubsec: cs harmonic maps}, for which we observe that different notions of harmonicity (weak harmonicity, stationarity and vanishing of the tension over the smooth locus) all coincide as long as the dimension of the domain manifold is greater or equal to $4$.

    \subsubsection{Definitions and notations}       \label{subsubsec: definitions and notations}

Let $\phi \col M^m \rightarrow N^n$ be a smooth map, where $(M,g)$ and $(N,h)$ are Riemannian manifolds. We shall denote by $\prescript{\phi}{}{TN} \rightarrow M$ the vector bundle $\phi^* (TN)$, and by $\prescript{\phi}{}{\nabla}$ the pull-back connection $\phi^*(\nabla^{h})$. Recall that we can regard the tangent map $\diff\phi$ as a section of the vector bundle $T^*M \otimes \prescript{\phi}{}{TN}$, and the local \emph{energy density} $e(\phi)$ is $\frac{1}{2} |\diff\phi|^2$ where the norm is computed with respect to $g$ on $TM$ and $h_\phi \coloneqq \phi^* h$ on $\prescript{\phi}{}{TN}$. The map $\phi$ is said to be \emph{harmonic} if it satisfies the associated Euler--Lagrange equations. 

To describe them, we need to introduce the \emph{tension} $\tau_g(\phi)$ of such a map. If $k \geq 1$, the vector bundle $T^*M^{\otimes k} \otimes \prescript{\phi}{}{TN}$ has a connection denoted by $\prescript{\phi}{}{\nabla}^{g}$, induced by the Levi-Civita connection $\nabla^{g}$ of $g$ on $T^*M^{\otimes k}$ and the connection $\prescript{\phi}{}{\nabla}$ on $\prescript{\phi}{}{TN}$. Thus we may define the \emph{Hessian}
\begin{equation}
    \Hess_g(\phi) \coloneqq \prescript{\phi}{}{\nabla}^{g} \diff\phi
\end{equation}
of $\phi$, which is a section of $T^*M^{\otimes 2} \otimes \prescript{\phi}{}{TN}$. Explicitly, 
\begin{equation*}
    \Hess_g(\phi) = (\phi^*\nabla^{h})_X \diff\phi(Y) - \diff\phi(\nabla^{g}_X Y) 
\end{equation*}
for any vector fields $X$, $Y$ on $M$. Then the tension $\tau(\phi) \in \Gamma(\prescript{\phi}{}{TN})$ is obtained by tracing the Hessian with respect to the metric $g$ (on the first two indices). That is,
\begin{equation*}
    \tau_g(\phi) = \sum_{i=1}^m \prescript{\phi}{}{\nabla}^{g} \diff\phi (e_i,e_i) 
\end{equation*}
where $e_1,\ldots,e_m$ is a local $g$-orthonormal frame of $TM$. Therefore, the harmonic map equation can be written $\tau_g(\phi) = 0$.

When $m$ is compact, harmonic maps satisfy a global variational principle. We can define the \emph{Dirichlet energy} $E(\phi)$ by 
\begin{equation*}
    E(\phi) \coloneqq \frac{1}{2} \int_M |\diff\phi|^2 \vol_g .
\end{equation*}
We can consider this functional as a smooth function of the Fr\'echet manifold $C^\infty(M,N)$ of smooth maps from $M$ to $N$. The tangent space $T_\phi C^\infty(M,N)$ can be identified with the space of sections $C^\infty(\prescript{\phi}{}{TN})$. Then one can compute
\begin{equation*}
    D_\phi E(u) = - \int_M \langle \tau_g(\phi), u \rangle_{h_\phi} \vol_g
\end{equation*}
for any $u \in C^\infty(\prescript{\phi}{}{TN})$. Thus the tension $\tau(\phi)$ is minus the gradient of the Dirichlet energy. 

There is also a second variation formula. Intrinsically (that is, without a choice of connection on the manifold of maps from $M$ to $N$), it can only be defined for a harmonic map. Given such a harmonic map $\phi$, the Hessian $D^2_\phi E$ of the energy functional reads:
\begin{equation*}
    D^2_\phi E(u,v) = \int_M \langle J^\phi_g(u), v \rangle_{h_\phi} \vol_g , ~~~~~ \forall u, v \in C^\infty (\prescript{\phi}{}{TN})
\end{equation*}
where $J^\phi_g \col C^\infty(\prescript{\phi}{}{TN}) \rightarrow C^\infty (\prescript{\phi}{}{TN})$ is the \emph{Jacobi operator} of $\phi$ (with respect to the metric $g$). Sometimes, we will simply use the notation $J^\phi$ if there is no ambiguity on the background metric $g$; on the other hand, the metric $h$ on the target $N$ will always be fixed, so we shall not need to specify it in our notations. 

It is well-known that the Jacobi operator is an elliptic and formally self-adjoint differential operator, which can be written as
\begin{equation}        \label{eq: Jacobi operator in general}
    J^\phi_{g}(u) = - \tr_{g}(\prescript{\phi}{}{\nabla}^{g} \prescript{\phi}{}{\nabla} u) - \tr_{g}(\mathrm{Rm}_h(u, \diff\phi) \diff\phi) = \prescript{\phi}{}{\nabla}^*\prescript{\phi}{}{\nabla} u - \tr_{g}(\mathrm{Rm}_h(u, \diff\phi) \diff\phi)
\end{equation}
for any $u \in C^\infty(\prescript{\phi}{}{TN})$. In the above expression, $\mathrm{Rm}_h$ is the Riemann curvature tensor of the metric $h$, which in our conventions reads
\begin{equation*}
    \mathrm{Rm}_h(U,V)W  = \nabla^h_U \nabla^h_V W - \nabla^h_V \nabla^h_U W - \nabla^h_{[U,V]} W .
\end{equation*} 
Moreover, $\prescript{\phi}{}{\nabla}^*$ is the formal adjoint of the operator $\prescript{\phi}{}{\nabla} \col C^\infty(\prescript{\phi}{}{TN}) \rightarrow C^\infty(T^*M \otimes \prescript{\phi}{}{TN})$ with respect to the metric $g$. The elements of the kernel of $J^\phi_g$ (i.e. the sections $u \in C^\infty(\prescript{\phi}{}{TN})$ such that $J^\phi_g u = 0$) are called \emph{Jacobi fields}.

    \subsubsection{Harmonic maps with isolated conical singularities}       \label{subsubsec: cs harmonic maps}

Henceforth, we shall denote by $\Ball^m_R \subset \R^m$ (or simply $\Ball_R$ if there is no ambiguity on the dimension) the open ball of radius $R$ and centre $0$ in $\R^m$, and by $\overline{\Ball}_R$ its closure. In this section, we shall also consider compact manifolds $M$ and $N$.

\begin{Def}     \label{Def: Adapted set of charts}
    Let $S \subset M$ be a finite subset and let $R > 0$. 
    \begin{enumerate}
        \item A set of charts $\{ \Upsilon_s \col (\Ball_R \subset \R^m) \to M\}_{s \in S}$ will be called \emph{adapted} if:
        \begin{itemize}
            \item for any $s \in S$, $\Upsilon_s(0) = s$, and $\Upsilon_s$ extends to a smooth embedding of $\overline{\Ball}_R \to M$;
            \item for any $s \neq s^\prime \in S$, $\overline{\Upsilon_s(\Ball_R)} \cap \overline{\Upsilon_{s^\prime}(\Ball_R)} = \emptyset$ .
        \end{itemize}

        \item A Riemannian metric $g$ on $M$ and an adapted set of charts $\{ \Upsilon_s \}_{s \in S}$ are said to be \emph{compatible} if for any $s \in S$, $\Upsilon^*_s g(0) = g_{\R^m}$ (where $g_{\R^m}$ denotes the standard Euclidean metric on $\R^m$).
    \end{enumerate}
\end{Def}

\begin{Def}     \label{Def: Conically singular maps}
    Let $S \subset M$ be a finite subset. A smooth map $\phi \col M \setminus S \rightarrow N$ will be called \emph{conically singular} with rate $\mu \in (0,1)$ if there exists the following set of data:
    \begin{itemize}
        \item $R >0$ and an adapted set of charts $\{\Upsilon_s \col \Ball_R \to M\}_{s \in S}$;
        \item a set of smooth dilation-invariant maps $\{ \phi_s \col \R^m \setminus \{0\} \to N \}_{s \in S}$; and
        \item a set of smooth sections $\{ u_s \in C^\infty(\Ball_R \setminus \{0\}, \prescript{\phi_s}{}{TN}) \}_{s \in S}$
    \end{itemize}
    such that for all $s \in S$, 
    \begin{equation*}
        \phi \circ \Upsilon_s(x) = \exp_{\phi_s(x)}(u_s(x)), ~~~  \forall x \in \Ball_R \setminus \{0\},
    \end{equation*}
    and the sections $u_s$ satisfy the decay condition
    \begin{equation*}
        |(\prescript{\phi_s}{}{\nabla})^k u_s | = \mathcal{O}(r^{\mu-k}), ~~~ \forall k \in \N_0
    \end{equation*}
    as $r \to 0$, where $r$ is the radius function on $\R^m$. Note that in the above, $\exp$ is the exponential map associated with the metric $h$ on $N$.
\end{Def}

\begin{rem}     \label{rem: independence of conically singular from choice of framing}
    Since we assumed $\mu \in (0,1)$, it is not difficult to see that the definition of a conically singular map as given above does not depend on the choice of adapted coordinate systems $\{\Upsilon_s \col \Ball_R \to M\}_{s\in S}$ or of background Riemannian metric on $M$. Indeed, if $\phi_0 \col \R^m \setminus \{0\} \to N$ is a smooth dilation-invariant map and $u \in \Gamma(B_R(0) \setminus \{0\}, \prescript{\phi_0}{}{TN})$ is a section that satisfies 
    \begin{equation*}
        \big\vert (\prescript{\phi_0}{}{\nabla})^k u \big\vert = \mathcal{O}(r^{\mu-k}) 
    \end{equation*}
    as $r\to 0$ for every $k \in \mathbb{N}_0$, and $\Psi \col \R^m \to \R^m$ is a diffeomorphism with $\Psi(0) = 0$ and $T_0 \Psi = A \in \mathrm{GL}(m,\R)$, then the function $\phi_A = \phi_0 \circ A$ is a smooth dilation-invariant function on $\R^m \setminus \{0\}$, and for some small enough $r > 0$ there exists a section $v\in \Gamma(\prescript{\phi_A}{}{TN})$ defined over $\Ball_r \setminus \{0\}$ such that
    \begin{equation*}
        \exp_{\phi_0(\Psi(x))} (u(\Psi(x))) = \exp_{\phi_A(x)}(v(x)), ~~~ \forall x \in \Ball_r \setminus \{0\},
    \end{equation*}
    and
    \begin{equation*}
        \big\vert (\prescript{\phi_A}{}{\nabla})^k v \big\vert = \mathcal{O}(r^{\min\{\mu,1\}-k})
    \end{equation*}
    as $r\to 0$ for every $k \in \mathbb{N}$. In particular, this means that if $g$ is a background Riemannian metric on $M$, we are free to choose an adapted coordinate system compatible with $g$.
\end{rem}

For singular maps between Riemannian manifolds which are at least of regularity $W^{1,2}$, the Dirichlet energy is still well defined. Therefore, one can consider different notions of harmonicity (weak harmonicity, stationarity, or vanishing of the tension over the smooth locus for instance), which differ by the type of variations which are allowed in order to test whether a map is a critical point of the Dirichlet energy functional. However, it turns out that when the dimension of the domain manifold is at least $4$, conically singular maps are well-behaved in the sense that these notions all coincide:

\begin{prop}        \label{prop: equivalent conditions cs harmonic}
    Let $(M^m,g)$, $(N^n,h)$ be compact Riemannian manifolds such that $m = \dim(M) \geq 4$ and $n = \dim(N) \geq 1$, and let $S \subset M$ be a finite subset. Then any conically singular map $\phi \col M \setminus S \to N$ can naturally be regarded as a map $\phi \in W^{1,2}(M,N)$. Moreover, the following conditions are equivalent:
    \begin{enumerate}[(i)]
        \item $\tau_g(\phi) = 0$ on $M \setminus S$.
        \item $\phi$ is a weakly harmonic map.
        \item $\phi$ is a stationary map.
    \end{enumerate}
\end{prop}

Since this proposition will not be used in the rest of this article (where we will always use condition (i) to characterise harmonicity), we shall defer its proof to \cref{app: cs as stationary maps}. Nevertheless, it serves to motivate the following definition:

\begin{Def}
    With the notations of \cref{prop: equivalent conditions cs harmonic}, we shall say that $\phi$ is a \emph{conically singular harmonic map} with respect to the Riemannian metric $g$ on $M$ and $h$ on $N$ if it is conically singular in the sense of \cref{Def: Conically singular maps} and the equivalent conditions (i), (ii) and (iii) of the previous proposition are satisfied.
\end{Def}

\begin{rem}
    If $\phi \in W^{1,2}(M,N)$ is a conically singular harmonic map in the sense of the previous definition, with singular set $S$, then in an adapted set of chats $\{\Upsilon_s\}_{s \in S}$ compatible with $g_M$ all the dilation-invariant tangent maps $\phi_s \col \R^m \setminus \{0\} \to N$ are harmonic with respect to the standard Euclidean metric on $\R^m$. Indeed, the map $\phi_s$ is the unique tangent map of $\phi$ at $s$, since the rescaled maps $\phi_{s,\varepsilon} \col \mathbb{B}_{\varepsilon^{-1} R} \rightarrow N$ defined as $\phi_{s,\varepsilon}(x) = \phi \circ \Upsilon_s(\varepsilon x)$ converge to $\varphi$ in $C^\infty_{\mathrm{loc}}$ as $\varepsilon \rightarrow 0$. In particular, if we write $\phi_s(x) = \varphi_s(x/|x|)$ for a smooth map $\varphi_s \col \Sp^{m-1} \to N$, then $\varphi_s$ is a smooth harmonic map with respect to the standard round metric.
\end{rem}

\begin{rem}
    While the assumption in \cref{Def: Conically singular maps} that all the derivatives of each $u_s$ decay polynomially is very strong, it occurs naturally for harmonic maps in a number of contexts. In fact, if the target manifold $(N,h)$ is \emph{analytical} and we assume that for any $s \in S$, all the Jacobi fields of the harmonic map $\varphi_s \col \Sp^{m-1} \rightarrow N$ are integrable, a classical result of Adams--Simon \cite{adams1988rates} shows that it is enough to assume that
    \begin{equation*}
        |u_s | + r|\nabla^{\phi_s} u_s| \to 0 ~~ \text{as} ~ r \to 0
    \end{equation*}
    to get the polynomial decay of all derivatives. On the other hand, and earlier result of Simon \cite{simon1983asymptotics,simon1993theorems} shows that the above condition holds for an energy-minimising harmonic map with isolated singularities, under the assumption that for each singularity, at least one of the tangent maps is smooth (away from $0$ in $\R^m$, where the target manifold is still assumed to be analytical and isometrically embedded into some Euclidean space). Note that when the target space is non-analytical, there are counterexamples to the uniqueness of the tangent maps at an isolated singularity \cite{white1992nonunique}.
\end{rem}

%% file: section2_conical.tex
    \section{Analysis in weighted spaces}       \label{sec: Analysis in weighted spaces}

This sections gathers the fundamental results of weighted analysis which underlie the proof of the deformation theorem in the next sections. We first describe the Jacobi operator of dilation-invariant harmonic maps $\phi \col \R^m \setminus \{0\} \to N$ and their Jacobi fields in \cref{subsec: Dilation-invariant harmonic maps}. This serves as motivation for the introduction of a good class of weighted Sobolev spaces in \cref{subsec: Weighted spaces}, which is well-suited for the analysis of conically singular problems. Finally, we use Lockhart--McOwen theory \cite{Lockhart1985ellipticOperators_on_noncompact_mfds} to describe the mapping properties of the Jacobi operator of a conically singular harmonic map between weighted spaces in \cref{subsec: The Jacobi operator of a CS harmonic map}.

    \subsection{Dilation-invariant harmonic maps}       \label{subsec: Dilation-invariant harmonic maps} 

This section discusses the Jacobi operator of dilation-invariant harmonic maps $\phi \col \R^m \setminus \{0 \} \to N$. In \cref{subsubsec: the jacobi operator of dilation-invariant maps}, we derive an explicit expression in terms of the Jacobi operator of $\varphi = \left. \phi \right|_{\Sp^{m-1}}$, in \cref{subsubsec: roots and polyhomogeneity} we discuss the roots of the Jacobi operator and the notion of polyhomogeneity, and finally in \cref{subsubsec: further remarks} we specialise the discussion to the case of maps satisfying conditions 1 an 2 in \cref{thm: main theorem introduction}.

    \subsubsection{The Jacobi operator of a dilation-invariant harmonic map}    \label{subsubsec: the jacobi operator of dilation-invariant maps}

Let $m \geq 2$ and consider the projection 
\begin{align*}
    \pr \col \R^m \setminus \{0\} & \longrightarrow \mathbb{S}^{m-1} \\ x \longmapsto x/|x|
\end{align*}
and let $\varphi \col (\Sp^{m-1},g_{\mathrm{\Sp^{m-1}}}) \rightarrow (N,h)$ be a smooth harmonic map with respect to the standard round metric on $\Sp^{m-1}$. Let $\phi = \varphi \circ \pr \col (\R^m \setminus \{0\}, g_{\R^m}) \to (N,h)$ be the associated dilation-invariant map, which is harmonic with respect to the Euclidean metric $g_{\R^m}$.  After identifying $\R^m \setminus \{0\}$ with the cone $C(\Sp^{m-1}) \simeq (0,\infty) \times \Sp^{m-1}$ endowed with the metric $g_{\R^m} \simeq \diff r^2 + r^2 g_{\Sp^{m-1}}$, a section $u \in \Gamma(\prescript{\phi}{}{TN}) \simeq  \Gamma(\pr^* (\prescript{\varphi}{}{TN}))$ can be seen as an $r$-dependent section $u(r)$ of the bundle $\prescript{\varphi}{}{TN}$ over $\Sp^{m-1}$. 

\begin{lem}
    Under the identification $\Gamma(\prescript{\phi}{}{TN}) \simeq  \Gamma(\pr^* (\prescript{\varphi}{}{TN}))$, the connection $\prescript{\phi}{}{\nabla}$ reads:
    \begin{equation*}
        \prescript{\phi}{}{\nabla} u = (\partial_r u ) \otimes \diff r + \prescript{\varphi}{}{\nabla} u ,
    \end{equation*}
    for all $u \in \Gamma(\pr^*(\prescript{\varphi}{}{TN}))$. Moreover, the formal adjoint $\prescript{\phi}{}{\nabla}^*$ of $\prescript{\phi}{}{\nabla}$ satisfies:
    \begin{equation*}
        \prescript{\phi}{}{\nabla}^* (v \otimes dr + A) = - \left( \frac{\partial}{\partial r} + \frac{m-1}{r} \right) v(r) + r^{-2} \prescript{\varphi}{}{\nabla}^* A
    \end{equation*}
    where $v \in \Gamma(\pr^*(\prescript{\varphi}{}{TN}))$ and $A \in \Gamma(\pr^*(T^* \Sp^{m-1} \otimes \prescript{\varphi}{}{TN}))$.
\end{lem}

\begin{proof}
    The formula for $\prescript{\phi}{}{\nabla}$ follows from the fact that $\prescript{\phi}{}{\nabla} = \pr^*(\prescript{\varphi}{}{\nabla})$. To see the expression of its formal adjoint,  we integrate by parts, assuming that $u$, $v$ and $A$ are compactly supported sections:
    \begin{multline*}
        \int_0^\infty \int_{\Sp^{m-1}} \langle \partial_r u \otimes dr + \prescript{\varphi}{}{\nabla} u , v \otimes dr + A \rangle_{g_{\R^m},h_\varphi} r^{m-1} \vol_{\Sp^{m-1}} \diff  r = \int_0^\infty \int_{\Sp^{m-1}} \langle \partial_r u , v \rangle_{h_\varphi} r^{m-1} \vol_{\Sp^{m-1}} \diff  r  \\ + \int_0^\infty \int_{\Sp^{m-1}} \langle \prescript{\varphi}{}{\nabla} u, A \rangle_{g_{\Sp^{m-1}},h_\varphi} r^{m-3} \vol_{\Sp^{m-1}} \diff  r .
    \end{multline*}
    For the first term on the right-hand side we have
    \begin{equation*}
        \int_0^\infty \int_{\Sp^{m-1}} \langle \partial_r u , v \rangle_{h_\varphi} r^{m-1} \vol_{\Sp^{m-1}} \diff  r = - \int_0^\infty \int_{\Sp^{m-1}} \langle u, \partial_r v + \tfrac{m-1}{r} v \rangle_{h_\varphi} r^{m-1} \vol_{\Sp^{m-1}} \diff  r
    \end{equation*}
    and for the second term 
    \begin{equation*}
        \int_0^\infty \int_{\Sp^{m-1}} \langle \prescript{\varphi}{}{\nabla} u, A \rangle_{g_{\Sp^{m-1}},h_\varphi} r^{m-3} \vol_{\Sp^{m-1}} \diff  r = \int_0^\infty \int_{\Sp^{m-1}} \langle u, \prescript{\varphi}{}{\nabla}^* A \rangle_{h_\varphi} r^{m-3} \vol_{\Sp^{m-1}} \diff  r
    \end{equation*}
    using the Fubini theorem and integration by parts. This is essentially our claim.
\end{proof}

With the previous lemma, we can write the rough Laplacian $\prescript{\phi}{}{\nabla}^*\prescript{\phi}{}{\nabla}$ as:
\begin{equation*}
    \prescript{\phi}{}{\nabla}^*\prescript{\phi}{}{\nabla} = -\frac{\partial^2}{\partial r^2} - \frac{m-1}{r} \frac{\partial}{\partial r} + r^{-2} \prescript{\varphi}{}{\nabla}^* \prescript{\varphi}{}{\nabla}  .
\end{equation*}
As a consequence, we obtain the following expression of the Jacobi operator $J^\phi$, which we shall denote by $J^\varphi_C$ to highlight that it is the cone over the harmonic map $\varphi$:

\begin{prop}        \label{prop: Jacobi operator of a cone}
    The Jacobi operators $J^\varphi_C = J^\phi$ and $J^\varphi$ are related via the formula
    \begin{equation*}
        J^\varphi_C = - \left( \frac{\partial^2}{\partial r^2} + \frac{m-1}{r} \frac{\partial}{\partial r} \right) + r^{-2} J^\varphi .
    \end{equation*}
\end{prop}

\begin{proof}
    Given the above expression for $\prescript{\phi}{}{\nabla}^*\prescript{\phi}{}{\nabla}$, it follows from  \eqref{eq: Jacobi operator in general} and the identity
    \begin{equation*}
        \tr_{g_{\R^m}}(\mathrm{Rm}_h(u,\diff\phi)\diff\phi) = r^{-2} \tr_{g_{\Sp^{m-1}}}(\mathrm{Rm}_h(u,\diff\varphi)\diff\varphi)
    \end{equation*}
    where the scaling factor $r^{-2}$ comes from the identity $g_{\R^m} = \diff r^2 + r^2 g_{\Sp^{m-1}}$.
\end{proof}

    \subsubsection{Roots and polyhomogeneity}       \label{subsubsec: roots and polyhomogeneity}

\begin{Def}
    Let $\varphi \col (\Sp^{m-1},g_{\Sp^{m-1}}) \rightarrow (N,h)$ be a smooth harmonic map.
    \begin{enumerate}
        \item A complex number $\zeta \in \C$ is called a \emph{root} of $J^\varphi_C$ if there exists a non-zero section $u \in \Gamma(\prescript{\varphi}{}{TN} \otimes \C)$ such that $J^\varphi_C(r^\zeta u) = 0$. 
        \item The \emph{multiplicity} of a root $\zeta$ is the largest integer $m \geq 1$ such that there exist $u_0,\ldots,u_{m-1} \in \Gamma(\prescript{\varphi}{}{TN} \otimes \C)$ such that $u_{m-1} \neq 0$ and $J^\varphi_C(r^\zeta \sum \log(r)^i u_i) = 0$.
        \item $\nu \in \R$ is called a \emph{critical rate} if the line $\nu + i \R \subset \C$ contains a root of $J^\varphi_C$. The set of critical rates will be denoted by $\mathcal{D}(J^\varphi_C) \subset \R$.
    \end{enumerate}
\end{Def}

Since the operator $J^\varphi$ is formally self-adjoint, the roots of $J^\varphi_C$ and their multiplicities can easily be described:

\begin{prop}        \label{prop: Description of the roots}
    Let $\zeta \in \C$. Then the following hold:
    \begin{enumerate}[(i)]
        \item $\zeta$ is a root of $J^\varphi_C$ if and only if $\zeta(\zeta + m - 2) = \lambda_\zeta$ is an eigenvalue of the Jacobi operator $J^\varphi$.
        \item Its multiplicity $m_\zeta$ is equal to its multiplicity as a root of the polynomial $t(t+m-2)-\lambda_\zeta$. In particular, $m_\zeta = 1$ if $\zeta \neq - \tfrac{m-2}{2}$, and if $\zeta = -\tfrac{m-2}{2}$ is a root of $J^\varphi_C$ it has multiplicity $2$.
    \end{enumerate}
\end{prop}

\begin{proof}
    If $u \in \Gamma(\prescript{\varphi}{}{TN} \otimes \C)$ and $\zeta \in \C$, a direct computation using \cref{prop: Jacobi operator of a cone} gives
    \begin{equation*}
        J^\varphi_C(r^\zeta u) = r^{\zeta-2}(J^\varphi u - (\zeta(\zeta-1) + (m-1) \zeta) u) =   r^{\zeta-2} (J^\varphi u - \zeta(\zeta + m -2) u)
    \end{equation*}
    which proves the first assertion. To prove the second assertion, we can explicitly solve the Jacobi equation by separation of variables. There is an orthogonal decomposition 
    \begin{equation*}
        L^2(\prescript{\varphi}{}{TN}) = \bigoplus_{\lambda \in \Spec(J^\varphi)} \ker(J^\varphi - \lambda),
    \end{equation*}
    and for any section $u \in \Gamma(\pr^*(\prescript{\varphi}{}{TN} \otimes \C))$ we denote by $u_\lambda(r)$ the projection of $u(r) \in \Gamma(\prescript{\varphi}{}{TN} \otimes \C)$ onto $\ker(J^\varphi - \lambda) \otimes \C$. Then the Jacobi equation $J^\varphi_C u = 0$ is equivalent to the collection of linear ODE systems
    \begin{equation*}
        - u^{\prime\prime}_\lambda(r) - \frac{m-1}{r} u^\prime_\lambda(r) + \frac{\lambda}{r^2} u_\lambda(r) = 0, ~~~ \lambda \in \Spec(J^\varphi).
    \end{equation*}
    Using the variable $t = \log(r)$, we see that when the polynomial $t(t+m-2) - \lambda$ has two distinct roots $\zeta^+_\lambda$ and $\zeta^-_\lambda$, i.e. when $\lambda \neq - \tfrac{(m-2)^2}{4}$, the solutions of this ODEs are of the form
    \begin{equation*}
        r^{\zeta^+_\lambda} u^+_\lambda + r^{\zeta^-_\lambda} u^-_\lambda, ~~~ u^+_\lambda, u^-_\lambda \in \ker(J^\varphi-\lambda) \otimes \C. 
    \end{equation*}
    Otherwise, if $\lambda = -\tfrac{(m-2)^2}{4}$, the polynomial $t(t+m-2) - \lambda$ has a double root $\zeta = -\tfrac{m-2}{2}$, and the solutions of the ODEs are of the form
    \begin{equation*}
        r^{-\tfrac{m-2}{2}}(v_0 + \log(r) v_1), ~~~ v_0, v_1 \in \ker(J^\varphi + \tfrac{(m-2)^2}{4}) \otimes \C.
    \end{equation*}
    This proves our claim.
\end{proof}

\begin{cor}
    The critical rates of $J^\varphi_C$ satisfy the following properties:
    \begin{enumerate}[(i)]
        \item $-\tfrac{m-2}{2} \in \mathcal{D}(J^\varphi_C)$ if and only if $\Spec(J^\varphi) \cap (-\infty, -\tfrac{(m-2)^2}{4}] \neq \emptyset$.
        \item $\mathcal{D}(J^\varphi_C) \cap (-\tfrac{m-2}{2},+\infty) = \left\{ \tfrac{- (m-2) + \sqrt{(m-2)^2 + 4 \lambda}}{2} ~ \big| ~ \lambda \in \Spec(J^\varphi) \cap (-\tfrac{(m-2)^2}{4},+\infty) \right\}$
        \item $\mathcal{D}(J^\varphi_C) \cap (-\tfrac{m-2}{2},+\infty) = \left\{ \tfrac{- (m-2) - \sqrt{(m-2)^2 + 4 \lambda}}{2} ~ \big| ~ \lambda \in \Spec(J^\varphi) \cap (-\tfrac{(m-2)^2}{4},+\infty) \right\}$
    \end{enumerate}
\end{cor}

\begin{Def}
    A section $u \in \Gamma(\pr^*(\prescript{\varphi}{}{TN}))$ is called \emph{polyhomogeneous of rate $\nu \in \R$} if there exists $\zeta_1,\ldots,\zeta_k \in \nu + i \R$ and a collection of sections $u_{ij} \in \Gamma(\prescript{\varphi}{}{TN} \otimes \C)$, $i = 1, \ldots, k$, $j = 0,\ldots,\ell_i$ such that
    \begin{equation*}
        u = \Re \left( \sum_{i=1}^k \sum_{j=0}^{\ell_i} r^{\zeta_i} \log^j(r) u_{ij} \right) .
    \end{equation*}
    Moreover, if $\nu \in \mathcal{D}(J^\varphi_C)$ the vector space of polyhomogeneous sections of rate $\nu$ satisfying $J^\varphi_C u = 0$ will be denoted by $\mathcal{K}(J^\varphi_C)_\nu$.
\end{Def}

As an immediate consequence of the previous results, we can explicitly describe the spaces $\mathcal{K}(J^\varphi_C)_\nu$ as follows:

\begin{prop}
    Let $\nu \in \mathcal{D}(J^\varphi_C)$.
    \begin{enumerate}[(i)]
        \item If $\nu \neq -\tfrac{m-2}{2}$, then
        \begin{equation*}
            \mathcal{K}(J^\varphi_C)_\nu = \{ r^\nu u ~ | ~ u \in \ker(J^\varphi-\lambda_\nu) \} , ~~~ \lambda_\nu = \nu(\nu+m-2) .
        \end{equation*}
        \item If $\nu = - \tfrac{m-2}{2}$, then $\mathcal{K}(J^\varphi_C)_{-\tfrac{m-2}{2}} = \mathcal{K}^\prime(J^\varphi_C)_{-\tfrac{m-2}{2}} \oplus \mathcal{K}^{\prime\prime}(J^\varphi_C)_{-\tfrac{m-2}{2}}$, where
        \begin{align*}
            \mathcal{K}^\prime(J^\varphi_C)_{-\tfrac{m-2}{2}} & = \left\{ r^{-\tfrac{m-2}{2}} (v_0 + \log(r) v_1) ~ | ~ v_0,v_1 \in \ker(J^\varphi + \tfrac{(m-2)^2}{4}) \right\} \\
            \mathcal{K}^{\prime\prime}(J^\varphi_C)_{-\tfrac{m-2}{2}} & = \bigoplus_{\alpha > 0} \left\{ r^{-\tfrac{m-2}{2}}(\cos(\alpha \log(r)) v_0 + \sin(\alpha \log(r) ) v_1) ~ | ~ v_0, v_1 \in \ker(J^\varphi + \tfrac{(m-2)^2}{4} + \alpha^2) \right\} .
        \end{align*}
    \end{enumerate}
\end{prop}

Observe that for any critical rate $\nu \in \R$, $\lambda_\nu = \lambda_{2-m-\nu}$, and therefore we can define a natural non-degenerate bilinear pairing $(\cdot,\cdot)_\nu \col \mathcal{K}(J^\varphi_C)_\nu \otimes \mathcal{K}(J^\varphi_C)_{2-m-\nu} \rightarrow \R$, which can be described as follows. Let $\chi \col [0,\infty) \rightarrow \R$ be a smooth cutoff function such that $\chi \equiv 1$ in a neighbourhood of $0$ and $\chi \equiv 0$ in a neighbourhood of $+\infty$. Then we can define
\begin{equation*}
    (\alpha,\beta)_\nu = \int_0^\infty \int_{\Sp^{m-1}} \langle J^\varphi_C(\chi \alpha), \beta \rangle_{h_\varphi} r^{m-1} \vol_{\Sp^{m-1}} \diff  r, ~~~ \forall (\alpha,\beta) \in \mathcal{K}(J^\varphi_C)_\nu \times \mathcal{K}(J^\varphi_C)_{2-m-\nu} .
\end{equation*}
Remark that this definition does not depend on the choice of cutoff function, since the difference between two cutoff functions is compactly supported and we can integrate by parts. To see the non-degeneracy of this pairing, we can give it another description:

\begin{prop}        \label{prop: Polyhomoegenous pairing}
    Let $\nu \in \mathcal{D}(J^\varphi_C)$.
    \begin{enumerate}[(i)]
        \item If $\nu \neq -\tfrac{m-2}{2}$ and $u, v \in \ker(J^\varphi-\lambda_\nu)$, 
        \begin{equation*}
            (r^\nu u, r^{2-m-\nu} v )_\nu = (m-2 + 2\nu) \langle u, v \rangle_{L^2} .
        \end{equation*}

        \item $\mathcal{K}^\prime(J^\varphi_C)_{-\tfrac{m-2}{2}}$ and $\mathcal{K}^{\prime\prime}(J^\varphi_C)_{-\tfrac{m-2}{2}}$ are orthogonal for the pairing $(\cdot,\cdot)_{-\tfrac{m-2}{2}}$.
        
        \item For any $u_1,v_1,u_2,v_2 \in \ker(J^\varphi+\tfrac{(m-2)^2}{4})$,
        \begin{equation*}
            (r^{-\tfrac{m-2}{2}}(u_1 + \log(r) v_1), r^{-\tfrac{m-2}{2}}(u_2+\log(r)v_2) )_{-\tfrac{m-2}{2}} = \langle v_1, u_2 \rangle_{L^2} - \langle u_1, v_2 \rangle_{L^2} .
        \end{equation*}
        In particular, the restriction of $(\cdot,\cdot)_{-\tfrac{m-2}{2}}$ to $\mathcal{K}^\prime(J^\varphi_C)_{-\tfrac{m-2}{2}}$ is non-degenerate and skew-symmetric.
    \end{enumerate}
\end{prop}

\begin{proof}
    Notice that we can write
    \begin{equation*}
        r^{m-1}J^\varphi_C = -\frac{\partial}{\partial r} \left( r^{m-1} \frac{\partial}{\partial r} \right) + r^{m-3} J^\varphi
    \end{equation*}
    and therefore for any $u \in \Gamma(\pr^*\prescript{\varphi}{}{TN})$ we have
    \begin{equation*}
        r^{m-1}J^\varphi_C(\chi(r)u(r)) - r^{m-1}\chi(r) J^\varphi_C u(r) = - \partial_r(r^{m-1} \chi^\prime(r) u(r)) - r^{m-1}\chi^\prime(r)u^\prime(r) .
    \end{equation*}
    For the proof of point (i), we obtain for $u \in \ker(J^\varphi-\lambda_\nu)$,
    \begin{equation*}
        r^{m-1}J^\varphi_C(\chi(r) r^\nu u) = - \partial_r( r^{m-1+\nu} \chi^\prime(r))u - \nu r^{m-2+\nu} \chi^\prime(r) u
    \end{equation*}
    since $J^\varphi_C (r^\nu u) = 0$. Since $\chi^\prime$ is compactly supported, we can integrate by parts and deduce that
    \begin{align*}
        \int_0^\infty \int_{\Sp^{m-1}} r^{m-1} \langle J^\varphi_C (\chi r^\nu u), r^{2-m-\nu} v \rangle_{h_\varphi} \vol_{\Sp^{m-1}} \diff  r & = (2-m-2\nu) \int_0^\infty \int_{\Sp^{m-1}} \chi^\prime(r)  \langle u, v \rangle_{h_\varphi} \vol_{\Sp^{m-1}} \diff  r \\
            & = (m-2 + 2\nu) \langle u, v \rangle_{L^2}
    \end{align*}
    since $\int_0^\infty \chi^\prime(r) \diff  r = -1$.

    Point (ii) follows directly from the fact that if $u \in \mathcal{K}^\prime(J^\varphi_C)_{-\tfrac{m-2}{2}}$ and $v \in \mathcal{K}^{\prime\prime}(J^\varphi_C)_{-\tfrac{m-2}{2}}$, $u(r)$ and $v(r)$ belong to different eigenspaces of $J^\varphi$ and hence are $L^2$-orthogonal for any $r \in (0,\infty)$.
    
    For point (iii), let us now assume that $\nu = - \tfrac{m-2}{2}$.  This time, we obtain
    \begin{multline*}
        r^{m-1} J^\varphi_C(r^{-\tfrac{m-2}{2}}(u_1 + \log(r) v_1)) = - \partial_r(r^{\tfrac{m}{2}}\chi^\prime(r))u_1 + \frac{m-2}{2} r^{\tfrac{m-2}{2}} \chi^\prime(r) u_1 \\ - \partial_r(r^{\tfrac{m}{2}} \log(r) \chi^\prime(r)) v_1 + \frac{m-2}{2} r^{\tfrac{m-2}{2}} \log(r) \chi^\prime(r) v_1 - r^{\tfrac{m-2}{2}} \chi^\prime(r) v_1 .
    \end{multline*}
    Integrating by parts, we have:
    \begin{align*}
        \int_0^\infty \int_{\Sp^{m-1}} & r^{m-1} \langle J^\varphi_C(\chi(r)r^{-\frac{m-2}{2}}(u_1 + \log(r) v_1), r^{-\frac{m-2}{2}}(u_2+\log(r)v_2) \rangle_{h_\varphi} \vol_{\Sp^{m-1}} \diff  r \\
            = & \int_0^\infty \int_{\Sp^{m-1}} \chi^\prime(r) (\tfrac{m-2}{2} \log(r) + r^{\frac{m}{2}} \partial_r(r^{-\frac{m-2}{2}} \log(r))) \langle u_1, v_2 \rangle_{h_\varphi} \vol_{\Sp^{m-1}} \diff  r \\
            & - \int_0^\infty \int_{\Sp^{m-1}} \chi^\prime(r) \langle v_1, u_2 \rangle_{h_\varphi} \vol_{\Sp^{m-1}} \diff  r \\
            & + \int_0^\infty \int_{\Sp^{m-1}} \chi^\prime(r) (\tfrac{m-2}{2} + r^{\frac{m}{2}} \partial_r(r^{-\frac{m-2}{2}})) \langle u_1, u_2 \rangle_{h_\varphi} \vol_{\Sp^{m-1}} \diff  r \\
            & + \int_0^\infty \int_{\Sp^{m-1}} \chi^\prime(r)(\tfrac{m-2}{2}\log^2(r)- \log(r) + r^{\frac{m}{2}} \log(r)  \partial_r (r^{-\frac{m-2}{2}} \log(r))) \langle v_1, v_2 \rangle_{h_\varphi} \vol_{\Sp^{m-1}} \diff  r \\
            = & \langle v_1, u_2 \rangle_{L^2} - \langle u_1, v_2 \rangle_{L^2}
    \end{align*}
    which proves our claim.
\end{proof}

\begin{rem}
    One could also derive an expression for the pairing on $\mathcal{K}^{\prime\prime}(J^\varphi_C)_{-\tfrac{m-2}{2}}$, but for the purpose of this article it will not be needed.
\end{rem}

    \subsubsection{Further remarks}     \label{subsubsec: further remarks}

If the map $\varphi \col (\Sp^{m-1},g_{\Sp^{m-1}}) \rightarrow (N,h)$ is non-trivial (its image is not reduced to a point), pre-composing $\varphi$ by isometries induces harmonic deformations of $\varphi$, and therefore the Killing fields $\xi$ on the sphere generate Jacobi fields $u = \diff \varphi(\xi)$ which span a non-trivial subspace of the kernel of $J^\varphi$. Moreover, the infinitesimal translations of $\R^m$ generate homogeneous Jacobi fields of degree $-1$ which locally span $T\Sp^{m-1}$, and hence $\nu = -1$ is always a root of $J^\varphi_C$, corresponding to the eigenvalue $-(m-3)$ of $J^\varphi$. Note that the infinitesimal translations generate an $m$-dimensional subspace of $\ker(J^\varphi + m-3)$ \cite{elsoufi1995indice} (for otherwise the dilation-invariant map $\varphi \circ \pr$ would be invariant under a translation and would be singular along at least a line in $\R^m$). It will be therefore useful to consider the following assumption:

\begin{hyp}     \label{assumption: Eigenvalues of varphi}
    The image of the harmonic map $\varphi \col \Sp^{m-1} \rightarrow (N,h)$ is not reduced to a point, and $-(m-3)$ and $0$ are the only non-positive eigenvalues of $J^\varphi$.
\end{hyp}

Under this assumption, the roots of the operator $J^\varphi_C$ have the following properties:

\begin{prop}    \label{prop: indicial roots of model operator}
    Let $m \geq 4$ and let $\varphi \col \Sp^{m-1} \rightarrow (N,h)$ satisfying \cref{assumption: Eigenvalues of varphi}. Then the following hold:
	\begin{enumerate}[(i)]
        \item All the roots of $J^\varphi_C$ are real.
        
		\item There exists $\varepsilon > 0$ such that the only critical rates of $J^\varphi_C$ in the interval $[-(m-2)-\varepsilon,\varepsilon]$ are
		\begin{equation*}
			-(m-2), ~~ -(m-3), ~~ -1 ~~ \text{and} ~~ 0 .
		\end{equation*}
		Moreover,
        \begin{itemize}
            \item $0$ and $-(m-2)$ are simple roots,
            \item if $m = 4$, $-1 = -(m-3)$ is a double root, and 
            \item if $m \geq 5$ the roots $-(m-3)$ and $-1$ are simple.
        \end{itemize}
        
		\item The other indicial roots are simple and given by 
		\begin{equation*}
			\nu_{\lambda}^\pm = - \frac{m-2 \mp \sqrt{(m-2)^2+4\lambda}}{2}
		\end{equation*}
		where $\lambda$ ranges across the positive eigenvalues of $J_\phi$. In particular, $\nu_\lambda^+ > 0$ and $\nu^-_\lambda < -(m-2)$.
	\end{enumerate}
\end{prop}

The spaces $\mathcal{K}(J^\varphi_C)_\nu$ also admit a simple description:

\begin{lem}     \label{lem: the space of polyhomogeneous kernel elements}
    Let $m \geq 4$ and suppose that $\varphi \col \Sp^{m-1} \rightarrow (N,h)$ satisfies \cref{assumption: Eigenvalues of varphi}. Then:
    \begin{enumerate}[(i)]
        \item If $\nu \in \R$ is a simple root of $J^\varphi_C$ (e.g., $m \geq 5$ or $m=4$ and $\nu \neq -1$), 
        \begin{equation*}
            \mathcal{K}(J^\varphi_C)_\nu = \{ r^\nu u ~ | ~ u \in \ker(J^\varphi - \lambda_\nu)\}, ~~~ \lambda_\nu = \nu(\nu+ m-2).
        \end{equation*}
        \item If $m = 4$ and $\nu = -1$, then
        \begin{equation*}
            \mathcal{K}(J^\varphi_C)_{-1} = \{ r^{-1}(u+\log(r) v) ~ | ~ u,v \in \ker(J^\varphi+1) \} .
        \end{equation*}
    \end{enumerate}
\end{lem}

    \subsection{Weighted spaces}        \label{subsec: Weighted spaces}

In this section, we discuss a class of weighted spaces adapted to the analysis of conically singular problems. The notions of a adapted bundles and weighted Sobolev spaces are introduced in \cref{subsubsec: adapted bundles}. In \cref{subsubsec: The tangent bundle as adapted bundle}, we fix certain conventions in order to regard the tangent bundle of a compact manifold as an adapted bundle. Finally, in \cref{subsubsec: Dilation-invariant bundle maps} and \cref{subsubsec: Bundle maps with dilation-invariant ends} we introduce a class of non-linear bundle maps, called \emph{bundle maps with dilation-invariant ends}, which will play a key role in \cref{sec: Banach manifolds} and \cref{sec: The deformation theorem}. The main results here are \cref{lem: Master lemma} and \cref{lem: Parametrised master lemma} concerning the differentiability of the Banach maps induced by non-linear bundle maps with dilation-invariant ends on weighted spaces of sections.

    \subsubsection{Adapted bundles and weighted spaces}     \label{subsubsec: adapted bundles}

Throughout this section, we consider a compact manifold $M$ of dimension $m$. Let $S \subset M$ be a finite subset. We fix an adapted set of charts $\{\Upsilon_s \col \Ball_1 \to M\}_{s \in S}$ (see \cref{Def: Adapted set of charts}), and pick a continuous function  $r \col M \rightarrow [0,\infty)$, smooth on $M \setminus S$, and such that $r \circ \Upsilon_s(x) = |x|_{g_{\R^m}}$ for any $s \in S$ and $x \in \Ball_1$ and $r \geq 1$ on $M \setminus (\cup_s \Upsilon_s(\Ball_1))$.

\begin{Def}
    Let $k \in \N_0$, $p \in [1,\infty)$, $\nu \in \R$, and let $g$ be a background metric on $M$. 
    
    \begin{enumerate}
    	
		\item We denote by $C^{k}_\nu(M \setminus S)$ the vector space of functions $f \col M \setminus S \rightarrow \R$ which are of class $C^{k}_{\mathrm{loc}}$ on $M \setminus S$ and such that the quantity
        \begin{equation*}
            \|f\|_{C^{k}_\nu} \coloneqq \sup_{0 \leq \ell \leq k} \sup_{x \in M} |r(x)^{\ell-\nu} (\nabla^{g})^\ell f(x)|
        \end{equation*}
        is finite. The above expression defines a norm which endows $C^{k}_\nu(M \setminus S)$ with the structure of a Banach space. The intersection $\cap_{k \in \N_0} C^{k}_\nu(M \setminus S)$ will be denoted by $C^\infty_\nu(M \setminus S)$.
    
    	\item We denote by $W^{k,p}_\nu(M \setminus S)$ the vector space of functions $f \col M \setminus S \rightarrow \R$ which are of class $W^{k,p}_{\mathrm{loc}}$ on $M \setminus S$ and such that the quantity
    \begin{equation*}
        \|f\|_{W^{k,p}_\nu} \coloneqq \sum_{\ell = 0}^k \left( \int_M r^{-m} (r^{\ell-\nu} |(\nabla^g)^\ell f|_g)^p \vol_g \right)^{\frac{1}{p}}
    \end{equation*}
    is finite. The above expression defines a norm which endows $W^{k,p}_\nu(M \setminus S)$ with the structure of a Banach space. When $k = 0$, we shall denote $W^{0,p}_\mu(M \setminus S)$ simply by $L^p_\nu(M \setminus S)$.
	\end{enumerate}
\end{Def}

\begin{rem}
    These definitions do not depend on the choice of background Riemannian metric; in particular, the $C^{k}_\nu$ and $W^{k,p}_\nu$-norms defined for two different choices of background metric on $M$ are equivalent. For this reason, we can make a convenient choice of back Riemannian metric on $M$, for instance by fixing a metric $g_0$ such that $\Upsilon_s^* g_0 = g_{\R^m}$ is flat on $\Ball_1$ for any $s \in S$.
\end{rem}

\begin{ex}
    For any $\nu \in \R$, $r^{\nu} \in C^\infty_{\nu}(M \setminus S)$, and $r^\nu \in W^{k,p}_{\nu-\varepsilon}(M \setminus S)$ for any $k \in \N_0$, $p \geq 1$ and $\varepsilon > 0$. Moreover, multiplication by the function $r^\nu$ defines an isomorphism of Banach spaces $r^\nu \col C^{k}_{\nu^{\prime}}(M \setminus S) \to C^{k}_{\nu^{\prime}+\nu}(M \setminus S)$ and $r^\nu \col W^{k,p}_{\nu^{\prime}}(M \setminus S) \to W^{k,p}_{\nu^{\prime}+\nu}(M \setminus S)$ for any $\nu^{\prime} \in \R$.
\end{ex}

In order to define weighted spaces of sections, we need an adapted class of vector bundles. First, we begin with the notion of dilation-invariant vector bundle on $\R^m \setminus \{0\}$.

\begin{Def}
    Let $E \to \R^m \setminus \{0\}$ be a vector bundle endowed with a bundle metric $h^E$ and a compatible connection $\nabla^E$. Then $(E,\nabla^E,h^E)$ is called \emph{dilation-invariant} if there exists a vector bundle with a metric and compatible connection $(E_0,\nabla^{E_0},h^{E_0})$ on $\Sp^{m-1}$ such that $(E,\nabla^E,h^E) = \pr^*(E_0,\nabla^{E_0},h^{E_0})$.
\end{Def}

\begin{ex}
    If $\varphi \col \Sp^{m-1} \to N$ is a smooth map and $\phi = \varphi \circ \pr \col \R^m \setminus \{0\} \to N$ is a dilation-invariant map, then 
    \begin{equation*}
        (\prescript{\phi}{}{TN},\prescript{\phi}{}{\nabla},h_\phi) = \pr^*(\prescript{\varphi}{}{TN},\prescript{\varphi}{}{\nabla},h_\varphi)
    \end{equation*}
    and thus $(\prescript{\phi}{}{TN},\prescript{\phi}{}{\nabla},h_\phi)$ is a dilation-invariant vector bundle.
\end{ex}

\begin{Def}
    Let $(E,\nabla^E,h^E)$ be a vector bundle over $M \setminus S$, endowed with a metric and compatible connection. We shall say that that $E$ is \emph{adapted} if there exist $\mu > 0$, $R \in (0,1)$ and:
    \begin{itemize}
        \item a set of dilation-invariant bundles $\{(E_s,\nabla^{E_s},h^{E_s})\}$; and
        \item a set of bundle isomorphisms $\{ \Psi_s \col \left. E_s \right|_{\Ball_R} \to \left. E \right|_{\Upsilon_s(\Ball_R)} \}_{s \in S}$ covering the charts $\{\Upsilon_s\}_{s \in S}$ on the ball $\Ball_R \subset \Ball_1$,
    \end{itemize}
    such that
    \begin{equation*}
        |(\nabla^{E_s})^k (\Psi_s^* h^E - h^{E_s})| = \mathcal{O}(r^{\mu-k}), ~~ \text{and} ~~ |(\nabla^{E_s})^k(\Psi_s^*\nabla^E - \nabla^{E_s})| = \mathcal{O}(r^{-1+\mu-k})
    \end{equation*}
    as $r \to 0$, for any $s \in S$ and $k \in \N_0$.
\end{Def}

With these definitions in hand, we can now define weighted spaces of sections of an adapted vector bundle, generalising the notion of weighted space of functions:

\begin{Def}     \label{Def: weighted spaces of adapted bundles}
    Let $g$ be a background Riemannian metric on $M$, $(E,\nabla^E,h^E)$ be an adapted vector bundle over $M \subset S$, and let $k \in \N_0$, $p \in [1,\infty)$ and $\nu \in \R$.
    \begin{enumerate}
    	\item We denote by $C^{k}_\nu(E)$ the set of sections $u \col M \setminus S \to E$ which are $C^{k}_{\mathrm{loc}}$ on $M \setminus S$ and such that the quantity
    	\begin{equation*}
        	\|u\|_{C^{k}_\nu} \coloneqq \sup_{0 \leq \ell \leq k} \sup_{x \in M} |r(x)^{\ell-\nu} (\nabla^E)^\ell u(x)|
		\end{equation*}
    	is finite.\footnote{The norms $|\cdot|$ on $T^*M^{\otimes j} \otimes E$ are computed using the adapted bundle metric $h^E$ on $E$ and the metric $g$ on $M$.} The norm $\|\cdot\|_{C^{k}_\nu}$ endows $C^{k}_\nu(E)$ with the structure of a Banach space.
	
		\item We denote by $W^{k,p}_\nu(E)$ the vector space of sections $u \col M \setminus S \rightarrow E$ which are of class $W^{k,p}_{\mathrm{loc}}$ on $M \setminus S$ and such that the quantity
        \begin{equation*}
            \|u\|_{W^{k,p}_\nu} \coloneqq \sum_{\ell = 0}^k \left( \int_M r^{-m} (r^{\ell-\nu} |(\nabla^E)^\ell u|)^p \vol_g \right)^{\frac{1}{p}}
        \end{equation*}
        is finite. The above expression defines a norm which endows $W^{k,p}_\nu(E)$ with the structure of a Banach space. The Banach space $W^{0,p}_\mu(E)$ will be simply denoted by $L^p_\mu(E)$.
    \end{enumerate}
\end{Def}

\begin{rem}
    As for functions, the definition of weighted spaces of sections does not depend on the choice of background Riemannian metric on $M$, and two different choices of metrics yield equivalent norms on $C^k_\nu(E)$ and $W^{k,p}_\nu(E)$.
\end{rem}

\begin{lem}
    Let $(E,\nabla^E,h^E)$ be an adapted bundle. Then for any $k \in \N_0$, $p \in [1,\infty)$ and $\nu \in \R$, the space $C^\infty_c(E)$ of compactly supported sections (away from $S$) is dense in $W^{k,p}_\nu(E)$.
\end{lem}

\begin{proof}
    Let $\chi \col [0,+\infty) \to \R$ be a smooth cutoff function such that $\chi \equiv 1$ on $[0,\frac{1}{2}]$ and $\chi \equiv 0$ on $[1,+\infty)$. Then for any $u \in W^{k,p}_\nu(E)$ and $0 \leq \ell \leq k$,
    \begin{equation*}
        r^{\ell-\nu}(\nabla^E)^\ell(\chi(\tfrac{r}{\varepsilon})u) = \sum_{j=0}^\ell (\tfrac{r}{\varepsilon})^j \chi^{(j)}(\tfrac{r}{\varepsilon}) r^{\ell-j-\nu} \diff r^{\otimes j} \otimes (\nabla^E)^{\ell-j} u
    \end{equation*}
    and since $|(\frac{r}{\varepsilon})^j \chi^{(j)}(\frac{r}{\varepsilon})| \leq \sup \chi^{(j)} < + \infty$ we deduce that there exists a constant $C > 0$ independent of $\varepsilon > 0$ such that
    \begin{equation*}
        \|\chi(\tfrac{r}{\varepsilon}) u\|_{W^{k,p}_\nu} \leq C \sum_{\ell = 0}^k \left( \int_{r \leq \varepsilon} r^{-m} (r^{\ell-\nu} |(\nabla^E)^\ell u|)^p \vol_g \right)^{\frac{1}{p}} \longrightarrow 0, ~~ \text{as} ~ \varepsilon \to 0 .
    \end{equation*}
    Thus $(1-\chi(\frac{r}{\varepsilon})) u \to u$ in $W^{k,p}_\nu(E)$ as $\varepsilon \to 0$. Using smoothings this proves the result.
\end{proof}

We list without proof a number of obvious properties of weighted spaces in relation to natural bundle operations:

\begin{lem}     \label{lem: Natural properties of weighted spaces}
    Let $(E,\nabla^E,h^E)$, $(F,\nabla^F,h^F)$ be adapted vector bundles.
    \begin{enumerate}[(i)]
        \item The vector bundles $E^*$, $E^{\otimes \ell}$, $\Lambda^\ell E$, $E \otimes F$, $\Hom(E,F)$ are adapted, endowed with the connections and compatible metrics induced by $(\nabla^E,h^E)$ and $(\nabla^F,h^F)$.

        \item For any $k \geq \ell \in \N_0$, $p \in [1,\infty)$ and $\nu_1 \geq \nu_2$, there are continuous embeddings $C^{k}_{\nu_1}(E) \hookrightarrow C^{\ell}_{\nu_2}(E)$ and $W^{k,p}_{\nu_1}(E) \hookrightarrow W^{\ell,p}_{\nu_2}(E)$.
        
        \item For any $k \in \N_0$, $p \in [1,\infty)$ and $\nu_1 > \nu_2$, there is a continuous embedding $C^k_{\nu_1}(E) \hookrightarrow W^{k,p}_{\nu_2}(E)$.

        \item For any $k \in \N_0$, $p \in [1,\infty)$ and $\nu_1, \nu_2 \in \R$, the tensor product of sections induces bounded bilinear maps $C^{k}_{\nu_1}(E) \times C^{k}_{\nu_2}(F) \to C^{k}_{\nu_1+\nu_2}(E \otimes F)$ and $C^{k}_{\nu_1}(E) \times W^{k,p}_{\nu_2}(F) \to C^{k}_{\nu_1+\nu_2}(E \otimes F)$

        \item For any $k \in \N_0$, $p \in [1,\infty)$ and $\nu_1, \nu_2 \in \R$, the evaluation map induces bounded bilinear maps $C^{k}_{\nu_1}(\Hom(E,F)) \times C^{k}_{\nu_2}(E) \to C^{k}_{\nu_1+\nu_2}(F)$ and $C^{k}_{\nu_1}(\Hom(E,F)) \times W^{k,p}_{\nu_2}(E) \to W^{k,p}_{\nu_1+\nu_2}(F)$
    \end{enumerate}
\end{lem}

To close this section, we finally state the relevant results concerning weighted Sobolev embeddings and weighted Sobolev multiplication:

\begin{prop}        \label{prop: Sobolev embedding and multiplication}
	Let $(E,\nabla^E,h^E)$, $(F,\nabla^F,h^F)$ be adapted bundles and let $k,\ell \in \N_0$, $p \in [1,\infty)$ and $\nu, \nu^\prime \in \R$.
    \begin{enumerate}[(i)]
        \item If the inequalities
        \begin{equation*} 
		      \frac{1}{p} < \frac{k-\ell}{m} ~~ \text{and} ~~ \nu \geq \nu^\prime,
        \end{equation*}
        then there is a continuous embedding $W^{k,p}_\nu(E) \hookrightarrow C^\ell_{\nu^\prime}(E)$.

        \item If $k \geq 1$ and the inequality $p > \frac{2m}{k}$ holds then the tensor product induces continuous bilinear maps $W^{k,p}_\nu(E) \times W^{k,p}_{\nu^\prime}(F) \to W^{k,p}_{\nu+\nu^\prime}(E \otimes F)$ and $W^{k,p}_\nu(\Hom(E,F)) \times W^{k,p}_{\nu^\prime}(E) \to W^{k,p}_{\nu+\nu^\prime}(F)$.
    \end{enumerate}
\end{prop}

\begin{proof}
    The weighted Sobolev embedding is a classical generalisation of the compact (unweighted) case, cf. \cite[Th. 1.2]{Bartnik1986mass_of_ALF} for instance. It can be proved by a scaling argument, using the dilation-invariance of the bundle $E$ near $S$ and the fact that the measure $r^{-m} \diff x^m$ is also dilation-invariant on $\R^m \setminus \{0\}$.
    
    Weighted Sobolev multiplication is a direct consequence of this embedding: indeed, under the assumptions that $k \geq 1$ and $p > \frac{2m}{k}$, there is a continuous embedding $W^{k,p}_\nu(E) \hookrightarrow C^{\lfloor\frac{k}{2}\rfloor}_\nu(E)$ and similarly for $F$. Thus if $u \in C^\infty_c(E), v \in C^\infty(F)$ and $0 \leq \ell \leq k$ we have
    \begin{multline*}
        r^{\ell-\nu-\nu^\prime} \nabla^\ell(u \otimes v) = \sum_{j \leq \frac{\ell}{2}} (r^{j-\nu} (\nabla^E)^j u) \otimes (r^{\ell-j-\nu^\prime} (\nabla^F)^{\ell-j} u) \\ + \sum_{j > \frac{\ell}{2}} (r^{j-\nu} (\nabla^E)^j u) \otimes (r^{\ell-j-\nu^\prime} (\nabla^F)^{\ell-j} u)
    \end{multline*}
    and using the Sobolev embedding $W^{k,p}_\nu(E) \hookrightarrow C^{\lfloor\frac{k}{2}\rfloor}_\nu(E)$ to bound the terms on the first line and the embedding $W^{k,p}_{\nu^\prime}(F) \hookrightarrow C^{\lfloor\frac{k}{2}\rfloor}_{\nu^\prime}(F)$ to bound the terms on the second line, we deduce that $\|u \otimes v \|_{W^{k,p}_{\nu+\nu^\prime}} \leq C \|u\|_{W^{k,p}_\nu} \|v\|_{W^{k,p}_{\nu^\prime}}$. The result follows by density of $C^\infty_c(E) \subset W^{k,p}_\nu(E)$ and $C^\infty_c(F) \subset W^{k,p}_{\nu^\prime}(F)$.
\end{proof}

\begin{rem}
    The assumption $p > \frac{2m}{k}$ for (ii) is not optimal: by induction, one could relax it to $p > \frac{m}{k}$, like for classical Sobolev multiplication on compact manifolds. Nevertheless, this version will be sufficient for our purpose, and has the advantage of being quick to prove. 
\end{rem}

    \subsubsection{The tangent bundle as an adapted bundle}        \label{subsubsec: The tangent bundle as adapted bundle}

An easy observation which will often be useful is the fact that if $(E,\nabla^E,h^E)$ is a vector bundle defined over the entire manifold $M$, it can naturally be seen as an adapted vector bundle over $M \setminus S$. Indeed, using the adapted set of charts $\{\Upsilon_s \col \Ball_1 \to M\}_{s \in S}$, we can use parallel transport in the radial direction in order to define an isomorphism of vector bundles
\begin{equation*}
    \Psi_s \col (\Ball_1 \setminus \{0\}) \times E_s \to E_{|\Upsilon(\Ball_1) \setminus \{s\}}
\end{equation*}
where $E_s$ is the fibre of $E$ over $s$ and the vector bundle $(\Ball_1 \setminus \{0\}) \times E_s$ is equipped with the trivial connection and the metric $h^{E_s}$. It is not difficult to see that $|(\nabla^{E_S})^k(\Psi_s^* h^E - h^{E_s})| = \mathcal{O}(r^{1-k})$ and $|(\nabla^{E_S})^k(\Psi_s^* \nabla^E - \nabla^{E_s})| = \mathcal{O}(r^{-k})$, e.g. $E_{|M\setminus S}$ is adapted with rate $\mu = 1$. Moreover, in terms of weighted spaces, if $u \in C^k(E)$ then $\nabla^\ell u \in C^{k-\ell}_0(T^*M^{\otimes \ell} \otimes E)$ for any $0 \leq \ell \leq k$.

This observation notably applies to the tangent bundle $TM$, the cotangent bundle $T^*M$ and tensor products thereof, equipped with any Riemannian metric $g$ and the associated Levi-Civita connection $\nabla^g$. This immediately implies the following result, which will help reducing the proof of many regularity results to an easy induction:

\begin{lem}     \label{lem: Covariant derivatives in weighted spaces}
    Let $E \to M \setminus S$ be an adapted bundle and $u \in W^{k,p}_{\mathrm{loc}}(E)$. Then for any $\nu \in \R$, $u \in W^{k,p}_\nu(E)$ if and only if for all $\ell = 0,1,\ldots,k$, $\nabla^\ell u \in L^p_{\nu-\ell}(T^*M^{\otimes \ell} \otimes E)$, where we regard $T^*M$ as an adapted bundle over $M \setminus S$.
\end{lem}

Let us also mention a point of subtlety, in order to avoid any later confusions. Given any Riemannian metric $g$ on $M$, the local model near the singular set $S$ for the bundle $(T^*M,\nabla^g,g)$, seen as an adapted bundle, is merely the bundle $(T^*(\R^m \setminus \{0\}),\nabla^{g_{\R^m}},g_{\R^m})$ equipped with the flat connection induced by the standard metric $g_{\R^m}$ on $\R^m$. On the other hand, there is an natural inclusion of the bundle $\pr^* (T^* \Sp^{m-1}) \hookrightarrow T^*(\R^m \setminus \{0\})$, and $\pr^*(T^* \Sp^{m-1})$ has a natural structure of dilation-invariant bundle by pulling back the round metric $g_{\Sp^{m-1}}$ and the associated Levi-Civita connection $\nabla^{g_{\Sp^{m-1}}}$. It is important to note that the embedding $\pr^* (T^* \Sp^{m-1}) \hookrightarrow T^*(\R^m \setminus \{0\})$ is \emph{not} compatible with the two dilation-invariant structures: indeed, if $\eta \in \Omega^1(\Sp^{m-1})$, then the norm $|\pr^* \eta|_{g_{\R^m}} = r^{-1} |\eta|_{g_{\Sp^{m-1}}}$, since in radial coordinates $g_{\R^m} = \diff r^2 + r^2 \pr^*g_{\Sp^{m-1}}$. Rather, if $\eta \in \Omega^1(\Sp^{m-1})$ then the rescaled section $r\pr^* \eta$ has constant norm, and indeed is parallel in the radial directions, and thus should be considered as a dilation-invariant section with respect to the dilation-invariant structure of $(T^*(\R^m \setminus \{0\}),\nabla^{g_{\R^m}},g_{\R^m})$ induced by radial parallel transport.

\begin{rem}     \label{rem: Factors of r}
    The above discussion implies that, if $(E,\nabla^E) = \pr^* (E_0,\nabla^{E_0})$ is a dilation-invariant bundle over $\R^m \setminus \{0\}$, then
    \begin{enumerate}
        \item for any dilation-invariant section $u = \pr^* u_0$ of $E$, $r \nabla^E u$ is a dilation-invariant section of $T^*\R^m \otimes E$; and similarly,
        \item for any connection $\tilde{\nabla}^{E_0}$ on $E_0$, the pull-back connection $\widetilde{\nabla}^{E} = \pr^* \widetilde{\nabla}^{E_0}$ has the property that $r (\nabla^{E}-\widetilde{\nabla}^{E}) \in \Gamma(T^*\R^m \otimes E)$ is a dilation-invariant section of $T^* \R^m \otimes E$.
    \end{enumerate}
    This is because $\nabla^E u = \pr^*(\nabla^{E_0} u_0)$ and $\nabla^{E}-\widetilde{\nabla}^{E} = \pr^*(\nabla^{E_0}-\widetilde{\nabla}^{E_0})$ for the bundle-valued $1$-forms $\nabla^{E_0} u_0, \nabla^{E_0}-\widetilde{\nabla}^{E_0} \in \Gamma(T^* \Sp^{m-1},E_0)$. This observation will be useful to keep in mind as a sanity check for various identities in the next sections.
\end{rem}

    \subsubsection{Dilation-invariant bundle maps}      \label{subsubsec: Dilation-invariant bundle maps}

Throughout this paper, it will be convenient to have a systematic argument at our disposal in order to prove that certain classes of vector bundles are adapted, and that certain sections thereof satisfy appropriate decay conditions. We develop such an argument in the next two sections, through the notions of \emph{dilation-invariant bundle maps} and \emph{bundle maps with dilation-invariant ends}. Let us first recall a few facts about the geometry of nonlinear bundle maps.

From now on, if $B$ is a base manifold (not necessarily compact), $E,F \to B$ are smooth vector bundles and $U \subseteq E$ is an open sub-bundle (which we will always assume to be convex and to contain the zero section $0_E$), we shall refer to any smooth base-point preserving map 
\begin{equation*}
    \bmap{f} \col (U \subseteq E) \to F
\end{equation*} 
as a \emph{(smooth) bundle map}. We emphasize here that we do not require $\bmap{f}$ to be linear along the fibres of $E$. We will also adopt the following notation: if $\bmap{A} \col (U \subseteq E) \to \Hom(E,F)$ is a smooth bundle map and $u \in \Gamma(U \subseteq E)$,  $v \in \Gamma(E)$ are sections such that $u$ takes values in $U$, we will write $\bmap{A}(u)(v) = \bmap{A}(u) \{ v\}$ to make explicit the fact that $\bmap{A}$ depends linearly on $v$ but may not be linear with respect to $u$. In particular, $\bmap{A}(u)\{v_1+ \lambda v_2\} = \bmap{A}(u)\{v_1\} + \lambda \bmap{A}(u) \{v_2\}$ for any $u \in \Gamma(U)$, $v_1,v_2 \in \Gamma(E)$ and $\lambda \in C^\infty(B)$. 

For a (possibly non-linear) bundle map $\bmap{f} \col (U \subseteq E) \to F$, there is a well-defined notion of \emph{vertical derivative}, which is the smooth bundle map $D^{\mathrm{v}} \bmap{f} \col (U \subseteq E) \to \Hom(E,F)$ defined as
\begin{equation}
    D^{\mathrm{v}} \bmap{f}(u)\{v\} = \left. \frac{\diff}{\diff t} \right|_{t=0} \bmap{f}(u+tv), ~~~~ \forall u \in U_b, v \in E_b, \forall b \in B.
\end{equation}
Let us point out here that this definition is intrinsic, in the sense that it does not depend on the choice of a connection on $E$. If in addition the bundles $E, F$ are endowed with connections $\nabla^E, \nabla^F$, then we may also define a notion of \emph{horizontal derivative}, $D^{\mathrm{h}} \bmap{f} \col (U \subseteq E) \to T^*B \otimes F$ which can be characterised as follows:

\begin{lem}     \label{lem: First derivative of bundle map}
    Suppose $(E,\nabla^E), (F,\nabla^F)$ are vector bundles with connections defined over a base manifold $B$ and let $\bmap{f} \col (U \subseteq E) \rightarrow F$ be a (not necessarily linear) bundle map. Then there exist a smooth bundle map
    \begin{equation*}
        D^{\mathrm{h}} \bmap{f} \col (U \subseteq E) \to T^*B \otimes F ,
    \end{equation*}
    called the horizontal derivative of $\bmap{f}$, such that
    \begin{equation*}
        \nabla^{F}(\bmap{f}(u)) = D^{\mathrm{h}} \bmap{f}(u) + D^{\mathrm{v}} \bmap{f}(u)\{\nabla^{E} u\}, ~~~~ \forall u \in \Gamma(U \subseteq E) .
    \end{equation*}
    Moreover, if $\bmap{f}$ vanishes along the zero section $0_E \in \Gamma(E)$ then so does $D^{\mathrm{h}} \bmap{f}_0$.
\end{lem}

\begin{proof}
    Around a point $b \in B$, we may choose coordinates $(x_1,\ldots,x_m)$ and trivialisations $e_1,\ldots,e_I$ of $E$, and $f_1,\ldots,f_J$ of $F$, such that $(\nabla^{E} e_i)_b = 0$ and $(\nabla^{F} f_j)_b = 0$ for any $1 \leq i \leq I$ and $1 \leq j \leq J$. Then for any section $u = \sum_i u_i e_i$ we can write $\bmap{f}(u) = \sum_j F_j(x_k, u_i)f_j$, and at the point $b$ we have
    \begin{equation*}
        (\nabla^E u)_b = \sum_{i} (\diff u_i)_b \otimes e_i, ~~~~ D^{\mathrm{v}} \bmap{f}(u)_b\{e_i\} = \sum_{j} \left. \frac{\partial F_j}{\partial u_i} (x_k,u_i)\right|_b f_j.
    \end{equation*}
    On the other hand,
    \begin{align*}
        \nabla^{F} \bmap{f}(u)_b & = \sum_j \left. \diff (F_j(x_k,u_i)) \right|_b \otimes f_j \\
        & = \sum_{k,j} \left. \frac{\partial F_j}{\partial x_k}(x_k,u_i) \right|_b \diff x_k \otimes f_j +\sum_{i,j} \left.  \frac{\partial F_j}{\partial u_i}(x_k,u_i)\right|_b \diff u_i \otimes f_j  \\
        & = D^{\mathrm{h}} \bmap{f}(u)_b + D^{\mathrm{v}} \bmap{f}(u)_b \{\nabla^{E} u\}
    \end{align*}
    where 
    \begin{equation*}
        D^{\mathrm{h}} \bmap{f}(u)_b = \sum_{k,j} \left. \frac{\partial F_j}{\partial x_k}(x_k,u_i) \right|_b dx_k \otimes f_j .
    \end{equation*}
    This expression defines a smooth bundle map $D^{\mathrm{h}} \bmap{f} \col (U \subseteq E) \to T^*B \otimes F$. In addition, if $\bmap{f}(0_E) \equiv 0$, then $\frac{\partial F_j}{\partial x_k}(x,0) = 0$ and therefore $D^{\mathrm{h}} \bmap{f}(0_E) \equiv 0$.
\end{proof}

Throughout the following sections, the bundle maps that will appear will be locally dilation-invariant near a finite set of singularities. In order to describe the action of such bundle maps on weighted Sobolev sections, we first consider the following model situation on $\R^m \setminus \{0\}$. We consider two vector bundles with metrics and compatible connections $(E_0,\nabla^{E_0},h^{E_0})$ and $(F_0, \nabla^{F_0},h^{F_0})$ over the sphere $\Sp^{m-1}$, and denote by $(E,\nabla^E,h^E)$ and $(F,\nabla^F,h^F)$ the corresponding dilation-invariant vector bundles over $\R^m \setminus \{0\}$. Let $U_0 \subseteq E_0$ be an open convex neighbourhood of the zero section and denote by $U = \pr^*(U_0) \subseteq E$ the associated dilation-invariant open neighbourhood of the zero section in $E$. Given a non-linear bundle map $\bmap{f}_0 \col (U_0 \subseteq E_0) \to F_0$, there is an induced non-linear bundle map $\bmap{f} = \pr^*(\bmap{f}_0) \col (U \subseteq E) \to F$. Such maps will be called \emph{dilation-invariant bundle maps}. Then we have the following result: 

\begin{lem}     \label{lem: Derivative of dilation-invariant bundle maps}
    Let $\bmap{f} \col (U \subseteq E) \to F$ be a dilation-invariant bundle map. Then the horizontal and vertical derivatives of $\bmap{f}$ have the following properties:
    \begin{enumerate}[(i)]
        \item The vertical derivative $D^{\mathrm{v}} \bmap{f} \col (U \subseteq E) \to \Hom(E,F)$ is dilation-invariant.
        \item With respect to the structure of dilation-invariant vector bundle of $T^*(\R^m \setminus \{0\})$ described in \cref{subsubsec: The tangent bundle as adapted bundle}, $r D^{\mathrm{h}} \bmap{f} \col (U \subseteq E) \to T^*(\R^m \setminus \{0\}) \otimes F$ is dilation-invariant.
    \end{enumerate}
\end{lem}

\begin{proof}
    The dilation-invariance of the vertical derivative is immediate from the expression $D^{\mathrm{v}} \bmap{f}(u)\{v\} = \frac{\diff}{\diff t} \bmap{f}(u+tv)$. On the other hand, if $u \in \Gamma(U)$ is a dilation-invariant section, then $\bmap{f}(u)$ is also dilation-invariant, and by \cref{rem: Factors of r} the sections $r \nabla^E u$ and $r\nabla^F (\bmap{f}(u))$ are also dilation-invariant. Thus $r D^{\mathrm{h}} \bmap{f}(u) = r(\nabla^F(\bmap{f}(u)) - D^{\mathrm{v}} \bmap{f}(u)\{\nabla^E u\})$ must be dilation-invariant. Since this hold for any dilation-invariant section $u$, the result follows.
\end{proof}

The previous lemma has for consequence the following technical result, which underlies all the decay and regularity statements of this article:

\begin{cor}     \label{cor: Estimates on derivatives for dilation-invariant maps}
    Let $\bmap{f} \col (U \subseteq E) \to F$ be a dilation-invariant bundle map, let $K_0 \subset U_0 \subseteq E_0$ be a compact convex subbundle, $K \coloneqq \pr^*(K_0) \subset E$, and let $R > 0$. 
    \begin{enumerate}[(i)]
        \item There exists a constant $C > 0$ such that
        \begin{equation*}
            |\bmap{f}(u_1)- \bmap{f}(u_0)| \leq C |u_1-u_0|,
        \end{equation*}
        and
        \begin{equation*}
        	r |\nabla^F(\bmap{f}(u_1)- \bmap{f}(u_0))| \leq C (|u_1-u_0| + r |\nabla^E(u_1-u_0)|) + C (|r\nabla^E u_0| + |r\nabla^E u_1|) |u_1-u_0| 
		\end{equation*}
        for any section $u_0,u_1$ of $K$ defined over the ball $\Ball_R \setminus \{0\}$.

        \item If $k \geq 2$ and $A > 0$, then there exists a constant $C^\prime > 0$ such that, if $u_1,u_0 \in \Gamma(K)$ are sections defined over the ball $\Ball_R \setminus \{0\}$ satisfying $r^\ell|(\nabla^{E})^\ell u_0|, r^\ell |(\nabla^F)^\ell u_1| \leq A$ for all $1 \leq \ell \leq \frac{k}{2}$, then
        \begin{multline*}
            r^k|(\nabla^F)^k (\bmap{f}(u_1) - \bmap{f}(u_0)) | \leq C^\prime \sum_{\ell = 0}^k r^{\ell} |(\nabla^E)^\ell (u_1-u_0)| \\ + C^\prime \left( \sum_{\frac{k}{2} < \ell \leq k} r^\ell|(\nabla^E)^\ell u_0| + r^\ell |(\nabla^E)^\ell u_1| \right) \sum_{0 \leq \ell \leq \frac{k}{2}} r^\ell |(\nabla^E)^\ell(u_1-u_0)|
        \end{multline*}
    \end{enumerate}
    In particular, if $\bmap{f}(0_E) \equiv 0$ and $u \in \Gamma(U \subseteq E)$ satisfies $|(\nabla^E)^k u | = \mathcal{O}(r^{\mu-k})$ as $r \to 0$ for all $k \in \N_0$, for some $\mu \geq 0$, then $\bmap{f}(u) \in \Gamma(F)$ satisfies $|(\nabla^F)^k \bmap{f}(u) | = \mathcal{O}(r^{\mu-k})$ as $r \to 0$ for all $k \in \N_0$.
\end{cor}

\begin{proof}
	Using the vertical derivative $D^{\mathrm{v}} \bmap{f}$ and the convexity of $K_0$, we have for any $u_0,u_1 \in \Gamma(K)$
    \begin{equation*}
        \bmap{f}(u_1)-\bmap{f}(u_0) = \int_0^1 D^{\mathrm{v}} \bmap{f}((1-t)u_0+tu_1)\{u_1-u_0\} \diff t .
    \end{equation*}
    Since $K_0$ is compact and convex, $(1-t)u_0+tu_1$ is a section of $K$ and as $D^{\mathrm{v}} \bmap{f}$ is dilation-invariant, we have $|D^{\mathrm{v}} \bmap{f}((1-t)u_0+tu_1)| \leq C$ for a constant $C > 0$ which depends on $K_0$ but not on $u_0,u_1$. In particular we have
    \begin{equation*}
        |\bmap{f}(u_1) - \bmap{f}(u_0)| \leq \int_0^1 |D^{\mathrm{v}} \bmap{f}((1-t)u_0+tu_1)| \cdot |u_1-u_0| \diff t \leq C |u_1-u_0|
    \end{equation*}
    which proves the first inequality.

    For the second inequality, we use the previous lemma: there exists a dilation-invariant bundle map $\tilde{\bmap{f}} = r D^{\mathrm{h}} \bmap{f} \col (U \subseteq E) \to T^*(\R^{m} \setminus \{0\}) \otimes F$ such that
    \begin{equation*}
        r\nabla^F \bmap{f}(u) = \tilde{\bmap{f}}(u) + D^{\mathrm{v}} \bmap{f}(u) \{ r\nabla^E u \}, ~~~~ \forall u \in \Gamma(U),
    \end{equation*}
    and the vertical derivative $D^{\mathrm{v}} \bmap{f}$ is also dilation-invariant. By the previous inequality, $|\tilde{\bmap{f}}(u_1) - \tilde{\bmap{f}}(u_0)| \leq C |u_1-u_0|$ and 
    \begin{align*}
    	|D^{\mathrm{v}} \bmap{f}(u_1) \{ r\nabla^E u_1\} - D^{\mathrm{v}} \bmap{f} (u_0) \{ r \nabla^E u_0\} | & \leq |D^{\mathrm{v}} \bmap{f}(u_1) - D^{\mathrm{v}} \bmap{f}(u_0)| \cdot |r \nabla^E u_1| \\
		& ~~~~~~~ + |D^{\mathrm{v}} \bmap{f}(u_0)| |r\nabla^E u_1 - r\nabla^E u_0| \\
		& \leq C r |\nabla^E u_1| \cdot |u_1-u_0| + C |r \nabla^E u_1 - r \nabla^E u_0|
    \end{align*}
    for some constant $C > 0$. The second inequality follows.
    
    For the generalisation to $k \geq 2$, an immediate induction from \cref{lem: Derivative of dilation-invariant bundle maps} yields that $(\nabla^E)^k \bmap{f}(u)$ can be written as a sum of terms of the form
    \begin{equation*}
    	r^{-j_0} \bmap{f}^\prime(u) \{ \nabla^{j_1} u \otimes \cdots \otimes \nabla^{j_\ell} u \}
    \end{equation*}
    where $\bmap{f}^\prime$ is a dilation-invariant bundle map defined on the appropriate bundle, and the integers $j_0 \geq 0$, $j_1,\ldots,j_\ell \geq 1$ satisfy $j_0 + \cdots + j_\ell = k$ and $j_1 \leq \cdots \leq j_\ell$. After multiplication with $r^k$, we obtain a term of the form $\bmap{f}^\prime(u) \{ (r^{j_1} \nabla^{j_1} u) \otimes \cdots \otimes (r^{j_\ell} \nabla^{j_\ell} u) \}$, and it is enough to prove the desired inequality holds for any such terms. Remarking that $j_1, \ldots, j_{\ell-1} \leq \frac{k}{2}$, we see that within our assumptions there exists some constant $C^\prime > 0$ such that we have the bound
    \begin{align*}
    	|\bmap{f}^\prime(u_1) \{ (r^{j_1} \nabla^{j_1} u_1) & \otimes \cdots \otimes (r^{j_\ell} \nabla^{j_\ell} u_1) \} - \bmap{f}^\prime(u_0) \{ (r^{j_1} \nabla^{j_1} u_0) \otimes \cdots \otimes (r^{j_\ell} \nabla^{j_\ell} u_0) \}| \\
		& \leq C^\prime |\bmap{f}^\prime(u_1)-\bmap{f}^\prime(u_0)| \cdot |r^\ell (\nabla^E)^\ell u_1| + C^\prime |\bmap{f}^\prime(u_0)| \cdot |r^{j_\ell} (\nabla^E)^{j_\ell} u_1 - |r^{j_\ell} (\nabla^E)^{j_\ell} u_0| \\
		& ~~~~ + C^\prime |\bmap{f}^\prime(u_0)| \cdot |r^{j_\ell} (\nabla^E)^{j_\ell} u_1| \cdot \sum_{i = 1}^{\ell-1} |r^{j_i} (\nabla^E)^{j_i}(u_1-u_0)| \\
		& \leq C^\prime |r^{j_\ell} (\nabla^E)^{j_\ell} (u_1-u_0)| + C^\prime |r^{j_\ell} (\nabla^E)^{j_\ell} u_1| \sum_{0 \leq j \leq \frac{k}{2}} |r^j (\nabla^E)^j(u_1-u_0)|
    \end{align*}
    where the constant $C^\prime$ may change from line to line. The first term is bounded by $C^\prime \sum_{j \leq k} |r^j (\nabla^E)^j(u_1-u_0)|$, and so is the second term if $j_\ell \leq \frac{k}{2}$ (since in that case $|r^{j_\ell} (\nabla^E)^{j_\ell} u_1| \leq A$). Otherwise $j_\ell > \frac{k}{2}$ and the second term above is bounded by the second term in the right-hand side of the inequality claimed in point (ii) of the lemma. This finishes the proof.
\end{proof}

	\subsubsection{Bundle maps with dilation-invariant ends}       \label{subsubsec: Bundle maps with dilation-invariant ends}

In addition to proving that certain vector bundles are adapted, we will often want to know that certain bundle maps between adapted bundles automatically induce smooth Banach maps between weighted Sobolev spaces of sections. In all cases, we will be able to boil this down to the following situation. 

Let $M$ be a compact manifold, $S \subset M$ be a finite subset and let $\{ \Upsilon_s \col \Ball_1 \to M\}_{s \in S}$ be an adapted set of charts. We consider two vector bundles $(E,\nabla^E,h^E)$ and $(F,\nabla^F,h^F)$ over $M \setminus S$ which are not only adapted but have \emph{dilation-invariant ends}: that is, for $R \in (0,1)$ small enough, there exist collections of dilation-invariant vector bundles $\{(E_s,\nabla^{E_s},h^{E_s}) \to \R^m \setminus \{0\}\}_{s \in S}$, $\{(F_s,\nabla^{F_s},h^{F_s}) \to \R^m \setminus \{0\}\}_{s \in S}$ and a collection of trivialisations $\Psi^E_s \col \left. E_s \right|_{\Ball_R} \to E$, $\Psi^F_s \col \left. F_s \right|_{\Ball_R} \to F$ such that
\begin{equation*}
    (\Psi^E_s)^* (\nabla^E,h^{E}) = (\nabla^{E_s},h^{E_s}), ~~~ (\Psi^F_s)^* (\nabla^F,h^F) = (\nabla^{F_s},h^{F_s}).
\end{equation*}
For any $\delta > 0$, we will denote by $U_\delta \subset E$ the open sub-bundle $\{ u \in E \mid |u|_{h^E} < \delta \}$. We similarly define $U_{s,\delta} \subset E_s$.

\begin{Def}		\label{Def: non-linear bundle map with dilation invariant ends}
    Let $\bmap{f} \col (U_\delta \subseteq E) \to F$ be a bundle map. We say that $\bmap{f}$ has \emph{dilation-invariant ends} if there exists a collection of dilation-invariant bundle maps $\bmap{f}_s \col (U_{s,\delta} \subseteq E_s) \to F_s$ such that 
    \begin{equation*}
        (\Psi_s^F)^{-1} \circ \bmap{f} \circ \Psi_s^E (u) = \bmap{f}_s(u), ~~~ \forall u \in U_{s,\delta} .
    \end{equation*}
\end{Def}

Given \cref{cor: Estimates on derivatives for dilation-invariant maps}, it is straightforward to deduce that a bundle map $\bmap{f} \col (U_\delta \subseteq E) \to F$ with dilation-invariant ends such that $\bmap{f}(0_E) = 0_F$ induces, for any $k \in \N_0$ and $\mu \geq 0$, a smooth Banach map $u \in C^{k}_\mu(U_\delta) \mapsto \bmap{f}(u) \in C^k_\mu(F)$, where we denote by $C^k_\mu(U_\delta) \subset C^k_\mu(E)$ the open subset of sections taking values in $U_\delta \subset E$. If $\bmap{f}(0_E) \neq 0_F$, then we can still deduce that the induced map $C^k_0(U_\delta) \to C^k_0(F)$ is a smooth Banach map. For weighted Sobolev spaces however, we need to take restrictions on the exponent, $p$, and the number of derivatives, $k$, for $\bmap{f}$ to define a smooth Banach map between weighted spaces of sections. Just like in the compact (unweighted) case, these restrictions correspond to the ones that are necessary for Sobolev multiplication to be a well-defined continuous bilinear map. The precise statement is as follows:

\begin{lem}     \label{lem: Master lemma}
    Let $\bmap{f} \col (U_\delta \subset E) \to F$ be a bundle map with dilation-invariant ends such that $\bmap{f}(0_E) = 0_F$. For any $k \geq 1$, $p \in [1,\infty)$ and $\mu \geq 0$ such that $p > \frac{2m}{k}$, let us denote by $\mathcal{U}^{k,p}_\mu \subset W^{k,p}_\mu(E)$ the subset
    \begin{equation*}
        \mathcal{U}^{k,p}_{\mu,\delta} \coloneqq \{ u \in W^{k,p}_\mu(E) ~ | ~ u(x) \in U_\delta,  \forall x \in M \}.
    \end{equation*}
    Note that since there is a sequence of continuous embeddings $W^{k,p}_\mu(E) \hookrightarrow C^0_\mu(E) \hookrightarrow C^0_0(E)$, $\mathcal{U}^{k,p}_{\mu,\delta}$ is an open subset of $W^{k,p}_{\mu}(E)$. Then $\bmap{f}$ induces a smooth Banach map $\mathcal{F} \col \mathcal{U}^{k,p}_{\mu,\delta} \to W^{k,p}_\mu(F), ~ u \mapsto \bmap{f}(u)$. Moreover, the derivative $D \mathcal{F} \col \mathcal{U}^{k,p}_{\mu,\delta} \to \mathcal{B}(W^{k,p}_\mu(E),W^{k,p}_\mu(F))$ coincides with the map induced by the vertical derivative of $\bmap{f}$, defined as  $D^{\mathrm{v}} \bmap{f}(u) \{v\} = \frac{\diff}{\diff t} \bmap{f}(u+tv)$:
        \begin{equation*}
            D_u \mathcal{F}\{v\} = D^{\mathrm{v}} \bmap{f}(u) \{v\}, ~~~ \forall u \in \mathcal{U}^{k,p}_{\mu,\delta}, ~ v \in W^{k,p}_\mu(E).
        \end{equation*}
\end{lem}

\begin{proof}
    Since there is a sequence of continuous embeddings $W^{k,p}_\mu(E) \hookrightarrow C^{\lfloor \frac{k}{2} \rfloor}_\mu(E) \hookrightarrow C^{\lfloor \frac{k}{2} \rfloor}_0(E)$, it is easy to deduce from \cref{cor: Estimates on derivatives for dilation-invariant maps}, the density of $C^\infty_c(E) \subset W^{k,p}_\mu(E)$ and the assumption $\bmap{f}(0_E) = 0_F$ that for any $u \in \mathcal{U}^{k,p}_{\mu,\delta}$, $\bmap{f}(u) \in W^{k,p}_\mu(F)$, and moreover for any $\delta^\prime < \delta$ and $A > 0$, there exists a constant $C > 0$ such that $\|\bmap{f}(u_1)-\bmap{f}(u_0)\|_{W^{k,p}_\mu} \leq C \|u_1-u_0\|_{W^{k,p}_\mu}$ for any $u_1,u_0 \in W^{k,p}_\mu(E)$ with $\|u_0\|_{C^{0}}, \|u_1\|_{C^0} \leq \delta^\prime$ and $\|u_0\|_{W^{k,p}_\mu}, \|u_1\|_{W^{k,p}_\mu} \leq A$. Thus $\bmap{f}$ induces a continuous map $\mathcal{F} \col \mathcal{U}^{k,p}_{\mu,\delta} \to W^{k,p}_\mu(F)$.
    
    To prove that $\mathcal{F}$ is $C^1$, we can use the fundamental theorem of calculus along the fibres of $E$: for any continuous sections $u_0, u_1 \in C^0(U_\delta \subset E)$,
    \begin{equation*}
        \bmap{f}(u_1) - \bmap{f}(u_0) = \int_0^1 D^{\mathrm{v}} \bmap{f}(u_t) \{u_1-u_0\} \diff t, ~~~~ u_t = (1-t)u_0 + t u_1 .
    \end{equation*}
    Since $D^{\mathrm{v}} \bmap{f}$ has dilation-invariant ends, it follows from the first part of the proof that $D^{\mathrm{v}} \bmap{f}(u) - D^{\mathrm{v}} \bmap{f}(0_E)$ induces a continuous map $\mathcal{U}^{k,p}_{\mu,\delta} \to W^{k,p}_\mu(\Hom(E,F)) \subset W^{k,p}_0(\Hom(E,F))$. On the other hand, by weighted Sobolev multiplication (\cref{prop: Sobolev embedding and multiplication}) there is a natural bounded linear map $W^{k,p}_0(\Hom(E,F)) \to \mathcal{B}(W^{k,p}_\mu(E), W^{k,p}_\mu(F))$. In addition, since $D^{\mathrm{v}} \bmap{f}(0_E)$ is a smooth section with dilation invariant-ends, $D^{\mathrm{v}} \bmap{f}(0_E) \in C^k_0(\Hom(E,F)) \hookrightarrow \mathcal{B}(W^{k,p}_\mu(E), W^{k,p}_\mu(F))$. This proves that the map $u \mapsto D^{\mathrm{v}} \bmap{f}(u)$ induces a continuous map $\mathcal{U}^{k,p}_{\mu,\delta} \to \mathcal{B}(W^{k,p}_\mu(E),W^{k,p}_\mu(F))$ and therefore $\mathcal{F}$ is continuously differentiable, with derivative $D \mathcal{F}$ induced by the vertical derivative $D^{\mathrm{v}} \bmap{f}$. Since $D^{\mathrm{v}} \bmap{f}$ is also dilation-invariant, an immediate induction applied to $D^{\mathrm{v}} \bmap{f} - D^{\mathrm{v}} \bmap{f}(0_E)$ yields that $\mathcal{F}$ is smooth.
\end{proof}

\begin{rem}     \label{rem: Derivative of Banach maps}
    The above lemma should not be taken as a practical way of computing the derivative of $\mathcal{F}$ by directly calculating the vertical derivative of $\bmap{f}$. Rather, its purpose is to establish that the abstract notion of Banach derivative in weighted Sobolev spaces coincides with the `naive' notion of directional derivative along the fibres of $E$, $D_u \mathcal{F}\{v\} = \tfrac{\diff}{\diff t} \bmap{f}(u+tv)$. The point is that the vertical derivative has a purely local expression, and can be evaluated for instance by calculating the directional derivative $\tfrac{\diff}{\diff t} \bmap{f}(u+tv)$ for smooth sections compactly supported away from the singular set $S$.
\end{rem}

The previous lemma also has a `parametrised version' which will often be useful: 

\begin{lem}     \label{lem: Parametrised master lemma}
    Let $\bmap{f} \col (U_\delta \subseteq E) \to F$ be a bundle map with dilation-invariant ends. Then for any $\mu \geq 0$, $k \in \N$ and $p \in [1,\infty)$ such that $p > \frac{2m}{k}$, the map
    \begin{equation*}
        \widetilde{\mathcal{F}} \col \mathcal{U}^{k,p}_{\mu,\delta/2} \times C^{k}_0(U_{\delta/2}) \to W^{k,p}_{\mu}(F), ~~ (u,w) \mapsto \bmap{f}(u+w) - \bmap{f}(w)
    \end{equation*}
    is a smooth Banach map.
\end{lem}

\begin{proof}
    By the previous lemma, the map $\tilde{\mathcal{F}}$ is well-defined. To prove smoothness, it is enough to prove that for any $0 \leq \ell \leq k$, $(u,w) \mapsto r^\ell(\nabla^F)^\ell (\bmap{f}(u+w)-\bmap{f}(w))$ defines a smooth Banach map from $\mathcal{U}^{k,p}_{\mu,\delta/2} \times C^{k}_0(U_{\delta/2})$ to $L^p_\mu(T^*M^{\otimes \ell} \otimes F)$. We begin with the case $\ell = 0$.
     
     Using again the fundamental theorem of calculus along the fibres of $U_\delta$, we have for any $u, w \in C^0_{\mathrm{loc}}(U_{\delta/2} )\times C^0_{\mathrm{loc}}(U_{\delta/2})$,
    \begin{equation*}
        \bmap{f}(u+w)-\bmap{f}(w) = \int_0^1 D^{\mathrm{v}} \bmap{f}(w+tu) \{u\} \diff t = \tilde{\bmap{f}}(u,w)\{u\}
    \end{equation*}
    for a smooth bundle map $\tilde{\bmap{f}} \col (U_{\delta/2} \oplus U_{\delta/2} \subset E \oplus E) \to \Hom(E,F)$ with  dilation-invariant ends. In particular, $\tilde{\bmap{f}}$ defines a smooth Banach map from $C^0_0(U_{\delta/2}) \times C^0_0(U_{\delta/2}) \to C^0_0(\Hom(E,F))$. Using the continuous embeddings $W^{k,p}_\mu(E) \hookrightarrow C^0_\mu(E) \hookrightarrow C^0_0(E)$ and $C^0_0(\Hom(E,F)) \subset \mathcal{B}(L^p_\mu(E),L^p_\mu(F))$ (cf. \cref{lem: Natural properties of weighted spaces} and \cref{prop: Sobolev embedding and multiplication}), this proves that $(u,w) \mapsto \bmap{f}(u+w)-\bmap{f}(w)$ defines a smooth Banach map from $\mathcal{U}^{k,p}_{\mu,\delta/2} \times C^{k}_0(U_{\delta/2})$ to $L^p_\mu(F)$.
    
    For the covariant derivatives, let us go back to the proof of \cref{cor: Estimates on derivatives for dilation-invariant maps}. We noted that $r^k (\nabla^F)^k \bmap{f}(u)$ can be written as a sum of terms of the form $\bmap{f}^\prime(u) \{ (r^{j_1} (\nabla^E)^{j_1} u) \otimes \cdots \otimes (r^{j_\ell}(\nabla^E)^{j_\ell} u)\}$ where $\bmap{f}^\prime$ is a dilation-invariant bundle map and $1 \leq j_1 \leq \ldots \leq j_\ell$ satisfy $j_1 + \cdots + j_\ell \leq k$. Thus it suffices to prove that any such term defines a smooth Banach map. Expanding the corresponding term in $r^k (\nabla^E)^k (\bmap{f}(u+w)-\bmap{f}(w))$ we obtain two types of terms: either terms of the form
    \begin{equation*}
    	\bmap{f}^\prime(u+w) \{r^{j_1} (\nabla^E)^{j_1} v_1 \otimes \cdots r^{j_\ell} (\nabla^E)^{j_\ell} v_\ell\}
    \end{equation*}
    where $v_i = u ~ \text{or} ~ w$ and at least one $v_i$ is equal to $w$, or
    \begin{equation*}
    	(\bmap{f}^\prime(u+w)-\bmap{f}^\prime(w))\{ r^{j_1} (\nabla^E)^{j_1} w \otimes \cdots r^{j_\ell} (\nabla^E)^{j_\ell} w\} .
    \end{equation*}
    For the first type of term, if $i_0$ is the largest value of $i$ such that $v_i = w$, then $(u,w) \to r^{j_{i_0}} (\nabla^E)^{j_{i_0}} u$ defines a smooth map into $L^p_\mu(T^*M^{\otimes j_{i_0}} \otimes F)$, and using the Sobolev embedding $W^{k,p}_\mu(E) \hookrightarrow C^{\lfloor \frac{k}{2}\rfloor}_\mu(E)$, all the other terms $(u,w) \to r^{j_i} (\nabla^E)^i v_i$ define a smooth map into $C^0_0(T^* M^{\otimes j_i} \otimes F)$ (because $j_i \leq \frac{k}{2}$ for any $i$ such that $v_{j_i} = u$). On the other hand, as before we can use the case $\ell = 0$ and the continuous embedding $W^{k,p}_\mu \hookrightarrow C^0_0$ to show that $(u,w) \to \bmap{f}^\prime(u+w) = \bmap{f}^\prime(w) + (\bmap{f}(u+w)-\bmap{f}(w))$ defines a smooth map into $C^0_0(\cdot)$ of the appropriate bundle. Using the continuous multiplication map $C^{0}_0(\cdot) \times L^p_\mu(\cdot) \to L^p_\mu(\cdot)$ (\cref{lem: Natural properties of weighted spaces}), it follows that the first type of terms define a smooth Banach map into $L^p_\mu(F)$. As for the second term, by the case $\ell = 0$, $\bmap{f}^\prime(u+w)-\bmap{f}^\prime(w)$ defines a smooth Banach map into $L^p_\mu(\cdot)$ of the appropriate bundle, and the tensor product $r^j_1 (\nabla^{j_1} w) \otimes \cdots r^{j_\ell} (\nabla^E)^{j_\ell} w$ defines a smooth Banach map into $C^0_0(\cdot)$, and as above this shows that $(u,w) \mapsto (\bmap{f}^\prime(u+w)-\bmap{f}^\prime(w))\{ r^{j_1} (\nabla^E)^{j_1} w \otimes \cdots r^{j_\ell} (\nabla^E)^{j_\ell} w\}$ defines a smooth Banach map from $\mathcal{U}^{k,p}_{\mu,\delta/2} \times C^{k}_0(U_{\delta/2})$ to $L^p_\mu(F)$. The lemma follows.
\end{proof}

        \subsection{The Jacobi operator of a conically singular harmonic map}       \label{subsec: The Jacobi operator of a CS harmonic map}

In this section, let $(M,g)$ and $(N,h)$ be two compact Riemannian manifolds, and as in the previous section let $S \subset M$ be a finite subset and $\{\Upsilon_s \col \Ball_1 \to M\}_{s \in S}$ an adapted set of charts, which we assume to be compatible with $g$. Our goal is to describe the mapping properties of the Jacobi operator of a conically singular harmonic map $\phi \col M \setminus S \to N$. 

We first prove in \cref{subsubsec: pulled-back tangent bundles} that that pulled-back tangent bundle $\prescript{\phi}{}{TN}$ naturally has the structure of an adapted bundle. We also introduce the notion of \emph{end-conical map}, which we will use in order to construct Banach manifolds of conically singular maps in \cref{sec: Banach manifolds}. We then describe the mapping properties of the Jacobi operator between weighted Sobolev spaces of sections in \cref{subsubsec: Mapping properties}. This is based on the general theory developed by Lockhart and McOwen \cite{Lockhart1985ellipticOperators_on_noncompact_mfds}, which has since then been generalised in the setting of $b$-calculus by Melrose \cite{Melrose1983bCalculus} (see also \cite{Melrose1993AtyiahPatodiSinger}). Good expositions for these materials can also be found in \cite{Bartnik1986mass_of_ALF}, \cite{Donaldson2002FloerHomology}, \cite{KarigiannisLotay2020conifolds}, and \cite{Pacard2008lecture_notes_connected_sums}. Finally, in \cref{subsubsec: pairing}, we discuss certain duality pairings that are useful to analyse the kernel and co-kernel of the Jacobi operator.

    \subsubsection{Pulled-back tangent bundles}     \label{subsubsec: pulled-back tangent bundles}

By means of illustrating how to use the formalism developed in the previous sections, we shall prove that for any conically singular map $\phi \col M \setminus S \to N$, the pulled-back bundle $(\prescript{\phi}{}{TN},\prescript{\phi}{}{\nabla},h_\phi)$ is adapted.

\begin{prop}        \label{prop: The pulled-back bundle has dilation-invariant ends}
    Let $\phi \col M \setminus S \to N$ be a conically singular map with rate $\mu > 0$. Then the pulled-back tangent bundle $(\prescript{\phi}{}{TN},\prescript{\phi}{}{\nabla},h_\phi)$ is adapted. 
    
    More precisely: if $R \in (0,1)$ is small enough so that for any $s \in S$, $\phi \circ \Upsilon_s = \exp_{\phi_s}(u_s)$ for a collection of dilation-invariant maps $\{ \phi_s \col \R^m \setminus \{0\} \to N\}_{s \in S}$ and sections $u_s \in \Gamma(\prescript{\phi_s}{}{TN})$, defined on the ball $\Ball_R \setminus \{0\}$, such that $\|u_s\|_{C^0} < \inj(N,h)$ and $|\prescript{\phi_s}{}{\nabla}^k u_s | = \mathcal{O}(r^{\mu-k})$ for all $k \in \N_0$, then the bundle maps $\Psi_s$ defined as
    \begin{align*}
        \Psi_s \col \prescript{\phi_s}{}{TN}_{|\Ball_R \setminus \{0\}} & \to \prescript{\phi}{}{TN}_{|\Upsilon_s(\Ball_R) \setminus \{s\}}, \\
        (x,v) & \mapsto D_{u_s(x)} \exp_{\phi_s(x)}\{v\} \coloneqq  \left. \tfrac{\diff}{\diff t} \right|_{t=0} \exp_{\phi_s(x)}(u_s(x)+t v),
    \end{align*}
    where $D_{u_s(x)} \exp_{\phi_s(x)}$ defines an isomorphism from $T_{u_s(x)} (T_{\phi_s(x)} N) \simeq T_{\phi_s(x)} N$ onto $T_{\phi(x)} N$, satisfy
    \begin{equation*}
        |\prescript{\phi_s}{}{\nabla}^k(\Psi_s^* h_{\phi} - h_{\phi_s})| = \mathcal{O}(r^{\mu-k}), ~~~ |\prescript{\phi_s}{}{\nabla}^k(\Psi_s^* (\prescript{\phi}{}{\nabla}) - \prescript{\phi_s}{}{\nabla})| = \mathcal{O}(r^{-1+\mu-k}),
    \end{equation*}
    as $r \to 0$ for all $k \in \N_0$.
\end{prop}

\begin{proof}
    Working in a neighbourhood of each singularity, we can consider perturbations of a dilation-invariant map $\phi_0 = \varphi_0 \circ \pr \col \R^m \setminus \{0\} \to N$. Let us define the open sub-bundle 
    \begin{equation*}
        \prescript{\phi_0}{}{U} = \{ u \in \prescript{\phi_0}{}{TN} \mid |u|_{h_{\phi_0}} < \inj(N) \}.
    \end{equation*}
    Then we can define a non-linear bundle map $\bmap{h}_{\phi_0} \col \prescript{\phi_0}{}{U} \to \mathrm{Sym}^2 (\prescript{\phi_0}{}{T^*N})$ by the expression
    \begin{equation*}
        \bmap{h}_{\phi_0}(u) \{v_1,v_2\} = h(D_u \exp_{\phi_0}\{v_1\}, D_u\exp_{\phi_0} \{v_2\}) .
    \end{equation*}
    The bundle map $\bmap{h}_{\phi_0}$ is obviously dilation-invariant, and has the following interpretation: if $u \in \Gamma(\prescript{\phi_0}{}{U})$ and we denote by $\phi_u$ the map $\exp_{\phi_0}(u) \col \R^{m} \setminus \{0\} \to N$ and by $\Psi_u$ the bundle isomorphism $D_u \exp_{\phi_0}$, then $\Psi^*_u h_{\phi_u} = \bmap{h}_{\phi_0}(u)$. Therefore it follows form \cref{cor: Estimates on derivatives for dilation-invariant maps} that for any section $u$ such that $|\prescript{\phi_0}{}{\nabla}^k u| = \mathcal{O}(r^{\mu-k})$ for all $k \in \N_0$, we have $|\prescript{\phi_0}{}{\nabla}^k(\Psi_u^*h_{\phi_u}-h_{\phi_0})| = \mathcal{O}(r^{\mu-k})$. 
    
    We can argue similarly for the estimates on $\Psi^*_u (\prescript{\phi_u}{}{\nabla}) - \prescript{\phi_0}{}{\nabla} \in \Omega^1(\End(\prescript{\phi_0}{}{TN}))$, except that it depends on the first derivative of $u$ as well. More precisely, we claim that there exist two dilation-invariant bundle maps,
    \begin{align*}
        \bmap{A}_{\phi_0} \col \prescript{\phi_0}{}{U} \to T^*\R^m \otimes \End(\prescript{\phi_0}{}{TN}), ~~~ \bmap{B}_{\phi_0} \col \prescript{\phi_0}{}{U} \to \Hom(\prescript{\phi_0}{}{TN}, \End(\prescript{\phi_0}{}{TN}))
    \end{align*}
    such that $\bmap{A}_{\phi_0}$ vanishes along the zero section, and for any $u \in \Gamma(\prescript{\phi_0}{}{U})$ we have
    \begin{equation*}
        \Psi_u^* (\prescript{\phi_u}{}{\nabla}) = \prescript{\phi_0}{}{\nabla} + r^{-1}\bmap{A}_{\phi_0}(u) + \bmap{B}_{\phi_0}(u)\{\prescript{\phi_0}{}{\nabla}u\} .
    \end{equation*}
    Given \cref{cor: Estimates on derivatives for dilation-invariant maps} this would prove the proposition. To prove our claim, we can proceed as in \cref{subsubsec: Dilation-invariant bundle maps}. Using local coordinates as in the proof of \cref{lem: First derivative of bundle map}, it is easy to see that there are bundle maps $\bmap{A}^\prime_{\phi_0}$, $\bmap{B}_{\phi_0}$ such that $\Psi_u^* (\prescript{\phi_u}{}{\nabla}) = \prescript{\phi_0}{}{\nabla} + \bmap{A}^\prime_{\phi_0}(u) + \bmap{B}_{\phi_0}(u)\{\prescript{\phi_0}{}{\nabla}u\}$, and we want to prove that $r \bmap{A}^\prime_{\phi_0}$ and $\bmap{B}_{\phi_0}$ are dilation-invariant. We argue as in the proof of \cref{lem: Derivative of dilation-invariant bundle maps}: if $u_0 \in \Gamma(\prescript{\varphi_0}{}{TN})$ and $u = \pr^* u_0$ is the associated dilation-invariant section of $\prescript{\phi_0}{}{TN}$, then the connection $\Psi_u^*(\prescript{\phi_u}{}{\nabla}) = \pr^*(D_{u_0} \exp_{\varphi_0})^*(\prescript{\varphi_u}{}{\nabla})$ is the pull back of a connection on $\prescript{\varphi_0}{}{TN}$ by the projection $\pr \col \R^m \setminus \{0\} \to \Sp^{m-1}$. By \cref{rem: Factors of r}, it follows that $r(\bmap{A}^\prime_{\phi_0}(u) + \bmap{B}_{\phi_0}(u)\{\nabla^E u\})$ is dilation-invariant. Since $r \nabla^E u$ is also dilation-invariant by the same remark, we easily deduce that $\bmap{A}_{\phi_0} = r \bmap{A}^\prime_{\phi_0}$ and $\bmap{B}_{\phi_0}$ are dilation-invariant. Finally, the vanishing of $\bmap{A}_{\phi_0}$ along the zero section is clear from its definition.
\end{proof}

The previous result also has a useful `global version'. Before stating it, we shall introduce a couple of definitions:

\begin{Def}     \label{Def: End-conical maps I}
    A smooth map $\phi_0 \col M \setminus S \to N$ will be called \emph{end-conical} (with respect to the adapted set of charts $\{\Upsilon_s \col \Ball_1 \to M\}_{s \in S}$) if there exist $R \in (0,1)$ and a set of dilation-invariant maps $\{\phi_s \col \R^m \setminus \{0\} \to N\}$ such that
    \begin{equation*}
        \phi_0 \circ \Upsilon_s (x) = \phi_s(x), ~~~~ \forall x \in \Ball_R \setminus \{0\} .
    \end{equation*}
\end{Def}

\begin{Def}     \label{Def: Bundle isomorphism Theta}
    Let $\phi_0 \col M \setminus S \to N$ be an end-conical map.
    \begin{enumerate}
        \item The open sub-bundle $\prescript{\phi_0}{}{U} \subset \prescript{\phi_0}{}{TN}$ will be defined as
    \begin{equation*}
        \prescript{\phi_0}{}{U} = \{ u \in \prescript{\phi_0}{}{TN} \mid |u|_{h_{\phi_0}} < \inj(N) \} \subset \prescript{\phi_0}{}{TN} .
    \end{equation*}

        \item For any $u \in \Gamma(\prescript{\phi_0}{}{U})$, the map $\exp_{\phi_0}(u) \col M \setminus S \to N$ will be denoted by $\phi_u$.

        \item For any $u \in \Gamma(\prescript{\phi_0}{}{U})$, we define a bundle isomorphism
        \begin{equation*}
            \prescript{\phi_0}{}{\Theta}_u \col \prescript{\phi_0}{}{TN} \to \prescript{\phi_u}{}{TN}, ~ v \mapsto D_u \exp_{\phi_0} \{v\}
        \end{equation*}
        where as in the previous lemma, 
        \begin{equation*}
            D_u \exp_{\phi_0} \{v\} \coloneqq \left. \tfrac{\diff}{\diff t} \right|_{t=0} \exp_{\phi_0}(u+tv)
        \end{equation*}
        is the vertical derivative of the exponential map. In particular for any $x \in M$ we can see $D_{u(x)} \exp_{\phi_0(x)}$ as a linear isomorphism $T_{u(x)}(T_{\phi_0(x)}N) \simeq T_{\phi_0(x)} N \to T_{\exp_{\phi_u(x)}}N$.
    \end{enumerate} 
\end{Def}

In order to see $T^*M$ as an adapted bundle, we shall also fix a background metric $g_0$ on $M$ such that $\Upsilon^*_sg_0 = g_{\R^m}$ on $\Ball_1$ for all $s \in S$ (and not just at the origin), and use parallel transport in the radial directions with respect to the Levi-Civita connection $\nabla^{g_0}$ as in \cref{subsubsec: The tangent bundle as adapted bundle}. Then essentially the same proof as the previous lemma yields the following result:

\begin{lem}     \label{lem: Bundle maps of end-conical maps}
    Let $\phi_0 \col M \setminus S \to N$ be an end-conical map. Then there exist non-linear bundle maps with dilation-invariant ends,
    \begin{align*}
        \bmap{h}_{\phi_0} & \col \prescript{\phi_0}{}{U} \to \mathrm{Sym}^2 (\prescript{\phi_0}{}{TN})^*, \\
        \bmap{A}_{\phi_0} & \col \prescript{\phi_0}{}{U} \to T^*M \otimes \End(\prescript{\phi_0}{}{TN}), \\
        \bmap{B}_{\phi_0} & \col \prescript{\phi_0}{}{U} \to \Hom(\prescript{\phi_0}{}{TN},\End(\prescript{\phi_0}{}{TN}))
    \end{align*}
    such that for any $u \in \Gamma(\prescript{\phi_0}{}{U})$, 
    \begin{equation*}
        \prescript{\phi_0}{}{\Theta}_u^* h_{\phi_u} = h_{\phi_0} + \bmap{h}_{\phi_0}(u), ~~~ \prescript{\phi_0}{}{\Theta}_u^*(\prescript{\phi_u}{}{\nabla}) = \prescript{\phi_0}{}{\nabla} + r^{-1} \bmap{A}_{\phi_0}(u) + \bmap{B}_{\phi_0}(u) \{ \prescript{\phi_0}{}{\nabla}u\} .
    \end{equation*}
    Moreover, $\bmap{h}_{\phi_0}$ and $\bmap{A}_{\phi_0}$ vanish along the zero section of $\prescript{\phi_0}{}{TN}$.
\end{lem}

The previous lemma has the following notable consequences:

\begin{lem}
    Let $u \in C^\infty_\mu(\prescript{\phi_0}{}{TN})$ for some $\mu > 0$ and suppose that $\|u\|_{C^0} < \inj(N)$. Then the map $\phi_u = \exp_{\phi_0}(u)$ is conically singular with rate $\mu$, and moreover for any $k \in \N_0$ we have
    \begin{equation*}
        |\prescript{\phi_0}{}{\nabla}^k(\prescript{\phi_0}{}{\Theta}_u^* h_{\phi_u} - h_{\phi_0}) | = \mathcal{O}(r^{\mu-k}), ~~ \text{and} ~~ |\prescript{\phi_0}{}{\nabla}^k(\prescript{\phi_0}{}{\Theta}_u^* (\prescript{\phi_u}{}{\nabla}) - \prescript{\phi_0}{}{\nabla})| = \mathcal{O}(r^{-1+\mu-k}) .
    \end{equation*}
    Moreover, for any $k \in \N_0$, $p \in [1,\infty)$ and $\nu \in \R$, $\prescript{\phi_0}{}{\Theta}_u$ induces an isomorphism of Banach spaces $W^{k,p}_\nu(\prescript{\phi_0}{}{TN}) \to W^{k,p}_\nu(\prescript{\phi_u}{}{TN})$.
\end{lem}

\begin{proof}
    The estimates on $\prescript{\phi_0}{}{\Theta}_u^*h_{\phi_u}-h_{\phi_0}$ and $\prescript{\phi_0}{}{\nabla}^k(\prescript{\phi_0}{}{\Theta}_u^* (\prescript{\phi_u}{}{\nabla}) - \prescript{\phi_0}{}{\nabla})$ follow from the previous lemma and \cref{cor: Estimates on derivatives for dilation-invariant maps}. On the other hand, since $\|u\|_{C^0} < \inj(N)$, $\prescript{\phi_0}{}{\Theta}_u \in \Hom(\prescript{\phi_0}{}{TN},\prescript{\phi_u}{}{TN})$ is invertible, and thus the pair of metric and compatible connections $\prescript{\phi_0}{}{\Theta}_u^*(\prescript{\phi_u}{}{\nabla},h_{\phi_u})$ define equivalent notions of weighted Sobolev norms as $(\prescript{\phi_0}{}{\nabla},h_{\phi_0})$. Thus $\prescript{\phi_0}{}{\Theta}_u$ induces a isomorphisms of Banach spaces between weighted Sobolev spaces of sections.
\end{proof}

A variant of the above lemma that will also be useful in order to deform the ends of conically singular maps is when the section $u \in \Gamma(\prescript{\phi_0}{}{TN})$ does not decay near the singularity, but instead is almost dilation-invariant, while remaining smaller in $C^0$-norm than the injectivity radius of $N$.

\begin{lem}     \label{lem: theta induces iso of Banach spaces}
    Let $u \in C^\infty_0(\prescript{\phi_0}{}{TN})$ be a section such that $\|u\|_{C^0} < (1-\delta) \inj(N,h)$ for some $\delta \in (0,1)$, and and assume that there exists $\tilde{u} \in C^\infty_\mu(\prescript{\phi_0}{}{TN})$ and a collection of dilation-invariant sections $v_s \in \Gamma(\prescript{\phi_s}{}{TN})$ such that for all $s \in S$, $u - \tilde{u} = \Psi_s(v_s)$, where $\Psi_s$ is defined as in \cref{prop: The pulled-back bundle has dilation-invariant ends}. Then $\phi_u = \exp_{\phi_0}(u)$ is conically singular with rate $\mu$, and $\prescript{\phi_0}{}{\Theta}_u$ induces an isomorphisms of Banach spaces $W^{k,p}_\nu(\prescript{\phi_0}{}{TN}) \to W^{k,p}_\nu(\prescript{\phi_u}{}{TN})$ for all $k \in \N_0$, $p \in [1,\infty)$ and $\nu \in \R$.
\end{lem}

\begin{proof}
    Using the previous lemma, it is enough to treat the case when the section $u$ has dilation-invariant ends, that is, $u = \Psi_s(v_s)$ in a neighbourhood of each $s \in S$. Then it is clear that $\phi_u$ is also an end-conical map, and $\prescript{\phi_0}{}{\Theta}_u \in \Hom(\prescript{\phi_0}{}{TN},\prescript{\phi_u}{}{TN})$ is dilation-invariant near the singular set $S$, and thus by \cref{lem: Natural properties of weighted spaces} it induces a bounded linear map between weighted spaces of sections. Moreover, since $\|u\|_{C^0} < (1-\delta) \inj(N,h)$ the bundle map $\prescript{\phi_0}{}{\Theta}_u$ is invertible and the inverse is dilation-invariant near $S$, whence the inverse induces a bounded linear map on weighted Sobolev spaces.
\end{proof}

    \subsubsection{Mapping properties of the Jacobi operator}      \label{subsubsec: Mapping properties}

In this section, let us fix a conically singular harmonic map $\phi \col (M \setminus S, g) \to (N,h)$ with rate $\mu > 0$. By \cref{prop: The pulled-back bundle has dilation-invariant ends}, the pulled-back bundle $\prescript{\phi}{}{TN}$ is adapted, and we have a family of bundle isomorphisms $\Psi_s \col \prescript{\phi_s}{}{TN}_{|\Ball_R} \to \prescript{\phi}{}{TN}_{|\Upsilon_s(\Ball_R) \setminus \{s\}}$ covering the adapted set of charts $\{\Upsilon_s \col \Ball_R \to M\}_{s \in S}$ for some $R \in (0,1)$, such that $\Upsilon_s^{*}g(0) = g_{\R^m}$ and $\phi_s \col \R^m \setminus \{0\} \to N$ are the dilation-invariant harmonic maps modelling the singularities of $\phi$. Our goal in this section is to study the properties of the Jacobi operator $J^\phi \col \Gamma(\prescript{\phi}{}{TN}) \to \Gamma(\prescript{\phi}{}{TN})$. Its properties are closely related to those of the model Jacobi operators $J^{\phi_s}$. Let us define for any $s \in S$ the operator $J^s = \Psi^{-1}_s \circ J^\phi \circ \Psi_s$, which is an elliptic operator acting on $\prescript{\phi_s}{}{TN}_{|\Ball_R}$. Then using similar arguments as in the previous part, one can see that $J^s$ may be written as
\begin{equation*}
    J^s = J^{\phi_s} + P_s
\end{equation*}
where the differential operator $P_s$ is of the form
\begin{equation*}
    P_s(v) = \bmap{P}_2(\prescript{\phi_s}{}{\nabla}^2 u) + r^{-1} \bmap{P}_1(\prescript{\phi_s}{}{\nabla} u) + r^{-2} \bmap{P}_0(u)
\end{equation*}
where the coefficients have the following decay properties:
\begin{itemize}
	\item $\bmap{P}_0 \in C^\infty_\mu(\Hom(\prescript{\phi_s}{}{TN},\prescript{\phi_s}{}{TN}))$,
    \item $\bmap{P}_1 \in C^\infty_\mu(\Hom(T^*M \otimes \prescript{\phi_s}{}{TN}, \prescript{\phi_s}{}{TN}))$, and
        \item $\bmap{P}_2 \in C^\infty_\mu(\Hom(T^*M^{\otimes 2} \otimes \prescript{\phi_s}{}{TN}, \prescript{\phi_s}{}{TN}))$.
\end{itemize}
Hence the operator $J^\phi$ is asymptotic to the operator $J^{\phi_s} = J^{\varphi_s}_C$ near each $s \in S$, where we write $\phi_s = \varphi_s \circ \pr$ for a smooth harmonic map $\varphi_s \col \Sp^{m-1} \to (N,h)$, up to error terms decaying like $r^\mu$. For such operators on adapted bundles, there is a well-established theory describing the mapping properties on weighted Sobolev spaces. We shall not attempt to give a detailed account of the theory in this section, and will restrict ourselves to stating the results relevant to us in the specific case of the Jacobi operator; see for instance \cite{Lockhart1985ellipticOperators_on_noncompact_mfds,Melrose1983bCalculus} for original references. 

From the expansion of the operator near the singularities, it is easy to deduce for instance from \cref{lem: Natural properties of weighted spaces} that $J^\phi$ induces a bounded map $J^\phi \col W^{k+2,p}_\nu(\prescript{\phi}{}{TN}) \to W^{k,p}_{\nu-2}(\prescript{\phi}{}{TN})$ for any $k \in \N_0$, $p \in (1,\infty)$ and $\nu \in \R$. Moreover, analogously to elliptic operators on compact manifolds, there are \emph{a priori estimates} of the form
\begin{equation}
    \|u\|_{W^{k+2,p}_\nu} \leq C (\|J^\phi u\|_{W^{k,p}_{\nu-2}} + \|u\|_{L^p_\nu}) 
\end{equation}
for any $u \in W^{k+2,p}_\nu(\prescript{\phi}{}{TN})$, where the constant $C > 0$ depends on $k,p,\nu$ but not on the choice of section. Unlike the compact case however, these estimates do not imply that $J^\phi$ induces a Fredholm map on weighted Sobolev spaces, essentially because the embedding $W^{2,p}_\nu(\prescript{\phi}{}{TN}) \hookrightarrow L^p_\nu(\prescript{\phi}{}{TN})$ fails to be compact. Instead, the Fredholm property depends on the choice rate $\nu \in \R$ and the behaviour of the model Jacobi operators $J^{\phi_s}$ near the singular set $S$.

In the remainder of this paper, we shall call $\{J^{\phi_s}\}_{s \in S}$ the set of \emph{indicial operators} of $J^\phi$, and $\nu \in \R$ an indicial root if it is a root of one of the indicial operators $J^{\phi_s} = J^{\varphi_s}_C$. The set of indicial roots will be denoted by $\mathcal{D}(J^{\phi})$. The Fredholm properties of the Jacobi operator are described in the following proposition, which from a general theorem of Lockhart--McOwen \cite{Lockhart1985ellipticOperators_on_noncompact_mfds}:

\begin{prop}\label{prop: mapping properties Jacobi operator}
    Let $\nu \in \R$, $k \in \N_0$ and $p \in (1,\infty)$. 
    \begin{enumerate}[(i)]
        \item The Jacobi operator $J^\phi$ induces a Fredholm map $J^\phi_\nu \col W^{k+2,p}_\nu(\prescript{\phi}{}{TN}) \rightarrow W^{k,p}_{\nu-2}(\prescript{\phi}{}{TN})$ if and only if $\nu$ is not an indicial root.

        \item If $\nu$ is not an indicial root, $\im(J^\phi_\nu) \subset W^{k,p}_{\nu-2}(\prescript{\phi}{}{TN})$ is the $L^2$-orthogonal complement of the finite-dimensional subspace $\ker(J^\phi) \cap C^{\infty}_{-m-\nu+2}(\prescript{\phi}{}{TN})$.

        \item If $\nu \in \mathcal{D}(J^\phi)$ and there are no other roots in the interval $[\nu-\varepsilon,\nu+\varepsilon]$, then
        \begin{equation*}
            \Ind(J^\phi_{\nu-\varepsilon})-\Ind(J^\phi_{\nu+\varepsilon}) = \sum_{s \in S} \dim(\mathcal{K}_\nu(J^{\varphi_s}_C)).
        \end{equation*}
    \end{enumerate}
\end{prop}

\begin{rem}
    It may be helpful to comment on the orthogonality condition in point (ii) and the duality between rate $\nu-2$ and rate $m-\nu+2$. If $J^\phi u \in W^{k,p}_{\nu-2}$ where $u \in W^{k+2,p}_\nu(\prescript{\phi}{}{TN})$ and $v \in \ker(J^\phi) \cap C^\infty_{-m-\nu+2}(\prescript{\phi}{}{TN})$, one can prove (cf. the next lemma) that since $\nu$ is not an indicial root, $v \in C^\infty_{-m-\nu+2+\varepsilon}(\prescript{\phi}{}{TN})$ for some small $\varepsilon > 0$. In particular, since $C^\infty_{-m-\nu+2+\varepsilon} \subset L^q_{-m-\nu+2+\varepsilon}$ for any $q > 1$, H\"older's inequality implies that $\langle J^\phi u, v \rangle_{h_\phi}$ is $L^1$; and in fact we have just enough control on the asymptotic behaviour to integrate by part and prove that $\langle J^\phi u, v \rangle_{L^2} = \langle u, J^\phi v \rangle_{L^2} = 0$. This implies that the image $\im(J^\phi_\nu) \subset W^{k,p}_{\nu-2}(\prescript{\phi}{}{TN})$ must be orthogonal to $\ker(J^\phi) \cap C^{\infty}_{-m-\nu+2}(\prescript{\phi}{}{TN})$. The difficult part is to prove that this orthogonality condition precisely characterises the image.
\end{rem}

Lockhart--McOwen theory also allows us to analyse the highest order terms of the kernels of the Jacobi operator:

\begin{prop}     \label{lem: kernels at different rates and leading order term}
    Let $\nu < \nu^\prime \in \R$.
    \begin{enumerate}[(i)]
        \item Suppose that there are no critical rates in the interval $[\nu,\nu^\prime]$. Then $\ker(J^\phi_\nu) = \ker(J^\phi_{\nu^\prime})$. 
        \item Suppose that $\nu \in \mathcal{D}(J^\phi)$, and let $\varepsilon \in (0,\mu)$ be small enough that $[\nu-\varepsilon,\nu+\varepsilon] \cap \mathcal{D}(J^\phi) = \{\nu\}$.  Then there is a map $\kappa_\nu = (\kappa^s_\nu)_{s\in S} \col \ker(J^\phi) \cap C^\infty_{\nu-\varepsilon}(\prescript{\phi}{}{TN}) \to \oplus_{s \in S} \mathcal{K}_\nu(J^{\varphi_s}_C)_\nu$ such that the following holds. For any $u \in \ker(J^\phi_\nu)$, there exists a section $\tilde u \in C^\infty_{\nu+\varepsilon}(\prescript{\phi}{}{TN})$, such that near any $s \in S$,
        \begin{equation*}
            u - \tilde{u} = \Psi_s(\kappa^s_\nu(u)) .
        \end{equation*}
    \end{enumerate}
\end{prop}

    \subsubsection{Duality pairings between kernel and obstructions}	\label{subsubsec: pairing}

We finish this section with a few technical results that will be useful in order to analyse the obstructions of the linearised deformation problem in \cref{sec: The deformation theorem}.

\begin{prop}
    Let $\nu \in \mathcal{D}(J^\phi)$, and let $\varepsilon \in (0,\mu)$ such that $[\nu-\varepsilon, \nu+\varepsilon] \cap \mathcal{D}(J^\phi) = \{\nu\}$. Suppose that $u \in \ker(J^\phi_{2-m-\nu-\varepsilon}) \subset C^\infty_{2-m-\nu-\varepsilon}(\prescript{\phi}{}{TN})$ and let $ v \in C^\infty_{\nu-\varepsilon}(\prescript{\phi}{}{TN})$ such that there exists $\tilde v \in C^\infty_{\nu+\varepsilon}(\prescript{\phi}{}{TN})$ and $(v_s)_{s \in S} \in \oplus_{s \in S} \mathcal{K}(J^{\varphi_s}_C)_\nu$ such that near $s \in S$, $v - \tilde v = \Psi_s v_s$.  Then the function $\langle J^\phi(v), u \rangle_{h_\varphi}$ is integrable on $M$, and moreover
    \begin{equation*}
        \int_M \langle J^\phi(v), u \rangle_{h_\varphi} \vol_{g} = \sum_{s \in S} (v_s,\kappa_{2-m-\nu}^s(u))_\nu .
    \end{equation*}
\end{prop}

\begin{proof}
    Given the expansion of $v$ near each singularity, $J^\phi(v) \in C^\infty_{\nu-2+\varepsilon}(\prescript{\phi}{}{TN})$ for any small enough $\varepsilon > 0$, and $u \in C^\infty_{2-m-\varepsilon-\varepsilon/2}(\prescript{\phi}{}{TN})$. Thus $|\langle J^\phi v, u \rangle_{h_\varphi}| = O(r_s^{-m+\varepsilon/2})$ near any $s \in S$, and therefore the function is integrable. 

    In order to calculate the integral, let $\chi$ be a cutoff function with $\chi(t) = 1$ for $t \in (-\infty,\tfrac{1}{2}]$ and $\chi = 0$ for $t \in [1,+\infty)$, and let $\chi_\tau = \chi(t/\tau)$. Then for any $\tau > 0$ small enough, the function $\chi_\tau(r)$ is smooth, and by integration by parts we have
    \begin{align*}
        \int_M \langle J^\phi(u), v \rangle \vol_{g} & = \int_M \langle J^\phi((1-\chi_\tau(r))u), v \rangle \vol_{g} + \int_M \langle J^\phi(\chi_\tau(r)u), v \rangle \vol_{g} \\
        & = \int_M \langle J^\phi(\chi_\tau(r)u), v \rangle \vol_{g} .
    \end{align*}
    Therefore
    \begin{equation*}
        \int_M \langle J^\phi(u), v \rangle \vol_{g} = \lim_{\tau \rightarrow 0} \int_M \langle J^\phi(\chi_\tau(r)u), v \rangle \vol_{g} = \lim_{\tau \rightarrow 0} \sum_{s \in S} \int_{\Upsilon_s(\Ball_R)} \langle J^\phi(\chi_\tau(r_s)v), u \rangle_{h_\phi} \vol_{g}
    \end{equation*}
    where $\Upsilon_s \col \Ball_R \to M$ form an adapted set of charts near the singular set $S$ and we can assume without loss of generality that $\Upsilon^*_sg(0) = g_{\R^m}$. Using the identification $\Psi_s \col \prescript{\phi_s}{}{TN}_{|\Ball_R} \rightarrow \prescript{\phi}{}{TN}_{|\Upsilon_s(\Ball_R)}$ described in \cref{subsubsec: Mapping properties}, we deduce from the definition of the pairing $(\cdot,\cdot)_\nu$ that (for $\tau > 0$ small enough) we can write
    \begin{multline*}
        \int_{\Upsilon_s(\Ball_R)} \langle J^\phi(\chi_\tau(r_s)v), u \rangle_{h_\phi} \vol_{g} - (v_s, \kappa_{2-m-\nu}^s(u) )_\nu= \\ \int_0^R \int_{\Sp^{m-1}} \left( \langle J^s(\chi_\tau(r) \Psi_s^{-1}v), \Psi_s^{-1} u \rangle_{\Psi_s^*h_{\phi}} \sqrt{\det(g_{ab})} - \langle J^{\varphi_s}_C(\chi_\tau(r) v_s, \kappa^s_{2-m-\nu}(u) \rangle_{h_{\phi_s}} \right) r^{m-1} \vol_{\Sp^{m-1}} \diff  r .
    \end{multline*}
    where $g_{ab}$ are the coefficients of the metric $g$ in coordinates, and $J^s = \Psi_s^{-1} \circ J^\phi  \circ \Psi_s$. Given the expansions of $J^s$, $u$, $v$ and $g_{ab}$ near $s$, we may deduce from \cref{cor: Estimates on derivatives for dilation-invariant maps} that the quantity $I_\tau$ between parentheses in the integrand at the right-hand side of the above identity may be written as 
    \begin{equation*}
        I_\tau = \tau^{-2}\chi^{\prime\prime}(\tfrac{r}{\tau}) a + \tau^{-1} \chi^\prime(\tfrac{r}{\tau})b + \chi(\tfrac{r}{\tau})c,
    \end{equation*}
    where the functions $a, b, c$ satisfy $a \in C^\infty_{-m+2+\varepsilon}$, $b \in a \in C^\infty_{-m+1+\varepsilon}$ and $c \in C^\infty_{-m + \varepsilon}$ for some small enough $\varepsilon > 0$. It follows that 
    \begin{multline*}
        \int_0^\tau \int_{\Sp^{m-1}} r^{m-1} \chi^{\prime\prime}(\tfrac{r}{\tau})a \vol_{\Sp^{m-1}} \diff  r = O(\tau^{2+\varepsilon}), ~~ \int_0^\tau \int_{\Sp^{m-1}} r^{m-1} \chi^\prime(\tfrac{r}{\tau}) b \vol_{\Sp^{m-1}} \diff  r= O(\tau^{1+\varepsilon}), \\ \int_0^\tau \int_{\Sp^{m-1}} r^{m-1} \chi(\tfrac{r}{\tau}) c \vol_{\Sp^{m-1}} \diff  r  = O(\tau^\varepsilon) .
    \end{multline*}
    Thus $\int_{\Upsilon_s(\Ball_R)} I_\tau \vol_{g} = O(\tau^\varepsilon)$ and thus $\lim_{\tau \to 0} \int_{\Upsilon_s(\Ball_R)} I_\tau \vol_{g} = 0$, which finishes the proof of the proposition.
\end{proof}

We will only need the following consequences of the previous proposition:

\begin{cor}     \label{cor: obstruction pairing}
    Let us assume that $m = \dim(M) \geq 4$ and that for any $s \in S$, the harmonic map $\varphi_s$ satisfies the \cref{assumption: Eigenvalues of varphi}. 
    \begin{enumerate}[(i)]
        \item Let $u \in \ker(J^\phi_{-m+2-\varepsilon})$ and let $u_s \in \ker(J^{\varphi_s})$ such that $\kappa^s_{-m+2}(u) = r^{-m+2} u_s$. Moreover, let $v \in C^\infty_0(\prescript{\phi}{}{TN})$ such that there exist $\tilde v \in C^\infty_\varepsilon(\prescript{\phi}{}{TN})$ and $(v_s) \in \oplus_s \ker(J^{\varphi_s})$ such that $v-\tilde v = \Psi_s v_s$ near $s \in S$. Then
        \begin{equation*}
            \int_M \langle J^\phi v, u \rangle_{h_\phi} \vol_{g} = - (m-2) \sum_{s \in S} \langle v_s, u_s \rangle_{L^2} .
        \end{equation*}

        \item Suppose that $m \geq 5$. Let $u \in \ker(J^\phi_{-m+3-\varepsilon})$ and let $u_s \in \ker(J^{\varphi_s}+m-3)$ such that $\kappa^s_{-m+3}(u) = r^{-m+3} u_s$. Moreover, let $v \in C^\infty_{-1}(\prescript{\phi}{}{TN})$ such that there exist $\tilde v \in C^\infty_{-1+\varepsilon}(\prescript{\phi}{}{TN})$ and $(v_s) \in \oplus_s \ker(J^{\varphi_s}+m-3)$ such that $v-\tilde v = r_s^{-1}\Psi_s v_s$ near $s \in S$. Then
        \begin{equation*}
            \int_M \langle J^\phi v, u \rangle_{h_\phi} \vol_{g} = - (m-4) \sum_{s \in S} \langle v_s, u_s \rangle_{L^2} .
        \end{equation*}

        \item Suppose that $m = 4$. Let $u \in \ker(J^\phi_{-1-\varepsilon})$ and let $u^0_s,u_s^1 \in \ker(J^{\varphi_s}+1)$ such that $\kappa^s_{-1}(u) = r^{-1} (u^0_s+\log(r) u^1_s)$. Moreover, let $v \in C^\infty_{-1}(\prescript{\phi}{}{TN})$ such that there exist $\tilde v \in C^\infty_{-1+\varepsilon}(\prescript{\phi}{}{TN})$ and $(v_s) \in \oplus_s \ker(J^{\varphi_s}+1)$ such that $v-\tilde v = r_s^{-1}\Psi_s v_s$ near $s \in S$. Then
        \begin{equation*}
            \int_M \langle J^\phi v, u \rangle_{h_\phi} \vol_{g} = - \sum_{s \in S} \langle u_s^1, v_s \rangle_{L^2} .
        \end{equation*}
    \end{enumerate}
\end{cor}

\begin{proof}
    This is an immediate consequence of the previous proposition and \cref{prop: Polyhomoegenous pairing}.
\end{proof}

%% file: section3_banach.tex
    \section{Banach manifolds of conically singular maps}       \label{sec: Banach manifolds}

The goal of this section is to construct the parameter spaces that will be used in the proof of the deformation theorem in \cref{sec: The deformation theorem}. In \cref{subsec: manifold structure}, we first construct Banach manifolds of conically singular maps $\phi \col M \setminus S \to N$, where the singular set $S$ is fixed but the tangent maps are allowed to vary within finite-dimensional subspaces of $C^\infty(\Sp^{m-1},N)$. For this we need to work with maps that have finite regularity, and our choice is to work with Sobolev maps. We then explain in \cref{subsec: variations of the tension} how to regard the tension of a conically singular map as a section of a Banach vector bundle lying over the entire parameter space combining the previously constructed Banach manifolds of conically singular maps, the deformations of the singular set and the variations of the background metric on $M$. The precise statement of this result is \cref{thm: Variations of the tension}, which also computes the derivative of the tension at a conically singular \emph{harmonic} map. This sets the stage for an application of the Implicit Function Theorem in order to prove \cref{thm: main theorem introduction} in the next section.

    \subsection{Manifold structure}     \label{subsec: manifold structure}

This section constructs Banach manifolds of conically singular Sobolev maps $\phi \col M \setminus S \to N$ with fixed singular set and tangent maps depending on finite-dimensional parameter spaces. \cref{subsubsec: Definitions and atlases} sets the precise definitions and describes the manifold structure. We then describe a useful class of Banach vector bundles in \cref{subsubsec: Banach vector bundles}. Finally, we explain how to take into account local deformations of the singular set in \cref{subsubsec: Moving the singular points}.

    \subsubsection{Definitions and atlases}     \label{subsubsec: Definitions and atlases}

In this section, let us fix a Riemannian manifold $(N,h)$, and a compact manifold $M$, of dimension $\dim(M) = m \geq 4$, together with a data set $\mathfrak{S} = (S,\{\Upsilon_s\}_{s \in S},\{\mathcal{M}_s\}_{s\in S})$ consisting of:
\begin{enumerate}
    \item a finite subset $S \subset M$,
    \item an adapted set of charts $\{\Upsilon_s \col (\Ball_1 \subset\mathbb{R}^m) \to M\}_{s\in S}$ as in \cref{Def: Adapted set of charts}.
    \item a set of finite-dimensional submanifolds $\{\mathcal{M}_s \subset C^\infty(\Sp^{m-1},N)\}_{s \in S}$.
\end{enumerate} 
The submanifolds $\{\mathcal{M}_s\subset C^\infty(\Sp^{m-1},N)\}_{s \in S}$ equivalently correspond to a collection of finite-dimensional manifolds $\{\mathcal{M}_s\}_{s \in S}$ together with smooth maps $\underline{\varphi_s } \col \mathcal{M}_s \times \Sp^{m-1} \to N$, such that the induced map $\mathcal{M}_s \to C^\infty(\Sp^{m-1},N)$ is an embedding (an injective immersion, proper onto its image with respect to the Fr\'echet topology). For simplicity, we shall henceforth denote the product $\times_{s \in S} \mathcal{M}_s$ by $\mathcal{M}$.

\begin{Def}     \label{Def: manifold structure for Sobolev maps with variable tangent}
    Let $p \in (m,\infty)$ and $\mu \in (0,1)$ be fixed throughout \cref{sec: Banach manifolds}. 
    \begin{enumerate}
        \item We shall denote by $\Map^\infty_\mu(M,N;\mathfrak{S})$ the set of smooth conically singular maps with rate $\mu$ and tangent maps determined by the data set $\mathfrak{S}$; that is, there exists $R \in (0,1)$, $(\varphi_s)_{s \in S} \in \mathcal{M}$ and sections $u_s \in \Gamma(\prescript{\phi_s}{}{TN})$, where $\phi_s = \varphi_s \circ \pr$, such that $\phi \circ \Upsilon_s = \exp_{\phi_s}(u_s)$ on $\Ball_R \setminus \{0\}$, and $|\prescript{\phi_s}{}{\nabla}^k u_s| = \mathcal{O}(r^{\mu-k})$ for all $k \in \N_0$.

        \item We shall denote by $\Map^\infty_{\mathrm{ec}}(M,N;\mathfrak{S}) \subset \Map^\infty_\mu(M,N;\mathfrak{S})$ the set of end-conical maps $\phi_0 \col M \setminus S \to N$ (cf. \cref{Def: End-conical maps I}) with tangent maps determined by the data set $\mathfrak{S}$; that is, a smooth map $\phi_0 \col M \setminus S \to N$ is in $\Map^{\infty}_{\mathrm{ec}}(M,N;\mathfrak{S})$ if and only if there exists $R \in (0,1)$ and $(\varphi_s)_{s \in S} \in \mathcal{M}$ such that $\left. \phi \circ \Upsilon_s  \right|_{\Ball_R} = \left. \varphi_s \circ \pr \right|_{\Ball_R}$ for all $s \in S$.

        \item For any $k \geq 2$, we shall denote by $\Map^{k,p}_\mu(M,N;\mathfrak{S})$ the set of continuous maps $\phi \col M \setminus S \to N$ such that there exists an end-conical map $\phi_0 \in \Map^\infty_{\mathrm{ec}}(M,N;\mathfrak{S})$ and a section $u \in \prescript{\phi_0}{}{\mathcal{U}}^{k,p}_\mu$, where
        \begin{equation*}
            \prescript{\phi_0}{}{\mathcal{U}}^{k,p}_\mu \coloneqq \{ u \in W^{k,p}_\mu(\prescript{\phi_0}{}{TN}) \mid \|u\|_{C^0} \leq \tfrac{1}{4} \inj(N) \},
        \end{equation*}
        such that
        \begin{equation*}
            \phi = \phi_u \coloneqq \exp_{\phi_0}(u) .
        \end{equation*}
        Such maps we be called conically singular of regularity $W^{k,p}_\mu$ with tangent maps determined by the data set $\mathfrak{S}$.
    \end{enumerate}
\end{Def}

\begin{rem}
    The above definition of $\Map^{k,p}_\mu(M,N;\mathfrak{S})$ is sensical since there is a continuous embedding $W^{k,p}_\mu(\prescript{\phi_0}{}{TN}) \hookrightarrow C^{k-1}_\mu(\prescript{\phi_0}{}{TN}) \subset C^1_\mu(\prescript{\phi_0}{}{TN})$, as we assumed $p > m$ and $k \geq 2$. In particular for any end-conical map $\phi_0$, $\prescript{\phi_0}{}{\mathcal{U}}^{k,p}_\mu$ is an open subset of $W^{k,p}_\mu(\prescript{\phi_0}{}{TN})$. Moreover, since $\mu > 0$ and we assumed that the submanifolds $\mathcal{M}_s \subset C^\infty(\Sp^{m-1},N)$ are embedded and that the charts $\Upsilon_s \col \Ball_1 \to M$ are fixed, there is a well-defined map $\Map^{k,p}_\mu(M,N;\mathfrak{S}) \to \mathcal{M}$ which assigns to a conically singular map of regularity $W^{k,p}_\mu$ its tangents maps along $S$ (or rather, the links thereof).
\end{rem}

We aim to equip $\Map^{k,p}_\mu(M,N;\mathfrak{S})$ with the structure of a Banach manifold. In order to define a convenient atlas of charts, we first need to fix some notations.

\begin{Def}     \label{Def: Banach manifold charts preliminary definitions}
    Fix a smooth cut-off function $\chi \col \R \to [0,1]$ such that $\chi(t) = 1$ if $t \leq \frac{1}{4}$ and $\chi(t) = 0$ if $t \geq \frac{3}{4}$. Let $\phi_0 \in \Map^\infty_{\mathrm{ec}}(M,N;\mathfrak{S})$ be an end-conical map and let $R \in (0,1)$ be small enough such that $\left. \phi_0 \circ \Upsilon_s \right|_{\Ball_R} = \varphi_s \circ \pr$, where $\varphi_s \in \mathcal{M}_s$. We define a linear map
    \begin{equation*}
        \chi_{\phi_0} \col \times_{s \in S} \Gamma(\prescript{\varphi_s}{}{TN}) \to \Gamma(\prescript{\phi_0}{}{TN})
    \end{equation*}
    such that for all $\underline{w} = (w_s)_{s \in S} \in \times_s \Gamma(\prescript{\varphi_s}{}{TN})$, $\chi_{\phi_0}(\underline{w}) = 0$ on $M \setminus (\cup_s \Upsilon_s(\Ball_R))$, and for all $s \in S$,
    \begin{equation*}
        \chi_{\phi_0}(\underline{w}) \circ \Upsilon_s(x) = \chi(\tfrac{|x|}{R}) w_s, ~~~~ \forall x \in \Ball_R \setminus \{0\},
        \end{equation*}
    where we naturally identify $\left. \prescript{\phi_0}{}{TN}\right|_{\Upsilon_s(\Ball_R)}$ with $\left.\pr^* (\prescript{\varphi_s}{}{TN})\right|_{\Ball_R}$.
\end{Def}

Now for any $s \in S$ and $\varphi_s \in \mathcal{M}_s$, let us pick a finite-dimensional submanifold $\prescript{\varphi_s}{}{\mathcal{W}}_s \subset \Gamma(\prescript{\varphi_s}{}{TN})$ such that $|w|_{h_{\varphi_s}} \leq \frac{1}{4}\inj(N)$ for any $w \in \prescript{\varphi_s}{}{\mathcal{W}}_s$, and $\exp_{\varphi_s}(\cdot) \col \prescript{\varphi_s}{}{\mathcal{W}}_s \to \mathcal{M}_s$ defines a diffeomorphism onto an open neighbourhood of $\varphi_s$ in $\mathcal{M}_s$. We will moreover define, for any $\phi_0 \in \Map^\infty_{\mathrm{ec}}(M,N;\mathfrak{S})$ whose tangent maps are given by $(\varphi_s)_{s \in S} \in \mathcal{M}$, the open subset
\begin{equation}
    \prescript{\phi_0}{}{\mathcal{W}} \coloneqq \times_{s \in S} \prescript{\varphi_{s}}{}{\mathcal{W}_s} \subset \times_s \Gamma(\prescript{\varphi_s}{}{TN}) .
\end{equation}
Notice that since we assumed that $|w|_{h_{\varphi_{s}}} < \tfrac{1}{4}\inj(N)$ for any $\underline{w} \in \prescript{\phi_0}{}{\mathcal{W}}$, $ \|u + \chi_{\phi_0}(\underline{w}) \|_{C^0} < \tfrac{1}{2}\inj(N)$ for any $(u,\underline{w}) \in \prescript{\phi_0}{}{\mathcal{U}}^{k,p}_\mu \times \prescript{\phi_0}{}{\mathcal{W}}$. Moreover, it is clear that $\exp_{\phi_0}(u+\chi_{\phi_0}(\underline{w})) \in \Map^{k,p}_\mu(M,N;\mathfrak{S})$, with tangent maps given by $\exp_{\varphi_s}(w_s)$ for all $s \in S$. In particular, we may use the exponential map to define an injective map
\begin{equation}
    \begin{aligned}
        \Psi_{\phi_0} \col \prescript{\phi_0}{}{\mathcal{U}}^{k,p}_\mu \times \prescript{\phi_0}{}{\mathcal{W}} & \to \Map^{k,p}_\mu(M,N;\mathfrak{S}), \\
        (u,\underline{w}) & \mapsto \exp_{\phi_0}(u+\chi_{\phi_0}(\underline{w})) .
    \end{aligned}
\end{equation}
Let us denote by $\prescript{\phi_0}{}{\mathcal{V}}^{k,p}_\mu$ the image of $\Psi_{\phi_0}$. We shall prove that 
\begin{equation*}
    \{\Psi_{\phi_0} \col \prescript{\phi_0}{}{\mathcal{U}}^{k,p}_\mu \times \prescript{\phi_0}{}{\mathcal{W}} \to \prescript{\phi_0}{}{\mathcal{V}}^{k,p}_\mu\}_{\phi_0 \in \Map^\infty_\mu(M,N;\mathfrak{S})}
\end{equation*}
is a smooth atlas of Banach charts on $\Map^{k,p}_\mu(M,N;\mathfrak{S})$. Note that it is obvious from the definitions that $\Map^{k,p}_\mu(M,N;\mathfrak{S})$ is covered by $\{\prescript{\phi_0}{}{\mathcal{V}}^{k,p}_\mu\}_{\phi_0 \in \Map^\infty_{\mathrm{ec}}(M,N;\mathfrak{S})}$. Thus in order to define a Banach manifold structure (and with it, a topology) on $\Map^{k,p}_\mu(M,N;\mathfrak{S})$ we just need to check the compatibility of the charts. This is the object of the next proposition:

\begin{prop}        \label{prop: banach manifold atlas}
    Let $\phi_0,\phi_1 \in \Map^\infty_{\mathrm{ec}}(M,N;\mathfrak{S})$ be end-conical maps such that $\prescript{\phi_0}{}{\mathcal{V}}^{k,p}_\mu \cap \prescript{\phi_1}{}{\mathcal{V}}^{k,p}_\mu \neq \emptyset$. Then $(\prescript{\phi_0}{}{\mathcal{U}}^{k,p}_\mu \times \prescript{\phi_0}{}{\mathcal{W}}) \cap \Psi_{\phi_0}^{-1}(\prescript{\phi_1}{}{\mathcal{V}}^{k,p}_\mu)$ is an open subset of $\prescript{\phi_0}{}{\mathcal{U}}^{k,p}_\mu \times \prescript{\phi_0}{}{\mathcal{W}} \subset W^{k,p}_\mu(\prescript{\phi_0}{}{TN}) \times \prescript{\phi_0}{}{\mathcal{W}}$. Moreover, the transition function
    \begin{equation*}
        \Psi_{\phi_1}^{-1} \circ \Psi_{\phi_0} \col (\prescript{\phi_0}{}{\mathcal{U}}^{k,p}_\mu \times \prescript{\phi_0}{}{\mathcal{W}}) \cap \Psi_{\phi_0}^{-1}(\prescript{\phi_1}{}{\mathcal{V}}^{k,p}_\mu) \to (\prescript{\phi_1}{}{\mathcal{U}}^{k,p}_\mu \times \prescript{\phi_1}{}{\mathcal{W}}) \cap \Psi_{\phi_1}^{-1}(\prescript{\phi_0}{}{\mathcal{V}}^{k,p}_\mu)
    \end{equation*}
    is a smooth Banach map\footnote{Note that all $\mathcal{M}_s \subset C^{\infty}(\Sp^{m-1},N)$ are assumed to be smoothly embedded \emph{finite-dimensional} submanifolds. This implies that $\prescript{\phi_0}{}{\mathcal{W}} \subset \mathcal{M}$ is a finite-dimensional manifold in its own right for any end-conical map $\phi_0$. Thus $\prescript{\phi_0}{}{\mathcal{U}}^{k,p}_\mu \times \prescript{\phi_0}{}{\mathcal{W}}$ carries the structure of smooth Banach manifold, and the smoothness of maps is understood in this sense.}. In particular the charts $\{\Psi_{\phi} \col \prescript{\phi}{}{\mathcal{U}}^{k,p}_\mu \times \prescript{\phi}{}{\mathcal{W}} \to \prescript{\phi}{}{\mathcal{V}}^{k,p}_\mu\}_{\phi \in \Map^\infty_\mu(M,N;\mathfrak{S})}$ endow $\Map^{k,p}_\mu(M,N;\mathfrak{S})$ with the structure of a smooth Banach manifold.
\end{prop}

\begin{proof}
    Let us prove the openness first. Remark first that $\phi \in \Map^{k,p}_\mu(M,N;\mathfrak{S})$ is in $\prescript{\phi_0}{}{\mathcal{V}}^{k,p}_\mu$ if and only if the following conditions are satisfied:
    \begin{itemize}
        \item the tangent maps of $\phi$ are of the form $\exp_{\phi_s}(w_s)$, where $\phi_s$ are the tangent maps of $\phi_0$ and $(w_s)_{s \in S} \in \prescript{\phi_0}{}{\mathcal{W}}$, and
        \item $\mathrm{dist}(\phi, \exp_{\phi_0}(\chi_\phi(\underline{w}))) < \frac{1}{4}\inj(N)$, where $\mathrm{dist}(\phi,\psi) \coloneqq \inf_{x \in M \setminus S} \mathrm{dist}_N(\phi(x),\psi(x))$ for any continuous map $\psi \col M \setminus S \to N$. 
    \end{itemize}
    Since for $i = 1,2$ the image of $\prescript{\phi_i}{}{\mathcal{W}}$ in $\mathcal{M}$ is open and there is a sequence of continuous embeddings $W^{k,p}_\mu(\prescript{\phi_i}{}{TN}) \hookrightarrow C^{k-1}_\mu(\prescript{\phi_i}{}{TN}) \hookrightarrow C^0_0(\prescript{\phi_i}{}{TN})$, the triangular inequality easily implies the openness of $(\prescript{\phi_0}{}{\mathcal{U}}^{k,p}_\mu \times \prescript{\phi_0}{}{\mathcal{W}}) \cap \Psi_{\phi_0}^{-1}(\prescript{\phi_1}{}{\mathcal{V}}^{k,p}_\mu)$.
    
    Next, we show that the transition function is smooth. Let us denote by $(\varphi_{i,s})_{s \in S} \in \mathcal{M}$ the tangent maps of $\phi_i$ for $i = 0,1$. We first remark that for any $(u,\underline{w}) \in (\prescript{\phi_0}{}{\mathcal{U}}^{k,p}_\mu \times \prescript{\phi_0}{}{\mathcal{W}}) \cap \Psi_{\phi_0}^{-1}(\prescript{\phi_1}{}{\mathcal{V}}^{k,p}_\mu)$ , we can express the transition function as
    \begin{equation*}
        \Psi_{\phi_1}^{-1} \circ \Psi_{\phi_0}(u,\underline{w}) = (\mathcal{F}(u,\underline{w}),\Xi(\underline{w})) \in \prescript{\phi_1}{}{\mathcal{U}}^{k,p}_\mu \times \prescript{\phi_1}{}{\mathcal{W}}
    \end{equation*}
    where $\Xi(\underline{w}) = (\Xi_s(w_s))_{s \in S}$ and $\Xi_s \col \exp_{\varphi_{1,s}}^{-1} \circ \exp_{\varphi_{0,s}}$ is the transition function between the charts $\prescript{\varphi_{0,s}}{}{\mathcal{W}}$ and $\prescript{\varphi_{1,s}}{}{\mathcal{W}}$ in $\mathcal{M}_s$. Thus the map $\Xi$ is smooth.

    To prove that the first component $\mathcal{F}$ of the transition function is smooth, we will use the notion of bundle map with dilation-invariant ends introduced in \cref{subsubsec: Bundle maps with dilation-invariant ends}. Recall that we introduced the open subbundles $\prescript{\phi_i}{}{U} = \{ u \in \prescript{\phi_i}{}{TN} \mid |u|_{h_{\phi_i}} < \inj(N) \}$ in \cref{subsubsec: pulled-back tangent bundles}. Since $\prescript{\phi_0}{}{\mathcal{V}}^{k,p}_\mu \cap \prescript{\phi_1}{}{\mathcal{V}}^{k,p}_\mu \neq \emptyset$, we have in particular $\mathrm{dist}(\phi_0,\phi_1) < \inj(N)$. Moreover, there exists a non-linear bundle map 
    \begin{equation*}
        \bmap{f} \col (U \subset \prescript{\phi_0}{}{TN}) \to \prescript{\phi_1}{}{TN} ,
    \end{equation*} 
    defined on the open sub-bundle whose fibre at $x \in M$ is 
    \begin{equation*}
        U_x \coloneqq \{ u \in \prescript{\phi_0}{}{U_x} \mid \mathrm{dist}_N(\exp_{\phi_0(x)}(v),\phi_1(x)) < \inj(N) \},
    \end{equation*} 
    such that $\exp_{\phi_0}(v) = \exp_{\phi_1}(\bmap{f}(v))$ for any $v \in \Gamma(U)$.
    This bundle map has dilation-invariant ends in the sense of \cref{Def: non-linear bundle map with dilation invariant ends} because $\phi_0,\phi_1$ are end-conical. Now for any $(u,\underline{w})$ in the domain of the transition function, $\chi_{\phi_0}(\underline{w})$ and $u + \chi_{\phi_0}(\underline{w})$ are sections of $U$, and moreover we have
    \begin{align*}
        \mathcal{F}(u,\underline{w}) & = \bmap{f}(u+\chi_{\phi_0}(\underline{w})) - \chi_{\phi_1}(\Xi(\underline{w})) \\ 
            & = (\bmap{f}(u+\chi_{\phi_0}(\underline{w})) - \bmap{f}(\chi_{\phi_0}(\underline{w}))) + (\bmap{f}(\chi_{\phi_0}(\underline{w})) - \chi_{\phi_1}(\Xi(\underline{w}))) .
    \end{align*}
    Since $\prescript{\phi_0}{}{\mathcal{W}}$ is finite-dimensional, the map $\chi_{\phi_0}$ induces a smooth map $\prescript{\phi_0}{}{\mathcal{W}} \to C^{k}_0(\prescript{\phi_0}{}{U}) \subset C^k_0(\prescript{\phi_0}{}{TN})$, and therefore by \cref{lem: Parametrised master lemma} the map $(u,\underline{w}) \to \bmap{f}(u+\chi_{\phi_0}(\underline{w}))-\bmap{f}(\chi_{\phi_0}(\underline{w}))$ defines a smooth map from the domain of the transition function into $W^{k,p}_\mu(\prescript{\phi_1}{}{TN})$. On the other hand, for sufficiently small $R \in (0,1)$, the section $\bmap{f}(\chi_{\phi_0}(\underline{w})) - \chi_{\phi_1}(\Xi(\underline{w})) \in \Gamma(\prescript{\phi_1}{}{TN})$ is compactly supported inside $M \setminus (\cup_{s \in S} \Upsilon_s(\Ball_R))$ for any $\underline{w} \in \prescript{\phi_0}{}{\mathcal{W}}$, and therefore it is clear that this expression defines a smooth map $\prescript{\phi_0}{}{\mathcal{W}} \to W^{k,p}_\mu(\prescript{\phi_1}{}{TN})$. Thus  $\mathcal{F}$ is a smooth Banach map.
\end{proof}

\begin{rem}
    It is clear from the construction of the atlas that the map $\Map^{k,p}_\mu(M,N;\mathfrak{S}) \to \mathcal{M}$ which assigns to a conically singular map of regularity $W^{k,p}_\mu$ its tangent maps is a smooth Banach submersion, since in the charts $\prescript{\phi_0}{}{\mathcal{U}}^{k,p}_\mu \times \prescript{\phi_0}{}{\mathcal{W}}$ of $\Map^{k,p}(M,N;\mathfrak{S})$ and $\prescript{\phi_0}{}{\mathcal{W}}$ of $\mathcal{M}$ it corresponds to the projection onto the second factor.
\end{rem}

\begin{rem}
    It is not difficult to see that the topology of $\Map^{k,p}_\mu(M,N;\mathfrak{S})$ is metrisable; in particular, it is a paracompact space.
\end{rem}

    \subsubsection{A family of Banach vector bundles}   \label{subsubsec: Banach vector bundles}

In this part, we let $k \geq 2$ and proceed to define a useful class of Banach vector bundles over the Banach manifold $\Map^{k,p}_\mu(M,N;\mathfrak{S})$. Let us first point out that if $\phi_0 \in \Map^\infty_\mu(M,N;\mathfrak{S})$ and $\phi \in \Map^{k,p}_\mu(M,N;\mathfrak{S})$ can be written as $\phi = \exp_{\phi_0}(u)$ for a section $u \in \prescript{\phi_0}{}{\mathcal{U}}^{k,p}_\mu$, the continuous embedding $W^{k,p}_\mu(\prescript{\phi_0}{}{TN}) \hookrightarrow C^{k-1}_\mu(\prescript{\phi_0}{}{TN})$ implies that $\prescript{\phi}{}{TN}$ has the structure of a vector bundle of class $C^{k-1}$ over $M \setminus S$ and $\prescript{\phi_0}{}{\Theta}_u \col \prescript{\phi_0}{}{TN} \to \prescript{\phi}{}{TN}$ is an isomorphism of $C^{k-1}$ vector bundles. Moreover, \cref{lem: Bundle maps of end-conical maps} and \cref{lem: Derivative of dilation-invariant bundle maps} imply that $h_\phi$ is a $C^{k-1}$-bundle metric and $\prescript{\phi}{}{\nabla}$ is a $C^{k-2}$ connection on $\prescript{\phi}{}{TN}$, and near the singular set $S$,
\begin{equation*}
    |\prescript{\phi_0}{}{\nabla}^\ell(\prescript{\phi_0}{}{\Theta}_u^* h_\phi - h_{\phi_0})| = \mathcal{O}(r^{\mu-\ell}), ~~~~ \forall 0 \leq \ell \leq k-1,
\end{equation*}
and
\begin{equation*}
    |\prescript{\phi_0}{}{\nabla}^\ell(\prescript{\phi_0}{}{\Theta}_u^* (\prescript{\phi}{}{\nabla}) - \prescript{\phi_0}{}{\nabla})| = \mathcal{O}(r^{-1+\mu-\ell}), ~~~~ \forall 0 \leq \ell \leq k-2 .
\end{equation*}
This bounds show that $\prescript{\phi}{}{TN}$ as the structure of an `adapted $C^{k-1}$-bundle' (the defining properties of an adapted bundle are satisfied up to order $k-1$), which is enough to define the weighted Sobolev spaces $W^{\ell,p}_\nu(\prescript{\phi}{}{TN})$ for any $\nu \in \R$ and $0 \leq \ell \leq k-1$, as in \cref{Def: weighted spaces of adapted bundles}. For any such $\ell$, we shall define a Banach vector bundle $\mathcal{E}^{\ell,p}_\nu \to \Map^{k,p}_\mu(M,N;\mathfrak{S})$ whose fibre above $\phi \in \Map^{k,p}_\nu(M,N;\mathfrak{S})$ can be identified with $W^{\ell,p}_\nu(\prescript{\phi}{}{TN})$.

Just like for finite-dimensional manifolds, Banach vector bundles over a Banach manifold can be defined through a collection of local trivialisations and gluing functions satisfying the usual cocycle conditions. To define a vector bundle $\mathcal{E}^{\ell,p}_\nu$ satisfying the desired property, we can use the atlas constructed in the previous part and start with the collection of trivial vector bundles
\begin{equation}        \label{eq: local trivialisations}
    \{ \prescript{\phi_0}{}{\mathcal{U}}^{k,p}_\mu \times \prescript{\phi_0}{}{\mathcal{W}} \times W^{\ell,p}_\nu(\prescript{\phi_0}{}{TN}) \}_{\phi_0 \in \Map^\infty_{\mathrm{ec}}(M,N;\mathfrak{S})} .
\end{equation}
For any end-conical map $\phi_0 \in \Map^\infty_{\mathrm{ec}}(M,N;\mathfrak{S})$ and $\phi = \exp_{\phi_0}(u+\chi_{\phi_0}(\underline{w})) \in \prescript{\phi_0}{}{\mathcal{V}}^{k,p}_\nu$, we can argue just as in the proof of \cref{lem: theta induces iso of Banach spaces} that the isomorphism of $C^{k-1}$-vector bundles
\begin{equation*}
    \prescript{\phi_0}{}{\Theta}_{u+\chi_{\phi_0}(\underline{w})} \col \prescript{\phi_0}{}{TN} \to \prescript{\phi}{}{TN}
\end{equation*}
induces an isomorphism of Banach spaces $W^{\ell,p}_\nu(\prescript{\phi_0}{}{TN}) \simeq W^{\ell,p}_\nu(\prescript{\phi}{}{TN})$ for any $\nu \in \R$ and $0 \leq \ell \leq k-1$. Thus it is natural to choose the gluing functions of $\mathcal{E}^{\ell,p}_\nu$ to be defined as follows, for any $\phi_0,\phi_1 \in \Map^\infty_{\mathrm{ec}}(M,N;\mathfrak{S})$ such that $\prescript{\phi_0}{}{\mathcal{V}}^{k,p}_\nu \cap \prescript{\phi_1}{}{\mathcal{V}}^{k,p}_\nu \neq \emptyset$:
\begin{equation}       \label{eq: gluing functions}
    \mathcal{G}_{\phi_1,\phi_0}(u_0,\underline{w}_0)\{u\} = \prescript{\phi_1}{}{\Theta}_{u_1+\chi_{\phi_1}(\underline{w}_1)}^{-1} \circ \prescript{\phi_0}{}{\Theta}_{u_0+\chi_{\phi_0}(\underline{w}_0)} \{u\}, ~~~~ \forall u \in W^{\ell,p}_\nu(\prescript{\phi_0}{}{TN})
\end{equation}
for any $(u_i,\underline{w}_i) \in \prescript{\phi_i}{}{\mathcal{U}}^{k,p}_\mu \times \prescript{\phi_i}{}{\mathcal{W}}$ such that $\Psi_{\phi_0}(u_0,\underline{w}_0) = \Psi_{\phi_1}(u_1,\underline{w}_1)$. These gluing functions obviously satisfy the cocycle condition, and thus we only need to prove that they indeed define a smooth Banach map from $(\prescript{\phi_0}{}{\mathcal{U}}^{k,p}_\nu \times \prescript{\phi_0}{}{\mathcal{W}}) \cap \Psi_{\phi_0}^{-1}(\prescript{\phi_1}{}{\mathcal{V}}^{k,p}_\mu)$ to the Banach space of bounded linear maps $\mathcal{B}(W^{\ell,p}_\nu(\prescript{\phi_0}{}{TN}),W^{\ell,p}_\nu(\prescript{\phi_1}{}{TN}))$. 

To see this, we proceed as in the proof of \cref{prop: banach manifold atlas}: it is not difficult to see that $\mathcal{G}_{\phi_1,\phi_0}$ can be expressed as a non-linear bundle map with dilation-invariant ends taking values in $\Hom(\prescript{\phi_0}{}{TN},\prescript{\phi_1}{}{TN})$. In particular, it follows from \cref{lem: Master lemma} that $\mathcal{G}_{\phi_1,\phi_0}(u,\underline{w}) - \mathcal{G}_{\phi_1,\phi_0}(0,\underline{w})$ defines of smooth Banach map from the domain of definition of the gluing functions into $W^{k,p}_\mu(\Hom(\prescript{\phi_0}{}{TN},\prescript{\phi_1}{}{TN})) \subset W^{k,p}_0(\Hom(\prescript{\phi_0}{}{TN},\prescript{\phi_1}{}{TN}))$, and thus by weighted Sobolev multiplication (\cref{prop: Sobolev embedding and multiplication}) this defines a smooth Banach map into $\mathcal{B}(W^{\ell,p}_\nu(\prescript{\phi_0}{}{TN}),W^{\ell,p}_\nu(\prescript{\phi_1}{}{TN}))$ for any $0 \leq \ell \leq k$. On the other hand, $\mathcal{G}_{\phi_0,\phi_1}(0,\underline{w})$ defines a smooth map into $C^k_0(\Hom(\prescript{\phi_0}{}{TN},\prescript{\phi_1}{}{TN})$, which also smoothly embeds into $\mathcal{B}(W^{\ell,p}_\nu(\prescript{\phi_0}{}{TN}),W^{\ell,p}_\nu(\prescript{\phi_1}{}{TN}))$. This proves that we can even define the vector bundles $\mathcal{E}^{\ell,p}_\nu$ in the case where $\ell = k$. We gather the outcomes of the above discussion in the following definitions:

\begin{Def}
    For any $\ell \in \{0,1,\ldots,k\}$ and $\nu \in \R$, we shall denote by $\mathcal{E}^{\ell,p}_\nu \to \Map^{k,p}_\mu(M,N;\mathfrak{S})$ the Banach vector bundle obtained from the collection of local trivialisations \eqref{eq: local trivialisations} and the gluing functions \eqref{eq: gluing functions}. By construction, these bundles have the following properties:
    \begin{enumerate}
        \item If $\ell \leq k-1$, the fibre $\mathcal{E}^{\ell,p}_{\nu,\phi}$ can be identified with $W^{\ell,p}_\nu(\prescript{\phi}{}{TN})$ for any $\phi \in \Map^{k,p}_\nu(M,N;\mathfrak{S})$. 

        \item In the case $k = \ell$, we still have an identification $\mathcal{E}^{k,p}_\nu \simeq W^{k,p}_\nu(\prescript{\phi}{}{TN})$ for any $\phi$ sufficiently regular, for instance when $\phi \in \Map^\infty_\mu(M,N;\mathfrak{S})$.
    \end{enumerate}
\end{Def}

\begin{rem}
    It is clear from the definitions that the vertical space of the smooth Banach submersion $\Map^{k,p}_\mu(M,N;\mathfrak{S}) \to \mathcal{M}$, defined as the sub-bundle of $T\Map^{k,p}_\mu(M,N;\mathfrak{S})$ tangent to the fibres, is naturally isomorphic to the Banach vector bundle $\mathcal{E}^{k,p}_\mu$. In particular for any $\phi \in \Map^\infty_\mu(M,N;\mathfrak{S})$ the vertical space is isomorphic to $W^{k,p}_\mu(\prescript{\phi}{}{TN})$, as one should naturally expect.
\end{rem}

    \subsubsection{Moving the singular points}     \label{subsubsec: Moving the singular points}

In order to prove \cref{thm: main theorem introduction}, we will see in \cref{sec: The deformation theorem} that it is necessary to allow the points of the singular set $S$ to move. Indeed, it turns out that the deformation problem for conically singular maps with a fixed singular set is always obstructed (this is related to the fact that $\nu=-1$ is always a critical rate of the Jacobi operator, as observed in \cref{subsubsec: further remarks}). Thus a naive idea would be to extend previous construction of the Banach manifolds of conically singular maps of regularity $W^{k,p}_\mu$, and allow for deformations not just of the tangent maps but of the singular set as well. An idea to try and construct an atlas of Banach charts for such a manifold would be to move around the singular points by pre-composing a conically singular map $\phi \col M\setminus S \to N$ with a diffeomorphism $f \col M \to M$ that realises a given translation at any $s \in S$. However, this runs into the issue that the pre-composition map
\begin{align*}
    W^{k,p}(M,N) \times \mathrm{Diff}^k(M) &\to W^{k,p}(M,N) \\
    (\phi,f) &\mapsto \phi \circ f
\end{align*}
is not even a continuous operation\footnote{As a side remark, one could try to use (weighted) H\"older spaces instead or use diffeomorphisms of higher regularity (even smooth), but this does not resolve the issue since the composition map $C^{k,\alpha}(M,N) \times \mathrm{Diff}^\ell(M) \to C^{k,\alpha}(M,N)$ is also not continuous, for any $\alpha \in (0,1)$ and $\ell \geq k+1$.}, so that there is no hope of constructing a Banach atlas in this way. Because of this, we will instead work with the Banach manifold $\Map^{k,p}_\mu(M,N;\mathfrak{S})$ of conically singular maps whose singularities remain fixed, and use a parametrised family of diffeomorphisms to realise (local) variations of the singular points by deforming the Riemannian metrics on $M$ instead of the conically singular maps. For the purpose of deforming harmonic maps this is of course equivalent, since $\phi \circ f^{-1}$ is harmonic with respect to $g$ if and only if $\phi$ is harmonic with respect to $f^*g$.

In order to do this, we shall first construct two families of vector fields on $M$ as follows. By our assumptions on the set of charts $\{\Upsilon_s \col \Ball_1 \to M\}_{s \in S}$, there exists $\delta > 0$ and a family of smooth extensions $\{\Upsilon_s \col \Ball_{1+\delta} \to M\}_{s \in S}$ such that $\Upsilon_s(\Ball_{1+\delta}) \cap \Upsilon_{s^\prime}(\Ball_{1+\delta}) = \emptyset$ if $s \neq s^\prime$. Let us pick a smooth cutoff function $\chi \col [0,+\infty) \to [0,1]$ such that $\chi \equiv 1$ on $[0,1]$ and $\chi \equiv 0$ on $[1+\delta,+\infty)$. In order to move around the singular set $S$, let us consider the smooth family of vector fields
\begin{align*}
        \xi_1 \col \Ball_1^S \to \mathfrak{X}(M)
\end{align*}
defined for any $\underline{\mathfrak{t}} = (\mathfrak{t}_s)_{s \in S} \in \Ball_1^S$ as
\begin{equation}
    \xi_1(\underline{\mathfrak{t}}) = \begin{cases} 0 ~~~~~~~~~~~~~~~~~~~~~~~~~\text{on} ~~ M \setminus (\cup_s \Upsilon_s(\Ball_{1+\delta})), \\
    (\Upsilon_s)_*(\chi(|x|) \mathfrak{t}_s) ~~~~~~~ \text{on} ~~ \Upsilon_s(\Ball_{1+\delta}) .
    \end{cases}
\end{equation}
Note that for all $\underline{\mathfrak{t}} \in \Ball_1^S$ the flow at time $1$ of $\xi_1(\underline{\mathfrak{t}})$ satisfies 
\begin{equation*}
    \textup{Flow}^{\xi_1(\underline{\mathfrak{t}})}_1(s) = \Upsilon_s(\mathfrak{t}_s), ~~~~ \forall s \in S.
\end{equation*}
We also consider a second smooth family of vector fields
\begin{align*}
    \xi_2 \col (\mathrm{Sym}^+_m(\R))^S \to \mathfrak{X}(M) ,
\end{align*}
where $\mathrm{Sym}^+_m(\R)$ is the set of positive-definite symmetric $m \times m$ matrices, such that for any $\underline{A} = (A_s)_{s \in S} \in (\mathrm{Sym}^+_m(\R))^S$ we have
\begin{equation}
    \xi_2(\underline{A}) = \begin{cases} 0 ~~~~~~~~~~~~~~~~~~~~~~~~~~~~~~~~~~~~~~~ \text{on} ~~ M \setminus (\cup_s \Upsilon_s(\Ball_{1+\delta})), \\
    (\Upsilon_s)_*(- \tfrac{1}{2} \chi(|x|) \log(A) \cdot x) ~~~~~~ \text{on} ~~ \Upsilon_s(\Ball_{1+\delta}),
    \end{cases}
\end{equation}
where $\log(A) \in \mathrm{Sym}_m(\R)$ is the logarithm of a positive-definite symmetric matrix. Its is clear that the flow of $\xi_2(\underline{A})$ fixes the singular set. Moreover, for any $\underline{A} \in (\mathrm{Sym}^+_m(\R))^S$ there exists $\varepsilon \in (0,1)$ such that the flow at time $1$ satisfies
\begin{equation*}
    \Upsilon_s^{-1} \circ \textup{Flow}^{\xi_2(\underline{A})}_1 \circ \Upsilon_s(x) = A_s^{-\frac{1}{2}} x, ~~~~ \forall s \in S, ~ \forall x \in \Ball_\varepsilon \subset \Ball_1.
\end{equation*}

Before explaining how to use the above vector fields to resolve the difficulties pointed out earlier, we need to introduce some notations. It will be convenient to denote by $\Met^{\ell}(M)$ the space of Riemannian metrics on $M$ of regularity $C^{\ell}$, and by $\Met^\ell(M;\mathfrak{S}) \subset \Met^\ell(M)$ the space of metrics of regularity $C^\ell$ such that $\Upsilon_s^* g(0) = g_{\R^m}$ for any $s \in S$. Note that $\Met^\ell(M;\mathfrak{S})$ is the intersection of $\Met^\ell(M)$ with a closed affine subspace of $C^\ell(\mathrm{Sym}^2_+(T^*M))$ and therefore $\Met^\ell(M;\mathfrak{S})$ is an open subset of a Banach space. We also define an evaluation map
\begin{equation*}
    \mathrm{ev}_S \col \Ball_1^S \times \Met^0(M) \to (\mathrm{Sym}^+_m(\R))^S
\end{equation*}
such that for any $(\underline{\mathfrak{t}},g) \in \Ball_1^S \times \Met^0(M)$, $\mathrm{ev}_S(\underline{\mathfrak{t}},g) = (A_s(\mathfrak{t}_s,g))_{s \in S}$ where $A_s(\mathfrak{t}_s,g)$ is the matrix of the coefficients of the quadratic form $(\Upsilon_s^*g)_{\mathfrak{t}_s}$ in the canonical basis of $\R^m$. We are now ready to define an appropriate family of diffeomorphisms:

\begin{Def}     \label{Def: diffeos f}
    For any $(\underline{\mathfrak{t}},g) \in \Ball_1^S \times \Met^0(M)$, we define the diffeomorphism $f_{\underline{\mathfrak{t}},g}$ by
    \begin{equation*}
        f_{\underline{\mathfrak{t}},g} = \textup{Flow}_1^{\xi_1(\underline{\mathfrak{t}})} \circ \textup{Flow}_1^{\xi_2(\mathrm{ev}_S(\underline{\mathfrak{t}},g))} \in \mathrm{Diff}(M).
    \end{equation*}
\end{Def}

\begin{prop}     \label{prop: regularity of pull-back of metrics}
    Let $k, \ell \in \N_0$.
    \begin{enumerate}[(i)]
        \item For any $\underline{\mathfrak{t}} \in \Ball_1^S$ and any $g \in \Met^{k}(M)$, $f_{\underline{\mathfrak{t}},g}^*g \in \Met^{k}(M;\mathfrak{S})$.

        \item If $\ell \geq 1$, then the induced map
        \begin{equation*}
            \mathcal{F} \col \Met^{k+\ell}(M) \times \Ball_1^S \to \Met^k(M;\mathfrak{S}), ~~ (\underline{\mathfrak{t}},g) \mapsto f_{\underline{\mathfrak{t}},g}^* g
        \end{equation*}
        is a Banach map of class $C^\ell$. Moreover, the derivative of this map satisfies
        \begin{equation*}
            D_{(\underline{0},g)} \mathcal{F} \{ \dot{\underline{\mathfrak{t}}} \} = f_{\underline{0},g}^*(\pounds_{\xi_g(\dot{\underline{\mathfrak{t}}})}g), ~~~ \forall g \in \Met^{k+\ell}(M), \forall \dot{\underline{\mathfrak{t}}} \in (\R^m)^S,
        \end{equation*}
        where $\pounds$ is the Lie derivative and $\xi_g(\cdot) \col (\R^m)^S \to \mathfrak{X}(M)$ is a family of vector fields such that for any $\underline{\dot{\mathfrak{t}}} \in (\R^m)^S$ and $s \in S$,
        \begin{equation*}
            \xi_g(\dot{\underline{\mathfrak{t}}})_s = (\Upsilon_s)_* \dot{\mathfrak{t}}_s .
        \end{equation*}
    \end{enumerate}
\end{prop}

\begin{proof}
    For the first point, let $(\underline{\mathfrak{t}},g) \in \Ball_1^S \times \Met^{k+\ell}(M)$, and let us simply denote $\mathrm{ev}_S(\underline{\mathfrak{t}},g)$ by $\underline{A} = \underline{A}(\underline{\mathfrak{t}},g) \in (\mathrm{Sym}^+_m(\R))^S$, so that $f_{\underline{\mathfrak{t}},g} = \textup{Flow}_1^{\xi_1(\underline{\mathfrak{t}})} \circ \textup{Flow}_1^{\xi_2(\underline{A})}$. Since $f_{\underline{\mathfrak{t}},g}$ is a smooth diffeomorphism, $f_{\underline{\mathfrak{t}},g}^* g \in \Met^{k+\ell}(M)$ is a metric of regularity $C^{k+\ell}$. On the other hand, using the local description of the flows of $\xi_1(\underline{\mathfrak{t}})$ and $\xi_2(\underline{A})$ near $S$, we see that there exists $\varepsilon \in (0,1)$ such that for any $s \in S$
    \begin{equation}        \label{eq: local flow}
        \Upsilon_{s}^{-1} \circ f_{\mathfrak{t},g} \circ \Upsilon_s(x) = A_s(\underline{\mathfrak{t}},g)^{-1/2} x + \mathfrak{t}_s , ~~~~ \forall x \in \Ball_\varepsilon.
    \end{equation}
    Since $A_s = A_s(\underline{\mathfrak{t}},g)$ is by definition the matrix of coefficients of the quadratic form $(\Upsilon_s^*g)_{\mathfrak{t}_s}$, it immediately follows that $(\Upsilon_s^* f_{\mathfrak{t},g}^*g)_{0} = g_{\R^m}$, that is, $f_{\mathfrak{t},g}^*g \in \Met^{k+\ell}(M;\mathfrak{S})$.

    For the second point, let us first remark that the composition of the flows at time $1$, $\textup{Flow}_1^{\xi_1(\underline{\mathfrak{t}})} \circ \textup{Flow}^{\xi_2(\underline{A})}_1$, defines a smooth family of diffeomorphisms $\Ball_1^S \times (\mathrm{Sym}^+_m(\R))^S \times M \to M$, and therefore it induces a smooth map $\Ball_1^S \times (\mathrm{Sym}^+_m(\R))^S \to \mathrm{Diff}^{k+1}(M)$ into the Banach manifold\footnote{The construction of a Banach manifold structure on $\mathrm{Diff}^{k+1}(M)$ is classical, for instance by fixing a background metric on $M$ and defining local charts modelled on open subsets of $C^{k+1}(TM)$ using the associated exponential map. Note that the composition of diffeomorphisms is continuous but not differentiable on $\mathrm{Diff}^{k+1}(M)$.} of diffeomorphisms of class $C^{k+1}$. On the other hand, the pull-back
    \begin{equation}        \label{eq: pb map}
        \Met^{k+\ell}(M) \times \mathrm{Diff}^{k+1}(M) \to \Met^{k}(M), ~~ (g,f) \mapsto f^* g
    \end{equation}
    is a Banach map of class $C^\ell$ (but no more). This essentially follows from the fact that if $X,Y,Z$ are compact manifolds and $k,\ell,\ell^\prime \in \N_0$, the composition $\Map^{k+\ell}(X,Y) \times \Map^{k+\ell^\prime}(Y,Z) \to \Map^{k}(X,Z)$ is of class $C^{\ell}$ but no more, no matter how large $\ell^\prime$ is (c.f. for instance the discussion of \cite[\S5]{eells1981deformations}). Finally, remark that the evaluation map
    \begin{equation*}
        \mathrm{ev}_S \col \Ball_1^S \times \Met^{k+\ell}(M) \to (\mathrm{Sym}^+_m(\R))^S
    \end{equation*}
    is obviously of class $C^{k+\ell}$. Gathering the previous observations and using the composition rules for Banach maps, we deduce that $\mathcal{F}$ is of class $C^\ell$. 

    To prove the statement on the derivative when $\ell \geq 1$, remark first that the derivative of the pull-back map \eqref{eq: pb map} along path of diffeomorphism is obviously given by the Lie derivative with respect to the corresponding family of vector fields. Now examining \eqref{eq: local flow}, it is clear that the vector field defined as $\left. \frac{\diff}{\diff t} \right|_{t = 0} f_{t \cdot \underline{\dot{\mathfrak{t}}},g} = \xi_g(\dot{\underline{\mathfrak{t}}}) \circ f_{\underline{0},g}$ satisfies $\xi_g(\underline{\dot{\mathfrak{t}}})_s = (\Upsilon_s)_* \dot{\mathfrak{t}}_s$, which finishes the proof of the proposition.
\end{proof}

    \subsection{Variations of the tension}      \label{subsec: variations of the tension}

The goal of this section is to show that, under the additional assumption that all the tangent maps are harmonic, the tension of a conically singular Sobolev map can be regarded as a continuously differentiable section of the Banach vector bundle $\mathcal{E}^{k,p}_{\mu-2}$ over the parameter space $\Map^{k+2,p}_{\mu}(M,N;\mathfrak{S}) \times \Ball_1^S \times \Met^{k+2}(M)$, and to compute its derivative at a conically singular harmonic map. \cref{subsubsec: tension as section} makes this statement precise (cf. \cref{thm: Variations of the tension}), and the remainder of this part is dedicated its proof. In \cref{subsubsec: expressions hessian}, we express the hessian of a map in terms of bundle maps with dilation-invariant ends, in order to use the formalism developed in \cref{subsubsec: Bundle maps with dilation-invariant ends} to prove \cref{thm: Variations of the tension} in \cref{subsec: proof of the main proposition}.

    \subsubsection{The tension as a section of a Banach vector bundle}      \label{subsubsec: tension as section}

In this section, we let $\mu \in (0,1)$, $p \in (m,\infty)$, $k \geq 0$ and $\mathfrak{S} = \{S,\{\Upsilon_s\}_{s \in S}, \{\mathcal{M}_s\}_{s \in S}\}$ be a data set as in the previous part, with the additional
\begin{hyp}
    For all $s \in S$, any $\varphi_s \in \mathcal{M}_s$ is harmonic with respect to the standard round metric on $\Sp^{m-1}$; thus the dilation-invariant map $\phi_s = \varphi_s \circ \pr$ is harmonic with respect to the standard Euclidean metric $g_{\R^m}$ on $\R^m$.
\end{hyp}
Because of this assumption, if $g$ is a Riemannian metric such that $\Upsilon_s^* g(0) = g_{\R^m}$ for all $s \in S$ (that is $g \in \Met^\infty(M;\mathfrak{S})$ in the notations introduced in the previous section) and $\phi \col M \setminus S \to N$ is a conically singular map whose tangent maps $\phi_s = \varphi_s \circ \pr$ satisfy $(\varphi_s)_{s \in S} \in \mathcal{M}$, the tension field $\tau(g,\phi) \in \Gamma(\prescript{\phi}{}{TN})$ has slower growth than the expected $\mathcal{O}(r^{-2})$ rate near the singular set. In fact if $\phi \in \Map^\infty_\mu(M,N;\mathfrak{S})$ and $g \in \Met^\infty(M;\mathfrak{S})$ we will see in a moment that the tension $\tau(g,\phi)$ is bounded by $\mathcal{O}(r^{\mu-2})$ near the singular set $S$. In particular, this holds for $g = f^*_{\underline{\mathfrak{t}},g^\prime} g^\prime$ where $\underline{\mathfrak{t}} \in \Ball_1^S$ any $g^\prime$ is any Riemannian metric, where $f^*_{\underline{\mathfrak{t}},g^\prime}$ is the diffeomorphism of \cref{Def: diffeos f}. 

The idea for proving the deformation theorem will therefore be to show that the tension fields $\tau(f_{\underline{\mathfrak{t}},g}^*g,\phi)$ fit together as a section of the Banach vector bundle $\mathcal{E}^{k,p}_{\mu-2}$ over the parameter space 
\begin{equation*}
    \Map^{k+2,p}_\mu(M,N;\mathfrak{S}) \times \Ball_1^S \times \Met^{k+2}(M) 
\end{equation*}
and to examine the conditions for the derivative of this section to be invertible, so as to use the Implicit Function Theorem. By means of writing the derivative of the tension in a reasonably concise way, let us define for any $\phi \in \Map^\infty_\mu(M,N;\mathfrak{S})$ and $g \in \Met^\infty(M)$ the finite-dimensional vector bundles
\begin{align}
    \prescript{\phi}{}{W} & \coloneqq \{ \prescript{\phi_0}{}{\Theta}_u(\chi_{\phi_0}(\dot{\underline{w}})) \mid \dot{\underline{w}} \in T_{\underline{\varphi}} \mathcal{M} \subset \times_s \Gamma(\prescript{\varphi_s}{}{TN})\} \subset \Gamma(\prescript{\phi}{}{TN}), \\
    \prescript{\phi}{}{V_g} & \coloneqq \{ \diff \phi(\xi_g(\dot{\underline{\mathfrak{t}}})) \mid \dot{\underline{\mathfrak{t}}} \in (\R^m)^S\} \subset \Gamma(\prescript{\phi}{}{TN}),
\end{align}
where $\phi_0$ is an end-conical map with the same tangent maps as $\phi$, $u \in \prescript{\phi_0}{}{\mathcal{U}}^{k,p}_\mu \cap C^\infty_\mu(\prescript{\phi_0}{}{TN})$ is a section such that $\phi = \exp_{\phi_0}(u)$, and $\xi_g(\underline{\mathfrak{t}})$ is the vector field introduced in \cref{prop: regularity of pull-back of metrics}. Note that in particular the properties of the family of vector fields $\xi_g(\cdot)$ imply that
\begin{equation*}
    \prescript{\phi}{}{V_g} \simeq (\R^m)^S \simeq T_{\underline{0}} \Ball_1^S. 
\end{equation*}
Moreover it is not difficult to see that $\prescript{\phi}{}{W} \subset C^{\infty}_0(\prescript{\phi}{}{TN})$ and $\prescript{\phi}{}{V_g} \subset C^\infty_{-1}(\prescript{\phi}{}{TN})$, and using the chart $\Psi_{\phi_0} \col \prescript{\phi_0}{}{\mathcal{U}}^{k,p}_\mu \times \prescript{\phi_0}{}{\mathcal{W}} \to \Map^{k,p}_\mu(M,N;\mathfrak{S})$ we obtain an identification\footnote{This identification is not canonical, since the subbundle $\mathcal{E}^{k+2,p}_\mu \subset T \Map^{k+2,p}_\mu(M,N;\mathfrak{S})$ has no canonical complement. This is why we needed to choose an end-conical map $\phi_0$ in order to define $\prescript{\phi}{}{W}_{\mathfrak{S}}$.}
\begin{equation}        \label{eq: tangent bundle of Map times B}
    T_{(\phi,\underline{0})}(\Map^{k+2,p}_\mu(M,N;\mathfrak{S}) \times \Ball_1^S) \simeq W^{k+2,p}_\mu(\prescript{\phi}{}{TN}) \oplus \prescript{\phi}{}{W} \oplus \prescript{\phi}{}{V_g} .
\end{equation}
Using this decomposition, the main theorem of this section can be stated as follows:

\begin{thm}     \label{thm: Variations of the tension}
    For any $\phi \in \Map^{k+2,p}_\mu(M,N;\mathfrak{S})$, $\underline{\mathfrak{t}} \in \Ball_1^S$ and $g \in \Met^{k+2}(M)$,
    \begin{equation*}
        \tau(f_{\underline{\mathfrak{t}},g}^*g,\phi) \in W^{k,p}_{\mu-2}(\prescript{\phi}{}{TN}),
    \end{equation*}
    and this defines a continuously differentiable Banach section
    \begin{equation*}
        \mathcal{T} \col \Map^{k+2,p}_{\mu}(M,N;\mathfrak{S}) \times \Ball_1^S \times \Met^{k+2}(M) \to \mathcal{E}^{k,p}_{\mu-2} .
    \end{equation*}
    Moreover, if $\phi \in \Map^\infty_\mu(M,N;\mathfrak{S})$ is harmonic with respect to $g \in \Met^\infty(M;\mathfrak{S})$ and we use \eqref{eq: tangent bundle of Map times B} to decompose the tangent space of $\Map^{k+2,p}_\mu(M,N;\mathfrak{S}) \times \Ball_1^S$, then the derivative\footnote{Recall that for any continuously differentiable section $\sigma$ of a Banach vector bundle $E \to B$ vanishing at $b \in B$, there is a well-defined (i.e. independent of a choice of local trivialisation) notion of derivative as a bounded linear map $D_b \sigma \col T_p B \to E_b$, essentially because the tangent space space of the total space of $E$ has a canonical splitting between vertical and horizontal parts along the zero section of $E$. This is not true when $\sigma(b)\neq 0$.} of $\mathcal{T}$ satisfies
    \begin{equation*}
        D_{(\phi,\underline{0},g)}\mathcal{T}\{\dot u + \dot w + \dot v\} = J^\phi_g(\dot u + \dot w + \dot v), ~~~~ \forall (\dot u, \dot w, \dot v) \in W^{k+2,p}_\mu(\prescript{\phi}{}{TN}) \oplus \prescript{\phi}{}{W} \oplus \prescript{\phi}{}{V_g} .
    \end{equation*}
\end{thm}

In a local trivialisation of the vector bundle $\mathcal{E}^{k,p}_{\mu-2}$ as described in \cref{subsubsec: Banach vector bundles}, the previous theorem is equivalent to the following statement:

\begin{prop}        \label{prop: variations of the tension}
    Let $\phi_0 \in \Map^\infty_{\mathrm{ec}}(M,N;\mathfrak{S})$ be an end-conical map. Then for any $(u,\underline{w},\underline{\mathfrak{t}},g) \in \prescript{\phi_0}{}{\mathcal{U}}^{k+2,p}_\mu \times \prescript{\phi_0}{}{\mathcal{W}} \times \Ball_1^S \times \Met^{k+2}(M)$,
    \begin{equation*}
        \prescript{\phi_0}{}{\mathcal{T}}(u,\underline{w},\underline{\mathfrak{t}},g)  \coloneqq \prescript{\phi_0}{}{\Theta}_{u+\chi_{\phi_0}(\underline{w})}^{-1} \tau(f_{\underline{\mathfrak{t}},g}^* g, \phi_{u+\chi_{\phi_0}(\underline{w})}) \in W^{k,p}_{\mu-2}(\prescript{\phi_0}{}{TN}) ,
    \end{equation*}
    and this induces a continuously differentiable Banach map
    \begin{equation*}
        \prescript{\phi_0}{}{\mathcal{T}} \col \prescript{\phi_0}{}{\mathcal{U}}^{k+2,p}_\mu \times \prescript{\phi_0}{}{\mathcal{W}} \times \Ball_1^S \times \Met^{k+2}(M) \to W^{k,p}_{\mu-2}(\prescript{\phi_0}{}{TN}) .
    \end{equation*}
    If moreover $u \in \prescript{\phi_0}{}{\mathcal{U}}^{k,p}_\mu \cap C^\infty_\mu(\prescript{\phi_0}{}{TN})$ is such that $\phi = \phi_{u}$ is harmonic with respect to some $g \in \Met^\infty(M;\mathfrak{S})$, then the derivative of $\prescript{\phi_0}{}{\mathcal{T}}$ at the point $(u,\underline{0},\underline{0},g)$ satisfies
    \begin{equation*}
        D_{(u,\underline{0},\underline{0},g)} \mathcal{T}\{\dot u+ \dot{\underline{w}} +\dot{\underline{\mathfrak{t}}}\} = \prescript{\phi_0}{}{\Theta}_{u}^{-1} \circ  J^{\phi}_g \circ \prescript{\phi_0}{}{\Theta}_{u} \{ \dot u+\chi_{\phi_0}(\dot{\underline{w}}) + \diff \phi (\xi_g(\dot{\underline{\mathfrak{t}}}))\}
    \end{equation*}
    where $J^\phi_g$ is the Jacobi operator of $\phi$ with respect to the metric $g$ and $\xi_g(\dot{\underline{\mathfrak{t}}}) \in \mathfrak{X}(M)$ the vector field defined in part (ii) of \cref{prop: regularity of pull-back of metrics}.
\end{prop}

The rest of this section is dedicated to the proof of \cref{prop: variations of the tension}. The idea is to rewrite the pull-back of the tension in terms of bundle maps with dilation-invariant ends and use the formalism developed in \cref{subsubsec: Bundle maps with dilation-invariant ends} in order to prove that $\mathcal{T}$ has the correct regularity and that its derivative is indeed given by the `naive' directional derivative, which corresponds to the claimed expression. 

    \subsubsection{Intrinsic expressions for the Hessian}       \label{subsubsec: expressions hessian}

In the remainder of this section, we shall fix an end-conical map $\phi_0 \in \Map^\infty_{\mathrm{ec}}(M,N;\mathfrak{S})$. For convenience, we shall also choose a background Riemannian metric $g_0$ on $M$ such that $\Upsilon^*_s g_0 = g_{\R^m}$ for $s \in S$ (remark that we require this property to hold on $\Ball_1$ and not just at the origin). Given the local expressions for the Christoffel symbols of the Levi-Civita connection, the following lemma is obvious:

\begin{lem}     \label{lem: Bundle map LC}
    Let us denote by $\mathrm{Sym}^2_+(T^*M) \subset \mathrm{Sym}^2(T^*M)$ the open sub-bundle of positive-definite quadratic forms. Then there exists a non-linear bundle map
    \begin{equation*}
        \bmap{D} \col \mathrm{Sym}_+^2(T^*M) \to \Hom(T^*M \otimes \mathrm{Sym}^2(T^*M) , T^* M \otimes \End(TM))
    \end{equation*}
    such that the Levi-Civita connections satisfy
    \begin{equation*}
        \nabla^{g} = \nabla^{g_0} + \bmap{D}(g)\{\nabla^{g_0} g\}, ~~~ \forall g \in \Met^1(M) .
    \end{equation*}
\end{lem}

We aim to write the Hessian of a map intrinsically in terms of bundle maps with dilation-invariant ends, in order to prove  \cref{prop: variations of the tension} in th next section. Before this, we remark that

\begin{lem}     \label{lem: Bundle map diff exp}
    There exists a non-linear bundle map with dilation-invariant ends
    \begin{equation*}
        \bmap{f} \col (\prescript{\phi_0}{}{U} \subset \prescript{\phi_0}{}{TN}) \to T^*M \otimes \prescript{\phi_0}{}{TN}
    \end{equation*}
    such that for any $u \in C^1_{\mathrm{loc}}(\prescript{\phi_0}{}{U})$, the map $\phi_u = \exp_{\phi_0}(u)$ satisfies
    \begin{equation*}
        \Theta_u^{-1} \diff \phi_u = \diff \phi_0 + \prescript{\phi_0}{}{\nabla} u + r^{-1} \bmap{f}(u),
    \end{equation*}
    where we write $\Theta_u \coloneqq \prescript{\phi_0}{}{\Theta}_u$ to make notations lighter.
\end{lem}

\begin{proof}
    By definition, $\phi_u = \exp_{\phi_0}(u)$, and as in \cref{lem: First derivative of bundle map} we can decompose the differential of the exponential map into a vertical component, 
    \begin{equation*}
        \diff^{\mathrm{v}}_u \exp_{\phi_0}\{v\} = \left. \tfrac{\diff}{\diff t} \right|_{t = 0} \exp_{\phi_0}(u+tv) = D_u \exp_{\phi_0} \{v\},
    \end{equation*} 
    and a horizontal component, $\diff^h \exp_{\phi_0}$, characterised by the property that
    \begin{equation*}
        \diff (\exp_{\phi_0}(u)) = D_u \exp_{\phi_0} \{\prescript{\phi_0}{}{\nabla} u\} + (\diff^h \exp_{\phi_0})(u), ~~~~ \forall u \in \Gamma(\prescript{\phi_0}{}{TN}) .
    \end{equation*}
    In particular, since $\Theta_u = D_u \exp_{\phi_0}$ for $u \in \Gamma(\prescript{\phi_0}{}{U})$, it follows that $\Theta_u^{-1} \diff \phi_u = \prescript{\phi_0}{}{\nabla} u + \bmap{f}^\prime(u)$ for some bundle map $\bmap{f}^\prime \col (\prescript{\phi_0}{}{U} \subset \prescript{\phi_0}{}{TN}) \to T^*M \otimes \prescript{\phi_0}{}{TN}$. To prove that $\bmap{f} \coloneqq r\bmap{f}^\prime$ has dilation-invariant ends, we argue as in \cref{lem: Derivative of dilation-invariant bundle maps}: for any section $u \in \Gamma(\prescript{\phi_0}{}{U})$ with dilation-invariant ends, $r \prescript{\phi_0}{}{\nabla} u$ and $r \Theta_u^{-1} \diff \phi_u$ must have dilation-invariant ends, and thus $r \bmap{f}^\prime(u)$ has dilation-invariant ends.
\end{proof}

Using \cref{lem: Bundle maps of end-conical maps}, \cref{lem: Bundle map LC} and \cref{lem: Bundle map diff exp}, it is straightforward to deduce

\begin{lem}     \label{lem: Expression of the Hessian}
    There are bundle maps with dilation-invariant ends, $\bmap{C} \col \prescript{\phi_0}{}{U} \to \prescript{\phi_0}{}{TN}$, $\bmap{E} \col \prescript{\phi_0}{}{U} \to T^*M^{\otimes 2} \otimes \prescript{\phi_0}{}{TN}$ and $\bmap{F} \col \prescript{\phi_0}{}{U} \to T^*M \otimes \prescript{\phi_0}{}{TN}$ such that for any $u \in C^2_{\mathrm{loc}}(\prescript{\phi_0}{}{U})$ and $g \in \Met^1(M)$, the Hessian of $\phi_u$ satisfies
    \begin{align*}
        \Theta_u^{-1} \mathrm{Hess}_g(\phi_u) & = \mathrm{Hess}_{g_0}(\phi_0) + \prescript{\phi_0}{}{\nabla}^2 u + r^{-2} \boldsymbol{\mathrm{E}}(u) + r^{-1} \boldsymbol{\mathrm{F}}(u) \{\prescript{\phi_0}{}{\nabla u}\} + \boldsymbol{\mathrm{C}}(u) \{\prescript{\phi_0}{}{\nabla}u \otimes \prescript{\phi_0}{}{\nabla}u\} \\ 
            & ~~~~ + \boldsymbol{\mathrm{D}}(g) \{\nabla^{g_0} g\} \{ \diff \phi_0 + r^{-1} \boldsymbol{\mathrm{A}}(u) + \boldsymbol{\mathrm{B}}(u) \{ \prescript{\phi_0}{}{\nabla}u\}\}
    \end{align*}
    where the bundle maps $\bmap{A},\bmap{B},\bmap{D}$ are as in \cref{lem: Bundle maps of end-conical maps} and \cref{lem: Bundle map LC}.
\end{lem}

\begin{rem}
    The precise nature of the bundle maps $\bmap{A}, \bmap{B}, \ldots , \bmap{F}$ is irrelevant for our purpose, the point being that the dilation-invariance of these maps near the singular set $S$ will be enough to deduce that the Banach section defined by the tension field has appropriate regularity and asymptotics near $S$ using the results of \cref{subsubsec: Bundle maps with dilation-invariant ends}.
\end{rem}

    \subsubsection{Proof of the main proposition}       \label{subsec: proof of the main proposition}

Given the preparatory work of the previous section, we shall deduce the the proof of \cref{prop: variations of the tension} as a consequence of the following two lemmas:

\begin{lem}     \label{lem: regularity w component}
    For any $(\underline{w},g) \in \prescript{\phi_0}{}{\mathcal{W}} \times \Met^{k+1}(M;\mathfrak{S})$,
    \begin{equation*}
        \Theta_{\chi_{0}(\underline{w})}^{-1} \tau(g,\phi_{\chi_{0}(\underline{w})}) \in C^{k}_{-1}(\prescript{\phi_0}{}{TN}) \subset W^{k,p}_{\mu-2}(\prescript{\phi_0}{}{TN}),
    \end{equation*}
    where we write $\chi_0 \coloneqq \chi_{\phi_0}$ to make notations lighter. Moreover, the induced map $\prescript{\phi_0}{}{\mathcal{W}} \times \Met^{k+1}(M) \to W^{k,p}_{\mu-2}(\prescript{\phi_0}{}{TN})$ is a smooth Banach map.
\end{lem}

\begin{proof}
    We can obviously write 
    \begin{equation*}
        \Theta_{\chi_{0}(\underline{w})}^{-1} \tau(g,\phi_{\chi_{0}(\underline{w})}) = (\tr_{g}-\tr_{g_0})(\Theta_{\chi_{0}(\underline{w})}^{-1} \mathrm{Hess}_{g}(\phi_{\chi_{0}(\underline{w})})) + \tr_{g_0}(\Theta_{\chi_{0}(\underline{w})}^{-1} \mathrm{Hess}_{g}(\phi_{\chi_{0}(\underline{w})})) .
    \end{equation*}
    We deal with each term separately. 
    
    On the one hand, it is clear that the map $\chi_{0}$ defines a smooth map $\prescript{\phi_0}{}{\mathcal{W}} \to C^{k+2}_0(\prescript{\phi_0}{}{TN})$, and examining the expression of the Hessian given in \cref{lem: Expression of the Hessian} and using the results of \cref{subsubsec: Bundle maps with dilation-invariant ends} we see that $\Theta_{\chi_{0}(\underline{w})}^{-1} \mathrm{Hess}_g(\phi_{\chi_{0}(\underline{w})})$ defines a smooth Banach map $\prescript{\phi_0}{}{\mathcal{W}} \times \Met^{k+1}(M;\mathfrak{S}) \to C^{k}_{-2}(T^*M^{\otimes 2} \otimes \prescript{\phi_0}{}{TN})$. On the other hand, by definition for any $g \in \Met^{k+1}(M;\mathfrak{S})$ we have $g_s = g_{0,s}$ for any $s \in S$, and thus $\Met^{k+1}(M;\mathfrak{S}) \subset g_0 + C^{k+1}_1(\mathrm{Sym}^2(T^*M))$. Given that we can express the trace with respect to a metric as a bundle map defined on $\mathrm{Sym}^2(T^*M)$, it follows from \cref{lem: Natural properties of weighted spaces} and \cref{lem: Master lemma} that the expression $(\tr_{g}-\tr_{g_0})(\Theta_{\chi_{0}(\underline{w})}^{-1} \mathrm{Hess}_{g}(\phi_{\chi_{0}(\underline{w})}))$ defines a smooth Banach map into $C^{k}_{-1}(T^*M^{\otimes 2} \otimes \prescript{\phi_0}{}{TN}) \subset W^{k,p}_{\mu-2}(T^*M^{\otimes 2} \otimes \prescript{\phi_0}{}{TN})$ since $\mu \in (0,1)$.
    
    To deal with the second term, another examination of the expression of the Hessian in \cref{lem: Expression of the Hessian} shows that
    \begin{multline*}
        \tr_{g_0}(\Theta_{\chi_{0}(\underline{w})}^{-1} \mathrm{Hess}_{g}(\phi_{\chi_{0}(\underline{w})})) = \Theta_{\chi_{0}(\underline{w})}^{-1} \tau(g_0,\phi_{\chi_{0}(\underline{w})}) \\ + \tr_{g_0}\bmap{D}(g)\{\nabla^{g_0}g\}\{\diff \phi_0 + r^{-1} \bmap{A}(\chi_0(\underline{w})) + \bmap{B}(\chi_0(\underline{w}))\{\prescript{\phi_0}{}{\nabla\chi_0(\underline{w})}\}\} .
    \end{multline*}
    Since $\chi_0$ defines a smooth Banach map $\prescript{\phi_0}{}{\mathcal{W}} \to C^{k+2}_0(\prescript{\phi_0}{}{TN})$, it is clear that the second term defines a smooth Banach map into $C^{k}_{-1}(\prescript{\phi_0}{}{TN}) \subset W^{k,p}_{\mu-2}(\prescript{\phi_0}{}{TN})$. In fact, so far we have only used the fact that $\chi_0(\underline{w}) \in C^{k+2}_0(\prescript{\phi_0}{}{TN})$, not the specific properties of $\prescript{\varphi_s}{}{\mathcal{W}}_s$. The only place where it plays a role is in the term $\tau(g_0,\phi_{\chi_0(\underline{w})})$: indeed by construction $\phi_{\chi_0(\underline{w})}$ is harmonic with respect to $g_0$ in a neighbourhood of $S$, and therefore $\tau(g_0,\phi_{\chi_0(\underline{w})})$ is uniformly compactly supported away form the singular locus $S$. Since $\prescript{\phi_0}{}{\mathcal{W}}$ is finite-dimensional, it is therefore clear that $\Theta_{\chi_0(\underline{w})}^{-1}\tau(g_0,\phi_{\chi_0(\underline{w})})$ defines a smooth Banach map $\prescript{\phi_0}{}{\mathcal{W}} \to C^{k}_\nu(\prescript{\phi_0}{}{TN})$ for any $\nu \in \R$, and in particular for $\nu = -1$. This finishes the proof.
\end{proof}

Before stating the next lemma, recall that we defined the open subbundle $\prescript{\phi_0}{}{U} \subset \prescript{\phi}{}{TN}$ of sections satisfying $|u|_{h_{\phi_0}} < \inj(N)$. We shall denote by $\prescript{\phi_0}{}{\tilde{U}} \subset \prescript{\phi_0}{}{U}$ the open subbundle of sections such that $|u|_{h_{\phi_0}} < \frac{1}{4}\inj(N)$.

\begin{lem}     \label{lem: restricted regularity of tension}
    For any $(u,w,g) \in \prescript{\phi_0}{}{\mathcal{U}}^{k+2,p}_\mu \times C^{k+2}_0(\prescript{\phi_0}{}{\tilde{U}}) \times \Met^{k+1}(M)$,
    \begin{equation*}
        \Theta_{u+w}^{-1} \tau(g,\phi_{u+w}) - \Theta_w^{-1} \tau(g,\phi_w) \in W^{k,p}_{\mu-2}(\prescript{\phi_0}{}{TN}).
    \end{equation*}
    Moreover, the induced map 
    \begin{equation*}
        \prescript{\phi_0}{}{\mathcal{U}}^{k+2,p}_\mu \times C^{k+2}_0(\prescript{\phi_0}{}{\tilde{U}}) \times \Met^{k+1}(M) \to W^{k,p}_{\mu-2}(\prescript{\phi_0}{}{TN})
    \end{equation*}
    is a smooth Banach map.
\end{lem}

\begin{proof}
    We will in fact prove a stronger statement: if $(u,w,g) \in \prescript{\phi_0}{}{\mathcal{U}}^{k+2,p}_\mu \times C^{k+2}_0(\prescript{\phi_0}{}{\tilde{U}}) \times \Met^{k+1}(M)$,
    \begin{equation*}
        \Theta_{u+w}^{-1} \mathrm{Hess}_g(\phi_{u+w}) - \Theta_w^{-1} \mathrm{Hess}_g(\phi_w) \in W^{k,p}_{\mu-2}(T^*M^{\otimes 2} \otimes \prescript{\phi_0}{}{TN}),
    \end{equation*}
    and the induced map $\prescript{\phi_0}{}{\mathcal{U}}^{k+2,p}_\mu \times C^{k+2}_0(\prescript{\phi_0}{}{\tilde{U}}) \times \Met^{k+1}(M) \to W^{k,p}_{\mu-2}(T^*M^{\otimes 2} \otimes \prescript{\phi_0}{}{TN})$ is a smooth Banach map. Since the trace $\Met^{k+1}(M) \times W^{k,p}_{\mu-2}(T^*M^{\otimes 2} \otimes \prescript{\phi_0}{}{TN}) \to W^{k,p}_{\mu-2}(\prescript{\phi_0}{}{TN})$ evidently defines a smooth Banach map this implies the lemma.
    
    Using the expression of the Hessian given in \cref{lem: Expression of the Hessian}, we can write
    \begin{equation}      \label{eq: pulled-back hessian}
        \begin{aligned}
            \Theta_{u+w}^{-1} & \mathrm{Hess}_{g}(\phi_{u+w}) - \Theta_w^{-1}\mathrm{Hess}_{g}(\phi_w) =  \prescript{\phi_0}{}{\nabla}^2 u + r^{-2}\widetilde{\bmap{E}}(u,w) \\
            & ~~~~ + r^{-1}(\bmap{F}(u+w) \{ \prescript{\phi_0}{}{\nabla}u\} + \widetilde{\bmap{F}}(u,w)\{\prescript{\phi_0}{}{\nabla} w\}) \\
            & ~~~~ + r^{-2}(\bmap{C}(u+w) \{ \prescript{\phi_0}{}{\nabla}u \otimes \prescript{\phi_0}{}{\nabla}u + 2 \prescript{\phi_0}{}{\nabla}w \otimes \prescript{\phi_0}{}{\nabla}u\} + \widetilde{\bmap{C}}(u,w)\{\prescript{\phi_0}{}{\nabla}w \otimes \prescript{\phi_0}{}{\nabla}w\}) \\
            & ~~~~ + \bmap{D}(g)\{\nabla^{g_0}g\}\{r^{-1} (\bmap{A}(u+w)-\bmap{A}(w)) + \bmap{B}(u+w)\{\prescript{\phi_0}{}{\nabla}(u+w)\} - \bmap{B}(w)\{\prescript{\phi_0}{}{\nabla}w\}\},
        \end{aligned}
    \end{equation}
    where
    \begin{equation*}
        \widetilde{\bmap{C}}(u,w) = \bmap{C}(u+w) - \bmap{C}(w), ~~ \widetilde{\bmap{E}}(u,w) = \bmap{E}(u+w) - \bmap{E}(w), ~~ \widetilde{\bmap{F}}(u,w) = \bmap{F}(u+w) - \bmap{F}(w) .
    \end{equation*}
    Using \cref{lem: Parametrised master lemma} and weighted Sobolev multiplication (\cref{prop: Sobolev embedding and multiplication}), it is clear that the terms on the first three lines of the right-hand side of \eqref{eq: pulled-back hessian} define smooth maps $\prescript{\phi_0}{}{\mathcal{U}}^{k+2,p}_\mu \times C^{k+2}_0(\prescript{\phi_0}{}{\tilde{U}}) \to W^{k,p}_{\mu-2}(T^*M^{\otimes 2} \otimes \prescript{\phi_0}{}{TN})$. As for the last line, the map $g \mapsto \bmap{D}(g)\{\nabla^{g_0}g\}$ defines a smooth Banach map $\Met^{k+1}(M) \to C^{k}(T^*M \otimes \End(TM)) \subset C^{k}_0(T^*M \otimes \End(TM))$, while the expression involving $u,w$ inside the brackets defines a smooth Banach map $\prescript{\phi_0}{}{\mathcal{U}}^{k+2,p}_\mu \times C^{k+2}_0(\prescript{\phi_0}{}{\tilde{U}}) \to W^{k+1,p}_{\mu-1}(T^*M \otimes \prescript{\phi_0}{}{TN})$ by \cref{lem: Parametrised master lemma} and weighted Sobolev multiplication. It follows that the last line defines a smooth map $\prescript{\phi_0}{}{\mathcal{U}}^{k+2,p}_\mu \times C^{k+2}_0(\prescript{\phi_0}{}{\tilde{U}}) \times \Met^{k+1}(M) \to W^{k+1,p}_{\mu-1}(\prescript{\phi_0}{}{TN}) \hookrightarrow W^{k,p}_{\mu-2}(\prescript{\phi_0}{}{TN})$.
\end{proof}

We are now finally equipped to prove \cref{prop: variations of the tension}, and thereby \cref{thm: Variations of the tension}:

\begin{proof}[Proof of \cref{prop: variations of the tension}]
    If $(u,\underline{w},\underline{{\mathfrak{t}}},g) \in \prescript{\phi_0}{}{\mathcal{U}}^{k+2,p}_\mu \times \prescript{\phi_0}{}{\mathcal{W}} \times \Ball_1^S \times \Met^{k+2}(M)$, then $\chi_0(\underline{w}) \in C^{k+2}_0(\prescript{\phi_0}{}{\tilde{U}})$ and by \cref{prop: regularity of pull-back of metrics}, $f_{\underline{\mathfrak{t}},g}^*g \in \Met^{k+1}(M;\mathfrak{S})$. Writing 
    \begin{multline*}
        \Theta_{u+\chi_0(\underline{w})}^{-1} \tau(f_{\underline{\mathfrak{t}},g}^*g,\phi_{u+\chi_0(\underline{w})}) = (\Theta_{u+\chi_0(\underline{w})}^{-1} \tau(f_{\underline{\mathfrak{t}},g}^*g,\phi_{u+\chi_0(\underline{w})}) - \Theta_{\chi_0(\underline{w})}^{-1} \tau(f_{\underline{\mathfrak{t}},g}^*g,\phi_{\chi_0(\underline{w})})) \\ + \Theta_{\chi_0(\underline{w})}^{-1} \tau(f_{\underline{\mathfrak{t}},g}^*g,\phi_{\chi_0(\underline{w})})
    \end{multline*}
    and using \cref{lem: regularity w component} and \cref{lem: restricted regularity of tension}, this implies that $\Theta_{u+\chi_0(\underline{w})}^{-1} \tau(f_{\underline{\mathfrak{t}},g}^* g, \phi_{u+\chi_0(\underline{w})}) \in W^{k,p}_{\mu-2}(\prescript{\phi_0}{}{TN})$. Since moreover the map $(\underline{\mathfrak{t}},g) \in \Ball_1^S \times \Met^{k+2}(M) \mapsto f_{\underline{\mathfrak{t}},g}^*g \in \Met^{k+1}(M;\mathfrak{S})$ is of class $C^1$ by \cref{prop: regularity of pull-back of metrics}, we deduce that $\prescript{\phi_0}{}{\mathcal{T}}$ defines a continuously differentiable Banach map from $\prescript{\phi_0}{}{\mathcal{U}}^{k+2,p}_\mu \times \prescript{\phi_0}{}{\mathcal{W}} \times \Ball_1^S \times \Met^{k+2}(M)$ to $W^{k,p}_{\mu-2}(\prescript{\phi_0}{}{TN})$.
    
    For the expression of the derivative of $\prescript{\phi_0}{}{\mathcal{T}}$, remark that by the results of \cref{subsubsec: Bundle maps with dilation-invariant ends} the Banach derivative of $\prescript{\phi_0}{}{\mathcal{T}}$ coincides with the naive directional derivative, which can be computed locally (c.f. \cref{rem: Derivative of Banach maps}). Hence it is immediate to deduce that if $u \in \prescript{\phi_0}{}{\mathcal{U}}^{k,p}_\mu \cap C^\infty_\mu(\prescript{\phi_0}{}{TN})$ is such that the map $\phi = \phi_{u}$ is harmonic with respect to some $g \in \Met^{k+2}(M;\mathfrak{S})$ and $(\dot u, \underline{\dot w}) \in W^{k+2,p}_\mu(\prescript{\phi_0}{}{TN}) \times T_{\underline{0}} \prescript{\phi_0}{}{\mathcal{W}}$, 
    \begin{equation*}
        D_{(u,\underline{0},\underline{0},g)} \prescript{\phi_0}{}{\mathcal{T}}\{ \dot u + \underline{\dot w}\} = \Theta_{u}^{-1} \circ J^{\phi}_g \circ \Theta_{u}(\dot u + \chi_0(\underline{\dot w})) .
    \end{equation*}
    As for the derivative in the direction of $\dot{\underline{\mathfrak{t}}} \in (\R^m)^S$, since $f_{\underline{0},g} = \mathrm{id}_M$ because $g \in \Met^{k+2}(M;\mathfrak{S})$ it follows from \cref{prop: regularity of pull-back of metrics} that $\left. \frac{\diff}{\diff t} \right|_{t = 0} f_{t \cdot \underline{\mathfrak{t}},g}^* g = \pounds_{\xi_g(\dot{\underline{\mathfrak{t}}})}g$. On the other hand since $\phi$ is smooth on $M \setminus S$ the relation $\left. \frac{\diff}{\diff t} \right|_{t = 0}\tau(f_{t \cdot \underline{\mathfrak{t}},g}^* g,\phi) = J^{\phi}_g(\diff \phi(\xi_g(\dot{\underline{\mathfrak{t}}})))$ must hold locally away from the singular set. Again, by \cref{rem: Derivative of Banach maps} this shows that
    \begin{equation*}
        D_{(u,0,\underline{0},g)} \prescript{\phi_0}{}{\mathcal{T}}\{ \dot{\underline{\mathfrak{t}}}\} = \Theta_{u}^{-1} J^{\phi}_g(\diff \phi(\xi_g(\dot{\underline{\mathfrak{t}}}))),
    \end{equation*}
    as claimed.
\end{proof}

\begin{rem}     \label{rem: derivatives of the metric}
    Examining the above proof, the only place where we need to use metrics of regularity $C^{k+2}$ instead of $C^{k+1}$ is when we use the fact that the map $(\underline{\mathfrak{t}},g) \in \Ball_1^S \times \Met^{k+2}(M) \to \Met^{k+1}(M;\mathfrak{S})$ is of class $C^1$. Indeed if we started with a metric $g \in \Met^{k+1}(M)$, the same proof would show that $\tau(f_{\underline{\mathfrak{t}},g}^*g,\phi) \in W^{k,p}_{\mu-2}(\prescript{\phi}{}{TN})$ for any $\phi \in \Map^{k+2,p}_\mu(M,N;\mathfrak{S})$, but the induced section $\Map^{k+2,p}_\mu(M,N;\mathfrak{S}) \times \Ball^S_1 \times \Met^{k+1}(M) \to \mathcal{E}^{k,p}_{\mu-2}$ would a priori only be continuous and not differentiable. This is eventually why we can only prove \cref{thm: main theorem introduction} for small $C^2$-variations of the metric, as opposed to small $C^1$-variations.
\end{rem}

%% file: section4_proof.tex
    \section{The deformation theorem}       \label{sec: The deformation theorem}

In this section we will prove with \cref{thm: main theorem introduction} our first main result of this article. For this, recall that we have introduced the continuously differentiable Banach section 
\begin{equation*}
    \mathcal{T} \col \Map^{k+2,p}_{\mu}(M,N;\mathfrak{S}) \times \Ball_1^S \times \Met^{k+2}(M) \to \mathcal{E}^{k,p}_{\mu-2}.
\end{equation*}
In order to prove \cref{thm: main theorem introduction} we want to apply the implicit function theorem and therefore we need to prove that the linearisation of $\mathcal{T}$ with respect to the first two variables is surjective (for an appropriate choice of singularity data $\mathfrak{S}$ and corresponding Banach manifold $\Map^{k+2,p}_{\mu}(M,N;\mathfrak{S})$). Note, however, that in the presence of non-trivial Killing vector fields on $N$ the linearisation of $\mathcal{T}$ has no chance of being surjective when considered as a section of $\mathcal{E}^{k,p}_{\mu-2}$. Indeed, \cref{prop: mapping properties Jacobi operator} implies that this linearisation is Fredholm of index zero (cf. the discussion prior to \cref{lem: image of augmented deformation operator}). Moreover, the deformations induced by the Killing vector fields contribute to the kernel and hence to the cokernel of this operator.

In order to remedy this, we will define in \cref{subsec: the tension revisited} a Banach subbundle $\mathcal{S}^{k,p}_{\mu-2} \subset \mathcal{E}^{k,p}_{\mu-2}$ that consists of sections perpendicular to the Killing vector fields. We will then show that the tension $\mathcal{T}$ actually maps into $\mathcal{S}^{k,p}_{\mu-2}$, thus overcoming the part of the cokernel of $\mathcal{T}$ spanned by these Killing fields. With this at hand, we will state and prove a more precise version of \cref{thm: main theorem introduction} in \cref{subsec: the deformation theorem}.

    \subsection{Orthogonality to the Killing fields}      \label{subsec: the tension revisited}

As in the previous section, we will assume that $\dim M = m \geq 4$ and fix a singular data set $\mathfrak{S} =(S,(\Upsilon_s)_{s\in S},(\mathcal{M}_s)_{s\in S})$ consisting of:
\begin{enumerate}
    \item a finite subset $S \subset M$,
    \item a set of adapted charts $\{\Upsilon_s \col (\mathbb{B}_1 \subset \mathbb{R}^m) \to M\}_{s\in S}$, 
    \item for each $s\in S$ a finite dimensional submanifold $\mathcal{M}_s \subset C^\infty(\Sp^{m-1},N)$ of maps harmonic with respect to the standard round metric on $\Sp^{m-1}$.
\end{enumerate}
Moreover, we fix $k\in \mathbb{N}_0$, $p\in (m,\infty)$, and $\mu \in (0,1)$ and let $\Map^{k+2,p}_\mu(M,N;\mathfrak{S})$, $\mathcal{E}^{k,p}_{\mu-2}$ be as defined in the previous section. Our goal for this section is to show that the tension $\mathcal{T}$, regarded as a section of $\mathcal{E}^{k,p}_{\mu-2}$, lies orthogonal to the space of (pulled back) Killing vector fields on $(N,h)$. For this, we first show that the set of pulled back Killing vector fields on $(N,h)$ span a subbundle of $\mathcal{E}^{k,p}_{\mu-2}$, at least on an appropriate open subset of the parameter space.

\begin{lem}    \label{lem: subbundle of pulled back Killing fields}
    Let $\widetilde{\mathcal{S}}_{\phi}^{k,p}\subset (\mathcal{E}^{k,p}_{\mu-2})_{\phi}$ for every $\phi\in \Map^{k+2,p}_\mu(M,N;\mathfrak{S})$ be defined as
    \begin{equation*}
        \widetilde{\mathcal{S}}_{\phi}^{k,p} \coloneqq \phi^*\mathfrak{iso}(M,h) \subset C^{k+1}_0(\prescript{\phi}{}{TN})\subset W^{k,p}_{\mu-2}(\prescript{\phi}{}{TN}) = (\mathcal{E}^{k,p}_{\mu-2})_{\phi} ,
    \end{equation*} 
    where $\mathfrak{iso}(M,h) \subset \mathfrak{X}(N)$ denotes the vector space of Killing vector fields on $(N,h)$. Then, on the open subset where $\dim \widetilde{\mathcal{S}}_{\phi}^{k,p} = \dim \mathfrak{iso}(M,h)$, $\widetilde{\mathcal{S}}^{k,p} \coloneqq \sqcup \widetilde{\mathcal{S}}_{\phi}^{k,p}$ forms a smooth vector subbundle of $\mathcal{E}^{k,p}_{\mu-2}$.
\end{lem}

\begin{proof}
    We define the following map: 
    \begin{align*}
        \mathcal{I} \col \Map^{k+2,p}_{\mu}(M,N;\mathfrak{S}) \times \mathfrak{iso}(N,h) &\to \widetilde{\mathcal{S}}^{k,p}\subset \mathcal{E}^{k,p}_{\mu-2} \\
        (\phi,\xi) &\mapsto \phi^*\xi \in \widetilde{\mathcal{S}}_{\phi}^{k,p}.
    \end{align*}
    Since $\mathfrak{iso}(N,h)$ is finite-dimensional and finite dimensional subspaces in Banach spaces always split, by \cite[Proposition~3.1 in Chapter~3]{Lang1999DifferentialGeometry} it is enough to prove that $\mathcal{I}$ is a smooth vector bundle homomorphism. 

    Let us pick an end-conical map $\phi_0 \in \Map^\infty_{\mathrm{ec}}(M,N;\mathfrak{S})$ and let $\Psi_{\phi_0} \col \prescript{\phi_0}{}{\mathcal{U}}^{k+2,p}_{\mu} \times \prescript{\phi_0}{}{\mathcal{W}} \to \prescript{\phi_0}{}{\mathcal{V}}^{k+2,p}_\mu \subset \Map^{k+2,p}_\mu(M,N;\mathfrak{S})$ be the local chart centred around $\phi_0$ as in \cref{prop: banach manifold atlas}. Moreover, for any $u \in W^{k+2,p}_{\mu}(\prescript{\phi_0}{}{TN}) \subset C^0(\prescript{\phi_0}{}{TN})$ with $\vert u \vert_{h_{\phi_0}}< \textup{inj}(N)$ we set $\phi_u \coloneqq \exp_{\phi_0}(u)$ and let $\prescript{\phi_0}{}{\Theta_u} \col \prescript{\phi_0}{}{TN} \to \prescript{\phi_u}{}{TN}$ be the isomorphism in \cref{Def: Bundle isomorphism Theta}. Using the local trivialisations 
    \begin{equation*}
        \left. \mathcal{E}^{k,p}_{\mu-2} \right|_{\prescript{\phi_0}{}{\mathcal{V}^{k+2,p}_\mu}} \simeq \prescript{\phi_0}{}{\mathcal{U}}^{k+2,p}_{\mu} \times \prescript{\phi_0}{}{\mathcal{W}} \times W^{k,p}_{\mu-2}\big(\prescript{\phi_0}{}{TN}\big) .
    \end{equation*}
    as in \cref{subsubsec: Banach vector bundles}, we need to show that the map 
    \begin{align*}
        \prescript{\phi_0}{}{\mathcal{U}}^{k+2,p}_{\mu} \times \prescript{\phi_0}{}{\mathcal{W}} \times \mathfrak{iso}(N,h) & \to  W^{k,p}_{\mu-2}(\prescript{\phi_0}{}{TN}) \\
        (u,\underline{w},\xi) &\mapsto (\prescript{\phi_0}{}{\Theta}_{u+\chi_{\phi_0}(\underline{w})})^{-1}(\phi_{u+\chi_{\phi_0}(\underline{w})}^*\xi)
    \end{align*}
    is smooth. To see this, first note that for every fixed $\xi \in \mathfrak{iso}(N,h)$ the expression on the right-hand side arises from a non-linear bundle map with dilation invariant ends as in \cref{lem: Master lemma}. Thus, the aforementioned lemma implies that for fixed $\xi \in \mathfrak{iso}(N,h)$ the expression on the right-hand side defines a smooth map $\prescript{\phi_0}{}{\mathcal{U}}^{k+2,p}_{\mu} \times \prescript{\phi_0}{}{\mathcal{W}} \to W^{k,p}_{0}(\prescript{\phi_0}{}{TN}) \subset W^{k,p}_{\mu-2}(\prescript{\phi_0}{}{TN})$. Since $\mathcal{I}$ depends linearly on $\xi \in \mathfrak{iso}(N,h)$ and the latter space is finite-dimensional, this proves the smoothness of $\mathcal{I}$ after choosing a basis of $\mathfrak{iso}(N,h)$.
\end{proof}

In the following, we will equip $\mathcal{E}^{k,p}_{\mu-2} \to \Map^{k+2,p}_{\mu}(M,N;\mathfrak{S}) \times (\mathbb{B}_1)^S\times \Met^{k+2}(M)$ with a bundle metric $\pairl \cdot,\cdot \pairr$ as follows: for any $(\phi, \underline{\mathfrak{t}},g)$ and $v_1,v_2 \in (\mathcal{E}^{k,p}_{\mu-2})_{(\phi,\underline{\mathfrak{t}},g)} \cong W^{k,p}_{\mu-2}(\prescript{\phi}{}{TN})$,
\begin{equation*}
    \pairl v_1,v_2 \pairr \coloneqq \int_M \langle v_1,v_2 \rangle_{h_{\phi}} \vol_{(f_{\underline{\mathfrak{t}},g})^*g} . 
\end{equation*}
Note that because $\mu>0$ and because we assume that $\dim M \geq 4$, this is indeed well-defined.

\begin{lem}
    The bundle metric $\pairl \cdot, \cdot \pairr$ on $\mathcal{E}^{k,p}_{\mu-2} \to \Map^{k+2,p}_\mu(M,N;\mathfrak{S})\times (\mathbb{B}_1)^S\times \Met^{k+2}(M)$ is of regularity $C^{k+2}$.
\end{lem}

\begin{proof}
    As in the proof of the previous proposition, we fix an end-conical map $\phi_0 \in \Map^{\infty}_{\textup{ec}}(M,N;\mathfrak{S})$ and use the local trivialisation
    \begin{align*}
        \left. \mathcal{E}^{k,p}_{\mu-2}\right|_{\prescript{\phi_0}{}{\mathcal{V}}^{k+2,p}_\mu} \cong \left(\prescript{\phi_0}{}{\mathcal{U}}^{k+2,p}_\mu \times \prescript{\phi_0}{}{\mathcal{W}} \right) \times W^{k,p}_{\mu-2}(\prescript{\phi_0}{}{TN}) .
    \end{align*}
    There, the metric $\pairl \cdot,\cdot \pairr$ becomes
    \begin{equation*}
        \pairl v_1,v_2\pairr_{(u,\underline{w},\underline{\mathfrak{t}},g)} = \int_M \big\langle \prescript{\phi_0}{}{\Theta_{u+\chi_{\phi_0}(\underline{w})}}(v_1),\prescript{\phi_0}{}{\Theta_{u+\chi_{\phi_0}(\underline{w})}}(v_2) \big\rangle_{h_{\phi_{u+\chi_{\phi_0}(\underline{w})}}} \vol_{(f_{\underline{\mathfrak{t}},g})^*g} , 
    \end{equation*} 
    where $u\in \prescript{\phi_0}{}{\mathcal{U}}^{k+2,p}_\mu$, $\underline{w} \in \prescript{\phi_0}{}{\mathcal{W}}$, $\underline{\mathfrak{t}}\in (\mathbb{B}_1)^S$, $g\in \Met^{k+2}(M)$, and $v_1,v_2 \in W^{k,p}_{\mu-2}(\prescript{\phi_0}{}{TN})$). 
    
    Now observe that 
    \begin{align*}
        \prescript{\phi_0}{}{\mathcal{U}}^{k+2,p}_\mu \times \prescript{\phi_0}{}{\mathcal{W}}  &\to W^{k+2,p}_0\big(\Hom(\prescript{\phi_0}{}{TN}\otimes \prescript{\phi_0}{}{TN},\mathbb{R})\big) \\
        (u,\underline{w}) &\mapsto \big\langle \prescript{\phi_0}{}{\Theta_{u+\chi_{\phi_0}(\underline{w})}}(\cdot),\prescript{\phi_0}{}{\Theta_{u+\chi_{\phi_0}(\underline{w})}}(\cdot) \big\rangle_{h_{\phi_{u+\chi_{\phi_0}(\underline{w})}}} \\
        & \qquad = \big\langle D_{u+\chi_{\phi_0}(\underline{w})} \exp_{\phi_0}(\cdot),D_{u+\chi_{\phi_0}(\underline{w})} \exp_{\phi_0}(\cdot) \big\rangle_{h}
    \end{align*}
    arises via a non-linear bundle map with dilation-invariant ends and is therefore smooth by \cref{lem: Master lemma} (and maps indeed into $W^{k+2,p}_0$). By \cref{prop: Sobolev embedding and multiplication}, it yields a smooth map 
    \begin{equation*}
        \prescript{\phi_0}{}{\mathcal{U}}^{k+2,p}_\mu \times \prescript{\phi_0}{}{\mathcal{W}} \to \Hom\big(W^{k,p}_{\mu-2}(\prescript{\phi_0}{}{TN})\otimes W^{k,p}_{\mu-2}(\prescript{\phi_0}{}{TN}),W^{k,p}_{2\mu-4}(\mathbb{R})\big).
    \end{equation*}
    Moreover, the mapping $g \mapsto \vol_g$ defines a smooth map $\Met^0(M) \to C^{0}(\Lambda^m T^*M)$ and that integration is smooth as well. So far, all maps have been smooth maps between Banach spaces. However, the map 
    \begin{align*}
        (\mathbb{B}_1)^S\times \Met^{k+2}(M) &\to \Met^0(M)\\
        (\underline{\mathfrak{t}},g)&\mapsto f_{\underline{\mathfrak{t}},g}^*g
    \end{align*} 
    is only of regularity $C^{k+2}$ (cf. \cref{prop: regularity of pull-back of metrics}) and therefore $\pairl \cdot,\cdot\pairr$ is also only of regularity $C^{k+2}$.
\end{proof}

\begin{Def}     \label{Def: Banach subbundle}
    Let $\mathcal{S}^{k,p}_{\mu-2} \subset \mathcal{E}^{k,p}_{\mu-2}$ be the orthogonal complement of $\widetilde{\mathcal{S}}^{k,p}$ with respect to $\pairl \cdot,\cdot \pairr$. By the above lemma, this is a Banach subbundle of $\mathcal{E}^{k,p}_{\mu-2}$ of regularity $C^{k+2}$ over the open subset of parameters $\mathcal{P}^{k+2,p}_\mu$ defined as
    \begin{equation*}
        \mathcal{P}^{k+2,p}_\mu \coloneqq \{ (\phi,\underline{\mathfrak{t}},g) \in \Map^{k+2,p}_\mu(M,N;\mathfrak{S}) \times \Ball_1^S \times \Met^{k+2}(M) \mid \dim \phi^* \mathfrak{iso}(N) = \dim \mathfrak{iso}(N) \}.
    \end{equation*}
\end{Def}

Recall now that the $C^1$-section $\mathcal{T} \col \Map^{k+2,p}_{\mu}(M,N;\mathfrak{S}) \times (\mathbb{B}_1)^S \times \Met^{k+2}(M) \to \mathcal{E}^{k,p}_{\mu-2}$ is given for every $(\phi,\underline{\mathfrak{t}},g)$ by 
\begin{equation*}
    \mathcal{T}(\phi,g,\underline{\mathfrak{t}}) = \tau((f_{g,\underline{\mathfrak{t}}})^*g,\phi)) \in W^{k,p}_{\mu-2}(\prescript{\phi}{}{TN}) = (\mathcal{E}^{k,p}_{\mu-2})_{(\phi,\underline{\mathfrak{t}},g)}  
\end{equation*}
as defined in \cref{thm: Variations of the tension}.

\begin{prop}        \label{prop: tension lies in subbundle}
    For every $(\phi,\underline{\mathfrak{t}},g) \in \mathcal{P}^{k+2,p}_\mu \subset \Map^{k+2,p}_{\mu}(M,N;\mathfrak{S}) \times (\mathbb{B}_1)^S \times \Met^{k+2}(M)$ we have 
    \begin{equation*}
        \mathcal{T}(\phi,\underline{\mathfrak{t}},g) \in (\mathcal{S}^{k,p}_{\mu-2})_{(\phi,\underline{\mathfrak{t}},g)}.
    \end{equation*} 
    Therefore $\mathcal{T}$ is naturally a continuously differentiable section of $\mathcal{S}^{k,p}_{\mu-2}$ over $\mathcal{P}^{k+2,p}_\mu$.
\end{prop}

\begin{proof}
    It suffices to show that the tension $\tau((f_{g,\underline{\mathfrak{t}}})^*g,\phi))$ is $L^2$-perpendicular to $\phi^*\xi$ for any Killing vector field $\xi \in \mathfrak{iso}(N,h)$. Let $\textup{Flow}_t^{\xi} \col N \to N$ for any $t \in \mathbb{R}$ be the time-$t$ flow of $\xi$. Since $\textup{Flow}_t^{\xi}$ is an isometry of $(N,h)$, the Dirichlet energy $E(\textup{Flow}_t^{\xi} \circ \phi)$ (with respect to the metrics $(f_{\underline{\mathfrak{t}},g})^*g \in \Met^{k+2}(M)$ and $h \in \Met^\infty(N)$) is $t$-independent. By the definition of the tension, this implies 
    \begin{equation*}
        - \int_M \big\langle \tau((f_{g,\underline{\mathfrak{t}}})^*g,\phi)), \phi^*\xi \big\rangle_{h_\phi} \vol_{(f_{g,\underline{\mathfrak{t}}})^*g} = \tfrac{\diff}{\diff t} E(\textup{Flow}_t^{\xi} \circ \phi)_{\vert t=0} = 0. \qedhere
    \end{equation*}
\end{proof}

    \subsection{Proof of the deformation theorem}        \label{subsec: the deformation theorem}

The following is an immediate consequence of the Implicit Function Theorem:

\begin{prop}    \label{prop: deformation theorem via IFT}
    Let $(\phi,g) \in \Map^{k+2,p}_\mu(M,N;\mathfrak{S}) \times \Met^{k+2}(M)$, such that $\phi$ is harmonic with respect to $g$ and $\phi^*\mathfrak{iso}(N,h) \cong \mathfrak{iso}(N,h)$. Moreover, assume that the partial derivative 
    \begin{align*}
        (D_1 \mathcal{T})_{(\phi,\underline{0},g)} \col T_{(\phi,\underline{0})} \big(\Map^{k+2,p}_\mu(M,N;\mathfrak{S}) \times (\mathbb{B}_1)^S\big) \to (\mathcal{S}^{k,p}_{\mu-2})_{(\phi,\underline{0},g)}
    \end{align*}
    is surjective with finite-dimensional kernel $\mathfrak{Ker}$. Then there exists an open neighbourhood $\widetilde{U}_1 \subset \mathfrak{Ker}$ of zero, an open neighbourhood $\widetilde{U}_2 \subset \Met^{k+2}(M)$ of $g$ and a $C^1$-map 
    \begin{align*}
         \mathfrak{h} \col \widetilde{U}_1 \times \widetilde{U}_2 &\to \Map^{k+2,p}_\mu(M,N;\mathfrak{S}) \times (\mathbb{B}_1)^S \\
         (\mathfrak{p},\tilde{g}) &\mapsto (\phi_{(\mathfrak{p},\tilde{g})},\underline{\mathfrak{t}}_{(\mathfrak{p},\tilde{g})})
    \end{align*} 
    with $\mathfrak{h}(0,g) = (\phi,\underline{0})$, such that $\mathcal{T}(\mathfrak{h}(\mathfrak{p},\tilde{g}),\tilde{g}\big) = 0$ for every $(\mathfrak{p},\tilde{g}) \in \widetilde{U}_1 \times \widetilde{U}_2$. In particular, for every $\tilde{g} \in \widetilde{U}_2 \subset \Met^{k+2}(M)$ the map $\phi_{(0,\tilde{g})} \circ f_{\tilde{g},\underline{\mathfrak{t}}_{(0,\tilde{g})}}^{-1}$ is harmonic with respect to $\tilde{g}$ and conically singular with singular set $\{\Upsilon_s((\mathfrak{t}_{(0,\tilde{g})})_s)\}_{s\in S}$ and tangent maps in $(\mathcal{M}_s)_{s\in S}$.
\end{prop}

In the following, we will derive sufficient conditions for the partial derivative $D_1\mathcal{T}$ to be surjective. Let $(\phi,g) \in \Map^\infty_\mu(M,N;\mathfrak{S}) \times \Met^{\infty}(M)$ such that $\phi$ is harmonic with respect to $g$ and $\dim \phi^* \mathfrak{iso}(N) = \dim \mathfrak{iso}(N)$. Let us moreover pick an end-conical map $\phi_0 \in \Map^\infty_{\mathrm{ec}}(M,N;\mathfrak{S})$ that has the same tangent maps as $\phi$, such that $\phi \in \prescript{\phi_0}{}{\mathcal{V}}^{k+2,p}_\mu$, and let $u \in C^\infty_\mu(\prescript{\phi_0}{}{TN})$ such that $\phi_0 = \phi_u$. Recall that in \cref{subsubsec: tension as section} we used $\phi_0$ to define a splitting
\begin{equation*}
    T_{(\phi,\underline{0})}(\Map^{k+2,p}_\mu(M,N;\mathfrak{S}) \times \Ball_1^S) \simeq W^{k+2,p}_\mu(\prescript{\phi}{}{TN}) \oplus \prescript{\phi}{}{W} \oplus \prescript{\phi}{}{V_g} ,
\end{equation*}
where $T_{\underline{\varphi}} \mathcal{M} \simeq \prescript{\phi}{}{W} = \{\prescript{\phi_0}{}{\Theta}_u (\chi_{\phi_0}(\dot{\underline{w}})) \mid \dot{\underline{w}} \in T_{\underline{\varphi}} \mathcal{M}\} \subset C^\infty_0(\prescript{\phi}{}{TN})$ and $(\R^m)^S \simeq \prescript{\phi}{}{V}_g = \{ \diff \phi(\xi_g(\dot{\underline{\mathfrak{t}}})) \mid \dot{\underline{\mathfrak{t}}} \in (\R^m)^S\} \subset C^\infty_{-1}(\prescript{\phi}{}{TN})$ are finite-dimensional. Moreover, for all $(\dot u, \dot w, \dot v) \in W^{k+2,p}_\mu(\prescript{\phi}{}{TN}) \oplus \prescript{\phi}{}{W} \oplus \prescript{\phi}{}{V_g}$,
\begin{equation*}
    (D_1 \mathcal{T})_{(\phi,\underline{0},g)}\{\dot u+ \dot w +\dot v\} = J^\phi_g \{\dot u+ \dot w +\dot v\}  \in W^{k,p}_{\mu-2}(\prescript{\phi}{}{TN}) \cong (\mathcal{E}^{k,p}_{\mu-2})_{\phi},
\end{equation*}
where $J^\phi_g$ is the Jacobi operator of $\phi$ with respect to $g$. In fact, using either the fact that $\mathcal{T}$ is a $C^1$ section of $\mathcal{S}^{k,p}_{\mu-2}$ or \cref{cor: obstruction pairing}, we see that $(D_1 \mathcal{T})_{(\phi,\underline{0},g)}$ maps into the closed subspace $(\mathcal{S}^{k,p}_{\mu-2})_{(\phi,\underline{0},g)}$, that is, the $L^2$-orthogonal subspace of $\phi^*\mathfrak{iso}(N)$ in $W^{k,p}_{\mu-2}(\prescript{\phi}{}{TN})$. In this setting, the following lemma is an immediate consequence of Lockhart-McOwen theory \cref{prop: mapping properties Jacobi operator}:

\begin{lem}    \label{lem: image of augmented deformation operator}
    Let $(\phi,g) \in \Map^\infty_\mu(M,N;\mathfrak{S}) \times \Met^{\infty}(M)$ be as above. Then if $\mu$ is not an indicial root of the Jacobi operator $J^\phi_g$, the map
    \begin{equation*}
        (D_1 \mathcal{T})_{(\phi,\underline{0},g)} \col T_{(\phi,\underline{0})}(\Map^{k+2,p}_\mu(M,N;\mathfrak{S}) \times \Ball_1^S) \to (\mathcal{S}^{k,p}_{\mu-2})_{(\phi,\underline{0},g)}
    \end{equation*}
    is Fredholm. Moreover, it is surjective if and only if for every 
    \begin{equation*}
        z \in \ker\big(J^{\phi}_g \big) \cap C^\infty_{-m+2-\mu}(\prescript{\phi}{}{TN}) \cap \big(\phi^*\mathfrak{iso}(N,h)\big)^{\perp}
    \end{equation*}
    there exist $(\dot w, \dot v) \in \prescript{\phi}{}{W} \oplus \prescript{\phi}{}{V}_{g}$ such that
    \begin{equation*}
        \int_M \big\langle J^{\phi}_g\{\dot w + \dot v\}, z  \big\rangle_{h_\phi} \vol_{g} \neq 0.
    \end{equation*}
\end{lem}

Henceforth, we shall denote the tangent maps of $\phi$ by $(\varphi_s)_{s \in S} \in \mathcal{M}$. In order to apply this proposition in the following, we can use the following properties of $\prescript{\phi}{}{W}$ and $\prescript{\phi}{}{V}_{g}$:

\begin{lem}    \label{lem: expansion of extra terms in augmented operator}
    There are injective linear maps
    \begin{equation*}
        \prescript{\phi}{}{\underline{\iota}} = (\prescript{\phi}{}{\iota}_s)_{s \in S} \col \prescript{\phi}{}{W} \to \bigoplus_{s \in S} \ker(J^{\varphi_s}), ~~ \prescript{\phi}{}{\underline{\jmath}} = (\prescript{\phi}{}{\jmath}_s)_{s \in S} \col \prescript{\phi}{}{V}_{g} \to \bigoplus_{s \in S} \ker(J^{\varphi_s}+m-3)
    \end{equation*}
    such that:
    \begin{enumerate}[(i)]
        \item for all $\dot w \in \prescript{\phi}{}{W}$ and $s \in S$, $\dot w - \Psi_s(\iota_s(\dot w)) \in C^\infty_{\varepsilon}(\prescript{\phi}{}{TN})$ in a neighbourhood of $s$, and
        \item for all $\dot v \in \prescript{\phi}{}{W}$ and $s \in S$, $\dot v - r^{-1}\Psi_s(\jmath_s(\dot v)) \in C^\infty_{-1+\varepsilon}(\prescript{\phi}{}{TN})$
    \end{enumerate}
    for some $\varepsilon > 0$, where $\Psi_s$ was defined in \cref{prop: The pulled-back bundle has dilation-invariant ends} and $r$ is the distance to the singular locus.
\end{lem}

\begin{proof}
    The existence of linear maps satisfying (i) and (ii) is obvious from construction of $\prescript{\phi}{}{W}$ and $\prescript{\phi}{}{V}_{g}$; the former is the space of infinitesimal deformations of the tangent maps of $\phi$ along $\mathcal{M}$, and the latter is the space of infinitesimal translations of the singular set. On the other hand, the injectivity of $\prescript{\phi}{}{\underline{\iota}}$ follows from our assumption that $\mathcal{M}_s$ is an embedded submanifold of $C^\infty(\Sp^{m-1},N)$ for all $s \in S$. As for $\underline{\jmath}$, this follows from the observation that the space of translation in $\R^m$ generate an $m$-dimensional subspace of $\ker(J^{\varphi_s}+m-3)$ for any $s \in S$ (cf. \cref{subsubsec: further remarks}).
\end{proof}

\begin{cor}     \label{cor: surjectivity of derivative}
    Let us choose $\mu \in (0,1)$ smaller than the lowest positive critical rate of $J^\varphi_g$. Then under the assumptions of \cref{lem: image of augmented deformation operator}, $(D_1 \mathcal{T})_{(\phi,\underline{0},g)}$ is surjective if and only if:
    \begin{enumerate}[(i)]
        \item For any $s \in S$, $-(m-3)$ and $0$ are the only non-positive eigenvalues of $J^{\varphi_s}$.
        \item $\dim \ker(J^{\varphi_s}+m-3) = m$ for all $s \in S$.
        \item $T_{\underline{\varphi}} \mathcal{M} + \underline{\varphi}^* \mathfrak{iso}(N) = \oplus_{s \in S} \ker(J^{\varphi_s})$.
        \item $\ker(J^\phi_g) \cap C^\infty_{-1}(\prescript{\phi}{}{TN}) = \phi^* \mathfrak{iso}(N)$.
    \end{enumerate}
    Moreover, if these conditions hold and (iii) is strengthened to $T_{\underline{\varphi}} \mathcal{M} = \oplus_{s \in S} \ker(J^{\varphi_s})$, then 
    \begin{equation*}
        \ker (D_1 \mathcal{T})_{(\phi,\underline{0},g)} \simeq \phi^* \mathfrak{iso}(N) .
    \end{equation*}
\end{cor}

\begin{proof}
    Let us fix $z \in \ker\big(J^{\phi}_{-m+2-\mu} \big) \cap \big(\phi^*\mathfrak{iso}(N,h)\big)^{\perp}$. We claim that, if conditions (i) and (iv) hold, then we can find an $s \in S$ and either a non-trivial $w_s \in \ker(J^{\varphi_s})$ or a non-trivial $v_s \in \ker(J^{\varphi_s}+(m-3))$ such that
    \begin{itemize}
        \item if $m \geq 5$ then
        \begin{equation*}
        z - \Psi_s(r^{-m+2} w_s) \in C^{\infty}_{-m+2+\varepsilon}(\prescript{\phi_u}{}{TN}) \quad \textup{or} \quad z - \Psi_s(r^{-m+3} v_s) \in C^{\infty}_{-m+3+\varepsilon}(\prescript{\phi_u}{}{TN})
        \end{equation*} 
    
        \item if $m = 4$ then 
        \begin{equation*}
        z - \Psi_s(r^{-2} w_s) \in C^{\infty}_{-2+\varepsilon}(\prescript{\phi_u}{}{TN}) \quad \textup{or} \quad z - \Psi_s(r^{-1} \log(r) v_s) \in C^{\infty}_{-1}(\prescript{\phi_u}{}{TN}) .
        \end{equation*}
    \end{itemize}
    
    Indeed, since $J^{\phi}_g$ is formally self-adjoint, one can conclude from \cref{prop: mapping properties Jacobi operator} together with \cref{prop: indicial roots of model operator}, \cref{lem: the space of polyhomogeneous kernel elements}, and condition (i) that 
    \begin{align*}
        \mathrm{Ind}(J^\phi_{-1+\varepsilon}) &= -\textstyle{\sum}_s \dim \ker(J^{\varphi_s}+(m-3)), \quad \textup{and} \\
    \mathrm{Ind}(J^{\phi}_\mu) &= -\textstyle{\sum}_s \dim \ker(J^{\varphi_s}+(m-3)) + \dim \ker(J^{\varphi_s}).
    \end{align*}
    This calculation together with condition (iv) implies
    \begin{align*}
        \dim \ker (J_{-m+3-\varepsilon}^{\phi}) &= \dim \mathfrak{iso}(N,h) + \textstyle{\sum}_s \dim \ker(J^{\varphi_s}+(m-3)) \\
        \dim \ker (J_{-m+2-\mu}^{\phi}) &= \textstyle{\sum}_s \dim \ker(J^{\varphi_s}+(m-3)) + \dim \ker(J^{\varphi_s}).
    \end{align*}
    
    Now recall from \cref{lem: kernels at different rates and leading order term} that there is a linear map $\kappa_{-m+3} \col \ker (J_{-m+3-\varepsilon}^{\phi}) \to \oplus_s \mathcal{K}(J^{\varphi_s}_C)_{-m+3}$ such that $v-\Psi_s(\kappa_{-m+3}(v)) \in C^{\infty}_{-m+3+\varepsilon}(\prescript{\phi}{}{TN})$ for every $v \in \ker (J_{-m+3-\varepsilon}^{\phi})$. Moreover, our fourth assumption implies that $\ker(\kappa_{-m+3}) = \ker (J_{-m+3+\varepsilon}^{\phi}) = \phi^* \mathfrak{iso}(N,h)$. Thus, $\kappa_{-m+3}$ restricts to an injective map on $\ker\big(J^{\phi}_{-m+3-\varepsilon} \big) \cap \big(\phi^*\mathfrak{iso}(N,h)\big)^{\perp}$. Similarly, we obtain that the analogously defined map $\kappa_{-m+2} \col \ker (J_{-m+2-\mu}^{\phi}) \to \oplus_s \mathcal{K}(J^{\varphi_s}_C)_{-m+2}$ restricts to an injective map on any complementary subspace of $\ker (J_{-m+3-\varepsilon}^{\phi}) \subset \ker (J_{-m+2-\mu}^{\phi})$. For every $z \in \ker(J^{\phi}_{-m+2-\mu})$ perpendicular to $\phi^*\mathfrak{iso}(N,h)$ there therefore exists an $s\in S$ such that the leading order term of $z$ close to $s$ is a non-trivial element of $\mathcal{K}(J^{\varphi_s}_C)_{-m+2}$ or $\mathcal{K}(J^{\varphi_s}_C)_{-m+3}$.

    Moreover, by \cref{lem: the space of polyhomogeneous kernel elements} we have
    \begin{align*}
        \mathcal{K}(J^{\varphi_s}_C)_{-m+2} = \big\{ r^{-m+2} w_s \mid w_s \in \ker(J^{\varphi_s}) \big\}
    \end{align*}
    and
    \begin{align*}
        \mathcal{K}(J^{\varphi_s}_C)_{-m+3} = 
        \begin{cases}
            \big\{ r^{-m+3} v_s \mid v_s \in \ker(J^{\varphi_s} + m -3) \big\} \quad &\textup{if $m>4$} \\
            \big\{ r^{-1} \big(\log(r)v_s + v^\prime_s)\big) \mid v_s, v^\prime_s \in \ker(J^{\varphi_s} + m -3)\big\} \quad &\textup{if $m=4$.}
        \end{cases}
    \end{align*}
    This proves our claim if $m>4$. For $m=4$ we additionally observe that by condition (iv), any $v \in \ker(J^{\phi}_{-1-\varepsilon}) \cap \phi^*\mathfrak{iso}(N,h)^{\perp}$ must have a point $s \in S$ at which the $r^{-1} \log(r) v_s$-term in its expansion is non-trivial.

    Using this, let us show if (i)--(iv) hold, then there exists $(\dot w,\dot v) \in \prescript{\phi}{}{W} \oplus \prescript{\phi}{}{V}_g$ such that
    \begin{equation*}
        \int_M \big\langle J^{\phi}_g\{\dot w + \dot v\}, z  \big\rangle_{h_\phi} \vol_{g} \neq 0.
    \end{equation*}
    First assume that we are in the situation where a non-trivial $w_{s_0} \in \ker(J^{\varphi_{s_0}})$ appears as the leading order term of $v$ close to $s_0$. Then the term of order $r^{-m+2}$ defines a non-trivial element $\underline{w} \in \oplus_s \ker(J^{\varphi_s})$, which by \cref{cor: obstruction pairing} must be orthogonal to $\underline{\varphi}^*\mathfrak{iso}(N,h)$ since $z$ is orthogonal $\mathfrak{iso}(N,h)$. On the other hand, condition (ii) is equivalent to saying that $\underline{\iota}$ maps surjectively onto a complement of $\underline{\varphi}^*\mathfrak{iso}(N,h)$. Thus \cref{cor: obstruction pairing} implies that we can find $\dot w \in \prescript{\phi}{}{W}$ such that $\int \langle J^\phi_g \{\dot w\}, z \rangle_{h_\phi} \vol_g \neq 0$.
    
    On the other hand, condition (iii) is equivalent to saying that $\underline{\jmath}$ is an isomorphism, and thus if the leading order term of $v$ is of the form $r^{-1}\log(r) v_s$, \cref{cor: obstruction pairing} again allows us to find some $\dot v\in \prescript{\phi}{}{V}_g$ such that $\int \langle J^\phi_g \{\dot w\}, z \rangle_{h_\phi} \vol_g \neq 0$.

    By \cref{lem: image of augmented deformation operator}, this proves that $(D_1 \mathcal{T})_{(\phi,\underline{0},g)}$ is indeed surjective if (i)--(iv) hold. In fact, the first part of the argument shows that, if (i) and (iv) hold, then the leading-order term at rates $r^{-m+2}$ and $r^{-1}\log(r)$ define an isomorphism of $\ker\big(J^{\phi}_g \big) \cap C^\infty_{-m+2-\mu}(\prescript{\phi}{}{TN}) \cap \big(\phi^*\mathfrak{iso}(N,h)\big)^{\perp}$ onto 
    \begin{equation*}
        (\oplus_s \ker(J^{\varphi_s})) \cap \underline{\varphi}^*\mathfrak{iso}(N,h)^\perp \oplus (\oplus_s \ker(J^{\varphi_s}+m-3)).
    \end{equation*}
    Thus if (iii) is strengthened to $T_{\underline{\varphi}} \mathcal{M} = \oplus_{s \in S} \ker(J^{\varphi_s})$, the statement on the kernel follows. On the other hand, if either of the conditions (i)--(iv) fails, it is easy to deduce from the index formulas that the dimension of the space of obstructions will be strictly greater than the dimension of $\prescript{\phi}{}{W} \oplus \prescript{\phi}{}{V}_g$. Thus these conditions are also necessary for $(D_1 \mathcal{T})_{(\phi,\underline{0},g)}$ to be surjective.
\end{proof}

We can now finally prove the main deformation theorem:

\begin{thm}     \label{thm: deformation theorem with sobolev regularity}
    Let $(\phi,g) \in \Map^\infty(M,N;\mathfrak{S}) \times \Met^\infty(M)$ such that $\phi$ is harmonic with respect to $g$, has tangent maps $\underline{\varphi} \in \mathcal{M}$ and $\dim \phi^*\mathfrak{iso}(N) = \dim \mathfrak{iso}(N)$. Let us pick $\mu \in (0,1)$ smaller than the lowest positive critical rate of $J^\phi_g$ and moreover assume:
    \begin{enumerate}
        \item For all $s \in S$, $-(m-3)$ and $0$ are the only non-positive eigenvalues of $\varphi_s$.
        \item For all $s \in S$, $\dim \ker(J^{\varphi_s}+(m-3)) = m$.
        \item For all $s \in S$, every Jacobi field $w_s \in \ker(J^{\varphi_s})$ is integrable and lies in $T_{\varphi_s} \mathcal{M}_s$.
        \item $\ker\big(J^{\phi}_g\big) \cap C^{\infty}_{-1}(\prescript{\phi}{}{TN}) = \phi^*\mathfrak{iso}(N,h)$.
    \end{enumerate}
    Then there exist
    \begin{itemize}
        \item a neighbourhood $\widetilde{U} \subset \Met^{k+2}(M)$ of $g$, and
        \item a $C^1$-family $(\phi_{\tilde{g}},\underline{\mathfrak{t}}_{\tilde{g}})_{\tilde{g}\in U} \subset \Map^{k+2,p}_{\mu}(M,N;\mathfrak{S}) \times \Ball_1^S$ with $(\phi_{g},\underline{\mathfrak{t}}_{g}) = (\phi,\underline{0})$,
    \end{itemize} 
    such that $\phi_{\tilde{g}} \circ f_{\tilde{g},\underline{\mathfrak{t}}(\tilde{g})}^{-1}$ is harmonic with respect to $\tilde{g} \in \widetilde{U}$ and conically singular along $\{\Upsilon_s((\mathfrak{t}_{(0,\tilde{g})})_s)\}_{s\in S}$ with tangent maps in $(\mathcal{M}_s)_{s\in S}$. Moreover, such $\phi_{\tilde{g}}$ is unique up to the action of $\mathrm{Iso}(N,h)$ in some $\mathrm{Iso}(N,h)$-equivariant neighbourhood of $\phi$.
\end{thm}

\begin{proof}
    The existence of the $C^1$-family $(\phi_{\tilde{g}},\underline{\mathfrak{t}}_{\tilde{g}})_{\tilde{g}\in \widetilde{U}}$ follows from \cref{prop: deformation theorem via IFT} and \cref{cor: surjectivity of derivative}. As for the uniqueness modulo the action of isometries on $N$, it follows from the fact that the action of $\mathrm{Iso}(N,h)$ on $\Map^{k+2,p}_\mu(M,N;\mathfrak{S})$ is smooth, and the kernel of $(D_1 \mathcal{T})_{(\phi,\underline{0},g)}$ is tangent to this action. The smoothness can be seen either in local charts, using the fact that post-composition by diffeomorphisms is a smooth operation (contrary to pre-composition), or proving essentially as in \cref{lem: subbundle of pulled back Killing fields} that the Killing fields of $N$ locally lift to a smooth sub-bundle of $T\Map^{k+2,p}_\mu(M,N;\mathfrak{S})$ and using the Picard--Lindel\"of theorem for ODEs in Banach spaces.
\end{proof}

We finish this section with the proof of \cref{thm: main theorem introduction}, which almost immediately follows from \cref{thm: deformation theorem with sobolev regularity}:

\begin{proof}[Proof of \cref{thm: main theorem introduction}]
    Let us first remark that, if condition 3 in \cref{thm: main theorem introduction} holds, then there are submanifolds $\mathcal{M}_s \subset C^\infty(\Sp^{m-1},N)$ consisting of harmonic maps and containing the tangent maps $\varphi_s$ of $\phi$, and such that $\ker(J^{\varphi_s}) = T_{\varphi_s} \mathcal{M}_s$. Therefore we can construct the Banach manifolds $\Map^{k,p}_\mu(M,N;\mathfrak{S})$ using the family of manifolds $\{\mathcal{M}_s\}_{s \in S}$ as part of the singular data set $\mathfrak{S}$. Hence conditions 1--3 in \cref{thm: deformation theorem with sobolev regularity} are satisfied. On the other hand, using Lockhart--McOwen theory, condition 4 in \cref{thm: main theorem introduction} is equivalent to condition 4 in \cref{thm: deformation theorem with sobolev regularity} (cf. \cref{lem: kernels at different rates and leading order term}).

    We may therefore apply \cref{thm: deformation theorem with sobolev regularity} for $k=0$, $p > m$ and $\mu \in (0,1)$ sufficiently small. This in particular implies that for any smooth Riemannian metric $\tilde{g}$, there exists $\phi_{\tilde{g}} \in \Map^{2,p}_\mu(M,N;\mathfrak{S})$ close to $\phi$ such that $\tilde{\phi} \coloneqq \phi_{\tilde{g}} \circ f_{\tilde{g},\mathfrak{t}(g)}^{-1}$ is harmonic and conically singular along $\tilde{S} \coloneqq \{\Upsilon_s((\mathfrak{t}_{(0,\tilde{g})})_s)\}_{s\in S} \subset M$. On the other hand, since $p > m$, the map $\tilde{\phi}$ is $C^1$ on $M \setminus \tilde{S}$, and since it is harmonic and $\tilde{g}$ is smooth this implies that $\tilde{\phi}$ is $C^\infty$ on $M \setminus \tilde{S}$ (see \cite[Th. 9.4.1]{Jost2011GeometricAnalysisBook} for instance). Finally, in order to see that $\tilde{\phi}$ is conically singular in the sense of \cref{Def: Conically singular maps}, one can apply the same reasoning for every $k \geq 1$ and use the uniqueness modulo isometries part of the previous theorem to deduce that $\phi_{\tilde{g}} \in \Map^\infty_\mu(M,N;\mathfrak{S})$, which finishes the proof since the diffeomorphism $f_{\tilde{g},\mathfrak{t}(g)}$ is smooth.
\end{proof}

%% file: section5_modelmaps.tex
    \section{Models for the tangent maps}       \label{sec: Models for the tangent maps}

In this section, we study several model tangent maps $\varphi \col (\Sp^{m-1}, g_{\Sp{m-1}}) \to (N,h)$ satisfying conditions 1--3 in \cref{thm: main theorem introduction}. We begin with a few preliminaries on harmonic fibrations in \cref{subsec: harmonic fibrations}. \cref{subsec: The identity map} and \cref{subsec: Hopf fibration as model map} respectively study the identity map $\Sp^{m-1} \to \Sp^{m-1}$ and the Hopf fibration $\Sp^3 \to \Sp^2$ and gather results that will be useful in order to construct explicit examples of harmonic maps satisfying the assumptions of \cref{thm: main theorem introduction} in \cref{sec: Examples from deformed suspensions}. Finally, \cref{subsec: higher Hopf fibrations} considers the higher-dimensional Hopf fibrations and proves \cref{thm: hopf fibrations intro}.

    \subsection{Generalities on harmonic fibrations}        \label{subsec: harmonic fibrations}

A special class of harmonic maps $\varphi \col \Sp^{m-1} \rightarrow N^n$ is provided by harmonic fibrations, that is, harmonic maps which are also Riemannian fibrations. It is well-known that a Riemannian fibration is harmonic if and only if its fibres are minimal \cite{eells1964harmonic}, and thus there are many examples of such harmonic fibrations: the identity map $\Sp^{m-1} \to \Sp^{m-1}$, the complex Hopf fibrations $\Sp^{2n+1} \to \C\P^{n}$, or the quaternionic Hopf fibrations $\Sp^{4n+3} \to \mathbb{H}\P^{n}$ for instance. Let us assume that $\varphi \col \Sp^{m-1} \to N$ is a harmonic fibration over a compact Einstein manifold $(N,h)$ with Einstein constant $\kappa > 0$:
\begin{equation*}
    \Ric(h) = \kappa h .
\end{equation*}
For such a map, it is straightforward to deduce from the general formula \eqref{eq: Jacobi operator in general} that the Jacobi operator of $\varphi$ with respect to the standard round metric (with constant sectional curvature $1$) on $\Sp^{m-1}$ takes the form
\begin{equation}        \label{eq: Jacobi for fibration over Einstein manifold}
    J^\varphi u = \prescript{\varphi}{}{\nabla}^* \prescript{\varphi}{}{\nabla} u - \kappa u, ~~~~ \forall u \in \Gamma(\prescript{\varphi}{}{TN}) .
\end{equation}
In order to write the connection $\prescript{\varphi}{}{\nabla}$ in a convenient way, recall that the structure of Riemannian fibration induces a splitting $T\Sp^{m-1} = \mathscr H \oplus \mathscr V$ into horizontal and vertical components; where $\mathscr V = \ker \diff \varphi$ and $\mathscr H$ is its orthogonal complement. Thus $\prescript{\varphi}{}{T\Sp^{m-1}}$ can be isometrically identified with $\mathscr H$ and $\diff \varphi \col T\Sp^{m-1} \rightarrow \mathscr H$ is the horizontal projection (orthogonally to the vertical space $\mathscr V$). 

Using this decomposition, the connection $\prescript{\varphi}{}{\nabla}$ for the Hopf fibrations can be written as follows:

\begin{lem}    \label{lem:  nablaphicircle}
    Let $m = 2(n+1)$ and $\varphi \col \Sp^{m-1} = \Sp^{2n+1} \rightarrow \C\P^{n}$ be the complex Hopf fibration, where we regard $\Sp^{2n+1}$ as the unit ball in $\C^{n+1}$. Let $\xi = I \partial_r$ be the vector field generating the $\mathrm{U}(1)$-action and let $\alpha = g(\xi, \cdot)$ be the dual $1$-form with respect to the round metric. Then for any $u \in \Gamma(\mathscr H)$,
    \begin{equation*}
        \prescript{\varphi}{}{\nabla} u = \mathscr H(\nabla u) - \alpha \otimes \nabla_u \xi,
    \end{equation*}
    where $\nabla$ denotes the Levi-Civita connection associated with the round metric on $\Sp^{2n+1}$.
\end{lem}

\begin{proof}
    For any $X \in \mathfrak{X}(\Sp^{2n+1})$ and $u \in \Gamma(\mathscr{H})$, let us define
    \begin{equation*}
        \overline{\nabla}_X u \coloneqq \mathscr{H}(\nabla_X u) - \alpha(X) \nabla_u \xi .
    \end{equation*}
    Since $\xi$ has constant unit length it is easy to see that the right-hand side is a horizontal vector field. Thus $\overline{\nabla}$ defines an affine connection on $\mathscr H$, and it is enough to show that $\overline{\nabla}_X u = \prescript{\varphi}{}{\nabla}_X u$ when $u$ is the horizontal lift of a vector field on $\C\P^n$, and either $X$ is a horizontal lift or $X = \xi$.

    Assume first that both are horizontal lifts. Then $\mathscr H(\nabla_X u)$ is the horizontal lift of $\nabla^{\textup{FS}}_{\varphi_* (X)} \varphi_*(u)$ (this is a general fact for Riemannian fibrations, see \cite{oneil1966fundamental} or \cite[Ch. 9]{besse1987einstein}), where $\nabla^{\textup{FS}}$ is the Levi-Civita connection of the Fubini--Study metric on $\C\P^n$, which by definition coincides with $\prescript{\varphi}{}{\nabla}_X u$. Hence $\prescript{\varphi}{}{\nabla}_X u = \overline{\nabla}_X u$ when $X$ and $u$ are horizontal lifts. 

    Now assume that $X = \xi$. Then $\prescript{\varphi}{}{\nabla}_\xi u = \nabla^{\textup{FS}}_{\varphi_*(\xi)} \varphi_*(u) = 0$ since $\xi$ is vertical. On the other hand,
    \begin{equation*}
        \nabla_\xi u = [\xi,u] + \nabla_u \xi = \nabla_u \xi
    \end{equation*}
    since $u$ is $\mathrm{U}(1)$-invariant (as a horizontal lift) and thus $[\xi,u] = 0$. Moreover $\nabla_u \xi$ is horizontal and thus we obtain $\overline{\nabla}_\xi u = 0 = \prescript{\varphi}{}{\nabla}_\xi u$ in this case too. This completes the proof.
\end{proof}

    \subsection{The identity map $\mathbb{S}^{m-1} \to \mathbb{S}^{m-1}$}      \label{subsec: The identity map}

Let us consider the sphere $\mathbb{S}^{m-1}$ with $m \geq 4$, endowed with the round metric (with constant sectional curvature $1$). The identity map $\id \col \mathbb{S}^{m-1} \to \Sp^{m-1}$ is harmonic, and since the round metric is Einstein with constant $\kappa =m-2$ it follows from \eqref{eq: Jacobi for fibration over Einstein manifold} that the Jacobi operator is given by
\begin{equation*}
    J^{\id} u = \nabla^*\nabla u - (m-2) u, ~~~ \forall u \in \mathfrak{X}(\mathbb{S}^{m-1}),
\end{equation*}
where $\nabla$ is the Levi-Civita connection of the round metric. On the other hand, the Hodge-de Rham Laplacian $\Delta_{\mathbb{S}^{m-1}}$ on $1$-forms is given by
\begin{equation*}
    \Delta_{\mathbb{S}^{m-1}} \alpha = \nabla^*\nabla \alpha + \Ric_{\mathbb{S}^{m-1}}(\alpha) = \nabla^*\nabla \alpha + (m-2) \alpha ,  ~~~ \forall \alpha \in \Omega^1(\mathbb{S}^{m-1}). 
\end{equation*}
Thus if we identify $\mathfrak{X}(\mathbb{S}^{m-1}) \simeq \Omega^1(\mathbb{S}^{m-1})$ using the round metric, we have
\begin{equation*}
    J^{\id} u = \Delta_{\mathbb{S}^{m-1}} u - 2(m-2) u, ~~~ \forall u \in \mathfrak{X}(\mathbb{S}^{m-1}) .
\end{equation*}
Note that since the natural action of $\mathrm{SO}(m)$ on $\Sp^{m-1}$ commutes with the Laplacian, all eigenspaces of $\Delta_{\Sp^{m-1}}$ are (finite-dimensional) representations of $\mathrm{SO}(m)$. The spectrum of the Laplacian on the round sphere and the corresponding eigenspace decomposition have been determined in \cite[Th. C]{folland1989harmonic}: there is an orthogonal decomposition
\begin{equation}        \label{eq: Eigenspace decomposition of the Laplacian on spheres}
    L^2(T^*\mathbb{S}^{m-1}) = \bigoplus_{\ell=1}^\infty (V_\ell \oplus V^\prime_\ell)
\end{equation}
where each $V_\ell$, $V^\prime_\ell$ is an irreducible representation of $\mathrm{SO}(m)$, such that $V_\ell$ is an eigenspace of $\Delta_{\mathbb{S}^{m-1}}$ with eigenvalue $\ell(m+\ell-2)$ and $V^\prime_\ell$ is an eigenspace of $\Delta_{\mathbb{S}^{m-1}}$ with eigenvalue  and $(\ell+1)(m+\ell-3)$. The lowest eigenvalue is $m-1$, with associated eigenspace $V_1$, and the second lowest eigenvalue is $2(m-2)$, with associated eigenspace $V^\prime_1$; all the other eigenvalues of $\Delta_{\mathbb{S}^{m-1}}$ are greater or equal to $2m$. For later use, we gather the following properties of the representations $V_\ell$, $V^\prime_\ell$, which can all be deduced from \cite{folland1989harmonic}:
\begin{itemize}
    \item For any $\ell \geq 1$, $V_\ell = \iota^* F_\ell$ where $\iota \col \Sp^{m-1} \hookrightarrow \R^m$ is the canonical inclusion and $F_\ell$ is a finite-dimensional vector space of $1$-forms on $\R^m$ whose coefficients are homogeneous harmonic polynomials of degree $\ell -1$.

    \item For any $\ell \geq 1$, $V^\prime_\ell = \iota^* F^\prime_\ell$ where $F^\prime_\ell$ is a finite-dimensional vector space of $1$-forms on $\R^m$ whose coefficients are homogeneous harmonic polynomials of degree $\ell$.

    \item For $\ell = 1$,
    \begin{equation*}
        F_1 = \left\{ \sum a_i \diff x_i \mid a_i \in \R \right\} \simeq (\R^m)^* .
    \end{equation*}
    In particular, $\dim(V_1) = m$.

    \item For $\ell = 1$,
    \begin{equation*}
        F^\prime_1 = \left\{ \sum a_{ij} x_i \diff x_j \mid a_{ji} = - a_{ij} \right\} \simeq \Lambda^2(\R^m)^* \simeq \mathfrak{so}_m.
    \end{equation*}
    In particular, $\dim(V^\prime_1) = \tfrac{n(n-1)}{2}$. 

    \item For $\ell = 2$,
    \begin{equation*}
        F_2 = \left\{ \sum a_{ij} x_i \diff x_j \mid a_{ji} = a_{ij}, \sum a_{ii} = 0 \right\},
    \end{equation*}
    that is, $F_2$ is isomorphic to the space of traceless symmetric matrices as a representation of $\mathrm{SO}(m)$. 
\end{itemize}

After identifying $\mathfrak{X}(\mathbb{S}^{m-1})$ and $\Omega^1(\mathbb{S}^{m-1})$, this implies that \eqref{eq: Eigenspace decomposition of the Laplacian on spheres} is the eigenspace decomposition of the Jacobi operator $J^{\id}_{\mathbb{S}^{m-1}}$; moreover, $V_\ell$ is associated with the eigenvalue $\ell^2 + (\ell-2)(m-2)$, and $V^\prime_\ell$ is associated with the eigenvalue $(\ell-1)(m+\ell-1)$. In particular,
\begin{itemize}
    \item the lowest eigenvalue of $J^{\id}_{\mathbb{S}^{m-1}}$ is $-m+3$, with corresponding eigenspace $V_1$ corresponding to the orthogonal projection of constant vector fields of $\R^m$ onto $\mathbb{S}^{m-1}$,
    \item the second lowest eigenvalue of $J^{\id}_{\mathbb{S}^{m-1}}$ is $0$, with eigenspace $V^\prime_1$ corresponding to the space of Killing fields on the sphere (one inclusion is clear, and equality follows by dimensionality), and
    \item all the other eigenvalues of $J^{\id}_{\mathbb{S}^{m-1}}$ are greater or equal to $2$.
\end{itemize}

    \subsection{The Hopf fibration $\mathbb{S}^3 \rightarrow \mathbb{S}^2$}     \label{subsec: Hopf fibration as model map}

In this section, we study the Hopf fibration $\Sp^3 \to \Sp^2$. We first derive a few key identities \cref{subsubsec: Computations for Hopf fibration} before giving an explicit expression of the Jacobi operator with respect to the round metric on $\Sp^3$ and describing its spectrum in \cref{subsubsec: Jacobi-Spectrum for Hopf fibration}. We then analyse the effect certain left-invariant deformations of the domain metric in \cref{subsubsec: The Jacobi operator for left-invariant metrics}, for later purpose in \cref{sec: Examples from deformed suspensions}.

    \subsubsection{Basic computations} \label{subsubsec: Computations for Hopf fibration}

In this section, we consider the Hopf fibration $\varphi \col \Sp^3 \rightarrow \Sp^2$, which is harmonic with respect to the round metrics on the spheres. On $\Sp^3 \simeq \mathrm{Sp}(1)$, we consider the action of $\Sp^3$ by left multiplication. There is an orthonormal basis $X_1,X_2,X_3$ of left-invariant vector fields for the round metric $g_{1} \coloneqq g_{\Sp^3}$ with constant sectional curvature $1$, such that $[X_i,X_{i+1}] = 2 X_{i+2}$ with cyclic indices. The vector field $\xi = X_3$ generates an isometric $\mathrm{U}(1)$-action \emph{on the right}, and we might assume that $\varphi$ is the quotient of $\Sp^3$ by this action. Therefore, the vector fields $X_1,X_2$ form a global frame of the vector bundle $\prescript{\varphi}{}{T\Sp^2}$, and moreover $\varphi$ is a Riemannian fibration with respect to the rescaled round metric $\tfrac{1}{4}g_{\Sp^2}$ on $\Sp^2$. Since $\tfrac{1}{4}g_{\Sp^2}$ is Einstein constant $\kappa = 4$, it follows from \eqref{eq: Jacobi for fibration over Einstein manifold} that the Jacobi operator satisfies

\begin{lem}
    With respect to the standard metrics, the Jacobi operator $J^\varphi$ of the Hopf fibration $\varphi \col \Sp^3 \rightarrow \Sp^2$ is given by
    \begin{equation*}
        J^\varphi u = \prescript{\varphi}{}{\nabla}^* \prescript{\varphi}{}{\nabla} u - 4 u, ~~~ \forall u \in \Gamma (\prescript{\varphi}{}{T\Sp^2}) .
    \end{equation*}
\end{lem}

In order to understand the connection $\prescript{\varphi}{}{\nabla}$, let us write a global section of $\prescript{\varphi}{}{T\Sp^2}$ as $u = u_1 X_1 + u_2 X_2$, where $u_1,u_2 \in C^\infty(\Sp^3,\R)$.

\begin{lem}     \label{lem: Connection of Hopf fibration}
    If $u = u_1 X_1 + u_2 X_2$ then
    \begin{equation*}
        \prescript{\varphi}{}{\nabla} u = \diff u_1 \otimes X_1 + \diff u_2 \otimes X_2 + 2 u_1 \alpha \otimes X_2 - 2u_2 \alpha \otimes X_1 .
    \end{equation*}
\end{lem}

\begin{proof}
    Using \cref{lem: nablaphicircle} and the fact that the Levi-Civita connection of the round metric on $\Sp^3$ satisfies $\nabla_{X_i} X_j = \frac{1}{2} [X_i,X_j]$, we have
    \begin{equation*}
        \prescript{\varphi}{}{\nabla} X_1 = \mathscr H(\nabla X_1) - \alpha \otimes \nabla_\xi X_1 = \frac{1}{2}(\mathscr H(\alpha \otimes [X_3,X_1] + X^2 \otimes [X_2,X_1]) - \alpha \otimes [X_1,X_3]) = 2 \alpha \otimes X_2
    \end{equation*}
    and similarly
    \begin{equation*}
        \prescript{\varphi}{}{\nabla} X_2 = -2 \alpha \otimes X_1 .
    \end{equation*}
    The expression of $\prescript{\varphi}{}{\nabla} u$ follows.
\end{proof}

We can therefore deduce a more explicit expression for the Jacobi operator:

\begin{prop}    \label{prop: Jacobi Operator Hopf fibration}
    For the Hopf fibration $\varphi \col \Sp^3 \rightarrow \Sp^2$,
    \begin{equation*}
        \prescript{\varphi}{}{\nabla}^*\prescript{\varphi}{}{\nabla} (u_1 X_1 + u_2 X_2) = (\Delta_{\mathbb{S}^3} u_1 + 4 u_1 + 4 \partial_\xi u_2) X_1 + (\Delta_{\mathbb{S}^3} u_2 + 4u_2 - 4 \partial_\xi u_1) X_2 .
    \end{equation*}
    In particular, the Jacobi operator of $\varphi$ with respect to the round metric is
    \begin{equation*}
        J^\varphi (u_1 X_1 + u_2 X_2) = (\Delta_{\mathbb{S}^3} u_1 + 4 \partial_\xi u_2) X_1 + (\Delta_{\mathbb{S}^3} u_2 - 4 \partial_\xi u_1) X_2 .
    \end{equation*}
    If we introduce the complex function $w = u_1+iu_2 \in C^\infty(\Sp^3,\C)$ this can be written as
    \begin{equation*}
        J^\varphi w = \Delta_{\mathbb{S}^3} w - 4 i \partial_\xi w .
    \end{equation*}
\end{prop}

\begin{proof}
    Using the previous lemma, we can integrate by parts:
    \begin{align*}
        \int_{\Sp^3} \langle \prescript{\varphi}{}{\nabla} u, \prescript{\varphi}{}{\nabla} u \rangle \vol_{\Sp^3}  & = \int_{\Sp^3}  (|\diff u_1|^2 + |\diff v_2|^2 + 4 u_1^2 + 4u_2) \vol_{\Sp^3} + 4 \int_{\Sp^3} (u_1 \partial_\xi u_2 - u_2 \partial_\xi u_1) \vol_{\Sp^3}  \\
            & = \int_{\Sp^3} \langle (\Delta_{\mathbb{S}^3} u_1 + 4 u_1 + 4 \partial_\xi u_2) X_1 + (\Delta_{\mathbb{S}^3} u_2 + 4u_2 - 4 \partial_\xi u_1) X_2, u_1 X_1 + u_2 X_2 \rangle \vol_{\Sp^3}
    \end{align*}
    where $\Delta_{\mathbb{S}^3} = \diff^* \diff$ is the scalar Laplacian for the standard round metric on $\Sp^3$. This proves the result since the differential operator $P_\xi(u_1 X_1 + u_2X_2) = \partial_\xi u_2 X_1 - \partial_\xi u_1 X_2$ on $\prescript{\varphi}{}{T\Sp^2}$ is easily found to be formally self-adjoint (see for instance the next section where we exhibit an explicit eigenspace decomposition of $L^2(\prescript{\varphi}{}{T\Sp^2})$).
\end{proof}

    \subsubsection{Spectrum of the Jacobi operator}     \label{subsubsec: Jacobi-Spectrum for Hopf fibration}

In this section we use coordinates $(x_0,x_1,x_2,x_3)$ on $\R^4$, which can be identified with the space of quaternions where $q = x_0 + x_1 I + x_2 J + x_3 K$. Then $X_1 = qJ$, $X_2 = qK$ and $\xi = X_3 = qI$, which we can also extend as vector fields on $\R^4$. Explicitly,
\begin{align*}
    X_3 & = - x_1 \partial_0 + x_0 \partial_1 + x_3 \partial_2 - x_2 \partial_3 = \xi, \\
    X_1 & = - x_2 \partial_0 - x_3 \partial_1 + x_0 \partial_2 + x_1 \partial_3 , \\
    X_2 & = - x_3 \partial_0 + x_2 \partial_1 - x_1 \partial_2 + x_0 \partial_3 .
\end{align*}
It is well-known that the eigenfunctions of $\Delta_{\Sp^3}$ are obtained by restriction of the harmonic homogeneous polynomials on $\R^4$ of degree $k \geq 0$. More precisely, if we denote by $W_k \subset C^\infty(\Sp^3,\C)$ the space generated by the restriction of the (complex-valued) harmonic polynomials of degree $k$, then
\begin{equation*}
    \Delta_{\Sp^3} w = k(k+2) w, ~~~~ \forall w \in W_k.
\end{equation*}

Since the differential operator $\partial_\xi$ is invariant under the action of $\mathrm{Sp}(1) \simeq \Sp^3$, it commutes with $\Delta_{\mathbb{S}^3}$ (essentially because the Laplacian is the Casimir element of the action). In particular, $\partial_\xi$ leaves invariant the eigenspaces of $\Delta_{\mathbb{S}^3}$, and thus we can co-diagonalise $\Delta_{\mathbb{S}^3}$ and $\partial_\xi$ on $C^\infty(\Sp^3,\C)$. To see this explicitly, let us introduce complex coordinates $(z_1,z_2)$ defined as
\begin{equation*}
    z_1 = x_0 + i x_1, ~~~ z_2 = x_2 - ix_3
\end{equation*}
so that if we identify $i \sim I$ we have
\begin{equation*}
    q = z_1 + J z_2 .
\end{equation*}
Then the flow $\textup{Flow}_{t}^\xi$ generated by $\xi$ is just given by $\textup{Flow}_{t}^\xi(z_1,z_2) = (\mathrm{e}^{it}z_1,\mathrm{e}^{it}z_2)$. If $w$ is a homogeneous complex polynomials of degree $k$ with holomorphic degree $\ell$ and anti-holomorphic degree $k-\ell$, then
\begin{equation*}
    \partial_\xi w = (2\ell - k)i w.
\end{equation*}
Let us denote by $W_{k,\ell} \subset W_k$ the space obtained by restriction of complex homogeneous harmonic polynomials in $\R^4 \simeq \C^2$ with holomorphic degree $\ell \in \{0,\ldots,k\}$ and anti-holomorphic degree $k-\ell$. Then for any $w \in W_{k,\ell}$ we have
\begin{equation*}
    J^\varphi w = (k(k+2) - 4 i^2 (2\ell-k) = (k(k-2) + 8 \ell) w .
\end{equation*}

\begin{lem}
    For any integers $k \geq 0$ and $\ell = 0, 1, \ldots, k$,
    \begin{equation*}
        J^\varphi w = (k(k-2) + 8\ell) w, ~~~~ \forall w \in W_{k,\ell}.
    \end{equation*}
\end{lem}

\begin{rem}
    Our calculations recovers a result of Urakawa \cite{urakawa1987stability}, whose computation of the spectrum of $J^\varphi$ coincides with ours up to a factor of $\tfrac{1}{8}$ which is due to different scaling conventions.
\end{rem}

The lowest eigenvalue of $J^\varphi$ is $-1$, corresponding to $k = 1$ and $\ell=0$, and all the other eigenvalues are non-negative. The eigenspace $W_{1,0}$ of $J^\varphi$ is $4$-dimensional, which corresponds to the $4$-dimensional space of homogeneous Jacobi fields of rate $-1$ for the operator $J^\varphi_C$: they correspond to the space of infinitesimal translations in $\R^4$.

The next eigenvalue is $0$, for which there is an $8$-dimensional eigenspace (the kernel of $J^\varphi$). A $2$-dimensional space coming from $(k,\ell) = (0,0)$: these correspond to pre-composing with multiplication \emph{on the right} by an element of $\Sp^3$ (there are only two independent directions because we quotient by the third). And a $6$-dimensional space coming from $(k,\ell) = (2,0)$: Three directions correspond to the space of rotations of $\Sp^2$ (or multiplication \emph{on the left} of $\Sp^3$, which commutes with the Hopf fibration) and because $\Sp^2$ is Kähler we get $3$ more directions of harmonic deformations of $\Sp^2$ by multiplying by the complex structure. Thus, every infinitesimal harmonic deformation of $\varphi$ corresponds to pre-composition by an isometry of $\Sp^3$ and post-composition by a conformal map on $\Sp^2$ and is therefore integrable. This shows that the Hopf fibration $\varphi \col \Sp^3 \to \Sp^2$ satisfies  conditions 1--3 in \cref{thm: main theorem introduction}.

    \subsubsection{The Jacobi operator for a family of left-invariant metrics on $\mathbb{S}^3$}     \label{subsubsec: The Jacobi operator for left-invariant metrics}

There exists a simple 3-dimensional family of metrics on $\Sp^3$ that contains the standard metric and for which the Hopf fibration is harmonic. In the subsequent section, we will use a carefully chosen path inside this class of metrics to construct harmonic maps $\Sp^4 \to \Sp^2$ and $\mathbb{CP}^2 \to \Sp^2$ satisfying the conditions in \cref{thm: main theorem introduction}. For this, we calculate the Jacobi operator of the Hopf fibration for these metrics and show that for certain metrics the only kernel elements are given by the pullback of Killing vector fields on $\Sp^2$. 

For $\beta_1,\beta_2,\beta_3>0$ we consider the metric 
\begin{equation}
    g_\beta \coloneqq \beta_1 (X^1)^2 + \beta_2 (X^2)^2 + \beta_3 (X^3)^2
\end{equation}
on $\Sp^3$. Dividing by $1/\beta_3$ we may assume that $\beta_3=1$. (Note that after a constant conformal rescaling $g \mapsto \lambda g$ the tension and the Jacobi operator scale $\tau_{\lambda g} (\phi) = (1/\lambda) \tau_g (\phi)$ and $J^\phi_{\lambda g} = (1/\lambda) J^\phi_g$.)

\begin{lem}     \label{lem: tension of Hopf fibration for deformed metric}
    The tension field of the Hopf fibration $\varphi \col \Sp^3 \to \Sp^2$ vanishes for any $\beta_1,\beta_2>0$, and hence the Hopf fibration remains a harmonic map with respect to the metric $g_\beta$ on $\Sp^3$.
\end{lem}

\begin{proof}
    Since the metric $g_\beta$ on $\Sp^3 = \mathrm{Sp}(1)$ is left-invariant, its Levi--Civita connection is given for any two left-invariant vector fields $V,W \in \mathfrak{X}(\Sp^3)^{L}$ by \begin{align*} \nabla^\beta_{V} W = \tfrac{1}{2}\big( [V,W] - \mathrm{ad}(V)^{*_\beta}W - \mathrm{ad}(W)^{*_\beta}V \big) \end{align*} where $\mathrm{ad}(\cdot)^{*_\beta}$ denotes the (metric) adjoint of the adjoint action with respect to the metric $g_\beta$ on $\mathfrak{sp}(1)$. A simple calculation shows
    \begin{align*}
        \langle\mathrm{ad}(X_i)^{*_\beta}X_j,X_k\rangle_\beta = \langle X_j , [X_i,X_k] \rangle_\beta &=  \beta_j \langle X_j, [X_i,X_k] \rangle_{\Sp^3} = \beta_j \langle [X_j,X_i],X_k \rangle_{\Sp^3} \\
    &= \tfrac{\beta_j}{\beta_\ell} \langle [X_j,X_i],X_k \rangle_\beta
    \end{align*}
    where $(j,i,\ell)$ is a permutation of $(1,2,3)$ and where $\langle \cdot,\cdot\rangle_{\Sp^3}$ denotes the standard (bi-invariant) metric on $\Sp^3$. Thus, 
    \begin{align*} 
        \nabla^\beta_{X_i} X_j = \tfrac{1}{2}\big( [X_i,X_j] - \mathrm{ad}(X_i)^{*_\beta}X_j - \mathrm{ad}(X_j)^{*_\beta}X_i \big) = \tfrac{\beta_\ell + \beta_j - \beta_i}{2\beta_{\ell}} [X_i,X_j] 
    \end{align*}
    and in particular, $\nabla^\beta_{X_i}X_i = 0$.

    Since we do not change the metric on $\Sp^2$, we deduce from \cref{lem: Connection of Hopf fibration} that 
    \begin{align*}
        \prescript{\varphi}{}{\nabla}^\beta \diff \varphi = (\nabla^\beta X^1) \otimes X_1 + (\nabla^{\beta} X^2) \otimes X_2 + 2 X^1 \otimes \alpha \otimes X_2 - 2 X^2 \otimes \alpha \otimes X_1
    \end{align*} 
    where $\alpha = g_{\Sp^3}(X_3,\cdot)=X^3$. The trace of the last two terms in the expression of $\prescript{\varphi}{}{\nabla}^\beta \diff \varphi$ therefore vanishes. For the first two terms we note that $(\nabla^\beta_{X_i} X^j)(X_i) = -X^j(\nabla^\beta_{X_i}X_i) = 0$ for any two $i,j \in \{1,2,3\}$. The tension field of the Hopf fibration with respect to the metric $g_\beta$ is therefore given by 
    \begin{equation*}
        \tau_{g_\beta}(\varphi) = \sum \tfrac{1}{\beta_i} (\prescript{\varphi}{}{\nabla}^\beta \diff \varphi) (X_i,X_i) = 0. \qedhere 
    \end{equation*}
\end{proof}

\begin{lem}     \label{lem: Jacobi operator for Hopf map with perturbation of metric}
    Let $\beta = (\beta_1,\beta_2,1)$. Then for any section $u \in C^\infty(\prescript{\varphi}{}{T\Sp^2})$ written as $u = u_1 X_1 + u_2 X_2$, the Jacobi operator of $\varphi$ with respect to $g_\beta$ has the expression
    \begin{equation*}
        J^{\varphi}_\beta u = J^{\varphi} u + \left(\sum_{i=1}^2 \left(1-\frac{1}{\beta_i}\right) \partial^2_{X_i} u_1 + 4 \left(1-\frac{1}{\beta_2}\right)u_1\right) X_1 + \left(\sum_{i=1}^2 \left(1-\frac{1}{\beta_i}\right) \partial^2_{X_i} u_2 + 4 \left(1-\frac{1}{\beta_1}\right)u_2\right) X_2
    \end{equation*}
    where $J^\varphi$ is the Jacobi operator associated with the round metric on $\Sp^3$. That is, if we identify $u$ with the complex-valued function $w = u_1+iu_2$, 
    \begin{equation*}
        J^\varphi_\beta w = J^\varphi w + \left(1-\frac{1}{\beta_1} \right) \partial_{X_1}^2 w + \left(1-\frac{1}{\beta_2}\right) \partial_{X_2}^2 w  + 2 \left(2-\frac{1}{\beta_1} - \frac{1}{\beta_2}\right) w + 2\left(\frac{1}{\beta_1}-\frac{1}{\beta_2}\right) \overline{w} .
    \end{equation*}
\end{lem}

\begin{proof}
    The following is analogous to the derivation of the Jacobi operator for the standard metric on $\Sp^3$ in \cref{prop: Jacobi Operator Hopf fibration}. Assuming $\beta_3=1$ and using the expression for $\prescript{\varphi}{}{\nabla}$ derived in \cref{lem: Connection of Hopf fibration} we obtain for two global sections $u= u_1 X_1 + u_2 X_2$ and $u^\prime = u_1^\prime X_1 + u_2^\prime X_2$ of $\prescript{\varphi}{}{T\Sp^2}$ that
    \begin{align*}
        \int_{\Sp^3} \langle \prescript{\varphi}{}{\nabla} u, \prescript{\varphi}{}{\nabla} u^\prime \rangle_\beta \mathrm{vol}_\beta &= \int_{\Sp^3} \langle \diff u_1, \diff u_1^\prime \rangle_\beta + \langle \diff u_2, \diff u_2^\prime \rangle_\beta + 4 (u_1 u_1^\prime + u_2 u_2^\prime ) \\
        &\qquad + 2 (\partial_\xi u_2 \cdot u_1^\prime + u_1 \cdot \partial_\xi u_2^\prime - \partial_\xi u_1 \cdot u_2^\prime - u_2 \cdot \partial_\xi u_1^\prime) \vol_{\beta} \\
        &= \int_{\Sp^3} (\Delta_{\beta} u_1) u_1^\prime + (\Delta_{\beta} u_2) u_2^\prime + 4 (u_1u_1^\prime + u_2 u_2^\prime) \\
        & \qquad + 4( \partial_\xi u_2 u_1^\prime - \partial_\xi u_1 u_2^\prime) \vol_\beta.
    \end{align*}
    Thus, 
    \begin{align*} 
        \prescript{\varphi}{}{\nabla}^{*_\beta}\prescript{\varphi}{}{\nabla} u = (\Delta_\beta u_1 + 4u_1 + 4 \partial_\xi u_2) X_1 + (\Delta_\beta u_2 + 4u_2 - 4 \partial_\xi u_1) X_2. 
    \end{align*}

    Since $(\Sp^2,g_{\Sp^2})$ is a space-form of constant sectional curvature 4, we have $\mathrm{Rm}_{\Sp^2}(x,y)z = 4 g_{\Sp^2}(y,z)x - 4 g_{\Sp^2}(x,z)y$. (\textit{Caution}: $g_{\Sp^2}$ is in our conventions not the standard metric on $\Sp^2$ but a rescaling thereof such that $(\varphi^*g_{\Sp^2})(X_1,X_1)= (\varphi^*g_{\Sp^2})(X_2,X_2)=1$) and therefore
    \begin{align*} 
        \tr_{g_\beta}(\mathrm{Rm}_{\Sp^2}(u,\diff \varphi)\diff \varphi) = \sum \tfrac{1}{\beta_i}\mathrm{Rm}_{\Sp^2}(u,\diff\varphi(X_i))\diff\varphi(X_i) = \tfrac{4}{\beta_1} u_2 X_2 + \tfrac{4}{\beta_2} u_1 X_1.
    \end{align*}

    The full Jacobi operator is therefore given by 
    \begin{align*}
        J^{\varphi}_\beta u &= \prescript{\varphi}{}{\nabla}^{*_\beta}\prescript{\varphi}{}{\nabla} u - \tr_{g_\beta}(\mathrm{Rm}_{\Sp^2}(u,\diff\varphi)\diff\varphi) \\
        &= \big(\Delta_\beta u_1 + 4(1-\tfrac{1}{\beta_2}) u_1 + 4 \partial_\xi u_2\big) X_1 + \big( \Delta_\beta u_2 + 4(1- \tfrac{1}{\beta_1}) u_2 - 4 \partial_\xi u_1\big) X_2
    \end{align*}
    where 
    \begin{align*}
        \Delta_\beta f &= -\sum \tfrac{1}{\beta_i} i_{X_i} \nabla^{\beta}_{X_i} (\diff f(X_j) \cdot X^j) = - \sum \tfrac{1}{\beta_i} \big(X_iX_i f + \diff f (X_j) \cdot (\nabla^{\beta}_{X_i} X^j) (X_i)\big) \\
        &= - \sum \tfrac{1}{\beta_i} X_iX_i f
    \end{align*}
    because $(\nabla^{\beta}_{X_i} X^j) (X_i) = -X^j (\nabla^{\beta}_{X_i} X_i) = 0.$ We therefore have 
    \begin{align*} 
        J^{\varphi}_\beta u = J^{\varphi} u + \big(\textstyle{\sum} (1-\tfrac{1}{\beta_i}) \partial^2_{X_i} u_1 + 4 (1-\tfrac{1}{\beta_2})u_1\big) X_1 + \big(\textstyle{\sum} (1-\tfrac{1}{\beta_i}) \partial^2_{X_i} u_2 + 4 (1-\tfrac{1}{\beta_1})u_2\big) X_2  
    \end{align*} 
    where $J^\varphi$ is the Jacobi operator of the Hopf fibration associated to the standard metric on $\Sp^3$.
\end{proof}

\begin{rem}     \label{rem: Explicit Jacobi operator}
    Note that if $\beta_1 = \frac{1}{1-\varepsilon}$ and $\beta_2 = \frac{1}{1+\varepsilon}$ for some $\varepsilon \in [0,1)$, then the Jacobi operator takes the particularly simple form
    \begin{equation*}
        J^\varphi_\beta = J^\varphi w + \varepsilon(\partial_{X_1}^2w -\partial_{X_2}^2w - 4 \overline{w}).
    \end{equation*}
    This will be useful for making explicit computations.
\end{rem}

\begin{rem}     \label{rem: pertrubed Jacobi operator preserves Wk}
    For any $\beta$, the operator $J^\varphi_\beta$ leaves invariant the finite-dimensional vector spaces $W_k = \oplus_{\ell=0}^k W_{k,\ell}$ (recall that $W_{k,\ell}$ is the space of harmonic homogeneous polynomials on $\R^4 \simeq \C^2$ of degree $k$ and holomorphic degree $\ell$, and $W_k$ the space of harmonic homogeneous polynomials of degree $k$). This follows from the fact that the metrics $g_\beta$ are left-invariant, and therefore $J^\varphi_\beta$ commutes with the Laplacian $\Delta_{\mathbb{S}^3}$, whose eigenspaces are exactly the spaces $W_k$.
\end{rem}

The above observation notably implies:

\begin{prop}        \label{prop: Jacobi fields of left-invariant metrics on S3}
    If $\varepsilon \in (0,1)$ is small enough, the only Jacobi fields of the Hopf fibration with respect to the metric $g_{\beta_\varepsilon}$ where $\beta_{\varepsilon} = (\frac{1}{1-\varepsilon},\frac{1}{1+\varepsilon},1)$ are the horizontal lifts of Killing fields on $\Sp^2$.
\end{prop}

\begin{proof}
    Since the operators $J^\varphi_{\beta_\varepsilon}$ leave invariant the spaces $W_k$, they induce by restriction a continuous family of bounded operators from the $L^2$-orthogonal complement of $W_0 \oplus W_2$ in $W^{2,2}(\prescript{\varphi}{}{T\Sp^2})$ to the $L^2$-orthogonal complement of $W_0 \oplus W_2$ in $L^2(\prescript{\varphi}{}{T\Sp^2})$ (where the $L^2$-inner product is taken with respect to the round metric on $\Sp^3$). For $\varepsilon = 0$ this map is invertible since $\ker(J^\varphi) \subset W_0 \oplus W_2$ and therefore the same is true for $J^\varphi_{\beta_\varepsilon}$ when $\varepsilon \in (0,1)$ small enough. Moreover, on $W_0 \subset \ker(J^\varphi)$, we see from \cref{rem: Explicit Jacobi operator} that $J^\varphi_{\beta_\varepsilon}$ acts as $w \mapsto -4\varepsilon \overline{w}$ which is invertible for $\varepsilon \neq 0$. Therefore, we deduce that for $\varepsilon > 0$ small enough, $\ker(J^\varphi_{\beta_\varepsilon}) \subset W_2$.
    
    Let us introduce the operator $Q w = \partial_{X_1}^2w -\partial_{X_2}^2w - 4 \overline{w}$, so that $J^\varphi_{\beta_\varepsilon} = J^\varphi + \varepsilon Q$. It is easy to see that $Q$ leaves invariant $W_{2,0} \oplus W_{2,2}$ and $W_{2,1}$ separately. Moreover, $J^\varphi$ is invertible on $W_{2,1}$, and therefore for $\varepsilon >0$ small enough it follows that $\ker(J^\varphi_{\beta_\varepsilon}) \subset W_{2,0} \oplus W_{2,2}$. Now in the (real) basis
    \begin{align*}
        &i \overline z_1^2 + i \overline z_2^2, ~ \overline z_1^2-\overline{z}_2^2, ~ \overline z_1 \overline z_2, ~ i\overline{z}_1^2-i\overline{z}_2^2, ~ i \overline z_1 \overline z_2, ~ \overline{z}_1^2 + \overline{z}_2^2, \\
        & iz_1^2+iz_2^2, ~ z_1^2- z_2^2, ~ z_1z_2, ~ iz_1^2-iz_2^2, ~ iz_1z_2, ~ z_1^2 + z_2^2,
    \end{align*}
    of $W_{2,0} \oplus W_{2,2}$ the matrix of $J^\varphi_{\beta_\varepsilon} = J^\varphi + \varepsilon Q$ can be written as
    \begin{equation*}
        16\begin{pmatrix}
             I &  \frac{1}{2}\varepsilon A \\ \frac{1}{2}\varepsilon  A & 0
        \end{pmatrix}
    \end{equation*}
    where the matrix $A$ takes the diagonal form
    \begin{equation*}
        A = \begin{pmatrix}
            1 & \\
            & -1 & \\
            & & -1 \\
            & & & 0 \\
            & & & & 0  \\
            & & & & & 0
        \end{pmatrix} \cdot
    \end{equation*}
    Thus the kernel of $J^\varphi_{\beta_\varepsilon}$ acting on $W_{2,0} \oplus W_{2,2}$ has dimension $3$ for $\varepsilon \neq 0$, and must therefore be spanned by the horizontal lifts of the Killing fields of $\Sp^2$.
\end{proof}

    \subsection{The higher-dimensional Hopf fibrations $\Sp^{2n+1} \to \C\mathbb{P}^n$}     \label{subsec: higher Hopf fibrations}

The goal of this part is to prove \cref{thm: hopf fibrations intro}, which in particular shows that all the complex Hopf fibrations $\Sp^{2n+1} \to \mathbb{CP}^n$ satisfy conditions 1--3 in \cref{thm: main theorem introduction}. In \cref{subsubsec: fundamental identities}, we use Sasaki--Einstein geometry in order to derive a few key identities, while \cref{subsubsec: index and nullity of hopf fibrations} carries out the representation-theoretical computations necessary to derive the index and nullity of the Hopf fibrations, and describe their infinitesimal harmonic deformations.

    \subsubsection{Fundamental identities}      \label{subsubsec: fundamental identities}

Computing the spectrum of the Hopf fibration $\varphi \col \Sp^{m-1} = \Sp^{2n+1} \to \C\P^{n}$ when $m = 2(n+1) \geq 6$ turns out to be much more difficult than for $\Sp^3 \to \Sp^2$. Since this is a homogeneous Riemannian fibration, there is a general formula for the Jacobi operator due to Urakawa \cite{urakawa1987stability} which boils down the computation of the spectrum to a representation-theoretical problem, but in practice it is still difficult to carry out the computations explicitly. To prove that the higher-dimensional Hopf fibrations do satisfy properties 1--3 in \cref{thm: main theorem introduction}, we will instead exploit the fact that $\Sp^{2n+1}$ (equipped with the standard round metric) is a \emph{Sasaki--Einstein manifold} (cf. \cite{sparks2011sasaki} for a comprehensive introduction). In particular, this property implies that
\begin{equation}        \label{eq: nablaisasaki}
    \nabla_u \xi = I u, ~~~~ \forall u \in \Gamma(\mathscr H),
\end{equation}
where $I \col \mathscr H \to \mathscr H$ is the endomorphism obtained by lifting the almost complex structure of $T\C\P^{n}$, under the identification $\mathscr{H} \simeq \prescript{\varphi}{}{T\C\P^{n}}$. In particular, $I^2 = -1$. Using this observation we obtain:

\begin{lem}     \label{lem: connection for higher Hopf fibrations}
    For any $u \in \Gamma(\mathscr{H})$,
    \begin{equation*}
        \prescript{\varphi}{}{\nabla} u = \nabla u - (Iu)^\sharp \otimes \xi - \alpha \otimes Iu
    \end{equation*}
    where $\nabla$ is the Levi-Civita connection of the round sphere on $\Sp^{2n+1}$ and $v^\sharp = g(v,\cdot)$ for any $v \in \mathfrak{X}(M)$.
\end{lem}

\begin{proof}
    By \cref{lem: nablaphicircle} and \eqref{eq: nablaisasaki} we have
    \begin{equation*}
        \prescript{\varphi}{}{\nabla} u = \mathscr{H}(\nabla u) - \alpha \otimes I = \nabla u - g(\xi,\nabla u) \otimes \xi - \alpha \otimes I u .
    \end{equation*}
    Now for any $X \in \mathfrak{X}(M)$, the orthogonality of $u$ and $\xi$ implies $g(\xi,\nabla_X u) = - g(\nabla_X \xi, u)$ and since $\xi$ is a Killing field, $g(\nabla_X \xi, u) = - g(X,\nabla_u \xi) = - g(X,I\xi)$. This proves the lemma.
\end{proof}

By definition, the Hopf fibration $\varphi \col \Sp^{2n+1} \to \C\P^{n}$ is the quotient by the $\mathrm{U}(1)$-action generated by the flow $\textup{Flow}_{t}^\xi$ of $\xi$, and therefore $\varphi$ is invariant under the action of $\mathrm{U}(1)$. In particular, the Jacobi operator $J^\varphi$ (with respect to the round metric on $\Sp^{2n+1}$) commutes with $\textup{Flow}_{t}^\xi$, which means that $\mathrm{U}(1)$ leaves invariant the eigenspaces of $J^\varphi$. Since moreover $J^\varphi$ commutes with the almost-complex structure $I$, we deduce that $J^\varphi$ leaves invariant the eigenspaces of the action of $\mathrm{U}(1)$, defined as
\begin{equation*}
    H_k \coloneqq \{ u \in \Gamma(\mathscr H) \mid \pounds_\xi u = kIu \} \subset \Gamma(\mathscr H), ~~~~ \forall k \in \Z .
\end{equation*}
Thus there exists an orthogonal decomposition of $L^2(\mathscr H)$ as the Hilbert sum
\begin{equation}
    L^2(\mathscr H) = \bigoplus_{k \in \Z} \bigoplus_{\lambda \in \Spec(J^\varphi)} H_{k,\lambda}, ~~~~ H_{k,\lambda} \coloneqq H_k \cap \ker(J^\varphi - \lambda)
\end{equation}
and we may estimate the spectrum of $J^\varphi$ by considering the action of the Jacobi operator on each of the spaces $H_k \subset \Gamma(\mathscr H)$ separately. 

In the following lemma, we relate the Jacobi operator of the Hopf fibration to the Jacobi operator of the identify map on $\Sp^{2n+1}$:

\begin{lem}     \label{lem: Jacobi operator of Hopf vs identity}
    Let $u \in H_k$, that is, $u \in \Gamma(\mathscr H)$ and $\pounds_\xi u = kIu$. Then
    \begin{equation*}
        \int_{\Sp^{m-1}} \langle J^\varphi u, u \rangle \vol_{\Sp^{m-1}} = \int_{\Sp^{m-1}} (\langle J^{\mathrm{id}} u, u \rangle - 2(k+2) |u|^2) \vol_{\Sp^{m-1}}
    \end{equation*}
    where $J^{\mathrm{id}}$ is the Jacobi operator of the identity with respect to the round metric on $\Sp^{m-1} = \Sp^{2n+1}$.
\end{lem}

\begin{proof}
    The Hopf fibration $\varphi \col \Sp^{2n+1} \to \C\P^{n}$ is a Riemannian fibration for a multiple of the Fubini--Study metric on $\C\P^{n}$ which is Einstein with constant $\kappa = 2(n+1) = m$ (the Einstein constant is always the same for regular Sasaki--Einstein manifolds, cf. \cite{sparks2011sasaki}), and thus from \eqref{eq: Jacobi for fibration over Einstein manifold} we obtain
    \begin{equation*}
        \int_{\Sp^{m-1}} \langle J^\varphi u, u \rangle \vol_{\Sp^{m-1}} = \int_{\Sp^{m-1}} (|\prescript{\varphi}{}{\nabla} u|^2 - m|u|^2)\vol_{\Sp^{m-1}} .
    \end{equation*}
    From \cref{lem: connection for higher Hopf fibrations} and taking into account that $|\xi|^2 = 1 = |\alpha|^2$ for the round metric, we deduce the following identity:
    \begin{equation*}
        |\prescript{\varphi}{}{\nabla}u|^2 = |\nabla u|^2 + 2 |u|^2 - 2 \langle \nabla u, \alpha \otimes Iu \rangle - 2 \langle \nabla u, (Iu)^\sharp \otimes \xi \rangle .
    \end{equation*}
    Now on the one hand we have
    \begin{equation*}
        \langle \nabla u, \alpha \otimes Iu \rangle = \langle \nabla_\xi u, I u \rangle = \langle \pounds_\xi u + \nabla_u \xi, I u \rangle = (k+1) |Iu|^2 = (k+1) |u|^2
    \end{equation*}
    since $\pounds_\xi u = k Iu$. On the other hand,
    \begin{equation*}
        \langle \nabla u, (Iu)^\sharp \otimes \xi \rangle = \langle \nabla_{Iu} u, \xi \rangle = - \langle u, \nabla_{Iu} \xi \rangle = - \langle u, I^2 u \rangle = |u|^2 
    \end{equation*}
    since $I^2 = -1$. Putting together the previous computations we obtain
    \begin{equation*}
        |\prescript{\varphi}{}{\nabla} u|^2 - m|u|^2 = |\nabla u|^2 - (m + 2(k+1)) |u|^2 .
    \end{equation*}
    But in \cref{subsec: The identity map} we showed that $J^{\mathrm{id}} = \nabla^* \nabla u - (m-2)|u|^2$, thus the result follows.
\end{proof}

    \subsubsection{Index and nullity}       \label{subsubsec: index and nullity of hopf fibrations}

In order to use the previous lemma in order to control the spectrum of $J^\varphi$, let us first make a few easy observations:
\begin{itemize}
    \item There is a $2(n+1)$-dimensional vector space $W_- \subset \ker(J^\varphi + 2n-1)$ spanned by the projections of the constant vector fields of $\R^m \simeq \C^{n+1}$ onto the orthogonal complement of $\textup{span}\{\xi,J\xi\}$. Indeed, as noted in \cref{subsubsec: further remarks}, the infinitesimal translations in $\R^m$ generate homogeneous Jacobi field of degree $-1$ for the dilation-invariant harmonic map $\R^m \setminus \{0\} \simeq \C^{n+1} \setminus \{0\} \to \C\P^{n}$ corresponding to the Hopf fibration, and thus the projection of those vector fields must be eigenvectors of $J^\varphi$ for the eigenvalue $-m+3 = -2n+1$. Note moreover that the complex structure $I$ on $\mathscr H$ leaves $W_-$ invariant.

    \item There is a finite-dimensional vector space $W_0 \subset \ker(J^\varphi)$, defined as the smallest $I$-invariant vector space containing all the projections onto the orthogonal complement of $\xi$ of the Killing fields generating the action of $\mathrm{SO}(m) = \mathrm{SO}(2(n+1))$ on $\Sp^{m-1} = \Sp^{2n+1}$. Indeed the projections of the Killing fields must definitely be Jacobi fields, and since $I$ commutes with $J^\varphi$ the whole space $W_0$ must be contained in $\ker(J^\varphi)$. To describe $W_0$ more precisely, remark that there is an orthogonal decomposition of the Lie algebra of $\mathrm{SO}(2(n+1))$ as
    \begin{equation*}
        \mathfrak{so}_{2(n+1)} = \mathfrak{u}_{n+1} \oplus \mathfrak{p}_{n+1} = \mathfrak{u}^\perp_{n+1} \oplus \R\xi \mathfrak{} \oplus \mathfrak{p}_{n+1}
    \end{equation*}
    where $\mathfrak{u}^\perp_{n+1}$ is the orthogonal complement of $\xi$ in $\mathfrak{u}_{n+1}$. In terms of dimensions, $\dim(\mathfrak{u}^\perp_{n+1}) = (n+1)^2-1 = n(n+2)$ and $\dim(\mathfrak{p}_{n+1}) = n(n+1)$. The space $\mathfrak{u}^\perp_{m+1}$ can in fact be identified with the Lie algebra of the group $\mathrm{PU}(n+1)$ of isometries of the Fubini--Study metric on $\C\P^n$, and its complexification $\mathfrak{u}^\perp_{n+1} \oplus i \mathfrak{u}^\perp_{n+1}$ is the Lie algebra of $\mathrm{PGL}(n+1,\C)$, the group of biholomorphisms of $\C\P^n$. On the other hand, it is not difficult to see that the vector space defined by the projection of elements of $\mathfrak{p}_{n+1}$ is invariant by $I$, and thus we obtain an isomorphism
    \begin{equation*}
        W_0 \simeq \mathfrak{u}^\perp_{n+1} \oplus I \mathfrak{u}^\perp_{n+1} \oplus \mathfrak{p}_{n+1} \simeq \mathfrak{pgl}_{n+1} \oplus \mathfrak{p}_{n+1}
    \end{equation*}
    where $\mathfrak{pgl}_{n+1}$ is embedded in $W_0$ by taking the horizontal lifts of holomorphic vector fields on $\C\P^n$. We deduce that (over $\R$) $\dim(W_0) = 2n(n+2) + n(n+1) = n(3n+5)$. Moreover, all the Jacobi fields of $W_0$ can be integrated to a harmonic deformation of $\varphi$, by pre-composing $\varphi$ by an isometry of $\Sp^{2n+1}$ and post-composing by a biholomorphism of $\C\P^n$.
\end{itemize}
By the above observations, to prove that $\varphi$ satisfies the desired properties it suffices to show that the restriction of the Jacobi operator $J^\varphi$ to the $L^2$-orthogonal complement of $W_- \oplus W_0$ has positive eigenvalues. To that end, let us denote by $H_k^\perp$ the $L^2$-orthogonal space of $W_- \oplus W_0$ in $H_k$. Then we have the following result:

\begin{lem}     \label{lem: projections}
    Let $k \in \Z$ and let $u \in H^\perp_k \subset L^2(\mathscr{H}) \subset L^2(T\Sp^{m-1})$. Let us denote, for any $\ell \geq 1$, by $u_\ell$ (respectively $u^\prime_\ell$) the orthogonal projection of $u$ onto the subspace $V_\ell \subset L^2(T\Sp^{2n+1})$ (resp. onto $V^\prime_\ell \subset L^2(T\Sp^{2n+1})$) defined in \cref{subsec: The identity map}. Then the following properties hold:
    \begin{enumerate}[(i)]
        \item $u_1 = 0 = u^\prime_1$.
        \item If $k \neq -2$ then $u_2 = 0$.
        \item For any $\ell \geq 3$, if $k \notin \{-\ell, -\ell+2, \cdots, \ell-2,\ell\}$ then $u_\ell = 0$.
        \item For any $\ell \geq 2$, if $k \notin \{-\ell-1,-\ell+1, \ldots, \ell-3,\ell-1\}$ then $u^\prime_\ell = 0$.
    \end{enumerate}
\end{lem}

\begin{proof}
    Properties (iii) and (iv) in fact hold for any $\ell \geq 1$. Indeed, since $\pounds_\xi u = k Iu$ it follows that the flow $\textup{Flow}_{t}^\xi$ of $\xi$ acts by $(\textup{Flow}_{t}^\xi)^* u = \mathrm{e}^{tkI} u$ (where we use the exponential of $I$ seen as an endomorphism of $\mathscr H$), and since the eigenspace decomposition $L^2(T\Sp^{m-1}) = \oplus_\ell (V_\ell \otimes V^\prime_\ell)$ is invariant under the action of $\mathrm{SO}(m)$ if follows that $(\textup{Flow}_{t}^\xi)^* u_\ell = \mathrm{e}^{tkI} u_\ell$ and $(\textup{Flow}_{t}^\xi)^* u^\prime_\ell = \mathrm{e}^{tkI} u^\prime_\ell$ for all $\ell \geq 1$. In particular, the dual $1$-forms $\alpha_\ell = u^\sharp_\ell = g(u_\ell,\cdot)$ and $\alpha^\prime_\ell = (u^\prime)^\sharp_\ell = g(u^\prime_\ell,\cdot)$ satisfy
    \begin{equation*}
        (\textup{Flow}_{t}^\xi)^* \alpha_\ell = g(\mathrm{e}^{tkI} u_\ell,\cdot) = g(u_\ell, \mathrm{e}^{-tkI} \cdot) =  \alpha_\ell  \circ \mathrm{e}^{-tkI}, ~~ (\textup{Flow}_{t}^\xi)^* \alpha^\prime_\ell = \alpha^\prime_\ell  \circ \mathrm{e}^{-tkI},
    \end{equation*}
    for all $\ell \geq 1$. Now for any $\ell \geq 1$, there exists a unique $(1,0)$-form $\eta_\ell$ on $\R^m \simeq \C^{n+1}$ whose coefficients are (complex-valued) harmonic polynomials of degree $\ell-1$ such that $\alpha_\ell = \iota^* \Re(\eta_\ell)$ (cf. \cref{subsec: The identity map} and the fact that there is a $1$-to-$1$ correspondence between real-valued $1$-forms and $(1,0)$-forms on $\C^{n+1}$), and from the above computation we deduce that $(\textup{Flow}_{t}^\xi)^* \eta_\ell = \mathrm{e}^{-kit} \eta_\ell$. In the same way, there is a unique $(1,0)$-form on $\C^{n+1}$ whose coefficients are harmonic polynomials of degree $\ell$ and $\alpha^\prime_\ell = \iota^* \Re(\eta^\prime_\ell)$, and this $1$-form must satisfy $(\textup{Flow}_{t}^\xi)^*\alpha^\prime_\ell = \mathrm{e}^{-ikt} \eta^\prime_\ell$. Now, any constant $(1,0)$-form $\eta$ on $\C^{n+1}$ satisfies $(\textup{Flow}_{t}^\xi)^* \eta = \mathrm{e}^{it} \eta$, and from this observation it is easy to deduce the more general fact that if $\eta$ is a $(1,0)$-form on $\C^{m+1}$ whose coefficients are homogeneous polynomials of degree $d$ and $(\textup{Flow}_{t}^\xi)^*\eta = \mathrm{e}^{i\delta t} \eta$ then $\delta \in \{-d+1,-d+3,\ldots,d-1,d+1\}$ unless $\eta = 0$. Applying this with $d = \ell-1$, $\delta = -k$ and $\eta = \eta_\ell$ we obtain point (iii), and with $d = \ell$, $\delta = -k$ and $\eta = \eta^\prime_\ell$ we obtain point (iv).

    It remains to prove the first two points. The fact that $u_1 = 0$ is clear since $W_-$ is the projection of $V_1$ onto the horizontal distribution $\mathscr H$ and we assumed $u$ to be orthogonal to $W_- \oplus W_0$ for the $L^2$-inner product. On the other hand, we have seen that $W_0$ can be decomposed as $W_0 \simeq \mathfrak{u}_{n+1}^\perp \oplus \mathfrak{p}_{n+1} \oplus I \mathfrak{u}_{n+1}^\perp$, where $\mathfrak{u}_{n+1}^\perp \oplus \mathfrak{p}_{n+1} \simeq \mathfrak{so}_{2(n+1)} / \R \xi$ is isomorphic to the projection of $V^\prime_1$ onto $\mathscr H$. This proves $u_1^\prime = 0$. It is not difficult to see that $I\mathfrak{u}_{m+1}$ is identified with the subspace of $C^\infty(\mathscr H)$ spanned by the projection of the vector fields on $\Sp^{m-1}$ whose dual $1$ forms can be written as $\alpha = \iota^* \Re(\eta)$ where $\eta = \sum a_{ij} \overline{z}_i\diff z_j$ where $a_{ij}$ is a traceless hermitian $(n+1) \times (n+1)$ matrix. Such $1$-forms on $\C^{n+1}$ span a subspace of $F_2$ in the notations of \cref{subsec: The identity map}, whose complement is spanned by those $1$-forms which can be written as $\Re(\sum a_{ij} z_i \diff z_j)$ where $a_{ij}$ is a symmetric $(n+1) \times (n+1)$ complex matrix. These $1$-forms obviously satisfy $(\textup{Flow}_{t}^\xi)^* \eta = \mathrm{e}^{2it} \eta$, and thus arguing as in the previous paragraph we deduce property (ii). 
\end{proof}

We are finally equipped to prove:

\begin{thm}
    For any $n \geq 2$, the index of the Hopf fibration $\varphi \col \Sp^{2n+1} \to \C\P^n$ is $2(n+1)$, its nullity is $n(3n+5)$, and the lowest positive eigenvalue of $J^\varphi$ is greater or equal to $1$. Moreover, all the Jacobi fields of $\varphi$ can be integrated by pre-composing $\varphi$ with an isometry of $\Sp^{2n+1}$ and post-composing $\varphi$ with a biholomorphism of $\mathbb{CP}^n$.
\end{thm}

\begin{proof}
    Let $k \in \Z$ and $u \in H_k^\perp \cap \ker(J^\varphi-\lambda)$, and assume that $u \neq 0$. We want to prove that $\lambda > 0$. On the one hand, $\int_{\Sp^{m-1}} \langle J^\varphi u, u \rangle \vol_{\Sp^{m-1}} = \lambda \int_{\Sp^{m-1}} |u|^2 \vol_{\Sp^{m-1}}$. On the other hand, by \cref{lem: Jacobi operator of Hopf vs identity} and \cref{lem: projections} we have
    \begin{align*}
        \int_{\Sp^{m-1}} \langle J^\varphi u, u \rangle \vol_{\Sp^{m-1}} & = \int_{\Sp^{m-1}} \langle (J^{\id} u, u \rangle - 2(k+2) |u|^2 ) \vol_{\Sp^{m-1}} \\
            & = \sum_{\ell \geq 2} (\ell^2 + (\ell-2)(m-2) - 2(k+2)) \|u_\ell\|^2_{L^2} \\
            & ~~~~~ + \sum_{\ell \geq 2} ((\ell-1)(m+\ell-1)-2(k+2)) \|u^\prime_\ell\|^2_{L^2} .
    \end{align*}
    When $k = -2$, we have $\ell^2 + (\ell-2)(m-2) \geq 4$ and $(\ell-1)(m+\ell-1) \geq m+1 \geq 7$ as $m \geq 6$, for any $\ell \geq 2$, and therefore we obtain $\int_{\Sp^{m-1}} \langle J^\varphi u, u \rangle \vol_{\Sp^{m-1}} \geq 4 \|u\|_{L^2}^2$, which implies that $\lambda \geq 4$ in this case. Now suppose that $k \neq -2$; then by the previous lemma, $u_2 = 0$. On the other hand, for any $\ell \geq 3$ such that $u_\ell \neq 0$, we know that $k \leq \ell$ by the previous lemma, and therefore 
    \begin{equation*}
        \ell^2 + (\ell-2)(m-2) - 2(k+2) \geq \ell^2 + (\ell-2)(m-2) - 2(\ell+2) = (\ell-2)(m-4) + \ell^2 - 8 \geq 1
    \end{equation*}
    since $\ell \geq 3$. In the same way, for any $\ell \geq 2$ such that $u^\prime_\ell \neq 0$, \cref{lem: projections} implies that $k \leq \ell - 1$ and therefore
    \begin{align*}
        (\ell-1)(m+\ell-1)-2(k+2) \geq (\ell-1)(m+\ell-1)-2(\ell+1) & = (m-2)(\ell-1) - 4 + (\ell-1)^2 \\
            & \geq (m-6) + (\ell-1)^2 \geq 1
    \end{align*}
    since $\ell \geq 1$ and $m \geq 6$ by assumption. Thus we deduce that $\int_{\Sp^{m-1}} \langle J^\varphi u, u \rangle \vol_{\Sp^{m-1}} \geq \|u\|_{L^2}^2$ if $k \neq -2$, and thus $\lambda \geq 1$ in this case. This proves that the eigenvalues of $J^\varphi$ on the $L^2$-orthogonal complement of $W_- \oplus W_0$ in $C^\infty(\mathscr H)$ are greater or equal to $1$, as desired.
\end{proof}

%% file: section6_examples.tex
    \section{Examples from deformed suspensions}        \label{sec: Examples from deformed suspensions}

In this section, we construct examples of metrics and harmonic maps satisfying the assumptions of \cref{thm: main theorem introduction}. The main difficulty consist in ensuring that the last assumption (condition 4) is satisfied. In order to do this, the idea is to consider a situation which is sufficiently symmetric for the kernel of the Jacobi operator (at the relevant rates) to be understandable, but sufficiently non-generic to be able to rule out potentially problematic Jacobi fields. We discuss our strategy in more detail in \cref{subsec: General stategy of examples}, before constructing explicit examples in the remainder of the section.

    \subsection{General strategy}       \label{subsec: General stategy of examples}

We will construct examples of conically singular harmonic maps satisfying the assumptions of \cref{thm: main theorem introduction} based on the following easy but useful observation.

\begin{prop}    \label{prop: t-independent harmonic maps on I x M}
    Let $I \subset \R$ be an open interval, let $\{g_t\}_{t \in I}$ be a family of Riemannian metrics on a compact manifold $\Sigma$ and let $M = I \times \Sigma$ be endowed with the metric $g = \diff t^2 + g_t$. We consider the map $\phi \col M \rightarrow N$ given by $\phi(t,x) = \varphi(x)$, where $\varphi \col \Sigma \rightarrow N$ is a smooth map (independent of $t$). Then the tension of $\phi$ satisfies
    \begin{equation*}
        \tau(\phi,g)_{(t,x)} = \tau(\varphi,g_t)_{x}, ~~~ \forall (t,x) \in I \times \Sigma.
    \end{equation*}
    Hence $\phi$ is harmonic if and only if $\varphi: \Sigma \rightarrow N$ is harmonic for the metric $g_t$ on $\Sigma$ for all $t \in I$.
\end{prop}

\begin{proof}
    For the sake of computing the tension of $\phi$, we may compute $\prescript{\phi}{}{\nabla} \diff \phi(X,Y)$ when $X = Y = \partial_t$ or when $X, Y$ are tangent to $\{ t \} \times M$ and are invariant under time translations. In the first case, we have 
	\begin{equation}
		\prescript{\phi}{}{\nabla} \diff \phi(\partial_t,\partial_t) = (\phi^* \nabla^h)_{\partial_t} \diff \phi(\partial_t) = 0
	\end{equation}
	since $\diff \phi(\partial_t) = 0$. In the other case, we see that
    \begin{equation*}
        \prescript{\phi}{}{\nabla} \diff \phi(X,Y) = (\phi^* \nabla^h)_X(\diff \phi(Y)) - \diff \phi(\nabla^{g}_X Y) = (\varphi^* \nabla^h)_X(\diff\varphi(Y)) - \diff\varphi(\nabla^{g_t}_X Y) - \diff \phi(\II_t(X,Y))
    \end{equation*}
    where $\II_t(X,Y)$ is the second fundamental form of $\{t\} \times \Sigma$ in $M$; it is colinear to $\partial_t$ and hence $\diff \phi(\II_t(X,Y)) = 0$ and the proposition follows.
\end{proof}

For such a harmonic map, the Jacobi operator takes a simple form:

\begin{prop}    \label{prop: Jacobi for warped product and constant map}
    In the above setup, let us assume that $\varphi \col \Sigma \rightarrow N$ is harmonic with respect to $g_t$ for any $t \in I$. Then for any $u \in C^\infty(\prescript{\phi}{}{TN}) \simeq C^\infty(I, C^{\infty}(\prescript{\varphi}{}{TN}))$, we have
    \begin{equation*}
        (J^\phi u)(t) = - u^{\prime\prime}(t) - \frac{1}{2} \tr_{g_t}(g^\prime_t) u^\prime(t) + J^\varphi_{g_t} u(t)
    \end{equation*}
    where $J^\varphi_{g_t}$ is the Jacobi operator of $\varphi$ for the metric $g_t$.
\end{prop}

\begin{proof}
    We have $(\prescript{\phi}{}{\nabla} u)(t) = u^\prime(t) \diff t + \prescript{\varphi}{}{\nabla} u$ where $\prescript{\varphi}{}{\nabla}= \varphi^*(\nabla^h)$ is independent of $t$. Now suppose that $v \in C^\infty(I, C^{\infty}(\prescript{\varphi}{}{TN}))$ is supported in a compact subset of $I$. Then we have
    \begin{align*}
        & \int_I \int_\Sigma \langle u^\prime(t), v^\prime(t) \rangle_{h_\varphi} \vol_{g_t} \diff t + \int_I \int_\Sigma \langle \prescript{\varphi}{}{\nabla} u(t) , \prescript{\varphi}{}{\nabla} v(t) \rangle_{g_t,h_\varphi} \vol_{g_t} \diff t \\
        & = - \int_I \int_\Sigma \langle u^{\prime\prime}(t), v(t) \rangle_{h_\varphi} \vol_{g_t} \diff t - \int_I \int_\Sigma \langle u^\prime(t), v(t) \rangle_{h} (\vol_{g_t})^\prime \diff t \\
        & ~~~ + \int_I \int_\Sigma \langle (\prescript{\varphi}{}{\nabla})^{*_{g_t}}\prescript{\varphi}{}{\nabla} u(t) , v(t) \rangle_{g_t,h_\varphi} \vol_{g_t} \diff t .
    \end{align*}
    On the other hand, $(\vol_{g_t})^\prime = \frac{1}{2} \tr_{g_t}(g^\prime_t)\vol_{g_t}$, and therefore we obtain
    \begin{equation*}
        ((\prescript{\phi}{}{\nabla})^*\prescript{\phi}{}{\nabla} u)(t) = - u^{\prime\prime}(t) - \tr_{g_t}(g^\prime_t) u^\prime(t) + (\prescript{\varphi}{}{\nabla})^{*_{g_t}}\prescript{\varphi}{}{\nabla} u(t) .
    \end{equation*}
    To conclude the proof, remark that
    \begin{equation*}
        \tr_{g} (\mathrm{Rm}_h(u,\diff \phi)\diff \phi) = \mathrm{Rm}_h(u,\tfrac{\partial \phi}{\partial t})\tfrac{\partial \phi}{\partial t} + \tr_{g_t}(\mathrm{Rm}_h(u,\diff\varphi)\diff\varphi) = \tr_{g_t}(\mathrm{Rm}_h(u,\diff \varphi)\diff \varphi)
    \end{equation*}
    since $\frac{\partial \phi}{\partial t} = 0$.
\end{proof}

In the subsequent sections, we will construct harmonic maps $\Sp^m \to \Sp^{m-1}$ and $\Sp^4 \to \Sp^2$ with two, and $\mathbb{CP}^2 \to \Sp^2$ with one conical singularity. Our strategy will be to identify $\Sp^m \setminus \{s_1,s_2 \} \simeq I \times \Sp^{m-1}$ (for the first two examples)  and $\mathbb{CP}^2 \setminus \{ \mathbb{CP}^1,s\} \simeq I \times \Sp^3$ (for the last example). We then start with a given harmonic map $\varphi \col (\mathbb{S}^{m-1},g_{\mathbb{S}^{m-1}}) \rightarrow (N,h)$ and seek a family of metrics $g_t$ for $t \in I$ on $\mathbb{S}^{m-1}$ such that $\varphi$ remains harmonic with respect to each metric $g_t$. By \cref{prop: t-independent harmonic maps on I x M} above $\phi(t,p) = \varphi(p)$ defines a ($t$-independent) harmonic map on $(I \times \Sp^{m-1}, \diff t^2+ g_t)$. For a judicious choice of metrics, $\diff t^2 + g_t$ admits a smooth completion as a Riemannian metric onto the sphere $\Sp^m$ (or the projective space $\mathbb{CP}^2$), and the map $\phi(t,p) = \varphi(p)$ is a harmonic map with two conical singularities (in the last case, $\phi$ extends smoothly over $\mathbb{CP}^1 \subset \mathbb{CP}^2$ so that the resulting map is singular only in the single point $s$). Our choice of $\varphi$ and $(N,h)$ will imply that conditions 1--3 in \cref{thm: main theorem introduction} are immediately satisfied. 

In order to understand the kernel of the Jacobi operator of $\phi$, the idea is to choose $g_t$ such that the Jacobi operators $J^\varphi_{g_t}$ for every $t\in I$ leave invariant the eigenspaces of $J^\varphi_{g_{\Sp^{m-1}}}$ (or finite direct sums thereof). In this case, the Jacobi equation for the map $\phi$ reduces by the previous proposition to a countable collection of vector-valued linear ODEs. We can then identify which of these ODEs could potentially result in Jacobi fields that are of order $\mathcal{O}(r^{-1})$ close to each singular point, and carefully design the path of metrics to rule out those which are not the pull-back of a Jacobi field of $(N,h)$.

The following key lemma will enable us to rule out all but finitely many potential obstructions; this is essentially the maximum principle, specialised to a class of ODEs that we will need to consider:

\begin{lem}     \label{lem: Preparatory ODE}
    Fix $T > 0$, $m \geq 4$, $\lambda > 0$, and let $a, b \col (-T,T) \to \R$ be two functions such that $b(t) > 0$ for any $t \in (-T,T)$ and for $\varepsilon > 0$ small enough, $a(t) = -\frac{m-1}{T-t}$ and $b(t) = \frac{\lambda}{(T-t)^2}$ for $t \in (T-\varepsilon,T)$. Let $f \col (-T,T) \rightarrow \R$ be a solution of the ODE
    \begin{equation*}
        -f^{\prime\prime}(t) - a(t) f^\prime(t) + b(t) f(t) = 0 .
    \end{equation*}
    Suppose that there exists $\tau \in (-T,T)$ such that for any $t \in (-T,\tau]$, $f(t) > 0$ and $f^\prime(t) > 0$. Then:
    \begin{enumerate}[(i)]
        \item $f(t) > 0$ and $f^\prime(t) > 0$ for all $t \in (-T,T)$.
        \item There exists a constant $C > 0$ such that $f(t) \geq C (T-t)^{-\frac{m-2 + \sqrt{(m-2)^2+4\lambda}}{2}}$ for all $t \in [0,T)$.
    \end{enumerate}
\end{lem}

\begin{proof}
    Let us define 
    \begin{equation*}
        \theta = \sup \{ t \in (-T,T) \mid \forall s \in (-T, t], ~ f^\prime(s) > 0 \} .
    \end{equation*}
    By our assumptions, it follows that $\theta \in (\tau, T]$, and we want to prove that $\theta = T$. By contradiction, let us assume that $\theta < T$. Then we must have $f^\prime(\theta) = 0$ and $f^\prime(t) > 0$ for any $t \in [\tau,\theta)$. In particular, $f(t) \geq f(\tau) > 0$ for any $t \in [\tau,\theta]$. By continuity, there exists $A, \delta > 0$ such that for $t \in [\theta-\delta,\theta]$ we have
    \begin{equation*}
        a(t) f^\prime(t) \leq A ~~~ \text{and} ~~~ b(t) f(t) \geq 2A
    \end{equation*}
    and from the ODE satisfied by $f$ we deduce that
    \begin{equation*}
        f^{\prime\prime}(t) \geq A > 0, ~~~ \forall t \in [\theta-\delta, \theta] .
    \end{equation*}
    This implies that $f^\prime(\theta) \geq f^\prime(\theta-\delta) > 0$, which yields a contradiction. Thus $\theta= T$, which proves point (i) of the lemma.

    This in particular implies that $f(t) \geq f(\tau) > 0$ for all $t \in [\tau,T)$. On the other hand, we can solve the ODE explicitly for $t \in [T-\varepsilon,T)$, and we must have
    \begin{equation*}
        f(t) = C_- (T-t)^{\nu_-} + C_+ (T-t)^{\nu_+}, ~~~ \forall t \in [T-1/2,T)
    \end{equation*}
    for some constants $C_-,C_+$, where
    \begin{equation*}
        \nu^-_\lambda = - \frac{m-2+ \sqrt{(m-2)^2+4\lambda}}{2} < -1 , ~~ \text{and} ~~ \nu^+_\lambda = \frac{-m+2 + \sqrt{(m-2)^2+4\lambda}}{2} > 0 .
    \end{equation*}
    Since $\lambda > 0$, the second term goes to $0$ as $t \rightarrow T$, and coupled with the condition that $f(t) \geq f(\tau)$ this implies that $C_- > 0$, which proves point (ii).
\end{proof}

    \subsection{Harmonic map from $\Sp^m$ to $\Sp^{m-1}$ with two conical singularities}        \label{subsec: map S^m to S^{m-1}}

To implement the strategy outlined in the previous section in a simple setting, let us fix a smooth, non-decreasing function $\kappa \col (0,\infty) \rightarrow (-\infty, 0]$, such that $\kappa(t) = \log(t)$ for $t \in (0,1/2]$ and $\kappa(t) = 0$ for $t \in [1,+\infty)$. For any $T \geq 2$, we may define $\kappa_T \col (-T,T) \rightarrow (-\infty,0]$ as $\kappa_T(t) \coloneqq \kappa(T+t)$ for $t \in (-T,0]$ and $\kappa_T(t) = \kappa(T-t) = \kappa(T-t)$ for $t \in [0,T)$. The simplest example of our ansatz is to consider the harmonic map $\varphi = \id \col \Sp^{m-1} \rightarrow \Sp^{m-1}$ and to define $\phi \col \Sp^m\setminus \{s_1,s_2\} \simeq (-T,T) \times \Sp^{m-1} \to \Sp^{m-1}$ as $\phi(t,\sigma) = \sigma$. If we equip $(-T,T) \times \Sp^{m-1}$ with the metric $\diff t^2 + \mathrm{e}^{2\kappa_T(t)}g_{\Sp^{m-1}}$, then $\phi$ is harmonic by \cref{prop: t-independent harmonic maps on I x M}. Moreover, this metric smoothly extends over the entire $\Sp^m$ and $\phi \col (\Sp^m, g_T \coloneqq \diff t^2 + \mathrm{e}^{2\kappa_T(t)}g_{\Sp^{m-1}}) \to (\Sp^{m-1},g_{\Sp^{m-1}})$ is a conically singular harmonic map with two singularities at $\{s_1,s_2\}$ modelled on $\varphi = \id$. By the discussion in \cref{subsec: The identity map}, $\phi$ satisfies conditions 1--3 in \cref{thm: main theorem introduction}.

We will now show that for a generic choice of $T$, the last condition in \cref{thm: main theorem introduction} is also satisfied. Diagonalising the Jacobi operator $J^{\id}_{g_{\Sp^{m-1}}}$ in an $L^2$-basis, the Jacobi equation reduces by \cref{prop: Jacobi for warped product and constant map} to the countable collection of ODEs
\begin{equation}    \label{eq: id example Jacobi ODE in Vk}
    - u_k^{\prime\prime}(t) - (m-1) \kappa_T^\prime(t) u_k^\prime(t) + \mathrm{e}^{-2\kappa_T(t)} \lambda_k u_k(t) = 0
\end{equation}
where $\lambda_k$ ranges across the eigenvalues of $J^{\id}_{g_{\Sp^{m-1}}}$. We denote the eigenspace corresponding to $\lambda_k$ by $V_{\lambda_k} \subset C^{\infty}(\prescript{\varphi}{}{T\Sp^{m-1}})$.

\begin{lem}
    Let $u$ be a Jacobi field such that $|u| = \mathcal{O}(r^{-1})$ close to each singular point (with $r$ denoting the distance to the respective singular point), then $u(t) \in V_{-(m-3)} \oplus V_0$ for any $t \in (-T,T)$.
\end{lem}

\begin{proof}
    Suppose that $\lambda_k > 0$. Then on $(-T,-T+1/2]$, the solutions of \eqref{eq: id example Jacobi ODE in Vk} read
\begin{equation*}
    u_k(t) = C_- (T+t)^{\nu^-_{\lambda_k}} + C_+ (T+t)^{\nu^+_{\lambda_k}}
\end{equation*}
where $\nu^-_{\lambda_k} < -1$ and $\nu^+_{\lambda_k} > 0$ are as in the proof of \cref{lem: Preparatory ODE}. Therefore, if $|u_k(t)| = \mathcal{O}((T+t)^{-1})$ we must have $C_- = 0$. After multiplying $u_k$ by $-1$ if necessary, we may assume that $C_+ \geq 0$. Thus for $t \in (-T,-T+\frac{1}{2})$ we have $u_k(t) = C_+ t^{\nu^+_{\lambda_k}} \geq 0$ and $u_k^\prime(t) = \nu^+_{\lambda_k} C_+ t^{\nu^+_{\lambda_k}-1} \geq 0$, where equality holds if and only if $C_+ = 0$. Thus it follows from \cref{lem: Preparatory ODE} (applied to the functions $a(t) = (m-1)\kappa^\prime_T(t)$ and $b(t) = e^{-2\kappa^\prime_T(t)}\lambda_k$) that $u_k(t)$ satisfies the bound $|u_k(t)| = \mathcal{O}((T-t)^{-1})$ if and only if $C_+ = 0$. This implies the lemma because $-(m-3)$ and $0$ are the only non-positive eigenvalues of $J^{\id}_{g_{\Sp^{m-1}}}$.
\end{proof}

In order to deal with the eigenspaces $V_{-(m-3)}$ and $V_0$, let us first introduce the following definition:

\begin{Def}
    We call $T \in [2,\infty)$ \emph{admissible} if the unique solution of the ODE
    \begin{equation*}
        f^{\prime\prime}(t) + (m-1)\kappa^\prime(t) f^\prime(t) + (m-3) \mathrm{e}^{-2\kappa(t)} f(t) = 0
    \end{equation*}
    such that $f(t) = t^{-1}$ on the interval $(0,1/2]$ satisfies $f(T) \neq 0$ and $f^\prime(T) \neq 0$. 
\end{Def}

The significance of this definition is the following:

\begin{prop}        \label{prop: Existence for identity}
    If $T \in [2,\infty)$ is admissible, the only Jacobi fields satisfying 
    \begin{equation*}
        |u(t)| = \mathcal{O}((T+t)^{-1}+(T-t)^{-1}) 
    \end{equation*}
    are the pull-back of Killing fields on $\Sp^{m-1}$.
\end{prop}

\begin{proof}
    Let us first prove that any such Jacobi field must take values in $V_0$. We only need to prove that if there is a solution $f \col (-T,T) \rightarrow \R$ of the scalar ODE
    \begin{equation*}
        f^{\prime\prime}(t) + (m-1) \kappa^\prime_T(t) f^\prime(t) + \mathrm{e}^{-2\kappa_T(t)} (m-3) f(t) = 0
    \end{equation*}
    such that $|f(t)| = \mathcal{O}((T+t)^{-1} + (T-t)^{-1})$, then $T$ is non-admissible. If such a solution exists, then there exist constants $A, B \neq 0$ such that $f(t) = A a(t+T)$ for $t \in (-T,0]$ and $f(t) = B a(T-t)$ for $t \in [0,T)$, where $a \col (0,\infty) \rightarrow \R$ is the unique solution of the ODE
    \begin{equation*}
        f^{\prime\prime}(t) + (m-1) \kappa^\prime(t) f^\prime(t) + \mathrm{e}^{-2\kappa(t)}(m-3) f(t) = 0
    \end{equation*}
    coinciding with $t^{-1}$ on $(0,1/2]$. Therefore, the couples
    \begin{equation*}
        (a(T),a^\prime(T)) = (f(0), f^\prime(0))/A ~~ \text{and} ~~ (a(T),-a^\prime(T)) = (f(0),f^\prime(0))/B
    \end{equation*}
    are colinear in $\R^2$, whence either $a(T) = 0$ or $a^\prime(T) = 0$. By definition, this does not occur when $T$ admissible.

    For the eigenvalue $0$, the Jacobi equation reduces to the scalar ODE
    \begin{equation*}
        f^{\prime\prime}(t) + (m-1) \kappa^\prime_T(t) f^\prime(t) = 0 
    \end{equation*}
    whose solutions on the interval $(-T,-T+1/2]$ are of the form
    \begin{equation*}
        f(t) = C + D (T+t)^{-(m-2)}
    \end{equation*}
    and since $m \geq 4$ it follows that the only solutions satisfying the required bounds are constant functions. The corresponding Jacobi fields are precisely the constant functions valued in the kernel of $J^{\id}_{g_{\Sp^{m-1}}}$, that is, the pull-backs of the Killing fields of $\Sp^{m-1}$.
\end{proof}

\begin{lem}     \label{lem: Point (i) of admissible values}
    For any $R \in [2,\infty)$, there is exactly one $T \in [R,R+\tfrac{\pi}{2\sqrt{m-3}})$ such that the unique solution of the ODE
    \begin{equation*}
            f^{\prime\prime}(t) + (m-1) \kappa^\prime(t) f^\prime(t) + (m-3) \mathrm{e}^{-2\kappa(t)} f(t) = 0 ,
    \end{equation*}
    such that $f(t) = t^{-1}$ on $(0,1/2]$ satisfies $f(T) = 0$ or $f^\prime(T) = 0$. Consequently, the admissible values are dense in $[2,\infty)$.
\end{lem}

\begin{proof}
    On the interval $[1,\infty)$, the ODE reads
    \begin{equation*}
        f^{\prime\prime}(t) + (m-3)f(t) = 0
    \end{equation*}
    and the solutions are of the form $A\cos(\sqrt{m-3}t+\theta)$ for $A, \theta \in \R$. Therefore $f(t)$ or $f^\prime(t)$ vanish with periodicity $\tfrac{\pi}{2\sqrt{m-3}}$.
\end{proof}

This yields our first example of conically singular harmonic map satisfying the assumptions of the deformation theorem, \cref{thm: main theorem introduction}:

\begin{thm}
    For a dense subset of values $T \in [2,+\infty)$, the conically singular harmonic map $\phi \col (\Sp^m,g_T) \to (\Sp^{m-1},g_{\Sp^{m-1}})$ satisfies the assumptions of \cref{thm: main theorem introduction} with respect to the metric $g_T \coloneqq \diff t^2 + \mathrm{e}^{2\kappa_T(t)}g_{\Sp^{m-1}}$ on $\Sp^m$.
\end{thm}

    \subsection{Harmonic map from $\mathbb{S}^4$ to $\Sp^2$ with two conical singularities}     \label{sec: map S^4 to S^2}

We will now apply our strategy to construct a metric on $\Sp^4$ and a harmonic map $\Sp^4 \to \Sp^2$ with two conical singularities modelled on the Hopf fibration that satisfy the assumptions of \cref{thm: main theorem introduction}. For this, fix a smooth, non-decreasing function $\chi \col \R \rightarrow [0,1]$ such that $\chi(t) = 0$ if $t \in (-\infty,1]$ and $\chi(t) = 1$ for $t \in [2,+\infty)$. For $T\geq 2$ define $\chi_T \col (-T,T) \rightarrow [0,1]$, such that $\chi_T(t) = \chi(T+t)$ for $t \in (-T,0]$ and $\chi_T(t) = \chi_T(-t) = \chi(T-t)$ for $t \in [0,T)$. Moreover, let $\kappa_T \col (-T,T) \to (-\infty,0]$ be as in the previous section. For $\varepsilon \in [0,1)$ we now equip $\Sp^4 \setminus \{s_1,s_2\} \simeq (-T,T) \times \Sp^3$ with the metric $g_{\varepsilon,T} \coloneqq \diff t^2 + \mathrm{e}^{2\kappa_T(t)} \widetilde{g}_{\varepsilon,T}(t)$, where for each $t \in (-T,T)$
\begin{equation*}
    \widetilde{g}_{\varepsilon,T}(t) \coloneqq \frac{1}{1- \varepsilon \chi_T(t)} X^1 \otimes X^1 + \frac{1}{1+\varepsilon \chi_T(t)} X^2 \otimes X^2 + X^3 \otimes X^3
\end{equation*}
is a metric on $\Sp^3$ of the form considered in \cref{subsubsec: The Jacobi operator for left-invariant metrics}. As before, $g_{\varepsilon,T}$ extends smoothly over the completion $\Sp^4$. Let $\phi \col (-T,T) \times \Sp^3 \rightarrow \Sp^2$ be the map defined as $\phi(t,\sigma) = \varphi(\sigma)$ where $\varphi \col \Sp^3 \rightarrow \Sp^2$ is the Hopf fibration. Then $\phi$ is harmonic by \cref{prop: t-independent harmonic maps on I x M} and \cref{lem: tension of Hopf fibration for deformed metric} and therefore defines a conically singular harmonic map $\phi \col \Sp^4 \to \Sp^2$. We have seen in \cref{subsubsec: Computations for Hopf fibration} that the first three conditions of \cref{thm: main theorem introduction} are immediately satisfied and we only need to show that for a suitable choice of $T$, the only Jacobi fields of $\phi$ that are of order $\mathcal{O}(r^{-1})$ close to both singular points are the pullbacks of Killing vector fields on $\Sp^2$.

As in \cref{subsec: Hopf fibration as model map}, we shall regard a section $u \in \Gamma(\prescript{\phi}{}{T\Sp^2})$ as a complex-valued function $u \col (-T,T) \times \Sp^3 \rightarrow \C$. By \cref{prop: Jacobi for warped product and constant map} we see that the Jacobi operator may be written as
\begin{equation*}
    J^\phi_{\varepsilon,T} u = - \frac{\partial^2u}{\partial t^2} - (3 \kappa_T^\prime(t) + \rho_{\varepsilon,T}(t)) \frac{\partial u}{\partial t} + \mathrm{e}^{-2\kappa_T(t)}J^\varphi_{\varepsilon,T}(t) u
\end{equation*}
where the function $\rho_{\varepsilon,T} \col (-T,T) \rightarrow \R$ is defined as
\begin{equation}        \label{eq: rho epsilon T}
    \rho_{\varepsilon,T}(t) = (1-\varepsilon\chi_T(t)) \left( \frac{1}{1-\varepsilon \chi_T(t)}\right)^\prime + (1+\varepsilon\chi_T(t)) \left( \frac{1}{1+\varepsilon \chi_T(t)}\right)^\prime = \frac{2\varepsilon^2 \chi_T^\prime(t) \chi_T(t)}{1-\varepsilon^2 \chi^2_T(t)} 
\end{equation}
and $J^\varphi_{\varepsilon,T}(t)$ is the Jacobi operator of the Hopf fibration associated with the metric $\widetilde{g}_{\varepsilon,T}(t)$. As determined in \cref{lem: Jacobi operator for Hopf map with perturbation of metric}, the Jacobi operator $J^\varphi_{\varepsilon,T}$ takes a particularly simple form:
\begin{equation*}
    J^\varphi_{\varepsilon,T}u = J^\varphi u + \varepsilon \chi_T(t) (\partial_{X_1}^2u - \partial_{X_2}^2u - 4 \overline{u}) 
\end{equation*}
where $J^\varphi$ is the Jacobi operator associated with the round metric on $\Sp^3$. For later purpose, let us record two important properties of the function $\rho_{\varepsilon,T}$:
\begin{itemize}
    \item $\rho_{\varepsilon,T}$ vanishes at order $1$ in $\varepsilon$, that is, $\rho_{0,T}(t) = \left. \frac{\partial \rho_{\varepsilon,T}}{\partial \varepsilon} \right|_{\varepsilon=0}(t) = 0$, for $t \in (-T,T)$, and

    \item $\rho_{\varepsilon,T}(t) = \varepsilon^2 \chi_T(t) \psi_{\varepsilon,T}(t)$ for a smooth function $\psi_{\varepsilon,T} \col (-T,T) \rightarrow [0,\infty)$ which remains uniformly bounded if $\varepsilon \in [0,1/2]$.
\end{itemize}
It will also be convenient to introduce the function $\rho_\varepsilon \col \R \rightarrow \R$ defined as
\begin{equation*}
    \rho_\varepsilon(t) = \frac{2 \varepsilon^2 \chi^\prime(t) \chi(t)}{1-\varepsilon^2\chi^2(t)} \cdot
\end{equation*}
Then $\rho_{\varepsilon,T}(t) = \rho_\varepsilon(T+t)$ for $t \in (-T,0]$, and $\rho_{\varepsilon,T}(t) = - \rho_{\varepsilon,T}(-t) = - \rho_\varepsilon(T-t)$ for $t \in [0,T)$.

As explained in \cref{subsubsec: The Jacobi operator for left-invariant metrics}, the operators $J^\varphi_{\varepsilon,T}(t)$ leave invariant the eigenspaces $W_k \subset C^\infty(\Sp^3,\C)$ of harmonic homogeneous polynomials. Thus, if we write $u = \sum_k U_k(t)$ where $U_k \col (-T,T) \rightarrow W_k$ is a vector-valued function, the Jacobi equation can be decomposed into the linear second-order ODE systems
\begin{equation*}
    -U^{\prime\prime}_k(t) - (3\kappa_T^\prime(t) + \rho_{\varepsilon,T}(t)) U^\prime_k(t) + \mathrm{e}^{-2\kappa_T(t)} J^k_{\varepsilon,T}(t)U_k(t) = 0
\end{equation*}
where the endomorphisms $J^k_{\varepsilon,T}(t) \in \End(W_k)$ depend smoothly on $(t,\varepsilon,T)$.  

In order to control the kernel of the Jacobi operator, we need a different notion of admissibility than in the previous example:

\begin{Def}     \label{Def: Admissible values}
    We shall call $T \in [2,\infty)$ \emph{admissible} if the following properties are satisfied.
    \begin{enumerate}
        \item The unique solution $f \col (0,\infty) \rightarrow \R$ of ODE
        \begin{equation*}
            f^{\prime\prime}(t) + 3 \kappa^\prime(t) f^\prime(t) + \mathrm{e}^{-2\kappa(t)} f(t) = 0 ,
        \end{equation*}
        such that $f(t) = t^{-1}$ for $t \in (0,1/2]$ satisfies $f(T) \neq 0$ and $f^\prime(T) \neq 0$.

        \item The unique solutions $X_\pm(\varepsilon,\cdot), Y_\pm(\varepsilon,\cdot) \col [1,\infty) \rightarrow \R^2$ of the ODE
            \begin{equation*}
                - U^{\prime\prime}(t) - \rho_\varepsilon(t) U^\prime(t) + 16 \begin{pmatrix} 1 & \pm \frac{1}{2} \varepsilon \chi(t) \\ \pm \varepsilon \chi(t) & 0 \end{pmatrix} U(t) = 0
            \end{equation*}
            satisfying the initial conditions
            \begin{equation*}
                X_\pm(\varepsilon,1) = 
                \begin{pmatrix}
                    1 \\
                    0
                \end{pmatrix}, ~ X^\prime_\pm(\varepsilon,1) = 
                \begin{pmatrix}
                    v_0 \\
                    0
                \end{pmatrix}, ~~~ \text{and} ~~~ Y_\pm(\varepsilon,1) = 
                \begin{pmatrix}
                    0\\
                    1
                \end{pmatrix}, ~ Y^\prime_\pm(\varepsilon,1) = 
                \begin{pmatrix}
                    0\\
                    0
                \end{pmatrix}
            \end{equation*}
            have the property that for $\varepsilon > 0$ small enough, $X^\prime_\pm(\varepsilon,T)$ and $Y^\prime_\pm(\varepsilon,T)$ are linearly independent in $\R^2$. Here, $v_0 \coloneqq f^\prime(1)$, where $f \col (0,1] \to \mathbb{R}$ is the unique solution of the ODE
                \begin{equation*}
                    -f^{\prime\prime}(t) - 3 \kappa^\prime(t) f^\prime(t) + 16 \mathrm{e}^{-2\kappa(t)} f(t) = 0
                \end{equation*}
            that satisfies $f(1) = 1$ and stays bounded. Note that \cref{lem: Preparatory ODE} implies that such a solution exists and that $v_0 >0$. Indeed, on the interval $(0,1/2]$ the solutions are of the form
            \begin{equation*}
                C_- t^{-1 - \sqrt{17}} + C_+ t^{-1+\sqrt{17}}.
            \end{equation*}
            In particular, the unique solution $\tilde{f}$ coinciding with $t^{-1+\sqrt{1+\lambda}}$ on $(1,1/2]$ has $\tilde{f}(1/2) > 0$ and $\tilde{f}^\prime(1/2)>0$, and therefore by (the proof of) \cref{lem: Preparatory ODE} it follows that $\tilde{f}(1) > 0$ and $\tilde{f}^\prime(1) > 0$, so that $v_0 = \tilde{f}^\prime(1) / \tilde{f}(1) > 0$.
    \end{enumerate}
\end{Def}

\begin{prop}     \label{thm: S^4 to S^2 example}
    Let $T \in [2,\infty)$ be an admissible value. Then there exists $\varepsilon(T) > 0$ such that for any $\varepsilon \in (0,\varepsilon(T))$, the only Jacobi fields of the harmonic map $\phi \col \mathbb{S}^4 \rightarrow \Sp^2$ with respect to the metric $g_{\varepsilon,T}$ satisfying the bounds
    \begin{equation*}
        |u| = \mathcal{O}(r_1^{-1} + r_2^{-1})
    \end{equation*}
    where $r_i$ is the distance to $s_i$, for $i=1,2$, are the lifts of the Killing fields of $\Sp^2$.
\end{prop}

To prove this proposition, we need a series of intermediate results:

\begin{lem}     \label{lem: Finite-dimension restriction}
    Fix $T \in [2,\infty)$. Then there exists $\varepsilon_0(T) > 0$ such that for any $\varepsilon \in [0,\varepsilon_0(T)]$ the following holds. Let $u$ be a Jacobi field with respect to the metric $g_{\varepsilon,T}$ such that
    \begin{equation*}
        |u| = \mathcal{O}((T+t)^{-3/2} + (T-t)^{-3/2}) .
    \end{equation*}
    Then $u(t) \in W_0 \oplus W_1 \oplus W_2$ for all $t \in (-T,T)$.
\end{lem}

\begin{proof}
    Let us fix a Sobolev exponent $p > 4$. Any Jacobi field satisfying this bounds in particular belongs to the weighted H\"older space $W^{3,p}_{-3/2}(\mathbb{S}^4,\C)$; notice that since (for fixed $T$) the metrics $g_{\varepsilon,T}$ only differ in compact subset supported away from $\{s_1,s_2\}$, defining these weighted spaces with respect to any of the metrics $g_{\varepsilon,T}$ for $\varepsilon \in [0,1]$ yields equivalent norms, so there is no ambiguity. Moreover, the Jacobi operators $J^\phi_{\varepsilon,T}\col W^{3,p}_{-3/2}(\mathbb{S}^4,\C) \rightarrow W^{1,p}_{-7/2}(\mathbb{S}^4,\C)$ form a continuous family of bounded Fredholm operators, since $-3/2$ is not an indicial root (cf. \cref{prop: indicial roots of model operator}).
    
    For any $\nu < 0$ and $k \geq 1$, let us consider the subspace $B^{k,p}_\nu \subset W^{k,p}_\nu((-T,T) \times \Sp^3,\C)$ of functions $u(t,\sigma)$ such that $u(t,\cdot)$ is orthogonal to $W_0 \oplus W_1 \oplus W_2$ for all $t \in (-T,T)$. Since $p > 4$, $k \geq 1$ and $\nu < 0$, we have a continuous embedding $W^{k,p}_\nu(\Sp^4,\C) \hookrightarrow C^0_0(\Sp^4,\C)$, and thus for any $t$, the map 
    \begin{equation*}
        u \in W^{k,p}_\nu((-T,T) \times \Sp^3,\C) \mapsto u(t,\cdot) \in C^0(\Sp^3) 
    \end{equation*}
    is continuous, and moreover the $L^2$-inner product is continuous on $C^0(\Sp^3)$, and therefore $B^{k,p}_\nu$ is the intersection of the kernels of continuous maps; therefore $B^{k,p}_\nu$ is closed for the $W^{k,p}_\nu$-norm. By restriction, the Jacobi operators $J^\phi_{\varepsilon,T}$ define a continuous family of bounded Fredholm operators $B^{3,p}_{-3/2} \rightarrow B^{1,p}_{-7/2}$. Therefore, to prove our claim it is enough to show that the restriction of $J^\phi_{0,T}$ defines an invertible map. On the other hand, since the operator is self-adjoint, the cokernel of $J^\phi_{0,T} \col B^{3,p}_{-3/2} \rightarrow B^{1,p}_{-7/2}$ is $\ker(J^\phi_{0,T}) \cap B^{3,p}_{-1/2} \subset B^{3,p}_{-3/2}$, and therefore it is enough to prove that $J^\phi_{0,T}$ has no kernel in $B^{3,p}_{-3/2}$.

    For $\varepsilon = 0$, $J^\varphi_{0,T}(t) = J^\varphi$ is the Jacobi operator of the Hopf fibration for the standard round metric on $\Sp^3$, which can be diagonalised and has eigenvalues $\lambda \geq 1$ on the orthogonal complement of $W_0 \oplus W_1 \oplus W_2$. Hence we need only prove that the only solutions of the real-valued scalar ODE
    \begin{equation*}
        - f^{\prime\prime}(t) - 3 \kappa^\prime(t) f^\prime(t) + \mathrm{e}^{-2\kappa(t)}\lambda f(t) = 0
    \end{equation*}
    satisfying the bound $|f(t)| = \mathcal{O}((T+t)^{-3/2} + (T-t)^{-3/2})$ vanish identically. If $f$ is such a solution, then on the interval $(-T,-T+1/2]$, we have
    \begin{equation*}
        f(t) = C_- (T+t)^{-1-\sqrt{1+\lambda}} + C_+ (T+t)^{-1+\sqrt{1+\lambda}}
    \end{equation*}
    for some constant $C_+,C_-$. For the bound to be satisfied, it is necessary that $C_-=0$. After exchanging $f$ and $-f$ (which satisfy the same ODE by linearity) if necessary, we may assume that $C_+ \geq 0$. In particular, $f(t), f^\prime(t) \geq 0$ on the interval $(-T,-T+\frac{1}{2}]$, with equality if and only if $C_+ = 0$. By \cref{lem: Preparatory ODE}, if $C_+ > 0$ then $f(t) \geq C (T-t)^{-1 - \sqrt{1+\lambda}}$ for some constant $C > 0$, which would contradict the bound $|f(t)| = \mathcal{O}((T-t)^{-3/2})$. Hence $C_+ = 0$ and $f$ vanishes identically, which finishes the proof.
\end{proof}

\begin{lem}     \label{lem: rule out W_1}
    If $T \in [2,\infty)$ is admissible, there exists $\varepsilon_1(T) > 0$ such that for any $\varepsilon \in [0,\varepsilon_1(T)]$ the following holds. If $u \in C^\infty((-T,T) \times \Sp^3,\C)$ is a Jacobi field with respect to the metric $g_{\varepsilon,T}$ satisfying the bound 
    \begin{equation*}
        |u(t)| = \mathcal{O}((T+t)^{-1}+(T-t)^{-1}) ,
    \end{equation*}
    then $u(t) \in W_0 \oplus W_2$ for any $t \in (-T,T)$.
\end{lem}

\begin{proof}
    Given the previous lemma, it is enough to prove that for $T$ admissible and $\varepsilon > 0$ small enough, there are no solutions $u \col (-T,T) \rightarrow W_1$ of the ODE
    \begin{equation*}
        -u^{\prime\prime}(t) - (3 \kappa^\prime_T(t) + \rho_{\varepsilon,T}(t)) u^\prime(t)  + \mathrm{e}^{-2\kappa_T(t)}J^1_{\varepsilon,T}(t)u(t) = 0
    \end{equation*}
    satisfying the bounds
    \begin{equation*}
        |u(t)| = \mathcal{O}((T+t)^{-1}) + \mathcal{O}((T-t)^{-1}) .
    \end{equation*}
    When $\varepsilon = 0$, we can argue as in the proof of \cref{prop: Existence for identity}, since $T$ is admissible. It remains to argue by continuity for $\varepsilon > 0$. 
    
    Let us denote by $V^\pm_{\varepsilon,T} \subset W_1 \oplus W_1$ the $4$-plane spanned by the couples $(u(0),u^\prime(0))$ where $u(t)$ is a solution of the ODE such that $|u(t)| = \mathcal{O}((T \mp t))$ as $t \rightarrow \pm T$. Then they vary continuously in the Grassmannian of $4$-planes of $W_1 \oplus W_1$ with respect to $\varepsilon$, and moreover $V^+_{0,T} \cap V^-_{0,T} = 0$. By dimensionality, it follows that $V^+_{0,T} \oplus V^-_{0,T} = W_1 \oplus W_1$, and since this is an open property we deduce that $V^+_{\varepsilon,T} \oplus V^-_{\varepsilon,T} = W_1 \oplus W_1$ for $\varepsilon$ small enough. Using dimensionality again, we must have $V^+_{\varepsilon,T} \cap V^-_{\varepsilon,T} = 0$ for $\varepsilon$ close enough to $0$, which proves the result.
\end{proof}

From now on, let us fix an admissible $T \in [2,\infty)$. We have boiled down our problem to two finite-dimensional linear ODE systems on $W_0$ and $W_2$. Hence we are free to choose $\varepsilon > 0$ as small as we want, and we can regard the equation as a perturbation of the case $\varepsilon = 0$, which simplifies many considerations. 

\begin{lem}     \label{lem: rule out W_0}
    If $\varepsilon > 0$ is small enough, there are no solutions $u(t)$ of the ODE in $W_0$ satisfying the bounds
    \begin{equation*}
        |u(t)| = \mathcal{O}((T+t)^{-1} + (T-t)^{-1}) .
    \end{equation*}
\end{lem}

\begin{proof}
    The operators $J^0_{\varepsilon,T}(t) \in \End(W_0)$ are co-diagonalisable, and in an appropriate (real) basis of $W_0$ they take the form
    \begin{equation*}
        J^0_{\varepsilon,T}(t) =
            \begin{pmatrix}
                -4\varepsilon \chi_T(t) & 0 \\
                0 & 4\varepsilon \chi_T(t))
            \end{pmatrix} \cdot
    \end{equation*}
    Let us consider a non-zero solution of the scalar ODEs
    \begin{equation*}
        -f^{\prime\prime}(t) - (3 \kappa_T^\prime(t) +\rho_{\varepsilon,T}(t))f^\prime(t)  \pm 4 \varepsilon \chi_T(t) f(t) = 0 .
    \end{equation*}
    On the interval $(-T,-T+1/2]$ the solutions are of the form
    \begin{equation*}
        f(t) = C_+ + C_- (T+t)^{-2}
    \end{equation*}
    and for the bounds to be satisfied we must have $C_- = 0$. Since $f$ does not vanish identically, we may assume that $C_+ = 1$ after multiplying by an appropriate constant. We deduce that $f(t) = 1$ for any $t \in (-T,-T+1]$. In particular, $f^\prime(-T+1) = 0$, and on the interval $[-T+1,T-1]$ it satisfies
    \begin{equation*}
        f^{\prime\prime}(t) + \rho_{\varepsilon,T}(t) f^\prime(t) = \pm 4 \varepsilon \chi_T(t) f(t) .
    \end{equation*}
    By continuity of the resolvent of the ODE with respect to the parameter $\varepsilon$, if we choose $\varepsilon > 0$ small enough, we will have $f(t) \geq 1/2$ on $[-T+1,T-1]$, and $|f^\prime(t)| < 1$. Since $\rho_{\varepsilon,T} = \varepsilon^2 \chi_T(t) \cdot \psi_{\varepsilon,T}(t)$ for the function $\psi_{\varepsilon,T}(t) = \chi^\prime_T(t)/ (1-\varepsilon^2\chi_T(t))$ which remains uniformly bounded, it follows that for $\varepsilon > 0$ small enough, $f^{\prime\prime}(t)$ will be of the sign of $\pm \chi_T(t) f(t)$. 
    
    In particular, $f^{\prime\prime}$ will either be non-negative or non-positive on $[-T+1,T-1]$ and not identically vanishing, and thus $f^\prime(T-1) \neq 0$. This implies that on the interval $[T-1/2,T)$, $f$ must be of the form
    \begin{equation*}
        f(t) = D_+ + D_- (T-t)^{-2}
    \end{equation*}
    where $D_- \neq 0$, so that does not satisfy $|f(t)| = \mathcal{O}((T-t)^{-1})$.
\end{proof}

\begin{lem}     \label{lem: S4 Hopf example excluding non-Killing Jacobi fields into W2}
    If $\varepsilon > 0$ is small enough, the space of Jacobi fields in $W_2$ satisfying the bound 
    \begin{equation*}
        |u(t)| = \mathcal{O}((T+t)^{-1} + (T-t)^{-1})
    \end{equation*}
    has dimension $3$, and is therefore spanned by the the horizontal lifts of Killing fields on $\Sp^2$.
\end{lem}

\begin{proof}
    In the proof of \cref{prop: Jacobi fields of left-invariant metrics on S3} we have seen that $J^2_{\varepsilon,T}(t)$ leaves invariant $W_{2,1}$ and $W_{2,0} \oplus W_{2,2}$, and that $J^\varphi$ has positive eigenvalues on $W_{2,1}$. We can therefore argue as in the proof of \cref{lem: Finite-dimension restriction} that for $\varepsilon > 0$ small enough, any Jacobi field with these bounds has no component in $W_{2,1}$. It remains to deal with the $W_{2,0} \oplus W_{2,2}$ part.

    As we noted in the proof of \cref{prop: Jacobi fields of left-invariant metrics on S3}, in an appropriate basis of $W_{2,0} \oplus W_{2,2}$, the Jacobi operator takes the form
    \begin{equation*}
        J^2_{\varepsilon,T}(t) = 16 \begin{pmatrix}
            I & \frac{1}{2} \varepsilon \chi_T(t) A \\
        \frac{1}{2} \varepsilon \chi_T(t) A & 0
        \end{pmatrix}
    \end{equation*}
    where
    \begin{equation*}
        A = \begin{pmatrix}
            1  \\
            & -1  \\
            & & -1   \\
            & & & 0 \\
            & & & & 0 \\
            & & & & & 0
        \end{pmatrix} \cdot
    \end{equation*}
    In particular, the operators $J^2_{\varepsilon,T}(t)$ leave invariant the decomposition $W_{2,0} \oplus W_{2,2} = K \oplus H \oplus P_1 \oplus P_2 \oplus P_3$, where
    \begin{itemize}
        \item $K \simeq \R^3$ is the space of horizontal lifts of Killing fields of $\Sp^2$.
        \item $H \simeq \R^3$ have a basis of eigenvectors $J^2_{\varepsilon,T}(t)$ corresponding to the eigenvalues $16$.
        \item $P_i \simeq \R^2$ and restriction of the Jacobi operator read is given by the $2 \times 2$ matrices
        \begin{equation*}
            16\begin{pmatrix}
                1 & \pm \frac{1}{2} \varepsilon \chi_T(t) \\
                \pm \frac{1}{2} \varepsilon \chi_T(t) & 0
            \end{pmatrix} \cdot
        \end{equation*}
    \end{itemize}
    For $\varepsilon > 0$ small enough, a Jacobi field with the appropriate bounds will be a section of $K \oplus P_1 \oplus P_2 \oplus P_3$, and we want to prove that when $T$ is admissible there are no components along $P_i$. On $P_i$, the ODE satisfied by Jacobi fields becomes
    \begin{equation*}
        - U^{\prime\prime}(t) - (3\kappa_T^\prime(t)+\rho_{\varepsilon,T}(t)) U^\prime(t) + 16 \mathrm{e}^{-2\kappa_T(t)} \begin{pmatrix}
                1 & \pm \frac{1}{2} \varepsilon \chi_T(t) \\
                \pm \frac{1}{2} \varepsilon \chi_T(t) & 0
            \end{pmatrix} U(t) = 0
    \end{equation*}
    If this ODE has a solution such that $|U(t)| = \mathcal{O}((T+t)^{-1} + (T-t)^{-1})$, then $U(t)$ must be uniformly bounded, and we deduce that there are constants $A, B, C, D \in \R$ such that $U(t) = A X_\pm(\varepsilon,T+t) + BY_\pm(\varepsilon,T+t)$ for $t \in [-T+1,0]$ and $U(t) = C X_\pm(\varepsilon,T-t) + D Y_\pm(\varepsilon,T-t)$ for $t \in [0,T-1]$, where $X_\pm(\varepsilon,\cdot), Y_\pm(\varepsilon,\cdot) \col [1,\infty) \rightarrow \R^2$ are the unique solutions of the ODE
    \begin{equation*}
        - U^{\prime\prime}(t) - \rho_\varepsilon(t) U^\prime(t) + 16 \begin{pmatrix} 1 & \pm \frac{1}{2} \varepsilon \chi(t) \\ \pm\frac{1}{2} \varepsilon \chi(t) & 0 \end{pmatrix} U(t) = 0
    \end{equation*}
    satisfying the initial conditions
        \begin{equation*}
            X_\pm(\varepsilon,1) = 
            \begin{pmatrix}
                1 \\
                0
            \end{pmatrix}, ~ X^\prime_\pm(\varepsilon,1) = 
            \begin{pmatrix}
                v_0 \\
                0
            \end{pmatrix}, ~~~ \text{and} ~~~ Y_\pm(\varepsilon,1) = 
            \begin{pmatrix}
                0\\
                1
            \end{pmatrix}, ~ Y^\prime_\pm(\varepsilon,1) = 
            \begin{pmatrix}
                0\\
                0
            \end{pmatrix}.
        \end{equation*}
    In particular, it follows that, at $t = 0$, 
    \begin{align*}
        A X_\pm(\varepsilon,T) + BY(\varepsilon,T) & = C X_\pm(\varepsilon,T) + DY_\pm(\varepsilon,T), ~~ \text{and} \\
        A X^\prime_\pm(\varepsilon,T) + BY^\prime_\pm(\varepsilon,T) & = - C X^\prime_\pm(\varepsilon,T) - DY^\prime_\pm(\varepsilon,T) .
    \end{align*}
    When $T$ is admissible, $X^\prime_\pm(\varepsilon,T)$ and $Y^\prime_\pm(\varepsilon,T)$ are independent, whence $A = - C$ and $B = -D$. Therefore, it follows that $A X_\pm(\varepsilon,T) + B Y_\pm(\varepsilon,T) = 0$ for $\varepsilon > 0$ small enough. On the other hand, $X_\pm(0,T)$ and $Y_\pm(0,T)$ are linearly independent, and so must be $X_\pm(\varepsilon,T)$ and $Y_\pm(\varepsilon,T)$ when $\varepsilon > 0$ is small enough, so that $A = B = 0$.

    This proves that if $T$ is admissible and $\varepsilon > 0$ is small enough, any Jacobi field $u$ with the bounds $|u| = \mathcal{O}((T+t)^{-1} + (T-t)^{-1})$ must satisfy $u(t) \in K$ for all $t \in (-T,T)$. On $K$, the Jacobi equation reduces to
    \begin{equation*}
        -U^{\prime\prime}(t) - (3 \kappa^\prime_T(t)+\rho_{\varepsilon,T}(t)) U^\prime(t) = 0
    \end{equation*}
    whose solutions are of the form
    \begin{equation*}
        U(t) = U_+ \oplus (T+t)^{-2} U_-
    \end{equation*}
    on $(-T,-T+1/2]$, whence the only solutions with appropriate bounds are constant maps $(-T,T) \rightarrow W_0$. These are precisely the horizontal lifts of the Killing fields of $\Sp^2$.
\end{proof}

It remains to show that there are admissible values in the interval $[2,\infty)$. In fact we have:

\begin{prop}        \label{prop: Existence of admissible values}
    There exists $T_0 \in [2,\infty)$ such that the non-admissible values of $T$ are isolated in $[T_0,\infty)$. In particular, the admissible values of $T$ are dense in $[T_0,\infty)$.
\end{prop}

The proof relies on the following lemma:

\begin{lem}     \label{lem: Point (ii) of admissible values}
    Let us fix $v > 0$, $\sigma \in \R \backslash \{0\}$ and a smooth function $\psi(\varepsilon,t) \col [0,1) \times [0,\infty) \rightarrow \R$ with all derivatives uniformly bounded and such that 
    \begin{equation*}
        \psi(0,t) = \frac{\partial \psi}{\partial \varepsilon}(0,t) = 0 .
    \end{equation*}
    Consider $X(\varepsilon,\cdot), Y(\varepsilon,\cdot) \col [0,\infty) \rightarrow \R^2$, the unique solutions of the ODE
    \begin{equation*}
        - U^{\prime\prime}(t) - \psi(\varepsilon,t) U^\prime(t) + \begin{pmatrix} 1 & \sigma \varepsilon \chi(t) \\ \sigma \varepsilon \chi(t) & 0 \end{pmatrix} U(t) = 0
    \end{equation*}
    satisfying the initial conditions
    \begin{equation*}
        X_\pm(\varepsilon,1) = 
        \begin{pmatrix}
            1 \\
            0
        \end{pmatrix}, ~ X^\prime_\pm(\varepsilon,1) = 
        \begin{pmatrix}
            v \\
            0
        \end{pmatrix}, ~~~ \text{and} ~~~ Y_\pm(\varepsilon,1) = 
        \begin{pmatrix}
            0\\
            1
        \end{pmatrix}, ~ Y^\prime_\pm(\varepsilon,1) = 
        \begin{pmatrix}
            0\\
            0
        \end{pmatrix}.
    \end{equation*}
    Then if $T \geq 0$ is large enough, there exists $\varepsilon(T) > 0$ such that for all $\varepsilon \in (0,\varepsilon(T))$, $X^\prime(\varepsilon,T)$ and $Y^\prime(\varepsilon,T)$ are linearly independent in $\R^2$.
\end{lem}

\begin{proof}
    Since the resolvent of the ODE is a smooth function of $(\varepsilon,t) \in [0,1) \times [0,\infty)$, there are expansions
    \begin{align*}
        X(\varepsilon,t) & = X_0(t) + \varepsilon X_1(t) + \frac{\varepsilon^2}{2} X_2(t) + \varepsilon^3X^{(3)}(\varepsilon,t) , \\
        Y(\varepsilon,t) & = Y_0(t) + \varepsilon Y_2(t) + \frac{\varepsilon^2}{2} Y_2(t) + \varepsilon^3Y^{(3)}(\varepsilon,t)
    \end{align*}
    where $X^{(3)}(\varepsilon,t)$ and $Y^{(3)}(\varepsilon,t)$ are smooth functions of $(\varepsilon,t) \in [0,1) \times [0,\infty)$. Moreover, 
    \begin{equation*}
        X_1(0) = X_1^\prime(0) = X_2(0) = X_2^\prime(0) = \begin{pmatrix}0 \\ 0\end{pmatrix}, ~~~ Y_1(0) = Y_1^\prime(0) = Y_2(0) = Y_2^\prime(0) = \begin{pmatrix}0 \\ 0\end{pmatrix}
    \end{equation*}
    and
    \begin{equation*}
        X_0(t) = \begin{pmatrix}
            \cosh(t) + v \sinh(t) \\
            0
        \end{pmatrix}, 
        ~~
        Y_0(t) = \begin{pmatrix}
            0 \\ 1
        \end{pmatrix} \cdot
    \end{equation*}
    The next orders of the expansion may be determined as solutions of an ODE, by differentiating the original ODE with respect to $\varepsilon$. For $X_1(t)$, we obtain the ODE
    \begin{equation*}
        - X_1^{\prime\prime}(t) + \begin{pmatrix} 1 & 0 \\ 0 & 0\end{pmatrix}X_1(t) + \sigma \chi(t) \begin{pmatrix} 0 & 1 \\ 1 & 0\end{pmatrix} X_0(t) = 0 .
    \end{equation*}
    The unique solution of the ODE, with the correct initial conditions, is
    \begin{equation*}
        X_1(t) = \begin{pmatrix}
            0 \\ \sigma \Psi_v(t)
        \end{pmatrix}, ~~~ \Psi_v(t) = \int_0^t \int_0^s \chi(u)(\cosh(u) + v \sinh(u)) \mathrm{d}u\mathrm{d}s .
    \end{equation*}
    Similarly, for $Y_1$ we have the equation 
    \begin{equation*}
        - Y_1^{\prime\prime}(t) + \begin{pmatrix} 1 & 0 \\ 0 & 0\end{pmatrix}Y_1(t) + \sigma \chi(t) \begin{pmatrix} 0 & 1 \\ 1 & 0\end{pmatrix} Y_0(t) = 0 .
    \end{equation*}
    If we denote by $\Phi(t)$ the unique solution of the ODE
    \begin{equation*}
        f^{\prime\prime}(t) = f(t) + \chi(t), ~~~ f(0) = f^\prime(0) = 0,
    \end{equation*}
    then we have
    \begin{equation*}
        Y_1(t) = \begin{pmatrix}
            \sigma \Phi(t) \\ 0
        \end{pmatrix} \cdot
    \end{equation*}
    For our purpose, we only need to determine $Y_2(t)$, as $X_2(t)$ will not be relevant. It is determined by the ODE
    \begin{equation*}
        -Y_2^{\prime\prime}(t) + \frac{\partial^2\psi}{\partial\varepsilon^2}(0,t) Y_0^\prime(t) + \begin{pmatrix} 1 & 0 \\ 0 & 0 \end{pmatrix} Y_2(t) + 2 \sigma \chi(t) \begin{pmatrix} 0 & 1 \\ 1 & 0 \end{pmatrix} Y_1(t) = 0 .
    \end{equation*}
    Since $Y_0^\prime(t) = 0$ for all $t$, we can see that
    \begin{equation*}
        Y_2(t) = \begin{pmatrix}
            0 \\
            2 \sigma^2 \Xi(t)
        \end{pmatrix}, ~~~ \Xi(t) = \int_0^t \int_0^s \chi(u)\Phi(u)  \mathrm{d}u\mathrm{d}s .
    \end{equation*}

    Gathering the previous results, we obtain
    \begin{equation*}
        X(\varepsilon,t) = \begin{pmatrix}
            f_v(t) + \mathcal{O}(\varepsilon^2) \\ \varepsilon \sigma \Psi_v(t) + \mathcal{O}(\varepsilon^2)
        \end{pmatrix}, ~~~ f_v(t) = \cosh(t) + v \sinh(t)
    \end{equation*}
    and 
    \begin{equation*}
        Y(\varepsilon,t) = \begin{pmatrix}
            \varepsilon \sigma \Phi(t) + \mathcal{O}(\varepsilon^3) \\ 1 +\varepsilon^2 \sigma^2 \Xi(t) + \mathcal{O}(\varepsilon^3)
        \end{pmatrix} \cdot
    \end{equation*}
    Now we can compute
    \begin{align*}
        \det \begin{pmatrix} X^\prime(\varepsilon,t) & Y^\prime(\varepsilon,t) \end{pmatrix} & = \det \begin{pmatrix}
            f^\prime_v(t) + \mathcal{O}(\varepsilon^2) & \varepsilon \sigma \Phi^\prime(t) + \mathcal{O}(\varepsilon^2) \\ \varepsilon \sigma \Psi_v^\prime(t) + \mathcal{O}(\varepsilon^2) & \varepsilon^2 \sigma^2 \Xi^\prime(t) + \mathcal{O}(\varepsilon^3) 
        \end{pmatrix} \\
        & = \varepsilon^2 \sigma^2 (f^\prime_v(t) \Xi^\prime(t)-\Psi^\prime_v(t) \Phi^\prime(t)) + \mathcal{O}(\varepsilon^3) .
    \end{align*}
    Therefore, it is enough to prove that the quantity 
    \begin{equation*}
        D_v(t) \coloneqq f^\prime_v(t) \Xi^\prime(t)-\Psi^\prime_v(t) \Phi^\prime(t)
    \end{equation*}
    does not vanish for $t$ large enough. 

    First, let us remark that for $t \geq 1$
    \begin{align*}
        \Psi^\prime_v(t) & = \int_0^t \chi(s) (\cosh(s) + v \sinh(s)) \mathrm{d}s = \sinh(t) + v \cosh(t) + C_v = f^\prime_v(t) + C_v
    \end{align*}
    where 
    \begin{equation*}
        C_v = - \sinh(1) - v \cosh(1) + \int_0^1 \chi(s) (\cosh(s) + v \sinh(s)) \mathrm{d}s .
    \end{equation*}
    Therefore,
    \begin{equation*}
        D_v(t) = f^\prime_v(t) (\Xi^\prime(t)-\Phi^\prime(t)) + C_v \Phi^\prime(t) .
    \end{equation*}
    On the other hand, the quantity $\Xi^\prime(t) - \Phi^\prime(t)$ satisfies the ODE
    \begin{equation*}
        (\Xi^\prime(t)-\Phi^\prime(t))^\prime = \Xi^{\prime\prime}(t) - \Phi^{\prime\prime}(t) = \chi(t) \Phi(t) - \Phi^{\prime\prime}(t)
    \end{equation*}
    and by definition $\Phi^{\prime\prime}(t) = \Phi(t) + \chi(t)$, whence
    \begin{equation*}
        (\Xi^\prime(t)-\Phi^\prime(t))^\prime = (\chi(t)-1)\Phi(t) - \chi(t) .
    \end{equation*}
    For $t \geq 1$, $\chi(t) = 1$ and therefore for $t \geq 1$ we have
    \begin{equation*}
        \Xi^\prime(t)-\Phi^\prime(t) = A - t , ~~~ A = 1 + \Xi^\prime(1) - \Phi^\prime(1) .
    \end{equation*}
    It follows that for $t \geq 1$, 
    \begin{equation*}
        D_v(t) = (A- t) f^\prime_v(t) + C_v \Phi(t) .
    \end{equation*}
    As $t \rightarrow \infty$, $f^\prime_v(t) = \sinh(t) + v \cosh(t) = \frac{1+v}{2} \mathrm{e}^t + \mathcal{O}(\mathrm{e}^{-t})$. On the other hand, since 
    \begin{equation*}
        \Phi^{\prime\prime}(t) = \Phi(t) + \chi(t) \leq \Phi(t) + 1
    \end{equation*}
    and the unique solution of the ODE
    \begin{equation*}
        f^{\prime\prime}(t) = f(t) + 1, ~~~ f(0) = f^\prime(0) = 0
    \end{equation*}
    is $f(t) =\cosh(t)-1$, we have $0 \leq \Phi(t) \leq \cosh(t)-1 = \frac{1}{2}\mathrm{e}^t + \mathcal{O}(\mathrm{e}^{-t})$, and hence
    \begin{equation*}
        D_v(t) = (A-t) \frac{1+v}{2} \mathrm{e}^t + \mathcal{O}(\mathrm{e}^t) \sim -\frac{(1+v)t}{2} \mathrm{e}^t \rightarrow - \infty
    \end{equation*}
    as $t \rightarrow \infty$, which finishes the proof.
\end{proof}

\begin{proof}[Proof of \cref{prop: Existence of admissible values}]
    By Lemma \cref{lem: Point (i) of admissible values}, the values of $T \in [2,\infty)$ such that point (i) fails are isolated. On the other hand, if we introduce the variable $\tau = 4t$, the ODE
    \begin{equation*}
        - U^{\prime\prime}(t) - \rho_\varepsilon(t) U^\prime(t) + 16 \begin{pmatrix} 1 & \pm \frac{1}{2} \varepsilon \chi(t) \\ \pm\frac{1}{2} \varepsilon \chi(t) & 0 \end{pmatrix} U(t) = 0
    \end{equation*}
    becomes
    \begin{equation*}
        - U^{\prime\prime}(\tau) - \frac{1}{4}\rho_\varepsilon(4\tau) U^\prime(\tau) + \begin{pmatrix} 1 & \pm \frac{1}{2} \varepsilon \chi(4\tau) \\ \pm\frac{1}{2} \varepsilon \chi(4\tau) & 0 \end{pmatrix} U(\tau) = 0
    \end{equation*}
    and choose $v = 4v_0$, $\psi_{\varepsilon}(\tau) = \frac{1}{4}\rho_{\varepsilon}(4\tau)$, $\sigma = \pm 1/2$ and swap $\chi(\tau)$ and $\chi(4\tau)$, it follows from \cref{lem: Point (ii) of admissible values} that point (ii) holds for $T$ large enough.
\end{proof}

The following theorem summarises the findings of this section:

\begin{thm}
    For a dense subset of values $T \in [T_0,\infty)$, there exists $\varepsilon(T) > 0$ such that for all $\varepsilon \in (0,\varepsilon(T))$, the conically singular harmonic map $\phi \col (\Sp^4, g_{\varepsilon,T}) \to (\Sp^2,g_{\Sp^2})$ satisfies the assumptions of \cref{thm: main theorem introduction} with respect to the metric $g_{\varepsilon,T} \coloneqq \diff t^2 + \mathrm{e}^{2\kappa_T(t)} \widetilde{g}_{\varepsilon,T}(t)$ on $\Sp^4$.
\end{thm}

    \subsection{A harmonic map $\mathbb{CP}^2 \to \mathbb{CP}^1$ with one conical singularity}      \label{sec: map CP2 to CP1}

Our final example is a harmonic map $\mathbb{CP}^2 \to \mathbb{CP}^1$ with one conical singularity modelled on the Hopf fibration. It is based on the meromorphic function $\phi \col \mathbb{CP}^2 \setminus \{[0:0:1]\} \to \mathbb{CP}^1$ defined by 
\begin{equation*}
    \phi([z_0:z_1:z_2]) = [z_0:z_1] .
\end{equation*}
When both projective spaces are equipped with their respective Fubini--Study metrics, $\phi$ is harmonic due to holomorphicity and conically singular with tangent map the Hopf fibration. The first three conditions of \cref{thm: main theorem introduction} are therefore again satisfied by the discussion in \cref{subsec: Hopf fibration as model map}. However, since the pull-back of all infinitesimal biholomorphic deformations of $\mathbb{CP}^1$ give rise to Jacobi fields of order $\mathcal{O}(1)$ close to the singular point, $\phi$ cannot satisfy the last condition of \cref{thm: main theorem introduction} for the Fubini--Study metric on $\mathbb{CP}^2$. As in the previous examples, we will therefore modify the metric on $\mathbb{CP}^2$ in a way that preserves the harmonicity of $\phi$ while at the same time ruling out any Jacobi field of order $\mathcal{O}(r^{-1})$ which is not the pull-back of a Killing vector field.

For this, recall that we can identify 
\begin{equation*}
    \mathbb{CP}^2 \setminus (\mathbb{CP}^1 \cup \{[0:0:1]\}) \cong (0,\tfrac{\pi}{2}) \times \Sp^3
\end{equation*}
as follows: the group $\mathrm{SU}(2)$ acts on $\mathbb{CP}^2$ by restricting the canonical $\mathrm{SU}(3)$-action to $\mathrm{SU}(2) \subset \mathrm{SU}(3)$. With appropriate conventions, the quotient of this action is simply $\mathbb{CP}^2/\mathrm{SU}(2) = [0,\tfrac{\pi}{2}]$, where the fibre over $0$ is $\mathbb{CP}^1=\{[z_0:z_1:0] \in \mathbb{CP}^2\} \subset \mathbb{CP}^2$ and the fibre over $\frac{\pi}{2}$ is $\{[0:0:1]\}$. Over the inner values, the action of $\mathrm{SU}(2)\cong \Sp^3$ is free, so that $\mathbb{CP}^2 \setminus (\mathbb{CP}^1 \cup \{[0:0:1]\}) \cong (0,\tfrac{\pi}{2}) \times \Sp^3$. Note that on this product $\phi$ is given by $\phi(t,\sigma) = \varphi(\sigma)$ where $\varphi \col \Sp^3 \to \Sp^2 \simeq \mathbb{CP}^1$ is the Hopf fibration. In practice, this means that for any $T > 0$ we can obtain $\mathbb{CP}^2$ as a completion of $(-T,T) \times \Sp^3$, where we add a copy of $\mathbb{CP}^1$ at $t = -T$ by collapsing the circles generated by the action of $\xi = X_3$, and at $t = T$ we add a circle point by collapsing the whole sphere $\Sp^3$. Then the map $\phi(t,\sigma) = \varphi(\sigma)$ extends smoothly on $\mathbb{CP}^2 \setminus \{*\}$ and has a conical singularity modelled on the Hopf fibration along $\{*\}$.

In order to define a convenient family of metrics on $\mathbb{CP}^2$, we will fix as before two smooth non-decreasing functions: $\kappa \col (0,+\infty) \to (-\infty,0]$ such that $\kappa(t) = \log(t)$ if $t \leq \frac{1}{2}$ and $\kappa(t) = 0$ if $t \geq 1$; and $\chi \col \R \to [0,1]$ such that $\chi(t) = 0$ if $t \leq 1$ and $\chi(t) = 1$ if $t \geq 2$. As in the previous section, we also define, for any $T \geq 2$, a cutoff function $\chi_T \col (-T,T) \to [0,1]$ such that $\chi_T(t) = \chi(t+T)$ for $t \in (-T,0]$ and $\chi(t) = \chi(-t) = \chi(T-t)$ for $t \in [0,T)$. Moreover, we shall define two auxiliary functions $\kappa^+_T, \kappa^-_T \col (-T,T) \to (-\infty,0]$ by
\begin{equation*}
    \kappa^+_T(t) = \kappa(T-t), ~~ \kappa^-_T(t) = \kappa(T+t), ~~~~ \forall t \in (-T,T) .
\end{equation*}
With these notations, we define for any $\varepsilon \in [0,1)$ and $T \geq 2$ the metric $g_{\varepsilon,T} \coloneqq \diff t^2 + e^{2\kappa^+_T(t)} \tilde{g}_{\varepsilon,T}(t)$ on $(-T,T) \times \Sp^3$, where $\tilde{g}_{\varepsilon,T}(t)$ is the family of metrics on $\Sp^3$ defined as
\begin{equation*}
    \tilde{g}_{\varepsilon,T}(t) \coloneqq \frac{1}{1-\varepsilon\chi_T(t)} X^1 \otimes X^1 + \frac{1}{1 + \varepsilon\chi_T(t)} X^2 \otimes X^2 + e^{2\kappa^-_T(t)} X^3 \otimes X^3 .
\end{equation*}
It is not difficult to see that the metric $g_{\varepsilon,T}$ can be completed to a smooth metric on $\mathbb{CP}^2$, such that the map $t \col \mathbb{CP}^2 \to [-T,T]$ satisfies $t^{-1}(-T) \simeq \mathbb{CP}^1$ and $t^{-1}(T)$ is a singleton. Moreover, with respect to this metric, the map $\phi \col \mathbb{CP}^2 \to \mathbb{CP}^1 \simeq \Sp^2$ given by $\phi(t,\sigma) = \varphi(\sigma)$ for all $(t,\sigma) \in (-T,T) \times \Sp^3$ (where $\varphi \col \Sp^3 \to \Sp^2$ is the Hopf fibration) extends to a harmonic map on $\mathbb{CP}^2 \setminus \{t^{-1}(T)\}$ with a single conical singularity modelled on the Hopf fibration.

Using \cref{prop: Jacobi for warped product and constant map} and regarding sections $u \in \Gamma(\prescript{\phi}{}{T\Sp^2})$ as complex-valued functions $u \col (-T,T) \times \Sp^3 \rightarrow \C$, we see that the Jacobi operator of $\phi$ takes the form
\begin{equation}
    J^\phi_{g_{\varepsilon,T}} u = - \frac{\partial^2 u}{\partial t^2} - (3 \kappa^\prime_{+,T}(t) + \kappa^\prime_{-,T}(t) + \rho_{\varepsilon,T}(t)) \frac{\partial u}{\partial t} + e^{-2\kappa_{+,T}(t)} J^\varphi_{\varepsilon,T}(t) u
\end{equation}
where $\rho_{\varepsilon,T}(T)$ is as in the previous section (cf. \eqref{eq: rho epsilon T}) and $J^\varphi_{\varepsilon,T}(t)$ is the Jacobi operator of $\tilde{g}_{\varepsilon,T}(t)$, which reads
\begin{equation}
    J^\varphi_{\varepsilon,T}(t) u = e^{-2\kappa_{-,T}(t)} J^\varphi u + \varepsilon \chi_T(t) (\partial^2_{X_1}u - \partial^2_{X_2} u - 4 \overline{u}) + (e^{-2\kappa_{-,T}(t)} -1) (\partial_{X_1}^2 + \partial_{X_2}^2 + 4) u.
\end{equation}
In particular, the Jacobi operator associated with the metric $g_{\varepsilon,T}$ coincide with the one derived in the previous section on the domain $(-T+1,T) \times \Sp^3$, but differs from it on the domain $(-T,-T+1) \times \Sp^3$, reflecting the fact that we are only collapsing the circle generated by the action of $\xi = X_3$ as $t \to -T$, rather than the full sphere $\Sp^3$. 

In order to study the Jacobi fields on $(-T,T) \times \Sp^3$ which extend smoothly to $\mathbb{CP}^2 \setminus \{ *\}$, let us write the Jacobi operator more explicitly in the interval $(-T,-T+1)$. In this region, $\kappa_{+,T}(t) = 0$ and $\chi_T(t) = 0$. Using the identity $\partial_{X_1}^2 + \partial_{X_2}^2 = - \Delta_{\Sp^3} - \partial_\xi^2$ (where as before $\xi = X_3$), we see that $J^\varphi_{\varepsilon,T}(t)$ leaves invariant the eigenspaces $W_{k,\ell}$ on the interval $(-T,-T+\frac{1}{2})$, and the operator $4 + \partial_{X_1}^2 + \partial_{X_2}^2$ has eigenvalue $-\alpha_{k,\ell}$ where
\begin{equation*}
    \alpha_{k,\ell} = k(k+2) -4 - (2\ell-k)^2 = 2(k + 2\ell(k-\ell) - 2) .
\end{equation*}
Moreover, for any function $u \col (-T,-T+1) \to W_{k,\ell}$ the Jacobi equation takes the form
\begin{equation}       \label{eq: ODE with mukl 1}
    - u^{\prime\prime}(t) - \kappa_{-,T}^\prime(t) u^\prime(t) + \left(e^{-2\kappa_{-,T}(t)}\mu_{k,\ell} + \alpha_{k,\ell} \right) u(t) = 0 
\end{equation}
where, since $\lambda_{k,\ell} = k(k-2) + 8 \ell$ is the eigenvalue of $J^\varphi$ acting on $W_{k,\ell}$, one can compute
\begin{equation*}
    \mu_{k,\ell} \coloneqq \lambda_{k,\ell} - \alpha_{k,\ell} = k(k-2)+8\ell + 2(2\ell(\ell-k) - k + 2) = (k-2)^2 + 4\ell(\ell-k+2) \geq 0
\end{equation*}
since $\ell \in \{0,\ldots,k\}$. Moreover, note that $\mu_{k,\ell} = 0$ if and only if $\ell = \frac{k}{2}-1$. 

On the interval $(-T,-T+\frac{1}{2})$, \eqref{eq: ODE with mukl 1} takes the particularly simple form
\begin{equation}       \label{eq: ODE with mukl 2}
    - u^{\prime\prime}(t) - \frac{1}{T+t} u^\prime(t) + \left(\frac{\mu_{k,\ell}}{(T+t)^{-2}} + \alpha_{k,\ell} \right) u(t) = 0 
\end{equation}
In order to solve the ODE in a neighbourhood of $-T$, we can use Frobenius' method (see \cite[Section~4.2]{Teschl2012ODE_and_DynSystems}), which implies that the two fundamental solutions of \eqref{eq: ODE with mukl 2} are of the form 
\begin{align*}
    w_+(t) \coloneqq (T+t)^{\sqrt{\mu_{k,\ell}}} h_+(t) ~~ \textup{and} ~~ w_-(t) \coloneqq (T+t)^{-\sqrt{\mu_{k,\ell}}} h_-(t) ~~~~ \textup{if $2\sqrt{\mu_{k,\ell}} \notin \mathbb{N}_0$}
\end{align*}
and 
\begin{align*}
    w_+(t) \coloneqq (T+t)^{\sqrt{\mu_{k,\ell}}} h_+(t) ~~ \textup{and} ~~ w_-(t) \coloneqq (T+t)^{-\sqrt{\mu_{k,\ell}}} h_-(t) + c \log(T+t) w_+(t) ~~~ \textup{if $2\sqrt{\mu_{k,\ell}} \in \mathbb{N}_0$,}
\end{align*}
where in both cases $h_{\pm}$ are analytic functions at $t=-T$ that satisfy $h_{\pm}(-T)=1$. Moreover, $c\in \mathbb{R}$ is a constant which is non-zero whenever $\mu_{k,\ell} = 0$. This shows that for every eigenspace $W_{k,\ell}$ and every $U \in W_{k,\ell} \setminus \{0\}$ there exists (up to scaling) precisely one Jacobi field $u \col (-T,-T+\frac{1}{2}) \to W_{k,\ell}$ colinear to $U$ that remains bounded near $t=-T$. In particular, only one of the two solutions of \eqref{eq: ODE with mukl 2} (namely, $w_+(t)$) has a chance of extending smoothly over the $\mathbb{CP}^1$-orbit at $t=-T$.

Similarly to the previous sections, we want to define an admissible range of values for $T$ so that, at least for small enough values of $\varepsilon$, we can ensure that any Jacobi field associated with $g_{\varepsilon,T}$ which smoothly extends across the $\mathbb{CP}^1$-orbit at $t = -T$ and satisfies the bound $\mathcal{O}((T-t)^{-1})$ must be the pull-back of a Killing field. As before, particular attention must be paid to the eigenspaces $W_{2,0}$ and $W_{2,2}$. For $(k,\ell) = (2,0)$ we compute that $\alpha_{2,0} = 0$ and $\mu_{2,0} = \lambda_{2,0} = 0$, so that \eqref{eq: ODE with mukl 1} reduces to the simple scalar ODE
\begin{equation}
    - f^{\prime\prime}(t) - \kappa^\prime_{-,T}(t) f^\prime(t) = 0 .
\end{equation}
This implies that, up to scaling, the only solution of the ODE on the interval $(-T,-T+1)$ extending across $t = -T$ is $f(t) \equiv 1$.

On the other hand, for $(k,\ell) = (2,2)$, we have $\alpha_{2,2} = 0$ and $\mu_{2,2} = \lambda_{2,2} = 16$, so that \eqref{eq: ODE with mukl 1} and \eqref{eq: ODE with mukl 2} explicitly reduce to the scalar ODEs
\begin{equation}
    - f^{\prime\prime}(t) - \kappa^\prime_{-,T}(t) f^\prime(t) + 16 e^{-2\kappa_{-,T}(t)}\mu_{k,\ell}f(t) = 0 
\end{equation}  \label{eq: ODE for k,l equal to 2}
and
\begin{equation}
    - f^{\prime\prime}(t) - \frac{1}{T+t} f^\prime(t) + \frac{16}{(T+t)^{-2}} f(t) = 0 .
\end{equation}
By the previous discussion, the only fundamental solution of the previous ODE which smoothly extends across $t = -T$ is $f(t) = (T+t)^4$, which satisfies $f(t), f^\prime(t) > 0$ on the interval $(-T,-T+\frac{1}{2})$. Thus point (i) of \cref{lem: Preparatory ODE} implies that the unique solution of \eqref{eq: ODE for k,l equal to 2} such that $f(-T+1) = 1$ and $f$ extends smoothly across $t = -T$ has $v_1 \coloneqq f^\prime(-T+1) > 0$.

The above discussion leads us to the following definition:

\begin{Def}     \label{Def: Admissible values 2}
    We shall call $T \in [2,\infty)$ \emph{admissible} if the following properties are satisfied.
    \begin{enumerate}
        \item The unique solution $f \col (-T,T) \to \R$ of the ODE
        \begin{equation*}
            - f^{\prime\prime}(t) - (3 \kappa^\prime_{+,T}(t) + \kappa^\prime_{-,T}(t)) f^\prime(t) + e^{-2\kappa_{+,T}(t)} (e^{-2\kappa_{-,T}}(t) - 2) f(t) = 0
        \end{equation*}
        satisfying $f(t) = \mathcal{O}(T-t)$ near $t = T$ and $f(t) = \mathcal{O}(1)$ near $t = -T$ is $f(t) \equiv 0$.

        \item The unique solutions $X_\pm(\varepsilon,\cdot), Y_\pm(\varepsilon,\cdot), Z_\pm(\varepsilon,\cdot) \col [1,\infty) \rightarrow \R^2$ of the ODE
            \begin{equation*}
                - U^{\prime\prime}(t) - \rho_\varepsilon(t) U^\prime(t) + 16 \begin{pmatrix} 1 & \pm \frac{1}{2} \varepsilon \chi(t) \\ \pm \varepsilon \chi(t) & 0 \end{pmatrix} U(t) = 0
            \end{equation*}
            satisfying the initial conditions
            \begin{align*}
                X_\pm(\varepsilon,1) & = 
                \begin{pmatrix}
                    1 \\
                    0
                \end{pmatrix}, ~ X^\prime_\pm(\varepsilon,1) = 
                \begin{pmatrix}
                    v_0 \\
                    0
                \end{pmatrix},
                ~~~ Y_\pm(\varepsilon,1) = 
                \begin{pmatrix}
                    0\\
                    1
                \end{pmatrix}, ~ Y^\prime_\pm(\varepsilon,1) = 
                \begin{pmatrix}
                    0\\
                    0
                \end{pmatrix} \\
                & ~~~~~~~  \text{and} ~~~ Z_\pm(\varepsilon,1) =
                \begin{pmatrix}
                    1 \\
                    0
                \end{pmatrix}, ~~Z^\prime_\pm(\varepsilon,1) =
                \begin{pmatrix}
                    v_1 \\
                    0
                \end{pmatrix} ,
            \end{align*}
            where $v_0 > 0$ is as in \cref{Def: Admissible values}, have the property that for $\varepsilon > 0$ small enough, 
            \begin{equation*}
                \begin{pmatrix}
                    X_\pm(\varepsilon,1) \\
                    - X_\pm^\prime(\varepsilon,1)
                \end{pmatrix}, ~
                \begin{pmatrix}
                    Z_\pm(\varepsilon,1) \\
                    Z_\pm^\prime(\varepsilon,1)
                \end{pmatrix}, ~
                \begin{pmatrix}
                    Y_\pm(\varepsilon,1) \\
                    Y_\pm^\prime(\varepsilon,1)
                \end{pmatrix}, ~ \text{and} ~
                \begin{pmatrix}
                    Y_\pm(\varepsilon,1) \\
                    - Y_\pm^\prime(\varepsilon,1)
                \end{pmatrix}
            \end{equation*}
            are linearly independent vectors in $\R^4$.
    \end{enumerate}
\end{Def}

\begin{prop}
    Assume that $T\geq 5$ is admissible. Then there exists $\varepsilon(T)>0$ such that for any $\varepsilon \in (0,\varepsilon(T))$, the only Jacobi fields of the harmonic map $\phi \col \mathbb{CP}^2 \to \mathbb{CP}^1$ for $g_{\varepsilon,T}(t)$ that satisfy the bound $\vert u \vert= \mathcal{O}((T-t)^{-1})$ are the pull-back of the Killing vector fields on $\mathbb{CP}^1$.
\end{prop}

\begin{proof}
    The argument is very similar to the proof of \cref{thm: S^4 to S^2 example}, so we will only explain how to adapt the arguments of the previous section.

    The first step is to show that, when $\varepsilon = 0$, the only Jacobi fields with respect $g_{0,T}$ which satisfy $|u| = \mathcal{O}((T-t)^{-1}$ must be functions $u \colon (-T,T) \to W_0 \oplus W_1 \oplus W_2$, so that the same continuity argument as in the proof of \cref{lem: Finite-dimension restriction} ensures that this condition still holds whenever $\varepsilon > 0$ is small enough. To see this, we can simply apply \cref{lem: Preparatory ODE} since $\alpha_{k,\ell} > 0$ for any $k \geq 3$ and $\ell \in \{0,\ldots,k\}$. In fact, since $J^\varphi_{\varepsilon,T}$ leave invariant the subspaces $W_{2,0} \oplus W_{2,2}$ and $W_{2,1}$ separately and $\alpha_{2,1} = \mu_{2,1} = 4 > 0$ we can also exclude the $W_{2,1}$-component in this way.

    Along $W_1$, remark first that for $\varepsilon = 0$, we can treat the components $W_{1,0}$ and $W_{1,1}$ separately. For $W_{1,1}$, the Jacobi field equation reduces to the scalar ODE
    \begin{equation*}
        - f^{\prime\prime}(t) - (3 \kappa^\prime_{+,T}(t) + \kappa^\prime_{-,T}(t)) f^\prime(t) + e^{-2\kappa_{+,T}(t)} (9e^{-2\kappa_{-,T}}(t) - 2) f(t) = 0
    \end{equation*}
    and we can rule out such solutions again using \cref{lem: Preparatory ODE} since $e^{-2\kappa_{+,T}(t)} (9e^{-2\kappa_{-,T}}(t) - 2) > 0$ on $(-T,T)$, since any non-trivial solution which remains bounded near $t = -T$ must satisfy $f(t)f^\prime(t) > 0$ in this region. On the other hand, on $W_{1,0}$ the Jacobi field equation reduces to the scalar ODE
    \begin{equation*}
        - f^{\prime\prime}(t) - (3 \kappa^\prime_{+,T}(t) + \kappa^\prime_{-,T}(t)) f^\prime(t) + e^{-2\kappa_{+,T}(t)} (e^{-2\kappa_{-,T}}(t) - 2) f(t) = 0
    \end{equation*}
    so that the problematic Jacobi fields along $W_{1,0}$ are ruled out by condition 1 when $\varepsilon = 0$. For small value of $\varepsilon > 0$, we can use the same stability argument as in the proof of \cref{lem: rule out W_1}.

    Finally, for the $W_0$-component, we note that for $\varepsilon = 0$ the Jacobi equation reduces to
    \begin{equation*}
        - f^{\prime\prime}(t) - (3 \kappa^\prime_{+,T}(t) + \kappa^\prime_{-,T}(t))f^\prime(t) + 4e^{-2\kappa_{+,T}(t)} (e^{-2\kappa_{-,T}(t)}-1) f(t) = 0 .
    \end{equation*}
    If the non-trivial solutions which are constant on the interval $(-T+1,T)$ do not remain bounded near $t = -T$, then this already rules out the problematic Jacobi fields (and we can use a stability argument for $\varepsilon > 0$ small enough). Otherwise, we can use the same argument as in the proof of \cref{lem: rule out W_0} to rule out problematic Jacobi fields for $\varepsilon > 0$ small enough, since we can make sure that any Jacobi field which is constant near $t = T$ will have non-trivial velocity at $t = -T + 1$, and therefore cannot remain bounded near $t = -T$.

    Finally, the matching problem for Jacobi fields along $W_{2,0} \oplus W_{2,2}$ is similar to the previous example, except that we loose the symmetry between $t$ and $-t$ of the problem, which implies that the velocities $v_0$ and $v_1$ have no reason to be equal. Instead, we want to prevent a solution of the form $A Z_\pm(\varepsilon,t-T) + B Y_\pm(\varepsilon,t-T)$ (coming from the left) from matching at $t = 0$ with a solution of the form $CX_\pm(\varepsilon,T-t) + DY_\pm(\varepsilon,T-t)$ (coming from the right). As in the proof of \cref{lem: S4 Hopf example excluding non-Killing Jacobi fields into W2}, condition 2 in the definition of admissibility is precisely what is needed to prevent such problematic matchings when $\varepsilon > 0$ is small enough. Thus we deduce in the same way that the only Jacobi fields in $W_{2,0} \oplus W_{2,2}$ are lifted from the Killing fields of $\mathbb{CP}^1$.
\end{proof}

\begin{lem}
    There exists $T_0 \geq 2$ such that the set of admissible values is a dense subset of $[T_0,\infty)$
\end{lem}

\begin{proof}
    As in the proof of \cref{lem: Point (i) of admissible values} in the previous section, it is easy to see that the first condition can only fail periodically. As for the second condition, the idea is again to compute the relevant $4 \times 4$ determinant at leading order in $\varepsilon > 0$. After a change of variables to get read of the factor of $16$ in the ODE (cf. the proof of \cref{prop: Existence of admissible values}) and some elementary transformations on the determinant, we can re-use the computations made in the proof of \cref{lem: Point (ii) of admissible values}: we want to prove that for $t$ large enough, the determinant
    \begin{equation*}
        \det 
        \begin{pmatrix}
            f_{\tilde{v}_0}(t) + \mathcal{O}(\varepsilon^2) & f_{\tilde{v}_1}(t) + \mathcal{O}(\varepsilon^2) & \varepsilon \sigma \Phi(t) + \mathcal{O}(\varepsilon^3) & 0 \\ 
            \varepsilon \sigma \Psi_{\tilde{v}_0}(t) + \mathcal{O}(\varepsilon^2) &  \varepsilon \sigma \Psi_{\tilde{v}_1}(t) + \mathcal{O}(\varepsilon^2) & 1 +\varepsilon^2 \sigma^2 \Xi(t) + \mathcal{O}(\varepsilon^3) & 0\\
            - f_{\tilde{v}_0}^\prime(t) + \mathcal{O}(\varepsilon^2) & f_{\tilde{v}_1}^\prime(t) + \mathcal{O}(\varepsilon^2) & 0 & \varepsilon \sigma \Phi^\prime(t) + \mathcal{O}(\varepsilon^3) \\ 
            - \varepsilon \sigma \Psi_{\tilde{v}_0}^\prime(t) + \mathcal{O}(\varepsilon^2) & \varepsilon \sigma \Psi_{\tilde{v}_1}^\prime(t) + \mathcal{O}(\varepsilon^2) & 0 & \varepsilon^2 \sigma^2 \Xi^\prime(t) + \mathcal{O}(\varepsilon^3)
        \end{pmatrix}
    \end{equation*}
    is non-vanishing, where $\tilde{v}_0 = 4v_0 > 0$, $\tilde{v}_1 = 4v_1 > 0$, and all the functions are as in the proof of \cref{lem: Point (ii) of admissible values}. It is not difficult to compute that this determinant is of the form
    \begin{multline*}
        \varepsilon^2\sigma^2 (f_{\tilde{v}_0}(t) (f^\prime_{\tilde{v}_1}(t)\Xi^\prime(t) - \Psi_{\tilde{v}_1}^\prime(t)\Phi^\prime(t)) + f_{\tilde{v}_1}(t) (f^\prime_{\tilde{v}_0}(t)\Xi^\prime(t) - \Psi_{\tilde{v}_0}^\prime(t)\Phi^\prime(t))) + \mathcal{O}(\varepsilon^3) \\ 
            = \varepsilon^2\sigma^2 (f_{\tilde{v}_0}(t) D_{\tilde{v}_1}(t) + f_{\tilde{v}_1}(t) D_{\tilde{v}_0}(t)) + \mathcal{O}(\varepsilon^3)
    \end{multline*}
    and from the computations made in the proof of \cref{lem: Point (ii) of admissible values} we see that this quantity $f_{\tilde{v}_0}(t) D_{\tilde{v}_1}(t) + f_{\tilde{v}_1}(t) D_{\tilde{v}_0}(t)$ tends to $-\infty$ as $t \to \infty$. 
\end{proof}

From this we finally obtain:

\begin{thm}
    For a dense subset of values $T \in [T_0,\infty)$, there exists $\varepsilon(T) > 0$ such that for all $\varepsilon \in (0,\varepsilon(T))$, the conically singular harmonic map $\phi \col (\mathbb{CP}^2, g_{\varepsilon,T}) \to (\mathbb{CP}^1, g_{\mathrm{FS}})$ satisfies the assumptions of \cref{thm: main theorem introduction} with respect to the metric $g_{\varepsilon,T} \coloneqq \diff t^2 + \mathrm{e}^{2\kappa_{+,T}(t)} \widetilde{g}_{\varepsilon,T}(t)$ on $\mathbb{CP}^2$.
\end{thm}

%% file: appendix_stationary.tex
    \section{Conically singular harmonic maps as stationary maps}       \label{app: cs as stationary maps}

The goal of this short appendix is to justify the claim made in \cref{subsubsec: cs harmonic maps} that if $\dim(M) \geq 4$, then for any conically singular map $\phi \col M \setminus S \to N$, the conditions of weak harmonicity, stationarity and vanishing of the tension field over the smooth locus are equivalent (\cref{prop: equivalent conditions cs harmonic}). This will be a consequence of \cref{lem: m geq 3} and \cref{lem: m geq 4} below.

Let us fix an isometric embedding $\iota \col (N,h) \hookrightarrow (\R^K,g_{\R^K})$ for some large enough $K \in \N$. Then we may give an extrinsic characterisation of conically singular maps: a map $\phi \col M \setminus S \to N$ is conically singular with rate $\mu > 0$ if an only if in an adapted set of charts $\{\Upsilon_s \col \Ball_R \to M\}_{s \in S}$, we can write
\begin{equation*}
    \iota \circ \phi \circ \Upsilon_s(x) = \varphi_s(\tfrac{x}{|x|}) + u_s(x), ~~~ \forall x \in \Ball_R \setminus \{0\}
\end{equation*}
where $\varphi_s \col \Sp^{m-1} \to \iota(N) \subset \R^K$ is a smooth map and the vector-valued function $u_s \col \Ball_R \setminus \{0\} \to \R^K$ satisfies $|\partial^a u_s| = \mathcal{O}(r^{\mu-|a|})$ as $r = |x| \to 0$, for all multi-indices $a = (a_1,\ldots,a_m)$ (where $\partial^au_s$ is a short-hand for $\partial_{x_1}^{a_1}\cdots \partial^{a_m}_{x_m} u_s$). In addition, remark that the tangent map $\phi_s = \varphi_s(x/|x|)$ satisfies $|\partial^a\phi_s| = \mathcal{O}(r^{-|a|})$ for any multi-index $a = (a_1,\ldots,a_m)$.

This extrinsic characterisation shows that when $m = \dim(M) \geq 3$, a conically singular map $\phi \col M \setminus S \to N$ may be seen as a Sobolev map from $M$ to $N$. Recall that, using the isometric embedding $\iota \col N \rightarrow \R^K$, the Lebesgue space $L^2(M,N)$ may be defined as
\begin{equation*}
    L^2(M,N) = \{ f \in L^2(M;\R^K) \mid f(p) \in \iota(N) ~ \text{for a.e.} ~ p \in M \} .
\end{equation*}
Similarly, we can define the Sobolev space $W^{1,2}(M,N)$ as
\begin{equation*}
    W^{1,2}(M,N) = \{ f \in W^{1,2}(M;\R^K) \mid f(p) \in \iota(N) ~ \text{for a.e.} ~ p \in M \} .
\end{equation*}

\begin{prop}
    Suppose that $m = \dim(M) \geq 3$, and let $\phi \col M \backslash S \rightarrow N$ be a conically singular map. Then $\phi \in W^{1,2}(M,N)$. 
\end{prop}

\begin{proof}
    Let us fix an adapted set of charts $\{\Upsilon_s \col \Ball_R \to M\}$, and identify the maps $\phi \col M \setminus S \to N$ and $\iota \circ \phi \col M \setminus S \to \iota(N) \subset \R^K$. Thus for any $s \in S$, $\phi \circ \Upsilon_s = \phi_s + u_s$ where $\phi_s = \varphi_s(x/|x|)$ for a smooth map $\varphi_s \col \Sp^{m-1} \to \iota(N)$ and $u_s$ satisfies $|\partial^a u_s| = \mathcal{O}(r^{\mu-|a|})$ for any multi-index $a = (a_1,\ldots,a_m)$.
    
    Since $N$ is compact, the map $\phi \col M \setminus S \rightarrow \iota(N) \subset \R^K$ is bounded and thus $\phi \in L^2(M,N)$. Moreover, for any $s \in S$ and $j = 1,\ldots, m$, $|\partial_{x_j}(\phi \circ \Upsilon_s)| = \mathcal{O}(r^{-1})$ since $|\partial_{x_j} \phi_s| = \mathcal{O}(r^{-1})$ and $|\partial_{x_j} u_s | = \mathcal{O}(r^{\mu-1})$. Since $m \geq 3$, it follows that $\diff\phi_{|M \setminus S}$ is $L^2$. It remains to see that $\diff\phi_{|M\setminus S}$ coincides with the weak (distributional) derivative of $\phi$. Since $\phi$ is smooth outside of $S$, this amounts to checking that for any $\R^K$-valued function $v \col \Ball_R \to \R^m$ supported in the interior of $\Ball_R$, any $s \in S$ and any $1 \leq j \leq K$, 
    \begin{equation*}
        \int_{\Ball_R} \langle \nabla(\phi^j_s + u^j_s), v \rangle \diff x = - \int_{\Ball_R}  (\phi^j_s + u^j_s) \divergence(v) \diff x
    \end{equation*}
    where the inner product $\langle \cdot, \cdot \rangle$ is taken with respect to the standard inner product on $\R^m$, $\nabla$ is the Euclidean gradient and $\divergence(v)$ the Euclidean divergence. Note that the integrands on both sides of the equality are $L^1$, and therefore these are definite integrals. Using the divergence theorem, we obtain
    \begin{align*}
        \int_{\Ball_R} \langle \nabla(\phi^j_s + u^j_s), v \rangle \diff x & = \lim_{\varepsilon \rightarrow 0} \int_{\Ball_R \setminus \Ball_\varepsilon} \langle \nabla(\phi^j_s + u^j_s), v \rangle \diff x \\
        & = \lim_{\varepsilon \rightarrow 0} - \int_{\Ball_R \setminus \Ball_\varepsilon}  (\phi^j_s + u^j_s) \divergence(v) \diff x + \int_{\Sp^{m-1}_\varepsilon} (\phi^j_s + u^j_s) \langle v, \nu_{\Sp^{m-1}_\varepsilon} \rangle \vol_{\Sp^{m-1}_\varepsilon} \\
        & = - \int_{\Ball_R}  (\phi^j_s + u^j_s) \divergence(v) \diff x + \lim_{\varepsilon \to 0} \int_{\Sp^{m-1}_\varepsilon} (\phi^j_s + u^j_s) \langle v, \nu_{\Sp^{m-1}_\varepsilon} \rangle \vol_{\Sp^{m-1}_\varepsilon}
    \end{align*}
    where $\Sp^{m-1}_\varepsilon = \partial \Ball_\varepsilon$ is the Euclidean sphere of radius $\varepsilon$, $\vol_{\Sp^{m-1}_\varepsilon}$ its standard volume form and $\nu_{\Sp^{m-1}_\varepsilon}$ the outwards pointing normal. Since $\phi^j_s + u^j_s$ and $\langle v , \nu_{\Sp^{m-1}_\varepsilon} \rangle$ are uniformly bounded functions and $\vol(\Sp^{m-1}_\varepsilon) \to 0$ as $\varepsilon \to 0$, this proves our claim.
\end{proof}

Let us now fix a Riemannian metric $g$ on $M$. For any map $\phi \in W^{1,2}(M,N)$, we may define its Dirichlet energy
\begin{equation*}
    E(\phi) = \frac{1}{2} \int_M |\diff\phi|^2_{g,g_{\R^K}} \vol_{g_M} .
\end{equation*}
If $u$ is continuously differentiable, this coincides with the intrinsic notion of energy as defined in \cref{subsubsec: definitions and notations}. The Euler--Lagrange equations of the (extrinsic) Dirichlet energy functional define an $\R^K$-valued PDE which reads (cf. \cite[Sec. 2]{helein2007harmonic} for instance)
\begin{equation}    \label{eq:weakel}
    \Delta_g \phi = \tr_g(\II_{\phi}(\diff\phi,\diff\phi)) 
\end{equation}
where $\II$ is the second fundamental form of the embedding $\iota \col N \hookrightarrow \R^K$ and $\Delta_g$ is the Laplacian operator with geometers' conventions (such that it has negative principal symbol). Note that if $\phi \in W^{1,2}(M,N)$ then $\II_{\phi}(\diff\phi,\diff\phi) \in L^1(M,\R^K)$, and therefore it represents a well-defined $\R^K$-valued distribution. Moreover, if $\phi$ is $C^2$ then $-\Delta_g \phi + \tr_g(\II_{\phi}(\diff\phi,\diff\phi))$ is a section of $\prescript{\phi}{}{TN} \subset M \times \R^K$, which coincides with the tension $\tau(\phi)$. By definition, a map $\phi \in W^{1,2}(M,N)$ is \emph{weakly harmonic} if it is a (distributional) solution of the above equation.

\begin{prop}        \label{lem: m geq 3}
    Suppose that $m = \dim(M) \geq 3$ and let $\phi \col M \setminus S \to N$ be a map with conical singularities. Then $\phi$ is weakly harmonic if and only if the tension $\tau(\phi)$ on $M \setminus S$.
\end{prop}

\begin{proof}
    If $\phi$ is weakly harmonic, the tension field $\tau(\phi)$ must vanish away from the singular locus. To prove the converse statement, remark that since $\phi$ is smooth and harmonic outside of the singular locus $S$, and $\phi \in W^{1,2}(M,N)$ globally, it is enough to prove that the Laplacian $\Delta_g \phi$ (seen as an $\R^K$-valued disribution on $M$) is $L^1$. Outside of $S$, $\Delta_g \phi = \tr_{g} \II_{\phi}(\diff\phi,\diff\phi)$ so that the function $|\Delta_g \phi|$ is $L^1$ on $M \setminus S$. Therefore it only remains to check for any smooth map $u \col M \rightarrow \R^K$ supported inside a small geodesic ball $B_\delta(s)$ centred at $s \in S$, the following identity holds
    \begin{equation*}
        \int_{B_\delta(s)} \langle \diff\phi, \diff u \rangle_{g,g_{\R^K}} \vol_{g} = \int_{B_\delta(s) \setminus \{s\}} \langle \Delta_g \phi, u \rangle_{g_{\R^K}} \vol_{g} .
    \end{equation*}
    As in the previous proof, since the integrands on both sides are $L^1$ we can integrate by parts and this boils down to proving that $\lim_{\varepsilon \rightarrow 0} \int_{\partial B_\varepsilon(s)} u^j (*_{g}\diff\phi^j)= 0$ (where we take coordinates $\{y^j\}_{j=1,\ldots,K}$ on $\R^K$). This identity holds since $|\diff\phi^j|= \mathcal{O}(r^{-1})$ near $s$ ($r$ being the distance to $s$), and $m-1 \geq 2$ so that the volume of $\partial B_\varepsilon(s)$ behaves as $\mathcal{O}(\varepsilon^{m-1}) \leq \mathcal{O}(\varepsilon^2)$.
\end{proof}

When $m = \dim(M) \geq 4$, conically singular maps $\phi \col M \setminus S \to N$ which are weakly harmonic (equivalently, whose tension field vanishes on the regular locus by the previous lemma) satisfy a stronger notion of harmonicity; namely, they are \emph{stationary}. That is, $\frac{\diff}{\diff t} E(\phi \circ \exp(tX)) = 0$ for any smooth vector field $X \in \mathfrak{X}(M)$.

\begin{lem}     \label{lem: m geq 4}
    Suppose that $m = \dim(M) \geq 4$. Then any conically singular map $\phi \col M \setminus S \to N$ which is weakly harmonic is stationary.
\end{lem}

\begin{proof}
    Since $\phi$ is weakly harmonic, stationarity is equivalent to the property that the tensor $T = \tfrac{1}{2} |\diff\phi|^2_{g,h} g - \phi^*h$ (which is in $L^1(\mathrm{Sym}^2T^*M)$) be covariantly divergence-free, in the distributional sense (c.f. \cite[\S3.1]{helein2007harmonic} for instance). This property holds outside of $S$ since $\phi$ is smooth on $M \setminus S$. On the other hand, in local adapted coordinates the tensor $T$ takes the form
    \begin{equation*}
        T_{ab} = \frac{1}{2} g^{k\ell} \langle \frac{\partial \phi}{\partial x^k}, \frac{\partial \phi}{\partial x^\ell} \rangle g_{ab} - \langle \frac{\partial \phi}{\partial x^a}, \frac{\partial \phi}{\partial x^b} \rangle
    \end{equation*}
    where the brackets $\langle \cdot , \cdot \rangle$ are defined with respect to the standard inner product of $\R^K$, using the isometric embedding $\iota \col N \hookrightarrow \R^K$. This expression shows that $|T|_{g,g_{\R^K}} = \mathcal{O}(r^{-2})$ and $|\nabla T|_{g,g_{\R^K}} = \mathcal{O}(r^{-3})$ near a singularity $s \in S$, and using integration by parts as before we easily deduce that when $m \geq 4$, $\nabla^{g} T$ is $L^1$ and coincides with the covariant derivative computed on $M \setminus S$. This prove that $T$ is covariantly divergence-free on $M$ in the sense of distributions, and hence $\phi$ is stationary.
\end{proof}